\documentclass[12pt]{amsart}
\usepackage{amsfonts}
\usepackage{amsmath}
\usepackage{amsxtra}
\usepackage{amssymb,latexsym}
\usepackage[mathcal]{eucal}
\usepackage{amscd}
\usepackage{tikz-cd}
\usepackage{hyperref}

\makeindex

\newtheorem{theo}{{\bfseries Theorem}}[section]
\newtheorem{prop}[theo]{{\bfseries Proposition}}
\newtheorem{lem}[theo]{{\bfseries Lemma}}
\newtheorem{cor}[theo]{{\bfseries Corollary}}
\newtheorem{df}[theo]{{\bfseries Definition}}

\newtheorem{ex}{{\bfseries Example}}

\newtheorem{note}[theo]{{\bfseries Notice}}

\def \Z {\mathbb Z}
\def \R {\mathbb R}
\def \Q {\mathbb Q}
\def \A {\mathcal A}

\def \CC {\mathcal C}
\def \E {\mathcal E}

\def \H {\mathcal H}

\def \I {\mathcal I}

\def \K {\mathcal K}
\def \NN {\mathcal N}

\def \S {\mathcal S}

\def \G {\mathcal G}
\def \O {\mathcal O}

\def \xx {{\mathbf x }}
\def \yy {{\mathbf y }}
\def \zz {{\mathbf z }}
\def \a {\alpha }

\def \ep {\epsilon}
\def \om {\omega}

\def \d {\delta}

\def \r {\rho}
\def \s {\sigma}
\def \t {\tau}

\usepackage{amssymb,latexsym}
\usepackage[mathcal]{eucal}
\usepackage{amscd}
\usepackage{amsmath}
\usepackage{graphicx}
\usepackage{amsfonts}
\usepackage{amscd}

\numberwithin{equation}{section}
\begin{document}

\begin{titlepage}
\large
\title[Discrete, Continuous and Hybrid Systems] {\bfseries Dynamical Systems:\\ Discrete, Continuous and Hybrid}
\author{Ethan Akin}
 \vspace{.7cm}

\address{Mathematics Department \\
    The City College \\ 137 Street and Convent Avenue \\
       New York City, NY 10031, USA     }
\email{ethanakin@earthlink.net}

\date{November, 2022}

\begin{abstract} The dynamics by iteration of a function on a compact metric space, sometimes called a cascade, can
be extended to the dynamics of a closed relation on such a space. Here we apply this relation dynamics to study semiflows (and their relation
extension) as well as hybrid dynamical systems which combine both continuous time and discrete time dynamics.  In a unified way we describe the attractor-
repeller structure, Conley's chain recurrence relation and the construction of Lyapunov functions for all of these systems.

\end{abstract}



\end{titlepage}
\maketitle

\tableofcontents
\newpage

\section{ \textbf{Introduction}}\vspace{.5cm}

These notes extend \cite{A93} which described the dynamics of a closed relation on a compact metric space. The goal is to provide a tool for the
study of hybrid systems on such spaces.\vspace{.25cm}

Section \ref{relation} \textbf{Closed Relation Dynamics} begins by reviewing and extending somewhat results from \cite{A93}, in particular
describing the attractor structure, the solution paths and the construction of Lyapunov functions for the discrete dynamical system associated with a
closed relation on a compact metric space.
When we turn to restrictions to a closed subset we encounter for a subset $C$ a
property which some authors call invariance, i.e. through each point of $C$ there exists a bi-infinite solution path which remains in $C$. When
the relation is a map, this is exactly invariance, but in general it is a somewhat different property and which we call viability. For a closed set $C$
we denote by $C_{\pm}$ the maximum viable subset of $C$, which is itself closed. If $C_{\pm}$ is contained in the interior of $C$,  then we call
$C_{\pm}$ an isolated viable set and $C$ an isolating neighborhood. We describe the construction of the so-called index pairs associated with an
isolated viable set.

Of special interest is the subsection on Anomalous Perturbations. If $C$ is an isolating neighborhood and we perturb the closed relation in a small
enough fashion, then $C$ remains an isolating neighborhood with respect to the new relation.  However, we show that in a very broad class of cases,
the viable subset can be eliminated.  That is, with respect to the new relation $C_{\pm} = \emptyset$. This elimination is not blocked even when the
topology of the index pair is quite non-trivial.

Finally, we compare the dynamics of the relation with that of the maps on the associated solution spaces.\vspace{.25cm}

Section \ref{semiflow} \textbf{Semiflow Relations} relates the relation dynamics to the dynamics of semiflow relations.  Semiflows were considered
in \cite{A93} but here we consider the relation version of a semiflow. This extends related work in \cite{BK}.
From the semiflow relation one is able to construct certain closed relations so that
the semiflow dynamics can be described using the relation dynamics. This allows us to extend the results of the previous section concerning
attractors, Lyapunov function, viable subsets and index pairs to the semiflow relation context.  \vspace{.25cm}

Section \ref{hybrid} \textbf{Hybrid Systems} extends the results to hybrid systems which combine the continuous time dynamics of a semiflow relation
with the discrete time dynamics of a closed relation. \vspace{.25cm}

NOTATION: With $\Z$ and $\R$ the integers and real numbers, respectively, and $\Z_+, \R_+$ the sets of the non-negative elements of each, we attach
points at infinity defining $\R^* = \{-\infty \} \cup \R \cup \{ \infty \}, \Z^* = \{-\infty \} \cup \Z \cup \{ \infty \} $ and
$\R_+^* =  \R_+ \cup \{ \infty \}, \Z_+^* = \Z_+ \cup \{ \infty \} $,

For a subset $C$ of a space $X$ we let $\overline{C} $ and $C^{\circ}$ denote the closure and interior, respectively, and
let $\partial C = \overline{C} \setminus C^{\circ}$, the boundary of $C$.

Note that for closed sets $A$ and $B$,  $\partial (A \cap B), \partial (A \cup B) \subset \partial A \cup \partial B$
because $(A \cap B)^{\circ} = A^{\circ} \cap B^{\circ}$
and so $(A \cap B) \setminus (A \cap B)^{\circ} \subset (A \setminus A^{\circ}) \cup (B \setminus B^{\circ})$ and
$(A \cup B) \setminus (A \cup B)^{\circ} \subset (A \setminus A^{\circ}) \cup (B \setminus B^{\circ})$.

We will call a sequence of sets $\{ A_n \}$ \emph{decreasing} if $A_{n+1} \subset A_n$ for all $n$.  If, in addition, $A_{n+1} \not= A_n$ we will
call the sequence \emph{strictly decreasing}. Similarly for \emph{increasing} and \emph{strictly increasing} sequences.\vspace{.5cm}

\textbf{Acknowledgements:}  This work was  a contribution to a project organized by Professor Ricardo Sanfelice to prepare a grant
proposal on applications of hybrid systems. I am grateful to have had the opportunity to work with these people, and I appreciate the comments and criticisms
of earlier versions of the work that several of them gave me.   In addition to myself, the group consisted of \vspace{.25cm}

Ricardo Sanfelice, Department of Electrical and Computer Engineering, University of California at Santa Cruz; \vspace{.25cm}

Rafal Goebel, Mathematics Department, Loyola University Chicago; \vspace{.2cm}

Miroslav Kramar, Mathematics Department,  University of Oklahoma; \vspace{.25cm}

Sanjit Seshia, Department of Electrical Engineering and Computer Science, University of California at Berkeley; \vspace{.25cm}

Andrew Teel, Department of Electrical and Computer Engineering, University of California at Santa Barbara.

\vspace{1cm}

\section{ \textbf{Closed Relation Dynamics}}\label{relation}\vspace{.5cm}

We follow the relation approach from \cite{A93}. Our spaces are all compact metric spaces.

A function $ f : X \to Y$ is usually described as a rule associating to every point $x$ in $X$
a unique point $y = f(x)$ in $Y$.  In set theory the function $f$ is defined to be the set of ordered
pairs $\{ (x,f(x)) : x \in X \}$. Thus, the function $f$ is a subset of the product $X \times Y$. It is
what is sometimes called the graph of the function. We will use  this language  so that, for example, the
identity map $1_X$ on $X$ is the diagonal subset $\{ (x,x) : x \in X \}$.  The notation is extended by
defining a \emph{relation from $X$ to $Y$}\index{relation}, written $F: X \to Y$, to be an arbitrary subset of $X \times Y$.
Then $F(x) = \{ y : (x,y) \in F \}.$  Thus, a relation is a function exactly when the set $F(x)$ is a
singleton set for every $x \in X$. When the relation $F$ is a function we use the notation $F(x)$ for the singleton set and
for the point it contains, allowing context to determine the reference.

As they are arbitrary subsets of $X \times Y$ we can perform set  operations like union, intersection, closure
and interior on relations.  In addition, for $F : X \to Y$ we define the \emph{inverse}\index{relation!inverse}%\index{$F^{-1}$}
 $\ F^{-1} : Y \to X$ by
\begin{equation}\label{eqrel01}
F^{-1} \  =_{def} \  \{ (y,x) : (x,y) \in F \}. \hspace{4cm}
\end{equation}

If $A \subseteq X$, then its \emph{image}\index{relation!image} is
\begin{equation}\label{eqrel02}
\begin{split}
F(A) \  =_{def} \  \{ y : (x,y) \in F \ \mbox{for some } \ x \in A \} \\
= \  \bigcup_{x \in A} F(x) \  = \  \pi_2 ((A \times Y) \cap F),\hspace{2cm}
\end{split}
\end{equation}
where $\pi_2 : X \times Y \to Y$ is the projection to the second coordinate. Clearly, for any collection $\{ A_i \}$ of subsets
of $X$, $F(\bigcup_i  A_i) = \bigcup_i  F(A_i)$.

The \emph{domain} of a relation $F : X \to Y$ is  \index{relation!domain} %\index{ $Dom(F)$ }
\begin{equation}\label{eqrel02a}
Dom(F) \ =_{def} \ \{ x : F(x) \not= \emptyset \} \ = \ F^{-1}(Y).
\end{equation}

If $G : Y \to Z$ is another relation, then the \emph{composition}\index{relation!composition} $G \circ F : X \to Z$ is the relation given by
\begin{equation}\label{eqrel03}
\begin{split}
G \circ F \  =_{def} \  \{ (x,z) : \mbox{ there exists} \ y \in Y  \hspace{1cm} \\ \mbox{such that} \
(x,y) \in F \ \mbox{and} \ (y,z) \in G \} \hspace{2cm} \\
= \  \pi_{13}((X \times G) \cap (F \times Z)), \hspace{3cm}
\end{split}
\end{equation}
where $\pi_{13} : X \times Y \times Z \to X \times Z$ is the projection map.%\index{ $G\circ F$ }
This generalizes
composition of functions and, as with functions, composition is associative. Clearly,
\begin{equation}\label{eqrel03a}
(G \circ F)^{-1} \ = \ F^{-1} \circ G^{-1}.
\end{equation}

If $Y = X$, so that $F : X \to X$, then we call $F$ a \emph{relation on $X$}\index{relation on $X$}. For a positive integer
$n$ we define $F^n$ to be the $n$-fold composition of $F$ with $F^0 =_{def} 1_X$ and
\begin{equation}\label{eqrel03b}
F^{-n} \ =_{def} \ (F^{-1})^n \ = \ (F^n)^{-1}.
\end{equation}
 This is well-defined because composition is associative.
Clearly, $F^m \circ F^n = F^{m+n}$ when $m$ and $n$ have the same sign, i.e. when $mn \geq 0$.
On the other hand, the equations $F \circ F^{-1} = F^{-1} \circ F = 1_X = F^0$ all hold if and only
if the relation $F$ is a bijective function.

If $F$ is a relation on $X$, then a subset $A \subset X$ is called $F$
\emph{+ invariant}\index{subset!invariant}\index{subset!+ invariant} (or \emph{invariant})
when $F(A) \subset A$ (resp. $F(A) = A$) (We will simply write + invariant or invariant when $F$ is understood).

A relation $F$ on $X$ is \emph{reflexive}\index{relation!reflexive} when $1_X \subset F$, \emph{symmetric}\index{relation!symmetric}
 when $F^{-1} = F$ and \emph{transitive}\index{relation!transitive} when $F \circ F \subset F$.
For example, with metric $d$ on $X$ and $\ep > 0$,
\begin{equation}\label{eqrel04}\begin{split}
V_{\ep} \  =_{def} \  \{ (x,y) : d(x,y) < \ep \}, \\
\bar V_{\ep} \  =_{def} \  \{ (x,y) : d(x,y) \le \ep \} \
\end{split} \end{equation}
are reflexive, symmetric relations with $V_{\ep}(x)$ the open ball with center $x$ and radius $\ep$.

For a relation $F$ on $X$, we define the \emph{orbit relation}\index{relation!orbit}\index{orbit relation}%\index{ $\O F $ }
\begin{equation}\label{eqrel05}
\O F \  =_{def} \  \bigcup_{n = 1}^{\infty} F^n. \hspace{2cm}
\end{equation}

Thus, $F$ is transitive if and only if $F = \O F$ and for any relation $F$ on $X$, $\O F$ is the smallest transitive relation which contains
$F$.

We call $F$ a \emph{closed relation}\index{relation!closed} when it is a closed
subset of $X \times Y$. Clearly, the inverse of a closed relation is closed and by compactness,
the composition of closed relations is closed.  If $A$ is a closed subset of $X$ and $F$ is a
closed relation, then the image $F(A)$ is a closed subset of $Y$. In particular, the domain of a closed relation is closed.

For relations being closed is analogous to being continuous for functions.  In fact, a function is continuous
if and only if, regarded as a relation, it is closed.  This is another application of compactness.

Define for $F : X \to Y$ and $V \subset Y$, %\index{$F^*(U)$}
\begin{equation}\label{eqrel06}
F^*(V) \  =_{def} \  \{ x : F(x) \subset V \} \ = \ X \setminus F^{-1}(Y \setminus V)
\end{equation}
If $V$ is an open subset of $Y$, then for a closed relation $F : X \to Y$ the set $F^*(V)$ is an open subset of $X$.

\begin{prop}\label{relprop00aa} Let $F: X \to Y, G : Y \to Z$ be relations with $A \subset X, B \subset Y, C \subset Z$ and $\{ B_i \}$ a collection of subsets of $Y$.
\begin{align}\label{eqrel06aa} \begin{split}
Y \setminus F(A) \ & = \ (F^{-1})^*(X \setminus A). \\
F^*(G^*(C)) \ &= \ (G \circ F)^*(C).  \\
F(A) \cap B \ = \ \emptyset \qquad &\Longleftrightarrow \qquad A \cap F^{-1}(B) \ = \ \emptyset. \\
\Longleftrightarrow \qquad F(A)\ \subset \ &X \setminus B  \qquad \Longleftrightarrow \qquad A \ \subset \ F^*(X \setminus B). \\
F^*(\bigcap_i \ B_i) \ & = \ \bigcap_i \ F^*(B_i).
\end{split}\end{align}\end{prop}

\begin{proof} $x \in F^*(G^*(C))$ if and only if $F(x) \subset G^*(C)$ if and only if $G(F(x)) \subset C$.

$F(A) \cap B  =  \emptyset$ when for all $(x,y) \in F, \ x \in A \Rightarrow y \not\in B$ and so if and only if $F(A)\ \subset X \setminus B$ if and only if
$A \ \subset \ F^*(X \setminus B)$ and, contrapositively $y \in B  \Rightarrow x \not\in A$, i.e.
$A \cap F^{-1}(B) = \emptyset.$

Finally, $F(x) \subset \bigcap_i  B_i$ if and only if $F(x) \subset B_i$ for all $i$.

\end{proof}

We call $F: X \to Y$ a \emph{surjective relation}\index{relation!surjective}  on $X$ when
$F^{-1}(Y) = Dom(F) = X$ and $F(X) = Dom(F^{-1}) = Y$ and so, of course,
$F^{-1}$ is surjective as well.

\begin{df}\label{reldf00ab} A closed,  surjective relation $F: X \to Y$  is called
\emph{irreducible}\index{relation!irreducible} when it satisfies the following
two conditions.
\begin{itemize}
\item For every closed subset $A$ of $X$, $F(A) = Y$ implies $A = X$, or, equivalently, for every open subset $V$ of $Y$, $V \not= \emptyset$ implies
$F^*(V) \not= \emptyset$.
\item For every closed subset $B$ of $Y$, $F^{-1}(B) = X$ implies $B = Y$, or, equivalently, for every open subset $U$ of $X$, $U \not= \emptyset$ implies
$(F^{-1})^*(U) \not= \emptyset$.
\end{itemize} \end{df}

The equivalence in the first statement follows by using $V = Y \setminus F(A)$ one way and $A = X \setminus F^*(V) $ the other and applying the first equation of
 (\ref{eqrel06aa}).

 \begin{prop}\label{relprop00ac} Let $H : X \to X, \ F: X \to Y, \ G : Y \to Z$ be irreducible relations  with $U \subset X, V \subset Y$ nonempty open subsets.
 \begin{itemize}
 \item[(a)] The composition $G \circ F : X \to Z$ is irreducible..

 \item[(b)] If  $1_X \subset H$, then $H^*(U)$ is a dense open subset of $U$.  In particular,
 $(F^{-1} \circ F)^*(U)$ is a dense open subset of $U$ and $(F \circ F^{-1})^*(V)$ is a dense open subset of $V$.

 \item[(c)]  If $U$ is dense in $X$, then $(F^{-1})^*(U)$ is dense in $Y$.  If $V$ is dense in $Y$, then $F^*(V)$ is dense in $X$.
\end{itemize} \end{prop}

\begin{proof} In each case, we need only provide the proofs for $F$ or $G \circ F$.  For the other direction the results follow by using $F^{-1}$ or
$(G \circ F)^{-1} = F^{-1} \circ G^{-1}$.

(a)  If $A$ is closed in $X$ and $G(F(A)) = Z$, then since $F(A)$ is closed in $Y$ and $G$ is irreducible, $F(A) = Y$ and so because $F$ is irreducible $A = X$. Thus, $G \circ F$ is irreducible.

(b) $x \in H^*(U)$ implies $x \in H(x) \subset U$, i.e. $H^*(U) \subset U$.  If $U'$ is an arbitrary nonempty open subset of $U$,
 then $H^*(U')$ is a nonempty subset of $U' \cap H^*(U)$. Hence, $H^*(U)$ is dense in $U$.

 Because $F$ is surjective, $1_X \subset F^{-1} \circ F$ and $1_Y \subset F \circ F^{-1}$.

 (c) If $V$ is an arbitrary nonempty open subset of $Y$, then $U' = F^*(V) \cap U$ is a nonempty open subset of $X$ because $U$ is open and dense.
 So $(F^{-1})^*(U')$ is a nonempty open subset of $ (F^{-1})^*(F^*(V)) \cap (F^{-1})^*(U) \subset V \cap (F^{-1})^*(U)$, by (b). So $(F^{-1})^*(U)$ is dense in $Y$.

 \end{proof}

  \begin{theo}\label{reltheo00ad}  Let $F: X \to Y$ be a closed surjective relation.
   \begin{itemize}
 \item[(a)] If $\ F^{-1}(\{ y \in Y : F^{-1}(y) $ is a singleton set$\})$ is dense in $X$ and
 $\ F(\{ x \in X : F(x) $ is a singleton set$\})$ is dense in $Y$, then
 $F$ is irreducible.

  \item[(b)] Define:
  \begin{align}\label{eqrel06ab} \begin{split}
  X_0 \ = \  \{ x : \ F^{-1}(F(x))& \ = \ \{ x \} \}, \\
   Y_0 \ = \  \{ x : \ F(F^{-1}(y))& \ = \ \{ y \} \}, \\
   X_1 \ = \ X_0 \cap F^*(Y_0), \qquad  &Y_1 \ = \ Y_0 \cap (F^{-1})^*(X_0).
\end{split}\end{align}
If   $F$ is irreducible, then $X_1, X_0$ are dense $G_{\d}$ subsets of $X$ \\ and $Y_1, Y_0$ are dense $G_{\d}$ subsets of $Y$ with
 \begin{align}\label{eqrel06ac}\begin{split}
 X_1 \ = \ \{ &x \in X : F(x) \  \text{ is a singleton set, contained in } Y_1 \ \}, \\
Y_1 \ = \ \{ &y \in Y : F^{-1}(y) \  \text{ is a singleton set, contained in } X_1 \ \}.
\end{split}\end{align}
The restriction $F \cap (X_1 \times Y_1)$ is a homeomorphism from \\ $X_1$ to $Y_1$.

 \item[(c)] Assume $Y = X$ so that $F$ is a closed relation on $X$. If   $F$ is irreducible, then there exists $W$ a dense $G_{\d}$ subset of $X$
 so that the restriction $f = F_W = F \cap (W \times W)$ is a homeomorphism on $W$ such that for $x \in W$, $F(x) = \{ f(x) \}$ and $F^{-1}(x) = \{ f^{-1}(x) \}$.
In particular, $W$ is invariant for $F$ and $F^{-1}$.
\end{itemize}
\end{theo}

\begin{proof} (a) If $A \subset X$ satisfies $F(A) = Y$, then $A \supset F^{-1}(\{ y \in Y : F^{-1}(y) $ is a singleton set$\})$. So if $A$ is
closed and the latter set is dense
we obtain $A = X$.  Applying this to $F^{-1}$ we see that $F$ is irreducible. Notice that if $F^{-1}(F(x)) = \{ x \}$ then
$ x \in  F^{-1}(\{ y \in Y : F^{-1}(y) $ is a singleton set$\})$.

(b) Let $\A_n$ be a cover of $X_0$ by open sets of diameter less than $1/n$.  From Proposition \ref{relprop00ac}(b) it follows that
$\bigcup \ \{ (F^{-1} \circ F)^*(U) : \ U \in \A_n \}$ is a dense open subset of $X$.  When we intersect over $n$, the Baire Category Theorem implies
 that $X_0$ is a dense, $G_{\d}$ subset of $X$. Similarly, $Y_0$ is a dense, $G_{\d}$ subset of $Y$. From
 Proposition \ref{relprop00ac}(c), Proposition \ref{relprop00aa} and the Baire Category Theorem again
 we obtain that $X_1$ and $Y_1$ are dense, $G_{\d}$ subsets of $X$ and $Y$, respectively.

 Now suppose that $(x,y) \in F$ so that $y \in F(x)$ and $x \in F^{-1}(y)$. If $x \in X_1$,
 then $F(x) \subset Y_0$ implies that $F(F^{-1}(y)) = \{ y \}$ and $x \in F^{-1}(y)$
 implies $F(x) \subset F(F^{-1}(y)) = \{ y \}$. That is, $F(x) $ is the singleton set $\{ y \}$ and $y \in Y_0$. Similarly, $y \in F(x)$ implies that
 $F^{-1}(y) \subset F^{-1}(F(x)) = \{ x \}$ and so $F^{-1}(y)$ is the singleton set $\{ x \}$. Since $x \in X_0$ it follows that $y \in Y_1$.
 Similarly, $y \in Y_1$ implies that $x \in X_1$. Thus the restriction $F \cap (X_1 \times Y_1)$ is a bijection from $X_1$ to $Y_1$.

 For continuity, let $\{ (x_n,y_n) \in F \}$ be a sequence with $x_n \in X_1$ converging to a point $x \in X_1$. If $y$ is any limit point of the
 $\{ y_n \}$ sequence, then $(x,y) \in F$. Since $x \in X_1$, $y$ is the unique point of $Y_1$ such that $(x,y) \in F$. This shows that
 $F \cap (X_1 \times Y_1)$ is a continuous map from $X_1$ to $Y_1$. Applying this to $F^{-1}$ we see that the restriction is a homeomorphism.

(c) From (b) it follows that there exist dense $G_{\d}$ subsets $X_1, Y_1$ of $X$ and a homeomorphism
$h : X_1 \to Y_1$ with $F(x) = \{ h(x) \}$ for $x \in X_1$ and $F^{-1}(x) = \{ h^{-1}(x) \}$
for $x \in Y_1$. Let $W_0$ be the dense $G_{\d}$ subset $W_0 = X_1 \cap Y_1$ so that
$h(W_0)$ is a dense $G_{\d}$ subset $Y_1$ and $h^{-1}(W_0)$  is a dense $G_{\d}$ subset $X_1$.
Inductively, define $W_n = W_{n-1} \cap h(W_{n-1}) \cap h^{-1}(W_{n-1})$
a dense $G_{\d}$ subset of $W_{n-1}$. Finally, let $W = \bigcap \{ W_n : n \in \Z_+ \}$ and let $f$ be the restriction of $h$ to $W$.

 \end{proof}

\textbf{Remark : } It should be noted that the closure of $F \cap (X_1 \times Y_1)$ might be a proper subset of $F$ in which case the projection map
$\pi_1 : F \to X_1$ is not an irreducible map.  For example, let $f$ be a homeomorphism on an infinite space $X$ without isolated points and let
$A$ be a nonempty closed, nowhere dense subset and $B = f(A)$. Let
$F = f \cup [A \times B]$. Because $f$ restricts to a homeomorphism
from $X \setminus A$ to $X \setminus B$ it follows that $F$ is irreducible on $X$ with $X \setminus A = X_0$ and $X \setminus B = Y_0$.

What is true in general is that the closure
$\overline{ F \cap (X_1 \times Y_1)}$ is the unique minimal element among the closed subsets $F_1 \subset F$ such that $F_1$ is a surjective relation on $X$.
\vspace{.5cm}

\begin{prop}\label{relprop00b} Let $\{ F_n : X \to Y \}$ and $\{ G_n : Y \to Z \}$ be decreasing sequences of closed relations with intersections
$F$ and $G$, respectively.  Let $\{ A_n \}$ be a decreasing sequence of closed subsets of $X$ with intersection $A$.
\begin{equation}\label{eqrel06a}
G \circ F \ = \ \bigcap_n \{ G_n \circ F_n \} \quad \text{and} \quad F(A) \ = \  \bigcap_n \{ F_n(A_n) \}.
\end{equation}
\end{prop}

\begin{proof} $G \circ F = \pi_{13}((F \times Z) \cap (X \times G)) = \bigcap_n \pi_{13}((F_n \times Z) \cap (X \times G_n))$
by compactness. Similarly use $F(A) = \pi_2 ( F \cap (A \times Y))$.

\end{proof}

\begin{cor}\label{relcor00a} If $F$ is a closed relation and $A$ is a  closed + invariant subset, then $A_{\infty} = \bigcap_{n = 1}^{\infty} \ F^n(A)$
is an invariant subset of $A$ which contains any other invariant subset of $A$. If $F^n(A) \not= \emptyset$ for all $n$, then $A_{\infty}$
is nonempty. In particular, if $Dom(F) = X$ and $A$ is nonempty, then $A_{\infty}$ is nonempty. \end{cor}

\begin{proof}  Since $F(A) \subset A$, the sequence $\{ F^n(A) \}$ is decreasing sequence of + invariant subsets.
Hence, the intersection $A_{\infty}$ is + invariant.  If $F^n(A) = \emptyset $ for some $n$, then
$F^m(A) = \emptyset$ for all $m \ge n$ and $A_{\infty} = \emptyset$ is the only invariant subset of $A$.  So we may assume $F^n(A) \not= \emptyset$
for all $n$.

Let $y \in A_{\infty}$  $y \in F^{n+1}(A)$  implies that $F^{-1}(y) \cap F^n(A)$ is nonempty. That is $\{ F^{-1}(y) \cap F^n(A) \}$
is a decreasing sequence of nonempty compact sets. Hence, the intersection $ F^{-1}(y) \cap A_{\infty}$ is nonempty.
Hence, $A_{\infty}$ is invariant. If $C$ is an invariant
subset of $A$, then, inductively, $C \subset F^n(A)$ for all $n$ and so $C \subset A_{\infty}$.

\end{proof}

Even when  $F$ is a closed relation on $X$, the transitive orbit relation $\O F$ need not be closed. Auslander's \emph{prolongation relation}\index{relation!prolongation}
\begin{equation}\label{eqrel07a}
\NN F \ =_{def} \  \overline{\O F} \hspace{1cm}
\end{equation}
 is closed, but need not be transitive.
We let $\G F$ denote the smallest closed, transitive relation %\index{ $\G F$ }
which
contains $F$.  We call it the \emph{infinite prolongation relation}\index{relation!infinite prolongation}.
 Conley's \emph{chain relation}\index{relation!chain} is
\begin{equation}\label{eqrel07}
\CC F \  =_{def} \  \bigcap_{\ep > 0} \ \O (V_{\ep} \circ F).
\end{equation}
As it is the intersection of transitive relations,%\index{ $\CC F$ }
 \ $\CC F$ is transitive.
It is closed as well (see \cite{A93} Proposition 1.8) and so it contains $\G F$. However, the containment is usually strict. For example,
$1_X$ is a closed, equivalence relation and so $\G 1_X = 1_X$.  On the other hand, we have:

\begin{prop}\label{relprop00} Let  $X$ be a connected space.  If $F$ is a reflexive relation on $X$, i.e. $1_X \subset F$, then
$\CC F = X \times X$. \end{prop}

\begin{proof} If $\ep > 0$, then $\O V_{\ep}$ is an open equivalence relation on $X$. As its equivalence classes are clopen and $X$ is
connected, $X$ is an equivalence class and so $\O V_{\ep} = X \times X$. Hence, $\CC 1_X = X \times X$.

\end{proof}

Thus, we obtain a tower of relations:
\begin{equation}\label{eqrel08}
F \ \subset \ \O F \ \subset \ \NN F \ \subset \ \G F \ \subset \ \CC F.
\end{equation}
\vspace{.25cm}

\begin{prop}\label{relprop01} Let $F, F_1$ be relations on $X$.

(a) For $\A = \O, \NN, \G, \CC$
\begin{equation}\label{eqrel09a}
F_1 \subset F \quad \Longrightarrow \quad  \A F_1 \subset \A F. \hspace{2cm}
\end{equation}

(b) For $\A = \O, \NN, \G, \CC$
\begin{equation}\label{eqrel09b}
\A (F^{-1}) \ = \ (\A F)^{-1} \hspace{3cm}
\end{equation}
and so we can omit the parentheses.

(c) For $\A = \O, \G, \CC$
\begin{equation}\label{eqrel09c}
F \cup ((\A F) \circ F) \ = \ \A F \ = F \cup (F \circ (\A F))
\end{equation}
\end{prop}

\begin{proof} See \cite{A93} Proposition 1.11.

\end{proof}

From (\ref{eqrel09c}) it follows that for $x \in X$
 \begin{equation}\label{eqrel09d}
\CC F(x)  \ \not= \ \emptyset  \quad \Longrightarrow \ F(x)  \ \not= \ \emptyset.
\end{equation}

For a relation $F$ on $X$ we define the \emph{cyclic set}\index{cyclic set},%\index{$|F|$}
\begin{equation}\label{eqrel10}
|F| \quad =_{def} \quad \{ x : (x,x) \in F \} \ = \ \pi_1(1_X \cap F).
\end{equation}
Thus, $|F|$ is a closed set when $F$ is a closed relation.

Following the nomenclature for the case when $F$ is a continuous map, for a closed relation $F$ on $X$ we call
\begin{itemize}
\item $|F|$, the set of \emph{fixed points} of $F$;
\item $|\O F|$, the set of \emph{periodic points} of $F$;
\item $|\NN F|$, the set of \emph{non-wandering points} of $F$;
\item $|\G F|$, the set of \emph{generalized non-wandering points} of $F$;
\item $|\CC F|$, the set of \emph{chain recurrent points} of $F$;
\end{itemize}

A continuous function $L : X \to \R$ is called a \emph{Lyapunov function}\index{Lyapunov function} for a relation $F$ on $X$ when
\begin{equation}\label{eqrel11}
(x,y) \in F \quad \Longrightarrow \quad L(x) \leq L(y),
\end{equation}
or, equivalently, with $\le \  =_{def} \  \{ (t,s) \in \R \times \R : t \le s \}$
\begin{equation}\label{eqrel12}
F \ \subset \ \ \le_L \quad  =_{def} \quad  (L \times L)^{-1}(\le).
\end{equation}
Since $\le_L$ is a closed transitive relation, a Lyapunov function for
$F$ is automatically a Lyapunov function for $\G F$.

Note that we follow the biologist view of a Lyapunov function, like fitness or entropy, to be increasing on orbits, rather than the physicist
view of a function, like free energy, which is decreasing on orbits.

A point $x \in X$ is a \emph{regular point}\index{Lyapunov function!regular point} for the Lyapunov function $L$ when
\begin{equation}\label{eqrel13}
\sup L|F^{-1}(x) \ < \ L(x) \ < \ \inf L|F(x). \hspace{1cm}
\end{equation}
Otherwise $x$ is a \emph{critical point}\index{Lyapunov function!critical point} for $L$. The set $|L|$
%\index{$|L|$}
of critical points is
closed because it is the domain of the closed relation $(F \cup F^{-1}) \cap ((L \times L)^{-1}(1_{\R}))$.

In Proposition 2.9 of \cite{A93} it is shown that a regular point satisfies that apparently stronger condition
\begin{equation}\label{eqrel14}
\sup L|\G F^{-1}(x) \ < \ L(x) \ < \ \inf L|\G F(x), \hspace{1cm}
\end{equation}
from which it follows that
\begin{equation}\label{eqrel15}
|\G F| \ \subset \ |L|. \hspace{3cm}
\end{equation}

If $F$ is a closed, transitive relation, then on the cyclic set $|F|$, the relation $F \cap F^{-1}$ is a closed equivalence
relation and for each $x \in |F|$ the equivalence class $F(x) \cap F^{-1}(x)$ is closed. On such an equivalence class, any Lyapunov function is
constant. For a closed relation $F$ the $\CC F \cap \CC F^{-1}$ classes of the set of chain recurrent points, $|\CC F|$, are called
the \emph{chain components}\index{chain components} of $F$ (although, following Smale, they are called \emph{basic sets} in \cite{A93}).

We sketch the  Lyapunov function  results.

We write $A \subset \subset B$ %\index{ $ \subset \subset $ }
when the closure of $A$ is contained in the
interior of $B$, i.e. $\overline{A} \subset B^{\circ}$.

\begin{prop}\label{relprop03a} Let $F$ be a closed, transitive relation on $X$. If $A$ is an $F$ + invariant subset and $A \subset \subset B$,
then there exists a closed $F$ + invariant subset $P$ such that  $A \subset \subset P \subset \subset B$ and an
open $F$ + invariant subset $Q$ with $A \subset \subset Q \subset \subset P$. \end{prop}

\begin{proof} Replacing $F$ by $F \cup 1_X$ we may assume that the closed relation $F$ is reflexive as well as transitive. Replacing $A$ by $\overline{A}$
we may assume that $A$ is closed.

Let $\{ U_n \}$ be a decreasing sequence of closed neighborhoods of $A$ with intersection $A$. Then $A = F(A) = \bigcap_n \{ F(U_n) \}$
by Proposition \ref{relprop00b}. So for large enough $n$, $P = F(U_n) \subset B^{\circ}$. By transitivity of $F$, $P$ is $F$ + invariant.
Because $F$ is reflexive, $U_n \subset F(U_n)$ and so $A \subset P^{\circ}$.

Let $Q = P^{\circ} \cap F^*(P^{\circ})$ as in (\ref{eqrel06}). Since $A$ is  + invariant it is a subset of the open set $Q$.
 If  $x \in Q$ and $y \in F(x)$, $y \in F(x) \subset P^{\circ}$. By transitivity  $F(y) \subset F(x) \subset P^{\circ}$. Thus, $y \in Q$.
 That is, $Q$ is $F$ + invariant.  Since $Q$ is open and $A$ is closed, $A \subset \subset Q$.  Since $Q \subset P \subset \subset B$,
 it follows that $Q  \subset \subset B$.

 \end{proof}

\begin{lem}\label{rellem03aa} Let $F$ be a closed, transitive relation on $X$. Assume that $A, B$ are disjoint, closed subsets of $X$ with
$A$  + invariant for $F$ and $B$  + invariant for $F^{-1}$. There exists
a Lyapunov function $L : X \to [0,1]$ such that $A = L^{-1}(1)$ and $B = L^{-1}(0)$.\end{lem}

\begin{proof} This is an extension of  \cite{A93} Lemma 2.10 which
 mimics the proof of the Urysohn Lemma.

 Let $\{ 0, 1, r_2,r_3, \dots \}$ be a counting of the rationals in  $[0,1]$. Let $U_0 = X, U_1 = A$. Inductively, we can use
 Proposition \ref{relprop03a} to choose a sequence of closed $F$ + invariant subsets of $X$ such that $r_n < r_m$ implies
 $U_m \subset \subset U_n$. We can choose them so that $\bigcap_{n \ge 2} U_n = A$ and $\bigcup_{n \ge 2} U_n = X \setminus B$. Define
 $L(x)$ by a Dedekind cut:
 \begin{equation}\label{eqrellyap01}
 L(x) \ = \ \sup \{ r_n : x \in U_n \} \ = \ \inf \{ r_m : x \not\in U_m \}
 \end{equation}

 Continuity follows as in Urysohn's Lemma, and $L$ is a Lyapunov function because the $U_n$'s are + invariant. Clearly, $x \in A$ if and only
 if $x \in U_n$ for all $n$ and $x \in B$ if and only if $x \not\in U_n$ unless $n = 0$. Hence, $A = L^{-1}(1)$ and $B = L^{-1}(0)$.

\end{proof}

The major Lyapunov function result is a sharpening of Lemma \ref{rellem03aa}.

\begin{theo}\label{reltheo02} Let $F$ be a closed, transitive relation on $X$.

(a) Assume that $A, B$ are disjoint, closed subsets of $X$ with
$A$  + invariant for $F$ and $B$  + invariant for $F^{-1}$.

 There exists a continuous function $L: X \to [0,1]$ with $B = L^{-1}(0)$, $A = L^{-1}(1)$ and
such that if $(x,y) \in F$, then $L(y) \ge L(x)$ with equality only when
 \begin{equation}\label{eqrellyap02}
x,y \in A, \quad x,y \in B \quad \text{or} \ \ (y,x) \in F
\end{equation}
 In particular, $L$ is a Lyapunov function
with $|F| \subset |L| \subset |F| \cup A \cup B$.

(b) There exists a continuous function $L: X \to [0,1]$
such that if $(x,y) \in F$, then $L(y) \ge L(x)$ with equality only when, in addition, $(y,x) \in F$. In particular, $L$ is a Lyapunov function
with $|L| = |F|$. \end{theo}

\begin{proof} This is an extension of \cite{A93} Theorem 2.12.

For any pair, $x, y \in X \setminus (A \cup B)$, let  $A_y = A \cup \{ y \} \cup F(y)$ and $B_x = B \cup \{ x \} \cup F^{-1}(x)$.
 Because $F$ is transitive, $A_y$ is $F$
+ invariant and $B_x$ is $F^{-1}$ invariant.

Let $Q = F \setminus [F^{-1} \cup (A \times A) \cup (B \times B)] \subset X \times X$.
If  $(x,y) \in Q$, then $A_y \cap B_x = \emptyset$.

 For any $(x,y) \in Q$,
 Lemma \ref{rellem03aa} implies that there exists $L_{(x,y)} : X \to [0,1]$ a Lyapunov function with $A_y = L_{(x,y)}^{-1}(1)$ and
$B_x = L_{(x,y)}^{-1}(0)$. Since $L_{(x,y)}(x) = 0$ and $L_{(x,y)}(y) = 1$, $(x,y)$ lies in the open set $O_{(x,y)} = \{ (u,v) \in X \times X :
L_{(x,y)}(u) < L_{(x,y)}(v)$.

Because $Q$ is a subset of a compact metric space, it satisfies the Lindel\"{o}f Property and so there is a sequence
$\{ (x_n,y_n) : n \in \Z_+ \}$ in $Q$ such that $\{ O_{(x_n,y_n)} \}$ is an open cover of $Q$.

Define the Lyapunov function $L : X \to [0,1]$ by
 \begin{equation}\label{eqrellyap03}
 L^*(x) \ = \ \sum_{n= 0}^{\infty} \ 2^{-(n+1)} L_{(x_n,y_n)}(x).
 \end{equation}

 Because $L^*$ is a Lyapunov function, $(x,y) \in F$ implies $L^*(x) \le L^*(y)$.

 If $(x,y) \in Q$, then since $(x,y) \in O_{(x_n,y_n)} $ for some $n$, it follows that $L^*(x) < L^*(y)$.

For part (b) with $A = B = \emptyset$ we use $L = L^*$ which  completes the proof of part (b).

 For part (a) we note that for $(x,y) \in F$,
 if $x \in A$, then $y \in A$ and $L^*(x) = L^*(y) = 1$ and if $y \in B$, then $x \in B$ and $L^*(x) = L^*(y) = 0$.

 We need a final adjustment to cover the cases $y \in A, x \not\in A$ and $x \in B, y \not\in B$.

 For $x \in X \setminus (A \cup B)$, we apply Lemma \ref{rellem03aa} again to get Lyapunov functions $L_x^+, L_x^- : X \to [0.1]$ with
 \begin{equation}\label{eqrellyap04}
 (L_x^+)^{-1}(0) = B, (L_x^+)^{-1}(1) = A_x, \quad \text{and} \quad (L_x^-)^{-1}(0) = B_x, (L_x^-)^{-1}(1) = A.
 \end{equation}
 With $L_x = \frac{1}{2}(L_x^+ \ + \ L_x^-)$ we have a Lyapunov function with $B \subset (L_x)^{-1}(0), A \subset (L_x)^{-1}(1)$
 and $L_x(x) = \frac{1}{2}$.

 As before $x$ lies in the open set $O_x = \{ u : 0 < L_x(u) < 1\} $ and we can choose $\{ O_{x_n} : n \in \Z_+ \}$ to be an open cover of
 $X \setminus (A \cup B)$.

 Define the Lyapunov function
  \begin{equation}\label{eqrellyap05}
 L^{**}(x) \ = \ \sum_{n= 0}^{\infty} \ 2^{-(n+1)} L_{x_n}(x).
 \end{equation}

 For $x \in X \setminus (A \cup B)$, $x \in O_n$ for some $n$ and so $0 < L^{**}(x) < 1$.

 Finally, let $L = \frac{1}{2}(L^* + L^{**})$.

 For $x \in X \setminus (A \cup B)$, $0 < L(x) < 1$ and so $L(x) < L(y)$ if $y \in A$ and $L(y) < L(x)$ if $y \in B$ and finally
 $L(x) = 0$ implies $x \in B$, $L(x) = 1$ implies $x \in A$.

\end{proof}

We call a function $L$ which satisfies the conditions of  part (b) of the theorem
a \emph{complete Lyapunov function} for the closed, transitive relation $F$.

Applying this result to the  transitive relations $\G F$ and $\CC F$ for an arbitrary closed relation $F$ on $X$ we obtain

\begin{cor}\label{relcor03} Let $F$ be a closed relation on $X$. There exists a Lyapunov function
$L$ for $F$ such that $|\G F| = |L|$. There exists a Lyapunov function
$L$ for $\CC F$ such that $|\CC F| = |L|$. \end{cor}

A closed set $U$ is an \emph{inward set} for $F$ \index{inward}\index{subset!inward} for $F$ when $F(U) \subset \subset U$.
Such a set is sometimes called a \emph{trapping region}\index{trapping region}.
An inward set is + invariant and a clopen + invariant set is inward. Clearly, if $\{ U_i \}$ is a finite collection of inward sets,
then $\bigcap_i  U_i$ is inward for $F$.  By Corollary \ref{relcor00a} $U_{\infty} = \bigcap_{n=1}^{\infty}  F^n(U)$
is an invariant set which we call the \emph{attractor} \index{attractor} associated with the inward set $U$.

\begin{theo}\label{reltheo04} Let $F$ be a closed relation on $X$ and $A$ be a closed subset of $X$.
\begin{enumerate}
\item[(a)] The following conditions are equivalent. When they hold we call $A$ a \emph{preattractor}\index{preattractor} for $F$.
\begin{itemize}
\item[(1)] $A$ is $F$ + invariant and there exists an $F$ + invariant neighborhood $U$ of $A$ such that $U_{\infty} \subset A$.

\item[(2)] $A$ is $F$ + invariant and there exists an inward set $U$ containing $A$ such that $U_{\infty} \subset A$.

\item[(3)] $A$ is $\G F$ + invariant and $A \cap |\G F|$ is relatively open (as well as closed) in $|\G F|$.

\item[(4)] $A$ is $\CC F$ + invariant and $A \cap |\CC F|$ is relatively open (as well as closed) in $|\CC F|$.
\end{itemize}

\item[(b)] The following conditions are equivalent. When they hold we call $A$ a \emph{attractor} for $F$.
\begin{itemize}
\item[(1)] $A$ is $F$  invariant and there exists a closed neighborhood $U$ of $A$ such that $\bigcap_{n=1}^{\infty} \ F^n(U) = A$.

\item[(2)] There exists an inward set $U$  such that $U_{\infty} = A$.

\item[(3)] $A$ is an $ F$  invariant preattractor

\item[(4)] $A$ is a $\CC F$  invariant preattractor
\end{itemize}

\item[(c)] The following conditions are equivalent.
\begin{itemize}
\item[(1)] $A$ is $\CC F$ + invariant.

\item[(2)] $A$ is the intersection of a (possibly infinite) collection of preattractors.

\item[(3)] The inward neighborhoods of $A$ form a base for the neighborhood system of $A$, i.e. if $A \subset \subset B$, then
there exists $U$ inward such that $A \subset \subset U \subset \subset B$.
\end{itemize}
\end{enumerate}
\end{theo}

\begin{proof} See \cite{A93} Theorem 3.3.

\end{proof}

The set  $U$ is inward for $F$ if and only if $F(U) \cap (X \setminus U^{\circ}) = \emptyset$, or, equivalently,
$U \cap F^{-1}(X \setminus U^{\circ}) = \emptyset$. It follows that $X \setminus U^{\circ}$ is inward for $F^{-1}$.
The associated $F^{-1}$ attractor $U_{-\infty} =_{def}  \bigcap_{n=1}^{\infty} F^{-n}(X \setminus U^{\circ})$
is called the \emph{repeller} \index{repeller}for $F$
which is dual to $U_{\infty}$ \index{repeller!dual} and $(U_{\infty},U_{-\infty})$ is called an \emph{attractor-repeller pair}
\index{attractor-repeller pair}.

If $(A,B)$ is an attractor-repeller pair, then by \cite{A93} Propositions 3.8 and 3.9:
\begin{align}\label{eqrel16}\begin{split}
|\CC F| \ \subset \ A \cup &B. \\
A = \CC F(A \cap |\CC F|) \quad \text{and}& \quad B = \CC F^{-1}(|\CC F| \setminus A)
\end{split}\end{align}

In particular, because a compact metric space contains only countably many clopen subsets, it follows that $X$ contains only
countably many attractors. Finally, if $\{ (A_n,B_n) \}$ counts the finite or countably infinite collection of attractor-repeller pairs for $F$, then
\begin{equation}\label{eqrel16a}
|\CC F| \ = \ \bigcap_n  (A_n \cup B_n),
\end{equation}
and a pair of points $x,y \in |\CC F|$ are in the same chain component if and only if they lie in the same set of attractors, i.e.
$x \in A_n \ \Leftrightarrow \ y \in A_n$ for all $n$.

The following sharpening of Corollary \ref{relcor03}, due to Conley,  is sometimes referred to as the \emph{Fundamental Theorem of Dynamical Systems}.

\begin{cor}\label{relcor03a} Let $F$ be a closed relation on $X$.There exists a continuous function $L: X \to [0,1]$
such that if $(x,y) \in \CC F$, then $L(y) \ge L(x)$ with equality only when, in addition, $(y,x) \in \CC F$. Furthermore,
$L$ takes distinct values on distinct chain components. In particular, $L$ is a Lyapunov function
for $\CC F$ such that $|\CC F| = |L|$. \end{cor}

\begin{proof} By Theorem \ref{reltheo02}
 there exists a continuous function $L_n: X \to [0,1]$ with $B_n = L^{-1}(0)$, $A_n = L^{-1}(1)$ and
such that if $(x,y) \in \CC F$, then $L(y) \ge L(x)$ with equality only when
 \begin{equation}\label{eqrellyap06}
x,y \in A_n, \quad x,y \in B_n \quad \text{or} \ \ (y,x) \in \CC F.
\end{equation}

Define:
 \begin{equation}\label{eqrellyap07}
 L(x) \ = \ \sum_n \frac{2}{3^{n+1}}\ L_n(x).
 \end{equation}

 If $(x,y) \in \CC F$, then $L(y) \ge L(x)$ with equality only when either $(y,x) \in \CC F$, which implies that $x, y$ are chain recurrent points
in the same chain component or else, for all $n$ either $x,y \in A_n$ or $x,y \in B_n$. By (\ref{eqrel16a}) and the remark thereafter
this also implies that $x, y$ are chain recurrent points
in the same chain component. Finally, if $x, y \in |\CC F|$ but in distinct chain components then for some $n$
either
$x \in A_n, y \in B_n$ or the reverse. Hence, distinct chain components are mapped to distinct points of the Cantor set.

\end{proof}

Note that $X$ and $\emptyset$ are inward for $F$ and for $F^{-1}$. We define
\begin{equation}\label{eqrel17}\begin{split}
X_-  \ =_{def} \  \bigcap_{n=1}^{\infty} \ F^n(X) \ = \ \{ x : F^{-n}(x) \not= \emptyset \ \ \text{for all} \ n \in \Z_+ \}\\
 X_+   \ =_{def} \  \bigcap_{n=1}^{\infty} \ F^{-n}(X)\ = \ \{ x : F^{n}(x) \not= \emptyset \ \ \text{for all} \ n \in \Z_+ \}\\
 X_{\pm}  \ =_{def} \  X_- \ \cap \ X_+. \hspace{5cm}
 \end{split}\end{equation}
$X_-$ is the maximum $F$ invariant subset of $X$.
It is an attractor with $\emptyset$ as dual repeller. On the other hand $X_+ $, the
maximum $F^{-1}$ invariant subset, is  a repeller dual to the attractor $\emptyset$.
For a general closed relation, the intersection $X_{\pm}$ need not be + invariant for either $F$ or $F^{-1}$.

We will later apply the following \emph{Index Construction}\index{Index Construction}.

    \begin{theo}\label{reltheoConley} For a closed relation $F$ on $X$, let $U, V$ be open subsets of $X$ with
    $X_- \subset U, X_{\pm} \subset V$. There exist closed sets $P_1 \supset P_2$ both inward for $F$ such that
    \begin{itemize}
    \item[(i)]  $X_- \subset P_1^{\circ}$ and $P_1 \subset U$.
    \item[(ii)]  $X_{\pm} \subset P_1^{\circ} \setminus P_2$ and $\overline{P_1 \setminus P_2} \subset V$.
    \end{itemize}
    \end{theo}

    \begin{proof}  We can choose $W_-, W_+$  open subsets of $X$ such that
  \begin{equation}\label{eqrel17Conley}
X_+ \ \subset \ W_+, \qquad  X_- \ \subset W_- \ \subset \ U, \qquad \overline{W_+ \cap W_-} \ \subset \ V,
   \end{equation}

   Because $X_-$ is an attractor, Theorem \ref{reltheo04} implies that there exists $P_1$ an $F$ inward neighborhood
   of $X_-$ with $P_1 \subset W_-$.  Because $X_+$ is a repeller, Theorem \ref{reltheo04} implies that
   there exists $Q_1$ an $F^{-1}$ inward  neighborhood
   of $X_+$ with $Q_1 \subset W_+$.

   As observed above, $X \setminus Q_1^{\circ}$ is $F$ inward and $X \setminus P_1^{\circ}$ is $F^{-1}$ inward.
   Let $P_2 = P_1 \cap (X \setminus Q_1^{\circ}) = P_1 \setminus Q_1^{\circ} $ and
   $Q_2 = Q_1 \cap (X \setminus P_1^{\circ}) = Q_1 \setminus P_1^{\circ} $. As it is the intersection of two $F$ inward sets,
   $P_2$ is $F$ inward, and, similarly, $Q_2$ is $F^{-1}$ inward.

   $P_1^{\circ} \setminus P_2 = P_1^{\circ} \cap Q_1^{\circ} \supset X_- \cap X_+ = X_{\pm}$. $P_1 \setminus Q_1 \subset P_2$ and so
     $P_1 \setminus P_2  \subset P_1 \cap Q_1 \subset W_- \cap W_+$. Hence, $\overline{P_1 \setminus P_2} \subset \overline{W_+ \cap W_-}  \subset  V$.

   \end{proof}

Recall that $F$ is a surjective relation  on $X$ when $Dom(F) = X = Dom(F^{-1})$, i.e. for every $x \in X$,
$F(x) \not= \emptyset$ and $F^{-1}(x) \not= \emptyset$. Clearly, $F$ is surjective exactly when $X = X_{\pm}$.

  We call a closed relation $F$ on $X$ \emph{chain transitive}\index{chain transitivity} when $\CC F = X \times X$.

    \begin{lem}\label{rellem08aab} If $F$ is chain transitive on $X$, then $X$ is the only nonempty inward subset for $F$.
    Conversely, if $F \not= \emptyset$ or $X$ contains more than one point, then $F$ is chain transitive on $X$
    when $X$ is the only nonempty inward subset for $F$. \end{lem}

    \begin{proof} If $F$ is chain transitive, then $X$ contains no proper $\CC F$ + invariant subset and so, in particular, no proper
    inward subset.

    For the converse we  exclude the case when $X$ is a singleton and $F = \emptyset$ and assume there is no proper inward subset.

    If there were $x \in X$ such that $F(x) = \emptyset$, then  $\{ x \}$ is inward. From our assumption it is not proper and so $X = \{ x \}$ and
    $F = \emptyset$.  But this is the excluded case. Hence, $F(x) \not= \emptyset$ for all $x$.

    If $F$ is not chain transitive, then there exist $x, y \in X$ such that $(x,y) \not\in \CC F$.
    Because $\CC F$ is closed and transitive, $\CC F(x)$ is a nonempty, closed $\CC F$ + invariant subset of $X$ contained in the
     proper open set of $X \setminus \{ y \}$. By Theorem \ref{reltheo04} (c)
     there exists an inward set $U$ with $\CC F(x) \subset U \subset X \setminus \{ y \}$. Thus, $U$ is a proper, nonempty inward subset.

     \end{proof}

      From (\ref{eqrel09d}) it follows that $Dom (\CC F) = X$ implies $Dom (F) = X$ and if $\CC F$ is surjective, then $F$ is surjective. Thus, if
      $F$ is chain transitive, then it is surjective.

For $n_1 \le n_2 \in \Z^* = \{ - \infty \} \cup \Z \cup \{ \infty \}$ we let $[n_1,n_2]$ denotes the \emph{$\Z$ interval}
\index{interval} $\{ n \in \Z : n_1 \le n \le n_2 \}$.
If $n_1, n_2 \in \Z$, then $n_2 - n_1$ is the \emph{length} of the interval \index{interval!length}.  Otherwise, it is an
\emph{infinite interval} \index{interval!infinite} with infinite length.

 A function $\xx : [n_1,n_2] \to X$ is called an \emph{orbit sequence}\index{orbit sequence} for $F$ or an $F$ \emph{solution path}\index{solution path}
 with length that of $[n_1,n_2]$
 when $n \in \Z$ with $n_1 \le n, n+1 \le n_2$ implies $\xx(n+1) \in F(\xx(n))$. If $n_1 \in \Z$ we say that the sequence begins
 \index{solution path!beginning} at $\xx(n_1)$ and
 if $n_2 \in \Z$ we say that the sequence terminates at $\xx(n_2)$\index{solution path!termination}.
 An \emph{infinite forward orbit sequence}
 \index{solution path!infinite forward} for $F$ is an orbit sequence defined on $\Z_+ = [0,\infty]$.
 A \emph{bi-infinite orbit sequence}
 \index{solution path!bi-infinite} for $F$ is an orbit sequence defined on $\Z = [-\infty,\infty]$.

 There are obvious operations on solution paths.
 \begin{itemize}
 \item {\bf Translation} If $\xx : [n_1,n_2] \to X$ is an solution path and $a \in \Z$, then the translate $Trl_a(\xx) : [n_1-a,n_2-a] \to X$
 given by $Trl_a(\xx)(n) = \xx(n+a)$ is a solution path.\index{solution path!translation}

 \item {\bf Composition} ] If $\xx : [n_1,n_2] \to X$ and $\yy : [n_2,n_3] \to X$ are solution paths with $\xx(n_2) = \yy(n_2)$, then
 the composition $\xx \oplus \yy : [n_1,n_3] \to X$
 is the solution path such that $ \xx \oplus \yy | [n_1,n_2] = \xx$ and $ \xx \oplus \yy | [n_2,n_3] = \yy$.\index{solution path!composition}

\item{\bf Inversion}  If $\xx : [n_1,n_2] \to X$ is a solution path for $F$, then $\bar \xx : [-n_2,-n_1] \to X$
defined by $ \bar \xx(n) = \xx(-n)$ is a solution path for $F^{-1}$.
\end{itemize}

 With the product topology the sequence spaces $X^{\Z_+}$, $X^{\Z_-}$  and $X^{\Z}$ are compact spaces which we equip with the metric:
 \begin{equation}\label{eqrel17a}
 d(\xx,\yy)  \ =_{def} \  \max \{ \min(d(\xx(n),\yy(n)),\frac{1}{|n|}) \}
 \end{equation}
 with $n$ varying over $\Z_+,\Z_-$ or $\Z$. Thus, for $\ep > 0$, $ d(\xx,\yy) < \ep$ if and only if $d(\xx(n),\yy(n)) < \ep$ for all
 $n$ with $|n| \le 1/\ep$.

 The shifts, the surjective map $S$ on $X^{\Z_+}$ and the homeomorphism $S$ on $X^{\Z}$ are defined by $S(\xx)(n) = \xx(n+1)$
 for all $n \in \Z_+$ and all $n \in \Z$, respectively. The coordinate projections $\pi_n : X^{\Z_+} \to X$ and
 $\pi_n : X^{\Z} \to X$ are defined by $\xx \mapsto \xx(n)$ for all $n \in \Z_+$ and all $n \in \Z$, respectively.

 Define the \emph{solution path spaces}\index{solution path space} (also called the \emph{sample path spaces}\index{sample path space})
 \begin{align}\label{eqrel18}\begin{split}
 \S_+(F)  \ =_{def} \  \{ \xx \in &X^{\Z_+}: \text{for all} \ \ n \in \Z_+, \ \xx(n+1) \in F(\xx(n)) \} \\
  \S_-(F)  \ =_{def} \  \{ \xx \in &X^{\Z_-}: \text{for all} \ \ n \in \Z_-, \ \xx(n) \in F(\xx(n-1)) \} \\
  \S(F)  \ =_{def} \  \{ \xx \in &X^{\Z}: \text{for all} \ \ n \in \Z, \ \xx(n+1) \in F(\xx(n))  \}.
  \end{split}\end{align}
 That is, $\S_+(F)$ is the space of all infinite forward orbit sequences and $\S(F)$ is the space of  all bi-infinite orbit sequences.
 In general, we will write $\S([n_1,n_2],F)$ (or just $\S([n_1,n_2])$ when $F$ is understood)
 for the set of solution paths defined on the interval $[n_1,n_2]$.

 \begin{prop}\label{relprop05} Assume $F$ is a closed relation on $X$.

 The solution path space $\S_+(F)$ is a closed, $S$ +invariant subspace of $X^{\Z_+}$ and
 $ \pi_0(\S_+(F)) = X_+$. In particular, $\pi_0(\S_+(F)) = X$ if and only if $Dom(F) = X$. If $F(X) = X$, then $\S_+(F)$ is
 $S$ invariant and $ \pi_n(\S_+(F)) = X_+$ for all $n \in \Z_+$.

 The solution path space $\S(F)$ is a closed, $S$ invariant subspace of $X^{\Z}$ and for all
 $n \in \Z, \ \pi_n(\S(F)) = X_{\pm}$. In particular, $\pi_n(\S(F)) = X$ if and only if $F$ is a surjective relation.
 \end{prop}

 \begin{proof}If $x = \pi_0(\xx)$, then $\xx(n) \in F^n(x)$ and so $x \in X_+$. Now  $X_+$ is $F^{-1}$
  invariant, and so $X_+ \subset F^{-1}(X_+)$. This implies that if $x_n \in X_+$, then there exists $x(n+1) \in X_+$ such that
  $\xx(n+1) \in F(\xx(n))$. So beginning with $x_0 = x \in X_+$ we can inductively construct an infinite forward orbit sequence and so $x \in \pi_0(\S_+(F))$.

  Clearly, $\S_+(F)$ is a closed, $S$ + invariant subset of $X^{\Z_+}$. Now assume $F(X) = X$, i.e. $Dom (F^{-1}) = X$. We also have
  $X_+ \supset F^{-1}(X_+)$ and this implies that for $\xx(0)\in X_+$,  there exists $y \in X_+$ such that $\xx(0) \in F(y)$. Define
  $\yy$ by $\yy(0) = y$ and $\yy(n) = \xx(n-1)$ for $n \ge 1$. Because $\yy \in \S_+(F)$ and $S(\yy) = \xx$ it follows that  $\S_+(F)$ is
  $S$ invariant.

  Finally, for $n \in \Z_+, \ \pi_n \circ \s^n = \s^n \circ \pi_0$. It follows that  $ \pi_n(\S_+(F)) = X_+$.

  It is obvious that $\S(F)$ is a closed, $S$ invariant subset of $X^{\Z}$ and that $\pi_0(\S(F)) \subset X_{\pm}$. If $x \in X_{\pm} \subset X_+$,
  then there exists a sample path defined on $[0,\infty]$ which begins at $x$. Applying the same result to $F^{-1}$, we obtain a sample path for $F$
  defined on $[- \infty,0]$ which terminates at $x$. Putting these together we obtain a bi-infinite orbit sequence $\xx$ with $\xx(0)= x$.
  As before we see that $ \pi_n(\S(F)) = X_{\pm}$ for all $n \in \Z$.

  \end{proof}

 \vspace{1cm}

 \subsection{{\bf Restriction  to a Closed Subset}}\label{restriction1}

\vspace{1cm}

If $F$ is a relation from a set $X$ to a set $Y$, i.e. $F \subset X \times Y$, and $A \subset X, B \subset Y$, then we can
restrict $F$ to obtain a relation from $A$ to $B$ by taking $F \cap (A \times B)$. For example, if $F$ is a function and
$B = Y$ then $F \cap (A \times Y)$ is the usual restriction $F|A$ of the function $F$ to the subset $A$ of its domain.

If $F$ is a  relation on a space $X$ and $C$ is a  subset of $X$, then the \emph{restriction}\index{restriction} of $F$ to $C$ is
\begin{equation}\label{eqrel20}
F_C \  \ =_{def} \  \ F \cap (C \times C). \hspace{3cm}
\end{equation}
When $F$ and $C$ are  closed, we can regard $F_C$ as a closed relation on $X$ with domain contained in $C$
or as a closed relation on the subspace $C$. Clearly,
\begin{equation}\label{eqrel20a}
(F_C)^{-1} \ = \ (F^{-1})_C, \hspace{3cm}
\end{equation}
and so we may omit the parentheses.

On the other hand, for $n > 1$ and $\A = \O, \NN, \G, \CC$ the obvious inclusions
$(F_C)^n \subset (F^n)_C$ and for $\A (F_C) \subset (\A F)_C$ might be strict. A partial exception occurs
when $C$ is $F$ + invariant. In that case, $(F_C)^n = (F^n)_C$ and for $\A = \O, \NN, \ \A (F_C) = (\A F)_C$.

An orbit sequence for $F_C$ is an orbit sequence for $F$ the terms of which lie in $C$. In particular,
\begin{align}\label{eqrel21}\begin{split}
\S_+(F_C) \ = \ \S_+(F) \cap C^{\Z_+}, &\quad \S_-(F_C) \ = \ \S_-(F) \cap C^{\Z_-}, \\ \S(F_C) \ =  \ \  &\S(F) \cap C^{\Z}.
\end{split}\end{align}

The definition (\ref{eqrel17}) applied to $F_C$ becomes
\begin{align}\label{eqrel17C}\begin{split}
C_-  \ &=_{def} \  \bigcap_{n=1}^{\infty} \ (F_C)^n(C) \ = \\
\{ x \in C : \  &(F_C)^{-n}(x) \not= \emptyset  \ \ \text{for all} \ n \in \Z_+ \} \\
C_+   \ &=_{def} \  \bigcap_{n=1}^{\infty} \ (F_C)^{-n}(C)\ = \\
\{ x \in C: \  &(F_C)^{n}(x) \not= \emptyset  \ \ \text{for all} \ n \in \Z_+ \}\\
 C_{\pm}  \ &=_{def} \   C_- \ \cap  \ C_+,
 \end{split}\end{align}
For $F$ and $C$ closed $C_+$ is the maximum repeller and $C_-$ the maximum attractor for $F_C$.

 From Proposition \ref{relprop05} we see that
\begin{equation}\label{eqrel22}
\pi_0(\S_+(F_C)) = C_+, \qquad \pi_0(\S(F_C)) = C_{\pm}.
\end{equation}

We recall the concept of $F$ invariance for a subset.

  \begin{prop}\label{relprop07a} Let $F$ be a  relation  on $X$ and $C$ be a  subset of $X$.

  The following conditions are equivalent.
  \begin{itemize}
 \item[(i)] $C$ is + invariant for $F$.

  \item[(ii)] $F(C) \subset C$.

  \item[(iii)] $C \subset F^*(C)$.

  \item[(iv)] For all $x \in C$, $F(x) \subset C$.

  \item[(v)] $F_C \ = \ F \cap (C \times X)$.

  \item[(vi)] For all $(x,y) \in F$, $x \in C \  \Longrightarrow \ y \in C$.

  \item[(vii)] $X \setminus C$ is + invariant for $F^{-1}$.
  \end{itemize}

  When $C$ is + invariant, $C_+ \ = \ X_+ \cap \ C$. \vspace{.25cm}

    The following conditions are equivalent.
  \begin{itemize}
 \item[(viii)] $C$ is invariant for $F$.

  \item[(ix)] $F(C) = C$.

   \item[(x)] $C$ is + invariant for $F$ and, in addition, for all $x \in C$,\\ $F^{-1}(x) \cap C \not= \emptyset$.

    \item[(xi)] $C$ is + invariant for $F$ and, in addition, $C = C_-$.
    \end{itemize}
    \end{prop}

 \begin{proof} The equivalence of (i)-(vi) are obvious and so (vii) is equivalent to the contrapositive of (vi). Clearly, $C$ is + invariant if and
 only if any solution path which begins in $C$ remains in $C$ and so $C_+ \subset C$.

 The equivalence of (viii)-(x) are clear as is (xi) $\Rightarrow$ (x). On the other hand, if (x) holds then given $x \in C$ we can inductively
 construct $\xx \in \S_-(F_C)$ with $\xx(0) = x$.  Thus, (xi) follows.

 \end{proof}

%\begin{align}\label{eqrel22aaa}\begin{split}
%C \  \text{ is + invariant for} \ F \quad &\Longleftrightarrow \quad F(C) \subset C \quad \Longleftrightarrow \\
%C \subset F^*(C) \quad \Longleftrightarrow \quad &F(x) \subset C \ \ \text{for all }  \ x \in C. \\
%C \  \text{ is  invariant for} \ F \quad &\Longleftrightarrow \quad F(C) = C \quad \Longleftrightarrow \\
%C \  \text{ is + invariant and }  \ F^{-1}&(x) \cap C \ \not= \emptyset \ \ \text{for all }  \ x \in C.
%  \end{split}\end{align}
Note that neither of the two properties $F$ invariance and $F^{-1}$ + invariance implies the other. For example, any nonempty set $A$ is
+ invariant for both $F$ and $F^{-1}$ when $F = \emptyset$, but $A$ is not invariant.

A number of authors use the term invariance to refer to a weaker property, \cite{KM}, \cite{B}, \cite{BM}. I will use the term \emph{viability}
instead.   I believe the issue arose historically because the two properties agree when $F$ is a homeomorphism. Other authors
use the term \emph{weak invariance}\index{subset!weakly invariant} for viability.

  \begin{df}\label{reldf07} Let $F$ be a  relation  on $X$ and $C$ be a  subset of $X$.

The following conditions are equivalent.  When they hold we say that $C$ is \emph{+ viable} for $F$. \index{subset!+ viable}
  \begin{itemize}
  \item[(i)] $C = Dom(F_C)$.

  \item[(ii)] $C \subset F^{-1}(C)$.

  \item[(iii)] $F(x) \cap C \not= \emptyset$ for all $x \in C$.

  \item[(iv)] $C = C_+$.

  \item[(v)]  $\pi_0(\S_+(F_C)) = C$.
  \end{itemize}

  We say that $C$ is \emph{- viable} for $F$ when it is + viable for $F^{-1}$, \index{subset!- viable}
  or, equivalently, when $C = C_-$.

  We say that $C$ is \emph{viable} for $F$ (= viable for $F^{-1}$) when it is both + and - viable. \index{subset!viable}
  So $C$ is viable when the following equivalent conditions hold.
   \begin{itemize}
   \item[(vi)] $F_C$ is a surjective relation on $C$.

  \item[(vii)] $C \subset F^{-1}(C) \cap F(C)$.

  \item[(viii)] $C = C_{\pm}$.

  \item[(ix)]  $\pi_0(\S(F_C)) = C$.
 \end{itemize}
 \end{df} \vspace{.5cm}

 Thus, $C$ is + viable when for every $x \in C$ there exists  an infinite forward solution path beginning at $x$ and remaining in $C$.
 Similarly, $C$ is viabile when for every $x \in C$ there is an bi-infinite  solution path which passes through $x$ and remains in $C$.
 Contrast this with invariance.  $C$ is + invariant when all solutions beginning at a point of $C$ remain in $C$.

 The equivalences among the various conditions are clear from (\ref{eqrel17C}).

    \begin{lem}\label{rellem06aa} Let $F$ be a closed relation on $X$.

    \begin{itemize}
    \item[(a)] Let $C \subset X$. If $C$ is + viable, - viable or
   viable, then the closure $\overline{C}$ satisfies the corresponding property.

    \item[(b)]  If $\{ C_i \}$  is a collections of subsets of $X$
 all of which are + viable, all - viable or all
   viable, then the union $\bigcup \{ C_i \}$ satisfies the corresponding property.

     \item[(c)]   If $\{ A_i  \}$ is a collections of closed subsets of $X$
   which is   totally ordered by inclusion,  all of which are + viable, all - viable or all
   viable, then the intersection $\bigcap \{ A_i \}$ satisfies the corresponding property.

    \item[(d)] Let $F_1 \subset F$. If $C \subset X$ is  + viable,  - viable or
   viable for $F_1$, then it satisfies the corresponding property for $F$.  If $C$ is + invariant for $F$, then it is + invariant for $F_1$.
    \end{itemize} \end{lem}

    \begin{proof} We do the proofs for + viability.

    (a) $F^{-1}(\overline{C})$ is closed and contains $F^{-1}(C) \supset C$ and so it contains $\overline{C}$.

    (b) $F^{-1}(\bigcup \{ C_i \}) = \bigcup \{ F^{-1}(C_i) \} \supset \bigcup \{ C_i \}$.

  %  If $x \in C$ then there is a sequence $\{ x_k \} \to x$ with $x_k \in C_{i_k}$.
%    Since $C_k$ is + viable, there exists $y_k \in F(x_k) \cap C_{i_k} \subset F(x_k) \cap C$.  By going to a subsequence
%    we may assume that $\{ y_k \} \to y$. Hence, $y \in C$ and $(x,y) \in F$ because $C$ and $F$ are closed.

   (c)  If $x \in \bigcap \{ A_i \}$, then  $\{ F(x) \cap A_{i} \}$ is a collection of closed nonempty sets totally ordered by inclusion. Hence, the
     intersection $F(x) \cap (\bigcap \{ A_i \})$ is nonempty by compactness.

     (d) Obvious.

    \end{proof}

  \begin{prop}\label{relprop06aaa} Let $F$ be a  relation  on $X$.

  (a) If a  subset $C$ is $F$ + invariant, then it $F$ invariant if and only if it is - viable.
  If $C$ is $F^{-1}$ + invariant,then it is $F^{-1}$ invariant if and only if it is + viable.
  In particular, an attractor is - viable and a repeller is + viable.

  (b) If $A$ is  + invariant and $B$ is + viable, then $A \cap B$ is + viable.

  (c) If $A$ is invariant for $F$, e.g. an $F$ attractor, and $B$ is invariant for $F^{-1}$, e.g. an $F$ repeller,
  then $A \cap B$ is viable for $F$.

  (d) If $A$ is + viable, then $F^{-1}(A)$ is + viable. If $A$ is - viable, then $F(A)$ is - viable.

  (e)  If $C$ is any  subset, then for both $F$ and $F_C$ \\  $C_+$ is + viable, $C_-$ is - viable
  and $C_{\pm}$ is viable. Furthermore, $C_+ / C_- / C_{\pm}$ is the maximum + viable / - viable / viable
  subset of $C$.

  (f) If $C$ is a closed subset, then $C_-$ is $F_C$ invariant and $C_+$ is $F_C^{-1}$ invariant.

  \end{prop}

  \begin{proof} (a)This is clear from (viii) $\Leftrightarrow$ (xi) in Proposition \ref{relprop07a}.

  (b) If $x \in A \cap B$, then $F(x) \subset A$ and $F(x) \cap B \not= \emptyset$.
 So $F(x) \cap (A\cap B) = F(x) \cap B \not= \emptyset $ for all $x \in A \cap B$.

  (c) $A$ is + invariant and by (a) $B$ is + viable. Hence, by (b) $A \cap B$ is + viable. Applying this to $F^{-1}$ we see that
  $A \cap B$ is - viable as well.

  (d) Let $\xx \in \S_+$ with $\xx(i) \in A$ for all $i \in \Z_+$ and so $\xx(i) \in F^{-1}(A)$ for all $i \in \Z_+$. If $x \in F^{-1}(\xx(0))$
  then let $\yy(0) = x$ and $\yy(i+1) = \xx(i)$ for all $i \in \Z_+$. Thus, $\yy \in \S_+$ with $\yy(0) = x$ and $\yy(i) \in F^{-1}(A)$ for all $i$.
  Applied to $F^{-1}$ we obtain the result for - viability.

  (e) If $\xx \in \S_+(F_C)$ beginning at $x \in C$, then for all $n \in \Z_+$, translating $\xx$ by $n$ we obtain an element of  $\S_+(F_C)$
  beginning at $\xx(n)$. That is, $\xx(n) \in C_+$ and so $\S_+(F_C) = \S_+(F_{C_+})$. From condition (v) of Definition \ref{reldf07} we
  see that $C_+$ is + viable. Similarly, $\S(F_C) = \S(F_{C_{\pm}})$ and so $C_{\pm}$ is viable.

  That each is the maximum $C$ subset of  its type is clear.

  (f) By (d) and (e) $F_C(C_-)$ is a - viable subset of $C$ and so is contained in $C_-$. Thus, $C_-$ is + invariant for $F_C$.
   By (a) applied to   $F_C$ it follows that $C_-$ is $F_C$ invariant. For $C_+$ apply the result to $F^{-1}$.

  \end{proof}

    \begin{prop}\label{relprop06a} Let $F$ be a closed relation on $X$.

     \begin{enumerate}
    \item[(a)] If $\K $ is a nonempty subset of $\S_+(F_C)$, then
       \begin{equation}\label{eqrel22b}
      \om [\K]  \ =_{def} \   \bigcap_{n=1}^{\infty} \ \overline{ \{ \xx(k) : \xx \in \K, k \geq n \}}
      \end{equation}
      is a nonempty, closed, viable  subset of $C_{\pm}$.

      In particular, if $\xx \in \S_+(F_C)$, then
         \begin{equation}\label{eqrel22a}
      \om [\xx]  \ =_{def} \   \bigcap_{n=1}^{\infty} \ \overline{ \{ \xx(k) :k \geq n \}}
      \end{equation}
      is a nonempty, closed, viable  subset of $C_{\pm}$. Furthermore, $\om [\xx]$ is a chain transitive subset of $X$.

      \item[(b)] Assume that $X_+ = X$, i.e. $Dom(F) = X$, and that $A$ is a closed subset of $X$. Let
      $\K(A)  \ =_{def} \  \{ \xx \in \S_+(F) : \xx(0) \in A \}$.

      \begin{itemize}
        \item[(i)] $\om [\K(A)] = Lim sup \{ F^k(A) \} = \bigcap_{n = 1}^{\infty} \overline{\bigcup \{ F^k(A) : k \ge n \}}$.

      \item[(ii)] If $\om[\K(A)] \subset A$, then $\om[\K(A)]$ is $F$ invariant and is the maximum - viable subset of $A$.

      \item[(iii)] If $A$ is $F$ + invariant, then $\om[\K(A)] = \bigcap_{k=1}^{\infty} F^k(A)$.

      \item[(iv)] If $\om[\K(A)] \subset \subset A$, then $\om[\K(A)]$ is an $F$ attractor.
      \end{itemize}
\end{enumerate}
      \end{prop}

      \begin{proof} (a) As it is the intersection of a decreasing sequence of nonempty compacta, $\om[\K]$ is closed and nonempty.

      Assume that $\yy(0) = Lim_{i \to \infty} \{ \xx_i(k_i) \}$ $\{ \xx_i \}$ a sequence in $\K$ and
      with $k_i$ increasing to infinity. By going to a subsequence
      and using a diagonal process we may assume that for each $n \in \Z$
     the sequence $\{ \xx_i(n + k_i) : k_i > -n \}$ converges to a point $\yy_n$. Since $(\xx_i(n + k_i),\xx_i(n + 1 + k_i)) \in F_C$
      it follows that $\yy \in \S(F_C)$. So $\yy(0) \in C_{\pm}$.  In fact, $\yy(n) \in \om (\K)$ for all $n$ and so $\om (\K)$ is viable.
      Since it is a viable subset of $C$, it is contained in $C_{\pm}$.

    %  If $y \in \om[\K]$, then there sequences $\{ \xx_i\in \K \}$ and $k_i \in \Z_+$ increasing to infinity such that $\{ (\xx_i)(k_i) \} \to y$.
%      By going to a subsequence we may assume $\{ \xx_i(0)\}$  converges to a point $x$.  Lemma  \ref{rellem06a}
%      implies that $(x,y) \in C_+ \times C_-$. Since $\xx(k) \in C_+$ for all $\xx \in A, k \in \Z_+$, it
%      follows that $y$ is in the closed set $C_+$ as well.
%      That is, $\om [\K] \subset C_- \cap C_+ = C_{\pm}$.

      In the case of a single $\xx$ we prove that $B = \om [\xx]$ is a chain transitive subset, i.e. $F_B$ is a chain transitive relation on $B$.
      We follow the proof of \cite{A93} Theorem 4.5.

      Given $y,z \in B$ and $\ep > 0$ we construct an $\ep$ chain for $F_B$ beginning $\ep$ close to $y$ and terminating $\ep$ close to $z$.

      First, we can choose a positive $\ep_1 < \ep/2$ such that
       \begin{equation}\label{eqrel22chain1}
       V_{\ep/2} \circ F_B \circ V_{\ep/2} \ \supset \ F \cap (\overline{V_{\ep_1}(B) \times  V_{\ep_1}(B)})
       \end{equation}
       because the right side decreases to $F_B$ as $\ep_1$ decreases to $0$.  Recall that $V_{\ep}$  is the open
        $\ep$ neighborhood of the diagonal $1_X$ in $X \times X$.

       Now choose $n$ so that
        \begin{equation}\label{eqrel22chain2}
         V_{\ep_1}(\om [\xx])   \ \supset \ \overline{ \{ \xx(k) :k \geq n \}}.
          \end{equation}

          There exist $k_1, k_2$ with $n < k_1 < k_2$ such that $d(y,\xx(k_1)), d(z,\xx(k_2)) < \ep_1$. By (\ref{eqrel22chain2}) and
          (\ref{eqrel22chain1}) there exists, for $k_1 \le i < k_2$, a pair $(y_i,z_i) \in F_B$ with $d(y_i,\xx(k_1+i)), d(z_i,\xx(k_1+i+1)) < \ep/2$.
          Hence, $d(z_i,y_{i+1}) < \ep$. Furthermore, $d(y,y_1), d(z_{k_2-1},z) < \ep$.

          Thus, $z \in (V_{\ep} \circ F_B)^{k_2-k_1}\circ V_{\ep}(y)$.

          As $\ep > 0$ was arbitrary, it follows that $ z \in \CC (F_B)(y)$. \vspace{.5cm}

      (b) \textbf{Claim:} If $\xx \in \S(F)$ with $\xx(-n) \in A$ for all $n \in \Z_+$, then $\xx(k) \in \om[\K(A)]$ for all $k \in \Z$.\vspace{.25cm}

      proof of the Claim: For any $m \in \Z_+$ let $\yy_m = Trl_{k-m}\xx$ so that $\yy_m(i) = \xx(i + k - m)$. For each $m > k$ the restriction
      $\yy_m|[0,\infty] \in \K(A)$ and $\yy_m(m) = \xx(k)$. Letting $m$ tend to infinity we see that $\xx(k) \in  \om[\K(A)]$.\vspace{.25cm}

      (i) If $y \in F^k(x)$, then because $y \in X_+$ there exists $\xx \in \S_+(F)$ with $\xx(0) = x$ and $\xx(k) = y$.
      Hence, $\bigcup \{ F^k(A) : k \ge n \} = \{ \xx(k) : \xx \in \K(A), \ k \geq n \}$. \vspace{.25cm}

      (ii) If $x \in B \subset A$ and $B$ is - viable, then there exists a solution path $\yy : [-\infty,0] \to B$ with $\yy(0) = x$. Since
      $x \in A \subset X_+$ we can extend $\yy$ to an element $\xx \in \S(F)$ and so the Claim implies that $\xx(k) \in \om[\K(A)]$.
      In particular, $x \in \om[\K(A)]$.

      By (a) $\om[\K(A)]$ is viable, so if $x \in \om[\K(A)]$ and $y \in F(x)$ there exists $\yy : [-\infty,0] \to \om[\K(A)]$ with
      $\yy(0) = x$ and $\zz \in \S_+(F)$ such that $\zz(0) = x$ and $\zz(1) = y$. By assumption $\om[\K(A)] \subset A$ and so the Claim
      applied to the composition $\xx = \yy \oplus \zz \in \S(F)$ yields $y \in \om[\K(A)]$. That is, $\om[\K(A)]$ is + invariant.
      Since it is viable, it is invariant by Proposition \ref{relprop06aaa}(a).\vspace{.25cm}

      (iii) If $A$ is + invariant, then $\{ F^k(A) \}$ is a decreasing sequence of closed subsets and so (iii) follows from (i).\vspace{.25cm}

      (iv)  If $\om[\K(A)] \subset \subset A$, then $\om[\K(A)]$ is invariant by (ii). From (i) and Theorem 3.3 (a)(i) of \cite{A93} it follows that
      $\om[\K(A)]$ is a preattractor. Since it is invariant, it is an attractor.

      \end{proof}

      If $\xx : [-\infty,0] \to X$ is a solution path for $F_C$ so that $\bar \xx \in S_+(F^{-1}_C)$, then
            \begin{equation}\label{eqrel22c}
      \a [\xx]  \ =_{def} \   \bigcap_{n=1}^{\infty} \ \overline{ \{ \xx(k) : k \le -n \}} = \om[\bar \xx].
      \end{equation}
       is a nonempty, closed, viable, chain transitive  subset of $C_{\pm}$ by Proposition \ref{relprop06a} applied to $\bar \xx$. \vspace{.25cm}

  Define $\nu_C, \bar \nu_C : C \to \Z_+ \cup \{ \infty \}$ :
  \begin{align}\label{eqrel23} \begin{split}
  \nu_C(x) & \ =_{def} \  \sup \{ n \in \Z_+ : (F_C)^n(x) \not= \emptyset \}, \\
  \bar \nu_C(x) & \ =_{def} \  \sup \{ n \in \Z_+ : (F_C)^{-n}(x) \not= \emptyset \}
  \end{split} \end{align}
  Notice that $x \in C$ implies $ (F_C)^0(x) = \{ x \}$ and $(F_C)^n(x) = \emptyset$ implies $(F_C)^m(x) = \emptyset$ for all $m \ge n$.

    \begin{prop}\label{relprop07aaa} The functions $\nu_C$ and $\bar \nu_C$ are usc, i.e. $\{ \nu_C < n \}$ and
    $\{ \bar \nu_C < n \}$ are open sets for any
    $n \in \Z_+$.
     \begin{align}\label{eqrel24} \begin{split}
     \nu_C(x) \ &= \ \infty \quad \Longleftrightarrow \quad x \in C_+, \\
  \bar \nu_C(x) \ &= \ \infty \quad \Longleftrightarrow \quad x \in C_-.
    \end{split} \end{align}

    If $y \in (F_C)^m(x)$, then $\nu_C(x) \ge m + \nu_C(y)$.  In particular, if $\nu_C(x) < \infty$
    and $y \in F^{\nu_C(x)}(x)$, then $\nu_C(y) = 0$, i.e.
    $F_C(y) = \emptyset$.
    \end{prop}

    \begin{proof} $\{ \nu_C < n \} = ((F_C)^n)^*(\emptyset)$ which is open.

    The equivalences of (\ref{eqrel24}) follow from (\ref{eqrel17C}).

    If $(F_C)^n(y) \not= \emptyset$, then $(F_C)^{m+n}(x) \supset (F_C)^n(y) \not= \emptyset$.

    \end{proof}

    Notice that $F_C(x) = \emptyset$ is equivalent to $F(x) \subset X \setminus C$. This is not true if we
     replace $F_C$ by $(F_C)^n$ for $n > 1$. To see what this means for larger $n$ we recall that for $A \subset X$
     $F^*(A) = \{ x : F(x) \subset A \}$ define, inductively,
    \begin{equation}\label{eqrel124a}
    F^{*1}(A) \ =_{def} \ F^*(A), \qquad F^{*(n+1)}(A) \ =_{def} F^*(A \cup F^{*n}(A)).
    \end{equation}
    Notice that, by induction, the sequence of sets $\{F^{*n}(A) : n = 1,2, \dots \}$ is increasing.

    \begin{lem}\label{rellem07abb} For $x \in C \subset X$ and $n \in \Z$ we have $(F_C)^n(x) = \emptyset$ if and only if
    $x \in F^{*n}(X \setminus C)$. \end{lem}

    \begin{proof}  This is clear for $n = 1$.  Now assume the result for $n \in \Z$.

    $(F_C)^{n+1}(x) = \emptyset$ if and only if for all $y \in F(x)$, either $y \in X \setminus C$ or $y \in C$ with
    $(F_C)^n(y) = \emptyset$, i.e. by induction hypothesis, $y \in F^{*n}(X \setminus C)$.  Thus,
    $F(x) \subset (X \setminus C) \cup F^{*n}(X \setminus C)$ and so $x \in F^{*(n+1)}(X \setminus C)$.

        If $x \in F^{*(n+1)}(X \setminus C)$, then for all $y \in F(x)$ either $y \in X \setminus C$ or $y \in C \cap F^{*n}(X \setminus C)$.  In the latter case, by induction hypothesis, $(F_C)^n(y) = \emptyset$. Hence, $(F_C)^{n+1}(x) = \emptyset$ in either case.

    \end{proof}

     When $x \in C$ and $F_C(x) = \emptyset$  we call $x$ a \emph{terminal point} \index{terminal point}
     for $F_C$.

  The following is a version of \cite{KM} Lemma 2.10.

  \begin{cor}\label{relcor08aaa} If $A$ is a closed subset of $C$, disjoint from $C_+$, then for sufficiently large $n \in \Z_+$, $(F_C)^n(A) = \emptyset$
  ,or, equivalently, $A \subset F^{*n}(X \setminus C)$.
  In particular, if $C_+ = \emptyset$, then for some positive integer $n$, $(F_C)^n = \emptyset$ or, equivalently, $C \subset F^{*n}(X \setminus C)$.
  \end{cor}

    \begin{proof} The usc function $\nu_C$ is finite on $A$ and so is bounded by compactness. The equivalence follows from Lemma \ref{rellem07abb}.

    \end{proof}

      The following two results provide an extension of \cite{KM}Lemma 2.9.

  \begin{lem}\label{rellem06a} Let $F$ be a closed relation on $X$ and $C$ be a closed subset of $X$.

  If $\{ n_k \}$ is a sequence in $\Z_+ $ with
  $\{ n_k \} \to \infty$ $\{ (x_k,y_k) \}$ is a sequence in $C\times C$
  converging to $(x,y)$ with $y_k \in (F_C)^{n_k}(x_k)$, then $(x,y) \in C_+ \times C_-$. \end{lem}

  \begin{proof}   For any $n \in \Z_+$, once $n_k > n$  there exist $ z_k \in (F_C)^{n}(x_k), w_k \in (F_C)^{-n}(y_k)$. If $z, w$ are limit points of the
  sequences $\{ z_k \}, \{ w_k \}$, then $z \in (F_C)^n(x), w \in (F_C)^{-n}(y)$. As $n$ was arbitrary, it follows that $x \in C_+, y \in C_-$.

  \end{proof}

  \begin{prop}\label{relprop06} Let $F$ be a closed relation on $X$ and $C$ be a closed subset of $X$.

   $\O (F_C) \ \cup \ (C_+ \times C_-)$ is a closed, transitive relation on $C$ with
  \begin{equation}\label{eqrel19}
  \NN (F_C) \ \subset \ \G (F_C) \ \subset \ \CC (F_C) \ \subset \ \O (F_C) \ \cup \ (C_+ \times C_-). \hspace{2cm}
  \end{equation}
  \end{prop}

  \begin{proof}  First we show that $\CC (F_C)$ is contained in $\O (F_C) \ \cup \ (C_+ \times C_-)$.

  Assume that $(x,y) \in \CC (F_C)$ with $x \not\in C_+$. For $U$ a closed neighborhood of $x$ (relative to $C$) which is disjoint from
  $C_+$, there exists a positive integer $N$ such that $(F_C)^N(U) = \emptyset$. From Proposition \ref{relprop00b} it follows that there
  exists $\ep > 0$ such that $(\bar V_{\ep} \circ F_C)^N(U) = \emptyset$.
  %We may choose $\ep$ small enough that $\bar V_{\ep}(x) \cap C \subset U$.

  There exist sequences $\{ x_k \}, \{ y_k \}, \{ n_k \}, \{ \ep_k \}$ with  $\{ (x_k,y_k) \} \to (x,y)$ and $y_k \in (V_{\ep_k}\circ F_C)^{n_k}(x_k)$
  and with $ \ep_k$ decreasing and tending to $0$.  We may assume $\ep > \ep_k$ and $x_k \in U$. Since $(\bar V_{\ep} \circ F_C)^n(x_k) = \emptyset$
  for all $n \ge N$, it follows that $n_k < N$. By going to a subsequence we may assume that there is a positive integer $n$ such that $n_k = n$ for all $k$.
  If $k \ge \ell$ then $(x_k,y_k) \in (\bar V_{\ep_{\ell}}\circ F_C)^{n}$ and so $(x,y) \in (\bar V_{\ep_{\ell}}\circ F_C)^{n}$. Letting $\ell \to \infty$,
  we have $(x,y) \in (F_C)^n \subset \O (F_C)$.

  Applying the argument to $F_C^{-1}$ we see that $(x,y) \in \CC (F_C)$ with $y \not\in C_-$ implies $(x,y) \in \O (F_C)$.

  The remaining case is $(x,y) \in (C_+ \times C_-)$.

  In any case, $\NN (F_C) = \overline{\O (F_C)}$ is a subset of $\G (F_C)$ which is, in turn, contained in $\CC(F_C)$.

  Since $\O (F_C) \ \cup \ (C_+ \times C_-) = \CC (F_C) \ \cup \ (C_+ \times C_-)$ it is a closed relation. Finally, we show it is transitive.

   Let $(x,y), (y,z) \in (\O (F_C)) \ \cup \ (C_+ \times C_-)$.

  Case 1: If $(x,y), (y,z) \in \O (F_C)$, then $(x,z) \in \O (F_C)$ because $\O (F_C)$ is transitive.

  Case 2: If $(x,y), (y,z) \in  C_+ \times C_-$, then $(x,z) \in C_+ \times C_-$.

  Case 3: If $(x,y) \in \O (F_C)$ and $(y,z) \in C_+ \times C_-$, then because $C_+$ is $(F_C)^{-1}$ invariant, $ x \in C_+$ and so $(x,z) \in C_+ \times C_-$.

  Case 4: Similarly, if $(x,y) \in C_+ \times C_-$ and $(y,z) \in \O (F_C)$, then $(x,z) \in C_+ \times C_-$ by $F_C$ invariance of $C_-$.

\end{proof}

\begin{prop}\label{relprop06new} Let $F$ be a closed relation on $X$ and $C$ be a closed subset of $X$.

 Let $U$ be a closed set which is a neighborhood of $C_{\pm}$ relative to $C$.
\begin{align}\label{eqrel19new1}\begin{split}
C_+ \ = \ \NN (F_C)^{-1}(C_{\pm}) \ = \ &\G (F_C)^{-1}(C_{\pm}) \ = \ \CC (F_C)^{-1}(C_{\pm}) \\
= \ \bigcap_{n \in \Z_+} \ &\overline{ \bigcup_{k \ge n} (F_C)^{-k}(U)}.
\end{split}\end{align}
\begin{align}\label{eqrel19new2}\begin{split}
C_- \ = \ \NN (F_C)(C_{\pm}) \ = \ &\G (F_C)(C_{\pm}) \ = \ \CC (F_C)(C_{\pm}) \\
= \ \bigcap_{n \in \Z_+} \ &\overline{ \bigcup_{k \ge n} (F_C)^k(U)}.
\end{split}\end{align}

\end{prop}

\begin{proof} By Proposition \ref{relprop06aaa} (f)  $C_+$ is $F_C^{-1}$ + invariant. So it is + invariant for the closed
relation $\O (F_C)^{-1} \cup C_- \times C_+$. From (\ref{eqrel19}) it follows that
\begin{equation}\label{eqrel19new5}
\NN (F_C)^{-1}(C_{\pm}) \subset  \G (F_C)^{-1}(C_{\pm}) \subset \CC (F_C)^{-1}(C_{\pm}) \subset C_+
\end{equation}
Clearly,
\begin{equation}\label{eqrel19new6}
\bigcap_{n \in \Z_+} \ \overline{ \bigcup_{k \ge n} (F_C)^{-k}(U)} \ \subset \ \bigcap_{n \in \Z_+} \  (F_C)^{-n}(C) \ = \ C_+.
\end{equation}

Now assume $x \in C_+$. There exists $\xx \in \S_+$ with $\xx(0) = x$. By Proposition \ref{relprop06a}
and the remark thereafter  $\om[\xx] \subset C_{\pm}$ and so if $\ep > 0$ then there exists $k(\ep) \in \Z_+$ and
$y(\ep) \in C_{\pm}$ such that $\xx(k(\ep)) \in V_{\ep}(y_{\ep})$. Letting $\ep$ tend to $0$ and going to a subsequence
we may assume that $y_{\ep} \to y \in C_{\pm}$.  So $ x \in \NN (F_C)^{-1}(y)$. If $\ep$ is small enough that $V_{\ep}(C_{\pm}) \subset U$,
then we can choose $k(\ep)$ so that $\xx(n) \in U$ for all $n \ge k(\ep)$. So for all such $n$ $x \in (F_C)^{-n}(U)$.

These prove the required reverse inclusions for (\ref{eqrel19new5}) and (\ref{eqrel19new6}).

\end{proof}

       The following concept of minimality is essentially one of those introduced in \cite{BEGK}.

\begin{prop}\label{relprop06ab} For a closed relation $F$ on $X$, let $C$ be a nonempty closed subset of $X$. The
following conditions are equivalent and when they hold we call $C$
a \emph{minimal viable subset}\index{minimal viable subset}\index{subset!minimal} for $F$ or just a minimal subset of $X$.
\begin{itemize}
\item[(i)] $C$ is viable and contains no proper viable subset.
\item[(ii)] $C$  is + viable contains no proper + viable subset.
\item[(iii)] If $A$ is a closed, nonempty subset of $C$, then $F^{-1}(A) \supset A$ if and only if $A = C$.
\item[(iv)] If $A$ is a closed, nonempty subset of $C$, then $F(A) \supset A$ if and only if $A = C$.
\item[(v)] $C$ is + viable and every infinite forward orbit sequence in $C$ is dense in $C$.
\item[(vi)] $C$ is viable and every bi-infinite orbit sequence in $C$ is dense in $C$.
\item[(vii)] $C$ is + viable and for every $\xx \in \S_+(C), \ \om[\xx] = C$.
\end{itemize}
In particular, if $C$ is a minimal viable subset for $F$, then it is a minimal viable subset for $F^{-1}$.

When $X$ itself is a minimal subset, then we say that $F$ is minimal.
\end{prop}

\begin{proof} (ii) $\Leftrightarrow$ (iii):  For a closed set $A$, $F^{-1}(A) \supset A$ if and only if $A$ is + viable.

(ii) $\Rightarrow$ (vii): For every $\xx \in \S_+(C), \ \om[\xx]$ is a viable subset of $C$.

(vii) $\Rightarrow$ (v): The closure of the orbit sequence of $\xx \in \S_+(C)$ contains $ \om[\xx]$.

(v) $\Rightarrow$ (ii): If $A$ is a nonempty + viable subset of $C$, then there exists $\xx \in \S_+(F_A) \subset \S_+(F_C)$. Since the orbit
sequence of $\xx$ is dense in $C$, it follows that $A = C$.

(vii) $\Rightarrow$ (vi): Since $\om[\xx] = C$ and $\om[\xx]$ is viable, $C$ is viable. The closure of the bi-infinite orbit
sequence of $\xx \in \S(C)$ contains $ \om[\xx]$.

(vi) $\Rightarrow$ (i): If $A$ is a nonempty  viable subset of $C$, then there exists $\xx \in \S(F_A) \subset \S(F_C)$. Since the orbit
sequence of $\xx$ is dense in $C$, it follows that $A = C$.

(i) $\Rightarrow$ (vii): $\om[\xx]$ is a viable subset of $C$.

Since viability is the same for $F$ and $F^{-1}$ it follows that $C$ is minimal for $F$ if and only if it is minimal for $F^{-1}$.
In particular, (iv) $\Leftrightarrow$ (i) follows from (iii) $\Leftrightarrow$ (i) for $F^{-1}$.

\end{proof}

The following is a version of Lemma 2.1 from \cite{KST}.

\begin{cor}\label{relprop06ac} If $F$ is minimal, then it is irreducible. There exists a dense $G_{\d}$ subset $W$ of $X$
and a homeomorphism $f : W \to W$ so that for $x \in W$, $F(x) = \{ f(x) \}$ and $F^{-1}(x) = \{ f^{-1}(x) \}$. \end{cor}

\begin{proof} If $A$ is a closed subset of $X$ and either $F(A) = X$ or $F^{-1}(A) = X$, then $A = X$ by (iii) and (iv) above.  This is
irreducibility. The rest is a special case of Theorem \ref{reltheo00ad}(c).

\end{proof}

Notice that, following the Remark after Theorem \ref{reltheo00ad}. If $f$ is a minimal homeomorphism on $X$, $x_0 \in X$ and a $A$ is a nonempty closed subset
of $X$ disjoint from the bi-infinite orbit $\{ f^n(x_0) : n \in \Z \}$ of $x_0$, then $F = f \cup [\{x_0 \} \times A]$
is a minimal closed relation and so is irreducible,
but $\pi_1 : F \to X$ is not irreducible because $f$ is a proper surjective subset of $F$.
In this case, $W = X \setminus \{ f^n(x_0) : n \in \Z \}$ with the homeomorphism
equal to  the restriction $f_W$. Notice that the closure of $f_W$ is $f$ which is a proper subset of $F$.

\begin{prop}\label{relprop06ad} Every nonempty + viable subset of $X$ contains a minimal viable subset. \end{prop}

\begin{proof}  This follows from Lemma \ref{rellem06aa} and the usual Zorn's Lemma argument.

\end{proof}

  Recall a closed relation $F$ on $X$ \emph{chain transitive}  when $\CC F = X \times X$ in which case $F$ is surjective.
    We call a closed subset $C$ a \emph{chain transitive subset}\index{subset!chain transitive} if $F_C $ on $C$ is chain transitive.
   In particular, if $C$ is a chain transitive subset, then $F_C$ is surjective and so $C$ is viable.

     \begin{prop}\label{relprop09aaa} If a closed subset $C$ satisfies $\CC F(C) \cap \CC F^{-1}(C) \subset C$, then the following hold.
         \begin{equation}\label{eqrel25}
         \CC (F_C) \ = \ (\CC F)_C, \qquad \text{and} \qquad  |\CC (F_C)| \ = \ |\CC F| \cap C.
         \end{equation}
         \end{prop}

         \begin{proof}  See \cite{A93} Theorem 4.5.

         \end{proof}

     \begin{cor}\label{relcor10aaa} The chain components of $F$ are the maximum chain transitive subsets of $X$.
     That is, if $x \in |\CC F|$, i.e. $x$ is a chain recurrent point,
     then the chain component $\CC F(x) \cap \CC(F^{-1}(x)$ is a chain
     transitive subset and every chain transitive subset of $X$ is contained in a chain component. \end{cor}

     \begin{proof} Since $ \CC (F_C) \subset (\CC F)_C$ it is clear that every chain transitive subset consists of chain recurrent points all of which
     are  $(\CC F) \cap (\CC F^{-1})$ equivalent. So every chain transitive subset is contained in a (unique) chain component. That the chain components
     are chain transitive subsets follows from Proposition \ref{relprop09aaa}.

     \end{proof}

 %   Define the \emph{extension}\index{extension} $\hat F_C$ of $F_C$:
%\begin{equation}\label{eqrel25}
%\hat F_C  \ =_{def} \  (F \cap (C \times X)) \cup 1_{X \setminus C^{\circ}}.
%\end{equation}
\vspace{.5cm}

\subsection{\textbf{Isolated  Subsets and the Conley Index}}\label{Conley}\vspace{.5cm}

 Let $C$ be a closed subset of $X$. While for a subset $A$ of $X$,
  $A^{\circ}$ denotes the interior in $X$, we will use $Int_C A$ to denote the interior of $A \subset C$ with respect to the relative topology
   of $C$.  This implies that there exists on open subset $O$ of $X$ such that $C \cap O = Int_C A$.  In particular,
   $C^{\circ} \cap O = C^{\circ} \cap Int_C A$ is a subset of $A$ open in $X$ and so is contained in
     $ A^{\circ}$. On the other hand, $A^{\circ}$ is open in
   $X$ and so in $C$. Hence, $A^{\circ} \subset Int_C A$ and it is clearly contained in $C^{\circ}$.  Thus, we have
 \begin{equation}\label{eqConley9aa}
 C \cap O \ = \ Int_C A \quad \text{and} \quad C^{\circ} \cap O \ = \ C^{\circ} \cap Int_C A \ = \ A^{\circ}.
  \end{equation}

For the Conley Index results we will follow \cite{KM}  and \cite{BM}.

  For $C$ a closed set in $X$ we define
  \begin{equation}\label{eqConley07a}
  \r_F(C)  \ =_{def} \  \overline{F(C) \setminus C}, \qquad \d_F(C)  \ =_{def} \  C \cap \r_F(C).
 \end{equation}
  Clearly, $\r_F(C)$ is the obstruction to $F$ + invariance for $C$.  That is, $C$ is $F$ + invariant if and only if $\r_F(C) = \emptyset$.
    Following
  \cite{BM} we call $\d_F$ the \emph{$F$ boundary}\index{$F$ boundary} of $C$.  Clearly,
  \begin{equation}\label{eqConley07b}
  \d_F(C) \subset \partial C, \qquad  F(C) \setminus C = \r_F(C) \setminus \d_F(C) \subset X \setminus C.
\end{equation}

  \begin{prop}\label{propConley03}  Assume $A \subset C$  are closed subsets of $X$.

\begin{enumerate}
\item[(a)] The following conditions are equivalent.

\begin{itemize}
\item[(i)] $A$ is an $F_C$ + invariant.

\item[(ii)] $F_C(A) = F(A) \cap C  \subset A$

\item[(iii)] $\r_F(A) \cap C \ =  \ \d_F(A)$.

\item[(iv)] $F(A) \setminus A \ = \ F(A) \setminus C \ \subset \ X \setminus C$.

%\item[(v)] $ \hat A \ = \ A \cup (F(A) \setminus C)$.
%
%\item[(vi)] $ \hat A \cap C \ = \ A$.
\end{itemize}

\item[(b)] Assume that $A$ is a closed, $F_C$ + invariant subset of $C$.

  We have
  \begin{equation}\label{eqConley07c}
   \d_F(A) \  \subset \ F(A) \cap \partial C \ \subset \ A \cap \partial C, \hspace{3cm}
   \end{equation}
        \begin{equation}\label{eqConley07d}
    \r_F(A) \ \subset \ \r_F(C) \qquad  \d_F(A) \ = \ C \cap \r_F(A) \ \subset \ \d_F(C).\hspace{1cm}
   \end{equation}
%       \begin{equation}\label{eqConley08}
% F(A) \setminus A \ = \ F(A) \setminus C \qquad \hat A \ = \ A \cup (F(A) \setminus C). \hspace{2cm}
%  \end{equation}
  %  \begin{equation}\label{eqConley07e}
%   \hat A \setminus C^{\circ} \ = \ (A \cap \partial C) \cup (F(A) \setminus C) \ = \ (A \cap \partial C) \cup \r_F(A).
%\end{equation}
%
%
%
%
%   The set $\hat A $ is a closed, $\widehat{F_C}$ + invariant subset of $\hat C$ and $x \in \hat A$ implies
%      \begin{equation}\label{eqConley07f}
%     % x \in \hat A \quad \Longrightarrow \quad
%     \widehat{F_C}(x) \ = \ \begin{cases} \{ x \} \cup \widehat{F_A}(x)  \ \text{for} \ x \in A \cap (\d_F(C) \setminus \d_F(A)), \\
%     \qquad \widehat{F_A}(x)  \ \quad \text{otherwise}. \end{cases}
%      \end{equation}
%
%
%  If $A$ is $F_C$ invariant, then $\hat A$ is $\widehat{F_C}$ invariant and \\ $\hat A = A \cup F(A) = \widehat{F_C}(A) =  F(A)$.
%
%  If $A$ is $\G (F_C)$ + invariant, then $\hat A$ is $\G (\widehat{F_C})$ + invariant. In that case it is $\G (\widehat{F_C})$
%  invariant if and only if it is $\widehat{F_C}$ invariant.

  \end{enumerate}
  \end{prop}

  \begin{proof}  (a)% From (\ref{eqConley03}) it is clear that $A \cup \widehat{F_C}(A) = A \cup F(A) = A \cup \r_F(A) = \hat A$
%  for any closed subset $A$ of $C$.

  (i) $\Leftrightarrow$ (ii) and (ii) $\Rightarrow$ (iv) are obvious.

  (iv) $\Rightarrow$ (iii): In any case, $\d_F(A) = \r_F(A) \cap A \subset \r_F(A) \cap C$. (iv) implies that $\r_F(A) \setminus C =
  \r_F(A) \setminus A$ and so $\r_F(A) \cap C  \subset  A$.

  (iii) $\Rightarrow$ (ii): $\r_F(A) \cup A = F(A) \cup A$ and $\d_F(A) \cup A = A$. So (iii) implies $(F(A) \cup A) \cap C = A$.
%
%  (iv) $\Rightarrow$ (v): Because $\hat A = A \cup (F(A) \setminus A)$.
%
%  (v) $\Rightarrow$ (vi): $A \cap C = A$.
%
%  (iii) $\Rightarrow$ (vi): In any case, $A \subset \hat A \cap C$. (iii) implies that $\hat A \cap C = (A \cup \r_F(A)) \cap C \subset A$.
%
%  (vi) $\Rightarrow$ (ii): (vi) says $(A \cup F(A)) \cap C = A$ which implies (ii) because $A \cap C = A$.

  (b) Now assume that $A$ is $F_C$ + invariant.

  From (a) (iv) $F(A) \setminus A = F(A) \setminus C \subset F(C) \setminus C$.
  Hence, $ \r_F(A) \subset \r_F(C)$ and from (a) (iii)
  $\d_F(A) = C \cap \r_F(A) \subset C \cap \r_F(C) = \d_F(C)$ (i.e. (\ref{eqConley07d})).

 Next $\d_F(A) \subset  \d_F(C) \subset \partial C$ implies that  $ \d_F(A) \subset F(A) \cap \partial C$ and the latter
  is contained in $A \cap \partial C$ by $F_C$ + invariance of $A$, proving (\ref{eqConley07c}).

   \end{proof}

Recall (\ref{eqrel17C})
\begin{equation}\label{eqConley01}\begin{split}
C_-  \ =_{def} \  \bigcap_{n=1}^{\infty} \ (F_C)^n(C) \ = \ \{ x \in C : (F_C)^{-n}(x) \not= \emptyset \ \ \text{for all} \ n \in \Z_+ \}\\
C_+   \ =_{def} \  \bigcap_{n=1}^{\infty} \ (F_C)^{-n}(C)\ = \ \{ x \in C: (F_C)^{n}(x) \not= \emptyset \ \ \text{for all} \ n \in \Z_+ \}\\
 C_{\pm}  \ =_{def} \  C_- \ \cap \ C_+,\hspace{5cm}
 \end{split}\end{equation}
When $F$ and $C$ are closed, $C_+$ is a repeller and $C_-$ an attractor for $F_C$ and $C_{\pm}$ is
the maximum viable subset of $C$ by Proposition \ref{relprop06aaa}(d).
% In particular, $C_-$ is $\G (F_C)$ invariant because it is $F_C$ invariant and $\CC (F_C)$ + invariant. Similarly, $C_+$ is
% $\G (F_C)^{-1}$ invariant. So Proposition
% \ref{propConley03}(b) applied with $A = C_-$ yields
%
%   \begin{equation}\label{eqConley07h}
%   \widehat{C_-} \ = \ F(C_-) \ = \ \widehat{F_C}(C_-) \ = \ C_- \cup (F(C_-) \setminus C)
%   \end{equation}
%   and this set is  $\G (\widehat{F_C})$ invariant.
%
%
%
%  \begin{cor}\label{corConley04} $\G (\widehat{F_C}) \subset \O (\widehat{F_C}) \cup (C_+ \times F(C_-))$ and
%  \begin{align}\label{eqConley07i}\begin{split}
%  \widehat{F_C} \cap (C \times C) \ &= \ F_C \cup 1_{\d_F(C)}, \\
%   \O(\widehat{F_C}) \cap (C \times C) \ &= \ \O (F_C) \cup 1_{\d_F(C)}, \\
%  \G (\widehat{F_C}) \cap (C \times C) \ &= \ \G (F_C) \cup 1_{\d_F(C)}.
%  \end{split}\end{align}
%  \end{cor}
%
%  \begin{proof} From (\ref{eqConley06}) and (\ref{eqConley04}) we have
%  $\G (\widehat{F_C}) \subset \widehat{F_C} \cup (\widehat{F_C} \circ \O (F_C)) \cup \widehat{F_C} \circ (C_+ \times C_-)$.
%
%  By (\ref{eqConley05}) $\widehat{F_C} \cup (\widehat{F_C} \circ \O (F_C)) = \O \widehat{F_C}$ and since $F(C_-) = \widehat{F_C}(C_-)$ by Porposition
%  \ref{propConley03} it follows that $\widehat{F_C} \circ (C_+ \times C_-) = C_+ \times F(C_-)$.
%
%  That $\widehat{F_C} \cap (C \times C) =  F_C \cup 1_{\d_F(C)}$ is clear from (\ref{eqConley03}).
%  The other two equations of (\ref{eqConley07i}) then follow easily from
%  Proposition \ref{propConley02}.
%
%  \end{proof}

Recall that, from  (\ref{eqrel19}) we have
  \begin{equation}\label{eqConley04}
  \O (F_C) \ \subset \ \G (F_C) \ \subset \ \CC (F_C) \ \subset \ \O (F_C) \ \cup \ (C_+ \times C_-). \hspace{1cm}
  \end{equation}

    \begin{prop}\label{propConley06}(a) Let $K$  be a closed subset of $X$ such that\\ $K \cap  C_+ = \emptyset$.
 If $K \subset C$, then $K \cup \CC(F_C)(K) = K \cup \G (F_C)(K) = K \cup \O (F_C)(K)$ is a closed $\CC(F_C)$ + invariant subset of $C$ which is
  disjoint from $C_+$.
%
%   \item[(ii)] $K \cup \G (\widehat{F_C})(K) = K \cup \O (\widehat{F_C})(K)$ is a closed $\G (\widehat{F_C})$ + invariant subset of $X$ which is
%  disjoint from $C_+$. \end{itemize}\vspace{.25cm}

  (b) If $A$  a closed, $F_C$ + invariant subset of $C$ such that $A \cap  C_{\pm} = \emptyset$, then $A \cap  C_{+} = \emptyset$.

 \end{prop}

 \begin{proof} (a) Because $K$ is disjoint from $C_+$, (\ref{eqConley04}) implies that $K \cup \CC(F_C)(K) = K \cup \G (F_C)(K) = K \cup \O (F_C)(K)$ and
 it is closed because $\CC (F_C)$ and $K$ are closed. Since $C_+$ is $F_C^{-1}$ invariant, it follows that $K \cup \O (F_C)(K)$ is disjoint from
 $C_+$.

  %We prove (ii) as (i) follows from (ii) because
%  $\G (\widehat{F_C}) \cap (C \times C) \ = \ \G (F_C) \cup 1_{\d_F(C)}$ by Corollary \ref{corConley04}.
%
% (ii): Because $K \cap C_+ = \emptyset$, Corollary \ref{corConley04} implies that
% $\G (\widehat{F_C})(K) = \O (\widehat{F_C})(K)$.  From transitivity of $\G (\widehat{F_C})$ it follows
% that $K \cup \G (\widehat{F_C})(K)$ is a $\G (\widehat{F_C})$ + invariant subset.
%
% Now suppose $y \in C_+ \cap \O(\widehat{F_C})(x)$
% for some $x \in K$, then by Proposition \ref{propConley02} $x \in C$ and
% $y \in C_+ \cap \O(F_C)(x)$. Because $C_+$ is $F^{-1}_C$ + invariant, it follows that $x \in C_+$. That is,$K \cup \O(\widehat{F_C})(K)$
% meets $C_+$ only when $K$ meets $C_+$. Contrapositively, $K \cap C_+ = \emptyset$ implies $(K \cup \O(\widehat{F_C})(K)) \cap C_+ = \emptyset$.

 (b): If $x \in A \cap C_+$, then there exists a solution path $\xx \in \S_+(F_C)$ with $\xx(0) = x$. Because $A$ is $F_C$ + invariant,
 $\xx(i) \in A$ for all $i \in \Z_+$. Because $A$ is closed, it follows that $\om[\xx] \subset A$. By Proposition \ref{relprop06a}
 $\om[\xx]$ is a nonempty subset of $C_{\pm}$. Hence, $A \cap C_{\pm} \not= \emptyset$.

 \end{proof}

\begin{df}\label{dfConley05a} Let $C$ be a closed subset of $X$ and $F$ be a closed relation on $X$.

(a) The set $C$ is  called an \emph{isolating neighborhood}
\index{isolating neighborhood} when $C_{\pm} \subset C^{\circ}$, i.e. its maximum viable subset is contained
in its interior. In that case, the viable set $A = C_{\pm}$ is called an
\emph{isolated viable set}\index{isolated viable set}\index{viable set!isolated}.

 $C$ is called a \emph{simple isolating neighborhood}  \index{isolating neighborhood!simple} when for
every $x \in \partial C = C \setminus C^{\circ}$ either $F(x) \cap C = \emptyset$ or $F^{-1}(x) \cap C = \emptyset$.

(b) The set $C$ is  called a \emph{- isolating neighborhood} (or a \emph{+ isolating neighborhood})
 when $C_{-} \subset C^{\circ}$ (resp. $C_{+} \subset C^{\circ}$). In that case, the - viable set $C_{-}$ is called an
\emph{isolated - viable set} (resp. the + viable set $C_{+}$ is called an
\emph{isolated + viable set}).

 $C$ is called a \emph{simple - isolating neighborhood}   (or a \emph{simple + isolating neighborhood}) when for
every $x \in \partial C, \ F^{-1}(x) \cap C = \emptyset$ (resp. for
every $x \in \partial C \ F(x) \cap C = \emptyset$) . \end{df}
\vspace{.5cm}

Clearly, a closed set $C$ is an isolating neighborhood for $F$ if and only if no bi-infinite $F$ orbit which is contained in $C$  meets
$\partial C = C \setminus C^{\circ}$. It follows that a simple isolating neighborhood is, indeed, an isolating neighborhood. Similarly, for
$\pm$ isolating neighborhoods.

\begin{prop}\label{propConley05b} For $C$  a closed subset of $X$, the following conditions are equivalent.
\begin{itemize}
\item[(i)]  $C$ is a simple isolating neighborhood for $F$.
\item[(ii)] $\partial C \subset F^*( X\setminus C) \cup (F^{-1})^*( X\setminus C).$
\item[(iii)] $\partial C \cap F(C) \cap F^{-1}(C) = \emptyset$.
\item[(iv)] $F_C(C) \cap F_C^{-1}(C) \subset C^{\circ}.$
\end{itemize}

The following conditions are equivalent.
\begin{itemize}
\item[(i)]  $C$ is a simple - isolating neighborhood for $F$.
\item[(ii)] $\partial C \subset  (F^{-1})^*( X\setminus C).$
\item[(iii)] $\partial C \cap  F(C) = \emptyset$.
\item[(iv)] $F_C(C)  \subset C^{\circ}.$
\end{itemize}
\end{prop}

\begin{proof}  It is clear that for $x \in X$
\begin{equation}\label{eqConleysimple1}
F^{-1}(x) \subset X \setminus C \quad \Leftrightarrow \quad x \in (F^{-1})^*(X \setminus C) \quad \Leftrightarrow \quad x \not\in F(C).
\end{equation}
From this and the analogous conditions for $F^{-1}$, the equivalences are clear.

\end{proof}

For an isolated viable set we define the associated \emph{stable subset} and \emph{unstable subset}\index{subset!stable}\index{subset!unstable}

\begin{theo}\label{theoConley07new} Let  $C$ be an isolating neighborhood for
$C_{\pm}$, i.e. $C_{\pm} \subset C^{\circ}$.

Define
\begin{align}\label{eqrel19new3}\begin{split}
W^s(C_{\pm}) \ &=_{def} \ \bigcup_{n \in \Z_+} \ F^{-k}(C_+), \\
W^u(C_{\pm}) \ &=_{def} \ \bigcup_{n \in \Z_+} \ F^k(C_-).
\end{split}\end{align}
$W^s(C_{\pm})$ is a + viable subset for $F$, $W^u(C_{\pm})$ is a - viable subset for $F$ and

\begin{align}\label{eqrel19new4}\begin{split}
x \in W^s(C_{\pm}) \  &\Longleftrightarrow \  \text{there exists} \  \xx \in \S_+(F), \ \text{with} \ \xx(0) = x, \ \om[\xx] \subset C_{\pm} \\
x \in W^u(C_{\pm}) \ &\Longleftrightarrow \  \text{there exists} \  \xx \in \S_-(F), \ \text{with} \ \xx(0) = x, \ \a[\xx] \subset C_{\pm}.
\end{split}\end{align}

\end{theo}

\begin{proof}  Using induction, Proposition \ref{relprop06aaa} (d) and Lemma \ref{rellem06aa} (b) it follows that  $W^s$ is + viable.
If $x \in W^s$ there exists $k \in \Z_+$ and $\xx \in \S([0,k],F)$ with $\xx(0) = x$ and $\xx(k) \in C_+$. So there exists
$\yy \in \S_+(F_C)$ with $\yy(0) = \xx(k)$. Extend $\xx$ by $\xx(i) = \yy(i-k)$ for $i \ge k$.  Thus, $\xx \in \S_+$ with
$\xx(0) = x$ and $\om[\xx] = \om[\yy] \subset C_{\pm}$ by Proposition \ref{relprop06a}.

If $\xx \in \S_+(F)$ with $\om[\xx(0)] \subset C_{\pm}$, then since $C_{\pm} \subset C^{\circ}$, there exists $k \in \Z_+$ such that
$\xx(n) \in C$ for all $n \ge k$.  Truncate to define $\yy \in \S_+(F_C)$ by
$\yy(i) = \xx(i+k)$. It follows that $\yy(i) \in C_+$ for all $i \in \Z_+$ and so $\xx(0) \in W^s$.

For $W^u$ use $F^{-1}$ as usual.

\end{proof}

From (\ref{eqrel19new4}) it is clear that the stable and unstable sets for an isolated viable set $A = C_{\pm}$ do not depend upon the choice of
of the isolating neighborhood $C$.\vspace{.25cm}

 For $C$ an isolating neighborhood for a viable subset $A $, we call a pair $(P_1, P_2)$ of
 closed subsets of $X$  an \emph{index pair} \index{index pair}  rel $C$
 for $A$ when
 the following conditions are satisfied:
 \begin{itemize}
 \item[(i)] $P_2 \ \subset \ P_1 \ \subset \ C$.

 \item[(ii)] $P_1$ and $P_2$ are $F_C$ + invariant.

 \item[(iii)] $A = C_{\pm} \ \subset \ P_1^{\circ} \setminus P_2 $..

  \item[(iv)] $P_1 \setminus P_2 \ \subset  \ C^{\circ}$, or, equivalently, $P_1 \cap \partial C \subset P_2$.
 \end{itemize}
 We will sometimes consider the strengthening of (iv)
  \begin{itemize}
 \item[(iva)] $\overline{P_1 \setminus P_2} \subset  \ C^{\circ}$, or,
 equivalently, $P_1 \cap \partial C$ is contained in the $P_1$ interior of $P_2$.
 \end{itemize}

 In \cite{BM} such a pair is called a weak index pair and the authors impose the additional condition
 $\d_F(P_1) \subset P_2$ which we will see is redundant (and indeed
  the authors themselves proved it so). We will sometimes consider a strengthening of this condition

  \begin{itemize}
 \item[(v)] $P_1 \cap \partial C \ = \ \d_F(P_1)$.
  \end{itemize}
%  Notice that by (\ref{eqConley07f}) if (ii) and (v) hold, then $\widehat{F_C} \cap (\hat P_1 \times \hat P_1) = \widehat{F_{P_1}} $.

 We call a pair $(P_1, P_2)$ of closed subsets of $X$  an \emph{index pair} \index{index pair}
 for a viable set $A$ when there exists an isolating neighborhood $C$ for $A$ such that   $(P_1, P_2)$ is an index pair   rel $C$
 for $A$.

The following is a version of Theorem 4.12 of \cite{BM}.

 \begin{theo}\label{theoConley07} Assume that $C$ is an isolating neighborhood for a viable set $A$.
 \begin{enumerate}
 \item[(a)] If $(P_1,P_2)$ is an index pair rel $C$ for $A$, then $P_1$ and $P_2$ are $\CC (F_C)$ + invariant sets with
 $C_- \subset P_1$ and $C_+ \cap P_2 = \emptyset$. In addition, $C \cap \r_F(P_1) = \d_F(P_1) \subset P_1 \cap \partial C \subset P_2$.

  \item[(b)] If $U$ and $V$ are open subsets of $X$ with $C_- \subset U$,  and $C_{\pm} \subset V \subset C$, then there exists
  a simple isolating neighborhood $C_0$ and
 an index pair $(P_1,P_2)$  rel $C$ for $A$ with $P_1 \subset U$, and  $\overline{P_1 \setminus P_2} \subset C_0 \subset V$.
 In particular, (iva) holds for $(P_1,P_2)$.

 \end{enumerate}
 \end{theo}

 \begin{proof} (a) If $x \in C_-$, there exists a solution path $\xx : [-\infty,0] \to C$ with $\xx(0) = x$ and by
 Proposition \ref{relprop06a} applied to (\ref{eqrel22c}) we have that $\a(\xx) \subset C_{\pm} = A $ which is
 contained in $ P_1^{\circ}$ by condition (iii).
 It follows that for sufficiently large $k \in \Z_+$, $\xx(-k) \in P_1$. Since $P_1$ is + $F_C$ invariant by (ii) it
 follows that $\xx(-k) \in P_1$ for all $k \in \Z_+$ in particular, $x \in P_1$.

 It follows from (\ref{eqConley04}) that $\CC (F_C)(P_1) = \O (F_C)(P_1) \cup C_- \subset P_1$ since $P_1$ is $F_C$ + invariant.

 Condition (iii) also implies that $C_{\pm}$ is disjoint from $P_2$.  Since $P_2$ is $F_C$ + invariant it is disjoint from
 $C_+$ by Proposition \ref{propConley06}(b).

  It follows from (\ref{eqConley04})again that $\CC (F_C)(P_2) = \O (F_C)(P_2)  \subset P_2$ since $P_2$ is  $F_C$ + invariant.

  Since $P_1 \setminus P_2$ is contained in $C^{\circ}$ it contains no point of $P_1 \cap \partial C$.
  By (\ref{eqConley07c}) $\d_F(P_1) \subset P_1 \cap \partial C$.
  Hence,by (\ref{eqConley07c}) $C \cap \r_F(P_1) = \d_F(P_1) \subset P_1 \cap \partial C \subset P_2$.

  (b) This follows from Theorem \ref{reltheoConley} applied to $F_C$ on $C$. We review and sharpen the proof.

Recall that for a subset $A$ of $X$,
  $A^{\circ}$ denotes the interior in $X$, we  use $Int_C A$ to denote the interior of $A \subset C$ with respect to the relative topology
   of $C$.  From (\ref{eqConley9aa}) we have
 \begin{equation}\label{eqConley9aax}
  C^{\circ} \cap Int_C A \ = \ A^{\circ}.
  \end{equation}

  Choose $W_-, W_+$ relatively open subsets of $C$ such that
  \begin{equation}\label{eqConley10}
  C_+ \ \subset \ W_+, \qquad  C_- \ \subset W_- \ \subset \ U, \qquad W_+ \cap W_- \ \subset \ V,
   \end{equation}
  and so, of course, $C_{\pm} = C_+ \cap C_- \subset W_+ \cap W_-$ and $W_+ \cap W_-$ is open in $X$ because $V$ is.

     Because $C_-$ is an attractor, Theorem \ref{reltheo04} implies that
     there exists $P_1$ an inward for $F_C$ closed neighborhood (with respect to $C$)
   of $C_-$ with $P_1 \subset W_-$.  That is, $C_- \subset F_C(P_1) \subset Int_C P_1 \subset W_-$.

 Similarly, because $C_+$ is a repeller, Theorem \ref{reltheo04} implies that
   there exists $Q_1$ an inward for $F_C^{-1}$ closed neighborhood (with respect to $C$)
   of $C_+$ with $Q_1 \subset W_+$.

        Hence, $C \setminus Int_C(Q_1)$ is $F_C$ inward.

      Let $P_2 = P_1 \cap (C \setminus Int_C(Q_1)) = P_1 \setminus Int_C(Q_1) $

   As it is the intersection of two $F_C$ inward sets,
   $P_2$ is $F_C$ inward.

   Observe that $P_1 \cap Q_1 \subset W_+ \cap W_- \subset V \subset C^{\circ}$.
   By (\ref{eqConley9aax}) $P_1^{\circ} \cap Q_1^{\circ} = (P_1 \cap Q_1)^{\circ}$
   is the same as $Int_C(P_1 \cap Q_1) = Int_C(P_1) \cap Int_C(Q_1) = Int_C(P_1) \setminus P_2 = P_1^{\circ} \setminus P_2$.

   Thus, $P_1^{\circ} \setminus P_2 = Int_C(P_1) \cap Int_C(Q_1)\supset C_- \cap C_+ = C_{\pm}$.

   Next,  $P_1 \setminus Q_1 \subset P_2$ and so
     $P_1 \setminus P_2  \subset P_1 \cap Q_1 \subset W_- \cap W_+$. Since $P_1$ and $Q_1$ are closed,
      $\overline{P_1 \setminus P_2}  \subset P_1 \cap Q_1 \subset W_+ \cap W_-  \subset  V$.

      Thus, (i)-(iv) and (iva) hold for $(P_1, P_2)$.

      Finally, let $C_0 = P_1 \cap Q_1$ so that   $\overline{P_1 \setminus P_2}  \subset C_0 \subset  V$.

      Since $C_{\pm} \subset P_1^{\circ} \setminus P_2 \subset (C_0)^{\circ}$ and $C_0 \subset C$, we have $(C_0)_{\pm} = C_{\pm}$.

      Because the closed set $C_0$ is contained in the open set $V$, $\partial C_0 = \partial_C(C_0) = C_0 \setminus Int_C(C_0)$.
      It follows that $\partial C_0 \subset \partial_C(P_1) \cup \partial_C(Q_1)$. Because $P_1$ is inward for $F_C$, it follows that
      $F_{C_0}(C_0) \subset F_C(P_1)$ is disjoint from $\partial_C(P_1)$. Similarly, $Q_1$ inward for
      $F_C^{-1}$ implies that $F_{C_0}^{-1}(C_0) \subset F_C^{-1}(Q_1)$ is disjoint from $\partial_C(Q_1)$. Thus, $F_{C_0}(C_0) \cap  F_{C_0}^{-1}(C_0)$
      is disjoint from  $\partial_C(P_1) \cup \partial_C(Q_1)$ and so from $\partial C_0$.
      That is, $F_{C_0}(C_0) \cap  F_{C_0}^{-1}(C_0) \subset C_0^{\circ}$. Hence, by Proposition \ref{propConley05b} $C_0$ is a simple isolating
      neighborhood.

  \end{proof}

   \begin{theo}\label{theoConley08}  A pair $(P_1, P_2)$  of closed subsets of $X$  is an index pair,
   i.e. there exists a viable set $A$ and a closed neighborhood $C$ of $A$ such that
   $(P_1, P_2)$ is an index pair rel $C$ for $A$ if and only if
 the following conditions are satisfied:
 \begin{itemize}
 \item[(i$'$)] $P_2 \subset P_1$.

 \item[(ii$'$)] $P_2$ is $F_{P_1}$ + invariant.

 \item[(iii$'$)] $(P_1)_{\pm} \subset  P_1^{\circ} \setminus P_2 $.

  \item[(iv$'$)] $\d_F(P_1) \subset P_2$.
 \end{itemize}

 Furthermore, $C$ can be chosen so that (v) holds.  In addition,
 $C$ can be chosen so that (iva) hold if and only if
  \begin{itemize}
   \item[($iva'$)] $\d_F(P_1)$ contained in the $P_1$ interior of $ P_2$.
 \end{itemize}
 \end{theo}

 \begin{proof}  It is clear from Theorem \ref{theoConley07}(a) that these conditions are necessary and,
 in particular, condition (iva) requires (iva$'$).

To construct the required $C$, we first find $C_1$ so that $P_1 \subset \subset C_1$ and $(P_1)_{\pm} = (C_1)_{\pm}$.
Let $\{ C_n \} $ be a decreasing sequence of closed sets with
$P_1 \subset C_n^{\circ}$
and $ P_1 = \bigcap_n \{ C_n \}$.

 If the condition fails, then for each $n$ there exists $\xx^n \in \S(F_{C_n})$ which is not entirely contained in $P_1$.
 By translation we may assume $\xx^n(0) \not\in P_1$. By going to a subsequence we obtain a limit $\xx \in S(F_{P_1})$ with
 $\xx(0) \in \partial P_1$ contradicting the assumption that $(P_1)_{\pm} \subset  P_1^{\circ}$.

 Fix such a $C_1$ and choose $\ep > 0$ so that for all $x \in P_1$, $V_{\ep}(x) \subset C_1$. The closed set $\r_F(P_1)$ satisfies
 $P_1 \cap \r_F(P_1) = \d_F(P_1) \subset P_2$. Let $ C $ be the closure of the set
  \begin{align}\label{eqConley13} \begin{split}
\{ \ \  y \in X :\ \  &\text{there exists} \ \ x \in P_1 \ \ \text{such that } \\  d(y,x) \ &\leq \ \frac{1}{2} \min[\ep, d(x,\r_F(P_1))]\ \  \}.
 \end{split}\end{align}
   For $y \in P_1$ we can use $x = y$ which shows that $P_1 \subset C$. Notice next that the definition of
   $\ep $ implies that $C \subset C_1$ and so $C_{\pm} = (P_1)_{\pm}$ and then (iii$'$) implies
    that $C$ is an isolating neighborhood for $A = (P_1)_{\pm}$.  By definition, $x \in C^{\circ}$ for all $x \in P_1 \setminus \d_F(P_1)$
    because for such $x$, $d(x,\r_F(P_1)) > 0$.

  So conditions (i), (iii) and (iv) hold by ($iv'$) and $P_1 \cap \partial C \subset \d_F(P_1)$.  Furthermore, (iva) follows from (iva$'$).

    Now suppose that $y \in C \cap \r_F(P_1)$. There is a sequence of pairs $\{ (x_n,y_n) \}$ with
    $x_n \in P_1$, $d(y_n,x_n) \leq \frac{1}{2} \min[\ep, d(x_n,\r_F(P_1))]$
    for all $ n$ and $ \{ y_n \} \to y$.  By going to a subsequence we may assume $ \{ x_n \} \to x \in P_1$
    and so $d(y,x) \leq \frac{1}{2} \min(\ep, d(x,\r_F(P_1)))$.
    Since $y \in \r_F(P_1), \ \ d(y,x) \geq  d(x,\r_F(P_1))$.  This can only happen if $d(y,x) = d(x,\r_F(P_1)) = 0$, i.e. $y = x$ and so $y \in \d_F(P_1)$.
%
%    If $x \in P_1$ and $y \in F_C(x) \subset  C \cap F(P_1)$, then if $y$ were not in $P_1$, it would be in $\r_F(P_1)$ and so,
%    by the argument of the preceding paragraph, in $\d_F(P_1) \subset P_1$.
    So $P_1$ is $F_C$ + invariant by Proposition \ref{propConley03} (a).

    If $x \in P_2$ and $y \in F_C(x)$, then $y \in P_1$, since $P_1$ is $F_C$ + invariant, and so $y \in F_{P_1}(x) \subset P_2$ by (ii$'$. That is,
    $P_2$, too, is $F_C$ + invariant. This completes the proof of (ii).

    From $F_C$ + invariance and (\ref{eqConley07d}), we have $\d_F(P_1) \subset \d_F(C) \subset \partial C$. That is,
  $\d_F(P_1) \subset P_1 \cap \partial C$. As we proved the reverse inclusion above it follows that (v) holds.

    \end{proof}

    \begin{theo}\label{theoConley09} For a closed subset $P_1$ of $X$, there exists $P_2$ such that $(P_1,P_2)$
    is an index pair if and only if the following conditions
     hold.
      \begin{itemize}

 \item[(i$''$)] $(P_1)_{\pm} \subset  P_1^{\circ} $, i.e. $P_1$ is an isolating neighborhood for $(P_1)_{\pm}$.

  \item[(ii$''$)] $\d_F(P_1) \cap (P_1)_{+} = \emptyset$.
 \end{itemize}

 If $P_0$ is any closed subset of $P_1$ such that $\d_F(P_1) \subset P_0$ and $P_0 \cap (P_1)_{+} = \emptyset$, then with
 $P_2 = P_0 \cup \CC (F_{P_1})(P_0) = P_0 \cup \O (F_{P_1})(P_0)$, the pair $(P_1,P_2)$ is an index pair. In particular,
 $\d_F(P_1) \cup \O (F_{P_1})(\d_F(P_1))$ is the smallest such
 set $P_2$.
 \end{theo}

 \begin{proof} For $P_0$ disjoint from $(P_1)_+, \  \CC (F_{P_1})(P_0) = \O (F_{P_1})(P_0)$ by (\ref{eqConley04}) and it is disjoint from
 $(P_1)_+$ by Proposition \ref{propConley06}.

  Condition (i$''$) is clearly necessary and if $(P_1,P_2)$ is an index pair rel $C$ then
 Theorem \ref{theoConley07} (a) implies (ii$''$) and so
 necessity follows from Theorem \ref{theoConley08}.

 Now assume that $P_0$ and $P_2$ are as described in the statement. Conditions (i$'$),(ii$'$) and (iv$'$)  are clear.
 From Proposition \ref{propConley06}(a) it follows that $P_2 \cap (P_1)_+ = \emptyset$ and this together with (i$''$) implies (iii$'$).

 \end{proof}

 Thus, a closed set $P_1$ which satisfies (i$''$) and (ii$''$) is a special sort of isolating neighborhood which
 we will call an \emph{isolating neighborhood of index type}
 \index{isolating neighborhood!of index type}. From Theorem \ref{theoConley07} it follows that every
 isolated viable subset admits a neighborhood base of
 isolating neighborhoods of index type.

 \begin{prop}\label{propConley09a} If $P_1$ is a simple isolating neighborhood, then it is an isolating neighborhood of index type. \end{prop}

 \begin{proof} Suppose that $x \in  \d_F(P_1)$ and so $x \in \partial P_1$.  There exists a sequence $\{ z_n \}$ in $P_1$ and $\{ x_n \}$ in $X \setminus P_1$
 such that $x_n \to x$ and $x_n \in F(z_n)$ for all $n$. We may assume $\{ z_n \}$ converges to $z \in P_1$ so the $z \in F_{P_1}^{-1}(x)$.

 If $P_1$ is a simple isolating neighborhood, then it must be that $F_{P_1}(x) = \emptyset$ and so $x \not\in (P_1)_+$.

 \end{proof}

 We now consider - isolating neighborhoods.

 \begin{theo}\label{theoConley09a}(a) Assume that $C$ is a - isolating neighborhood, i.e. $C_- \subset C^{\circ}$ and let $U$ be an open set
 with $C_- \subset U \subset C$. There exists a closed set $P_1 \subset U$ such that $C_- \ \subset \ F_C(P_1) \ \subset \ P_1^{\circ}$. In particular,
 $P_1$ is an inward set for $F_C$ with associated $F_C$ attractor $C_-$. Furthermore, $(P_1, \emptyset)$ is an index pair for $C_{\pm}$ rel $C$ and
 $\r_F(P_1) \subset X \setminus C$.

 (b) If $C$ a closed subset with $(C)_{\pm} \subset C^{\circ}$, then
  $(C, \emptyset)$ is an index pair if and only if
 $\d_F(C)  = \emptyset$, i.e. $F(C) \setminus C$ is a closed set.
In that case, there exists a closed set $C_1$ with $C \subset C_1^{\circ}$ such that
 $C_{\pm} = (C_1)_{\pm}$ and $C_- = (C_1)_-$. In particular, $C_-$ is a - isolated set with
 - isolating neighborhood $C_1$.

  %
%  (c) For a closed set $C$, $C_{\pm} = C_-$ if and only if $C_{\pm}$ is $F_C$ invariant.
 \end{theo}

 \begin{proof} (a) $C_-$ is the maximum attractor for $F_C$ and $U$ is a neighborhood of $C_-$ with $U \subset C$. Hence, there exists
 $P_1$ an $F_C$ inward set with $C_- \subset P_1 \subset U$ and since $P_1 \subset C$, $C_-$ is the associated attractor for $P_1$. To say that
 $P_1$ is $F_C$ inward is to say that $F_C(P_1)$ is contained in the $C$ interior of $P_1$.  Since $P_1$ is contained in the open set $U \subset C$
 it follows that the $C$ interior of $P_1$ is $P_1^{\circ}$. Clearly, $(P_1, \emptyset)$ satisfies (i)-(iii) and (iva).  Since
 $\d_F(P_1) \subset P_2 = \emptyset$ it follows that $\r_F(P_1) \subset X \setminus C$.

 (c)  As remarked above, if $(C, \emptyset)$ is an index pair, then  $\d_F(C) \subset P_2 = \emptyset$.  On the other hand, if
  $\d_F(C) = \emptyset$, then using $P_0 = \emptyset$, Theorem \ref{theoConley09} implies that $(C, \emptyset)$ is an index pair.

 Assume that $\d_F(C) = \emptyset$ and so there exists $\ep > 0$ so that $\r_F(C) = F(C) \setminus C$ has distance greater than $\ep$ from
from $C$. By the initial step of the proof of Theorem \ref{theoConley08}, there exists $C_1$ a closed neighborhood of $C$ which is contained in
the $\ep$ neighborhood of $C$ and such that $C_{\pm} = (C_1)_{\pm}$.  By choice of $\ep$, $\ C_1 \cap \r_F(C) = \emptyset$. Clearly,
$C_- \subset (C_1)_-$.

Now let $x \in (C_1)_-$. There exists $\xx \in \S([-\infty,0],C_1)$ with $\xx(0) = x$. By Proposition \ref{relprop06a} and the remark thereafter,
$\a(\xx) \subset (C_1)_{\pm} = C_{\pm} \subset C^{\circ}$. Hence, there exists $k > 0$ such that  $\xx(-i) \in C$ for all $i \ge k$. Now if
$i > 0$ and $\xx(-i) \in C$, then $\xx(-(i-1)) \in C_1 \cap F(\xx(i)) \subset C_1 \cap F(C)$ and this is contained in $C$ because $C_1$ is
disjoint from $\r_F(C)$. Hence, by induction, $\xx(-i) \in C$ for all $i \in \Z_+$.  Hence, $\xx \in \S([-\infty,0],C)$ and so $x = \xx(0) \in C_-$.

Because $(C_1)_- = C_- \subset C \subset (C_1)^{\circ}$ it follows that $C_1$ is a - isolating neighborhood for $(C_1)_- = C_-$.

\end{proof}

For + isolating neighborhood subsets we apply the - isolating results to $F^{-1}$.

It can happen that $\d_F(C) = \emptyset$ but that $C$ is not $F$ + invariant, i.e. $\r_F(C)$, while disjoint from $C$, is nonempty.
We obtain stronger results when $\r_F(C) = \emptyset$.

 \begin{theo}\label{theoConley09aa} Assume that $Dom(F) = X$.

 (a) Assume that $C$ is an $F$ + invariant subset of $X$ so that $\emptyset = \r_F(C) = \d_F(C)$.
 If $C$ is isolating neighborhood of $A = C_{\pm}$, then
 $A = C_-$ is an attractor for $F$ and $(C, \emptyset)$ is index pair for $A$.

 (b) If $C$ is an inward set for $F$ with associated attractor $A$, then $C$ is a simple isolating neighborhood for $C_-  = C_{\pm} = A$.
 \end{theo}

 \begin{proof} (a) Recall that $C$ is + invariant if and only if  $\emptyset = \r_F(C)$ and in that case $F(x) = F_C(x)$ for $x \in C$.
 From Theorem \ref{theoConley09a}(b)
 $(C, \emptyset)$ is index pair for $A = C_{\pm}$. Because $Dom(F) = X$, $x \in C_-$ implies $F(x) = F_C(x)$ is nonempty.
 Because $C_-$ is $F_C$ + invariant, $F_C(x) = F_{C_-}(x)$ for $x \in C_-$. It follows that $C_-$ is + viable and so $C_- = C_{\pm}$.
 By Theorem \ref{reltheo04}, $A$ is an attractor.

 (b) If $C$ is inward, then it is + invariant and the associated attractor is $A = \bigcap_{k=1}^{\infty} F^k(C) = \bigcap_{k=1}^{\infty} (F_C)^k(C) = C_-$ .
If $C$ is an inward set, then $F^{-1}(x) \cap C = \emptyset$ for all $x \in \partial C$ and so $C$ is a smple - isolating neighborhood.

 \end{proof}

 It is clear that $C$ is an inward set if and only if it is a simple - isolating set which is $F$ + invariant. \vspace{.5cm}

Isolated sets satisfy the following perturbation property.

 \begin{theo}\label{theoConley09b} Let $F$ be a closed relation on $X$, $C$ be a closed subset $X$ and $U$ be an open subset of $X$ with $U \subset C$.

 \begin{enumerate}
 \item[(a)] Assume $C$ is an isolating neighborhood with $C_{\pm} \subset U$. There exists $\ep > 0$ so that if $F_1$ is a closed relation
 contained in $V_{\ep} \circ F \circ V_{\ep}$, then $C$ is an isolating neighborhood for $F_1$ with the associated viable set $C_{\pm}$ for
 $F_1$ contained in $U$.

  \item[(b)] Assume $C$ is a - isolating neighborhood with $C_{-} \subset U$ (or + isolating neighborhood with $C_{+} \subset U$).
  There exists $\ep > 0$ so that if $F_1$ is a closed relation
 contained in $V_{\ep} \circ F \circ V_{\ep}$, then $C$ is a - isolating neighborhood for $F_1$ with the associated - viable set $C_{-}$ for
 $F_1$ contained in $U$ (resp. $C$ is a + isolating neighborhood for $F_1$ with the associated + viable set $C_{+}$ for
 $F_1$ contained in $U$).
 \end{enumerate}
 \end{theo}

 \begin{proof} (a) Let $W_-, W_+$ be open sets with $C_- \subset W_-$, $C_+ \subset W_+$ and $W_- \cap W_+ \subset U$. By  Corollary \ref{relcor08aaa}
there exists $n \in \Z_+$ such that $(F_C)^n(C \setminus W_+) = \emptyset$ and $(F_C)^{-n}(C \setminus W_-) = \emptyset$. As $\ep > 0$ decreases to
 $0$, the closed relations $(\bar V_{\ep} \circ F \circ \bar V_{\ep})_C$ decrease with intersection
 $F_C$. Inductively applying Proposition \ref{relprop00b},
 compactness yields that for sufficiently small $\ep > 0$ $[(\bar V_{\ep} \circ F \circ \bar V_{\ep})_C]^n(C \setminus W_+) = \emptyset $
 and $[(\bar V_{\ep} \circ F \circ \bar V_{\ep})_C]^{-n}(C \setminus W_-) = \emptyset $. This implies that for
 $F_1 \subset \bar V_{\ep} \circ F \circ \bar V_{\ep}$
 the maximum viable subset of $C$ is contained in $W_+ \cap W_-$.

 (b) Similarly, in this case there exists $n$ such that $(F_C)^{-n}(X \setminus U) = \emptyset$ and so for $\ep > 0$
 $[(\bar V_{\ep} \circ F \circ \bar V_{\ep})_C]^{-n}(C \setminus U) = \emptyset $. This implies that  for
 $F_1 \subset \bar V_{\ep} \circ F \circ \bar V_{\ep}$
 the maximum - viable subset of $C$ is contained in $U$.

 Alternatively, we can use the fact that if $P \subset C^{\circ} $ is an inward set for $F_C$, then it is an inward set for
 $(\bar V_{\ep} \circ F \circ \bar V_{\ep})_C$ provided $\ep > 0$ is small enough.

 The proof for + viability is similar, or else we apply the - viability result for $F^{-1}$.

 \end{proof}

If $P_1$ is a nonempty, closed subset of $X$ and $P_2$ is a closed subset of $P_1$, then we define $P_1/P_2$ to be the quotient space with
$P_2$ identified to a single point $[P_2]$. When $P_2 = \emptyset$, the point $[P_2]$ is an isolated point of $P_1/P_2$. We regard
$P_1/P_2$ to consist of the points of $P_1 \setminus P_2$ together with the base point $[P_2]$.   We let
$\pi : P_1 \to P_1/P_2$ denote the quotient map, which is surjective except when $P_2 = \emptyset$.  Observe that if $B$ is a closed
subset of $X$ with $P_1 \cap B \subset P_2$, then we can identify $P_1/P_2$ with $(P_1 \cup B)/(P_2 \cup B)$. In particular, we have the
quotient map $\pi$ from $P_1 \cup B$ onto $P_1/P_2$ mapping $B$ to $[P_2]$.

We now apply this with $(P_1,P_2)$ an index pair for $F$ on $X$ with $B = \r_F(P_1)$. Define the closed relation $F_{P_1/P_2}$ on $P_1/P_2$ by
\begin{equation}\label{Conley13b}
F_{P_1/P_2} \ =_{def} \ (\pi \times \pi)(F \cap [P_1 \times (P_1 \cup \r_F(P_1))]) \ \cup \ ([P_2],[P_2]).
\end{equation}
Thus, for $x \in P_1 \setminus P_2 = (P_1/P_2) \setminus [P_2]$ and $(x,y) \in F$
\begin{align}\label{eqConley13c}\begin{split}
(x,y) \in F_{P_1/P_2} \qquad &\Longleftrightarrow \qquad y \in P_1 \setminus P_2 = (P_1/P_2) \setminus [P_2], \\
(x,[P_2]) \in F_{P_1/P_2} \qquad &\Longleftrightarrow \qquad y \in P_2 \cup \r_F(P_1).
 \end{split}\end{align}
 Since $P_2$ is $F_{P_1}$ + invariant, $(x,y) \in F$ with $x \in P_2$ implies that $y \in P_2 \cup \r_F(P_1)$ and so
 $(x,y)$ projects to $([P_2],[P_2])$.

\begin{theo}\label{theoConley11a} Assume that $F$ is a closed relation on $X$ with $Dom(F) = X$ and that $(P_1, P_2)$ is an index
pair for $F$. For the closed relation $F_{P_1/P_2}$ on  $P_1/P_2, \ \ Dom F_{P_1/P_2} = P_1/P_2$.

The singleton $\{ [P_2] \}$ is
an attractor for $F_{P_1/P_2}$ with dual repeller $(P_1)_+\subset P_1 \setminus P_2$. Furthermore, $\pi((P_1)_-) \cup \{ [P_2] \}$ is
an attractor for $F_{P_1/P_2}$ with  dual repeller $\emptyset$.
If $P_1$ is $F$ + invariant and so $P_2 = \emptyset$, then the isolated point $\{ [P_2] \}$ is also a repeller for $F_{P_1/P_2}$ with dual
attractor $(P_1)_- = (P_1)_{\pm} \subset P_1 \setminus P_2$.
 \end{theo}

\begin{proof} If $x \in P_1$, then $F(x)$ is a nonempty subset of $P_1 \cup \r_F(P_1)$ and so $F_{P_1/P_2}(x)$ is nonempty.

We first consider the Theorem \ref{theoConley09aa} case with $P_1 \ \ F$ + invariant and so $P_1 = (P_1)_+$ which implies $P_2 = \emptyset$.
In that case, the singleton $\{ [P_2] \}$ is  invariant for both $F_{P_1/P_2}$ and $F_{P_1/P_2}^{-1}$. Since it is clopen, it is inward
for both $F_{P_1/P_2}$ and $F_{P_1/P_2}^{-1}$ and so is both an attractor and repeller for $F_{P_1/P_2}$. Regarded as a repeller, its
dual attractor is $(P_1)_- = (P_1)_{\pm}$. Regarded as an attractor, its dual repeller is $P_1 = P_1 \setminus P_2 \subset P_1/P_2$.
Finally, $(P_1)_- \cup \{ [P_2] \}$ is an attractor with dual repeller $\emptyset$.

Now assume that $P_1$ is not + invariant for $F$ and so $\r_F(P_1) \not= \emptyset$. If $x \in P_1$ and $(F_C)^n(x) = \emptyset$, then
$(F_{P_1/P_2})^n(x) = \{ [P_2] \}$. To see this let $\xx \in \S([0,n],F)$ with $\xx(0) = x$. By assumption, some $\xx(k) \not\in P_1$ and
if $k$ is the minimum such then $\xx(k) \in \r_F(P_1)$ and so $(F_{P|1/P_2}(x(k)) = [P_2]$.

Since $P_2$ is disjoint from $(P_1)_+$ it follows from Corollary \ref{relcor08aaa} that for some $n \in \Z_+,  \ \ P_2 \subset F^{*n}(X \setminus P_1)$.
So there exists $U$ closed with $P_2 \subset U^{\circ} \subset U \subset F^{*n}(X \setminus P_1)$. By Lemma \ref{rellem07abb}
$(F_{P_1})^n(U \cap P_1) = \emptyset$ and so $(F_{P_1/P_2})^n(U \cap P_1) = [P_2]$. It follows from  Theorem \ref{reltheo04}
that $\{[P_2]\}$ is an attractor for $F_{P_1/P_2}$. Since $x \not\in (P_1)_+$ implies $ (F_{P_1})^n(x) = \emptyset$ for some $n$, it follows
that $(P_1)_+$ is the dual repeller.

If $x \not\in (P_1)_-$, then $(F_{P_1})^{-n}(x) = \emptyset$ for some $n$ and so $\pi((P_1)_-) \cup \{ [P_2] \}$ is
an attractor for $F_{P_1/P_2}$ with  dual repeller $\emptyset$.

\end{proof}

   For pairs $(P_1,P_2)$, $(Q_1,Q_2)$,  we define
  \begin{equation}\label{eqConley14}%\begin{split}
%  (R_1,R_2) \prec (P_1,P_2) \qquad \Longleftrightarrow \qquad R_1 \subset P_1, \ \ \text{and} \ \ R_2 \supset R_1 \cap P_2. \\
  (P_1,P_2) \wedge (Q_1,Q_2)  \ =_{def} \  (P_1 \cap Q_1,P_1 \cap Q_1 \cap (P_2 \cup Q_2)). \hspace{.5cm}
 % \end{split}
  \end{equation}

  Notice that
   \begin{equation}\label{eqConley14a}\begin{split}
   \r_F(P_1 \cap Q_1) = \overline{[F(P_1 \cap Q_1) \cap (X \setminus P_1)]} \cup \overline{[F(P_1 \cap Q_1) \cap (X \setminus Q_1)]} \\
   \subset \r_F(P_1) \cup \r_F(Q_1), \hspace{4cm} \\
   \text{and so} \qquad \d_F(P_1 \cap Q_1) \subset [Q_1 \cap\d_F(P_1)] \cup [P_1 \cap \d_F(Q_1)]. \hspace{2cm}\\
    \r_F(P_1 \cup Q_1) = \overline{[F(P_1)  \setminus P_1]\setminus Q_1} \cup \overline{[F(Q_1)   \setminus Q_1]\setminus P_1} \hspace{2cm}\\
     \subset \r_F(P_1) \cup \r_F(Q_1). \hspace{4cm}
   \end{split}
   \end{equation}

 \begin{prop}\label{propConley10} Let $A$ be a viable subset of $X$. If $(P_1,P_2)$ is an
 index pair rel $C$ for $A$ and $(Q_1,Q_2)$ is an index pair rel $D$ for $A$, then
  $(R_1,R_2) = (P_1,P_2) \wedge (Q_1,Q_2)$  is an index pair rel $C \cap D$ for $A$.

 % In particular, if $(P_1,P_2)$ is an
% index pair rel $C$ for $A$ and $(Q_1,Q_2)$ is an index pair rel $D$ for $A$, and $ (Q_1,Q_2) \prec (P_1,P_2) $
% then $(Q_1,Q_2)$  is an index pair rel $C \cap D$ for $A$.

\end{prop}

 \begin{proof} There exist open subsets $U_1, V_1, U_2, V_2$ of $X$ with $A \subset V_1 \subset P_1 \setminus P_2 \subset U_1 \subset C$, and
$A \subset V_2 \subset Q_1 \setminus Q_2 \subset U_2 \subset D $. Note that
$R_1 \setminus R_2 = P_1 \cap Q_1 \cap (X \setminus P_2) \cap (X \setminus Q_2)$
which equals $(P_1 \setminus P_2) \cap (Q_1 \setminus Q_2)$. Hence,
$A \subset V_1 \cap V_2 \subset R_1 \setminus R_2 \subset U_1 \cap U_2 \subset C \cap D $.
Thus, $(R_1,R_2)$ satisfies condtion (iii) and condition (i) is clear.

Now $F_{C \cap D} = F_C \cap F_D$. Assume $(x,y) \in F_{C \cap D}$. If $x \in R_1 = P_1 \cap Q_1$, then $y \in R_1$ by $F_C$ + invariance of $P_1$ and
$F_D$ + invariance of $Q_1$. Thus, $R_1$ is $F_{C \cap D}$ + invariant. If
$x \in R_2 = R_1 \cap (P_2 \cup Q_2)$ then, as before, $y \in R_1$. If $x \in P_2$ then
$y \in P_2$ by $F_C$ + invariance of $P_2$ and similarly, $x \in Q_2$ implies $y \in Q_2$.
So $R_2$ as well is $F_{C \cap D}$ + invariant. This is condition (ii).

\end{proof}

\begin{cor}\label{corConley11} If $(P_1,P_2)$ and $(Q_1,Q_2)$ index pairs for a viable
subset $A$, then $(P_1,P_2) \wedge (Q_1,Q_2)$ is an index pair for $A$.

If $P_1$ and $Q_1$ are isolating neighborhoods of index type for $A$, then $P_1 \cap Q_1$
is an isolating neighborhood of index type for $A$.

If $P_1$ and $Q_1$ are simple isolating neighborhoods, then $P_1 \cap Q_1$
is a simple isolating neighborhood. \end{cor}

 \begin{proof} These are immediate first from Theorem \ref{theoConley08} together with Proposition \ref{propConley10} and then from
 Theorem \ref{theoConley09} together with Proposition \ref{propConley10}.

 If $P_1$ and $Q_1$ are simple, then $F(P_1)\cap F^{-1}(P_1)$ is disjoint from $\partial P_1$ and
 $F(_1)\cap F^{-1}(Q_1)$ is disjoint from $\partial Q_1$. Hence,
 $F(P_1 \cap Q_1)\cap F^{-1}(P_1 \cap Q_1) \subset F(P_1)\cap F^{-1}(P_1)\cap F(Q_1)\cap F^{-1}(Q_1)$ is disjoint from $\partial P_1 \cup \partial Q_1$
 and so from $\partial (P_1 \cap Q_1)$. Thus $P_1 \cap Q_1$ is simple.

 \end{proof}

 \begin{cor}\label{corConley11a} If $\{ P_{2i} \}$ is a finite collection of subsets of $P_1$ and for each $i$  $(P_1,P_{2i})$ is an index pair, then
 $(P_1, \bigcup_i  P_{2i})$ is an index pair. \end{cor}

 \begin{proof}  While this is easy to check directly, it follows from Corollary \ref{corConley11} and induction because
 $(P_1,P_{2i} \cup P_{2j}) = (P_1,P_{2i}) \wedge (P_1,P_{2j})$.

 \end{proof}

 \begin{df}\label{dfConley12} Let $P_1$ be an isolating neighborhood of index type for a viable set $A$ and let $Q_1$ be an isolating
 neighborhood for $A$. We write $Q_1 \prec P_1$ when
 \begin{equation}\label{eqConley15}
 Q_1 \ \subset \ P_1, \qquad \text{and} \qquad \r_F(Q_1) \cap (P_1)_+ \ = \ \emptyset.
 \end{equation}\end{df} \vspace{.5cm}

 \begin{prop}\label{propConley13} If $P_1$ is an isolating neighborhood of index type for a viable set $A$ and  $Q_1$ is an isolating
 neighborhood for $A$ with $Q_1 \subset P_1$, then $Q_1 \prec P_1$ if and only if there exists $P_2$ such that $(P_1,P_2)$ is an
 index pair for $A$ with $P_1 \cap \r_F(Q_1) \subset P_2$. In that case, $(Q_1,Q_1 \cap P_2)$ is an index pair for $A$.  In particular,
 $Q_1 \prec P_1$  implies that $Q_1$ as well as $P_1$ is an isolating neighborhood of index type.

 If $\{ Q_{1i} \}$ is a finite collection of isolating
 neighborhoods for $A$ with $Q_{1i} \prec P_1$ for all $i$, there exists $P_2$ such that $(P_1,P_2)$ is an
 index pair for $A$ with $P_1 \cap \r_F(Q_{1i}) \subset P_2$ for all $i$ and so for each $i$ $(Q_{1i},Q_{1i} \cap P_2)$ is an index pair for $A$.

 \end{prop}

 \begin{proof} If such a set $P_2$ exists, then $(P_1)_+ \cap (P_1 \cap \r_F(Q_1)) \subset (P_1)_+ \cap P_2 = \emptyset$ by
 Theorem \ref{theoConley07} (a). Since $(P_1)_+ \subset P_1$, this implies $(P_1)_+ \cap  \r_F(Q_1) = \emptyset$. Furthermore,
 $\d_F(Q_1) \subset Q_1 \cap (P_1 \cap \r_F(Q_1)) \subset Q_1 \cap P_2$. Since $A \subset (Q_1)^{\circ} \cap (X \setminus P_2)$
    it follows from Theorem \ref{theoConley08} that $(Q_1, Q_1 \cap P_2)$ is an index pair for $A$. In particular, $Q_1$ is of index type.

    On the other hand, if $Q_1 \prec P_1$ and $P_0$ is any closed subset of $P_1$ such that $(P_1 \cap \r_F(Q_1)) \cup \d_F(P_1) \subset P_0$ and
    $P_0 \cap (P_1)_+ = \emptyset$, then  with
 $P_2 = \O (F_{P_1})(P_0)$ we obtain the required index pair $(P_1,P_2)$ by Theorem \ref{theoConley09}.

 For a collection $\{ Q_{1i} \}$, choose $P_{2i}$ such that $(P_1,P_{2i})$ is an
 index pair for $A$ with $P_1 \cap \r_F(Q_{1i}) \subset P_{2i}$. Let $P_2 = \bigcup_i P_{2i}$ and apply Corollary \ref{corConley11a}.

 \end{proof}

  \begin{cor}\label{corConley13a} If $Q_1 \prec P_1$, then $(Q_1)_+ \ = \ Q_1\cap (P_1)_+$. In fact, if
  $\xx \in \S_+(P_1)$ with $\xx(0) \in Q_1$, then $\xx \in \S_+(Q_1)$. \end{cor}

  \begin{proof}  In any case, if $Q_1 \subset P_1$ then $(Q_1)_+ \subset Q_1\cap (P_1)_+$.

  Now assume $Q_1 \prec P_1$ and  $\xx \in \S_+(P_1)$ with $\xx(0) \in Q_1$. Choose $P_2 \supset \r_F(Q_1)$ so that
  $(P_1,P_2)$ is an index pair.

  If $\xx(k) \not\in Q_1$ for some $k \in Z_+$ we can
  choose $k$ to be the smallest such index. Then $k > 1$ and $\xx(k-1) \in Q_1$.  Hence, $\xx(k) \in (F(Q_1) \setminus Q_1) \cap P_1 \subset P_2$.
  On the other hand, $\xx(i) \in (P_1)_+$ for all $i$. So $\xx(k) \not\in Q_1$ for some $k$ contradicts $P_2 \cap (P_1)_+ = \emptyset$.

  \end{proof}

 It can happen that $Q_1 \subset P_1$ are both isolating neighborhoods of index type, but not $Q_1 \prec P_1$.  In particular, for
 $Q_1, P_1$ isolating neighborhoods of index type, it need not be true that $Q_1 \cap P_1 \prec P_1$.

 \begin{prop}\label{theoConley14} Let $A$ be a viable set.  The relation $\prec$ on the set of isolating neighborhoods of index type for
 $A$ is a partial order, i.e. it is reflexive, anti-symmetric and transitive. \end{prop}

 \begin{proof}  The relation is reflexive because the neighborhoods are of index type, i.e. $P_1 \prec P_1$ because
 $\d_F(P_1) = \r_F(P_1) \cap P_1$ is disjoint from $(P_1)_+$ by condition (ii$'$) of Theorem \ref{theoConley09}.

 Anti-symmetry follows because set inclusion is anti-symmetric.

 Now assume that $R_1 \prec Q_1$ and $Q_1 \prec P_1$.  That is, $\overline{F(R_1) \setminus R_1} \cap (Q_1)_+ = \emptyset$ and
 $\overline{F(Q_1) \setminus Q_1} \cap (P_1)_+ = \emptyset$.  We must show that $\overline{F(R_1) \setminus R_1} \cap (P_1)_+ = \emptyset$.

 Suppose instead that $x \in \overline{F(R_1) \setminus R_1} \cap (P_1)_+ $. This means there exists $\xx \in \S_+(F_{P_1})$ with
 $\xx(0) = x$. If $\xx(i) \in Q_1$ for all $i$, then $\xx \in \S_+(F_{Q_1})$.  This contradicts $\overline{F(R_1) \setminus R_1} \cap (Q_1)_+ = \emptyset$.
 So we may assume that $\xx(k) \not\in Q_1$ for some $k$ and choose $\xx(k)$ to be the first term of the sequence which is not in $Q_1$.
 If $k = 0$, i.e. $x \not\in Q_1$, then $x \in \overline{F(R_1) \setminus R_1} \subset F(R_1) \subset F(Q_1)$ implies $x = \xx(k) \in F(Q_1) \setminus Q_1$.
 If $k > 1$, then $\xx(k-1) \in Q_1$ implies $\xx(k) \in F(Q_1) \setminus Q_1$. But $\xx(k) \in (P_1)_+$. This contradicts
  $\overline{F(Q_1) \setminus Q_1} \cap (P_1)_+ = \emptyset$. It follows that no such $x$ exists.

  \end{proof}

 \begin{theo}\label{theoConley15} Let $\{ P_{1i} \}$ be a finite collection of isolating neighborhoods of index type for a viable set $A$
  and let $C$ be an isolating
 neighborhood for $A$. There exists $Q_1 \subset C$ an isolating neighborhood for $A$ with $Q_1 \prec P_{1i}$ for all $i$.
 In particular, $Q_1 \prec \bigcap_i \ P_{1i}$  \end{theo}

 \begin{proof} Replacing $C$ by $(\bigcap_i P_{1i}) \cap C$ we may assume $C \subset P_{1i}$ for all $i$.

 Define the closed relation $\tilde F = F \cap (X \times (X \setminus C^{\circ}))$.

 If $A$ is a closed $F_C$ + invariant subset of $C$, then by  Proposition \ref{propConley03},
 $F(A) \setminus A = F(A) \setminus C \subset F(A) \cap (X \setminus C^{\circ}))$.  It follows that $\r_F(A) \subset \tilde F(A)$.

 Next observe that for the closed $F_C$ + invariant set $C_-$, and for each $i$ $\tilde F(C_-) \cap (P_{1i})_+ = \emptyset$.

 Suppose instead that there exists $(x,y) \in F$ with $x \in C_-$ and $y \in (P_{1i})_+ \cap (X \setminus C^{\circ})$. Then
 $x \in (P_{1i})_- \supset C_-$ and $y \in P_{1i}$.  Because $(P_{1i})_-$ is $F_{P_{1i}}$ + invariant, it follows that $y \in (P_{1i})_-$.
 But $y \in (P_{1i})_+$ and so $y \in (P_{1i})_{\pm} = A$. This contradicts the inclusion $A \subset C^{\circ}$.

 Now let $U = (\tilde F)^*(X \setminus (\bigcup_i (P_{1i})_+) )$. This is an open subset of $X$ which contains $C_-$.

 By Theorem  \ref{theoConley07} there exists an index pair $(Q_1, Q_2)$ rel $C$ for $A$ with $Q_1 \subset U$.
 Hence, $Q_1$ is a closed $F_C$ + invariant subset of $C$ which implies $\r_F(Q_1) \subset \tilde F(Q_1)$. Since $Q_1 \subset U$
 it follows that $\tilde F(Q_1) \subset X \setminus (\bigcup_i (P_{1i})_+)$.
 That is $\r_F(Q_1) \cap (\bigcup_i (P_{1i})_+) = \emptyset$. Since $Q_1 \subset C \subset P_{1i}$, it follows that $Q_1 \prec P_{1i}$.

 Because $(\bigcap_i \ P_{1i})_+ \subset \bigcup_i (P_{1i})_+$, it follows that $Q_1 \prec \bigcap_i \ P_{1i}$.

 \end{proof}

\subsection{{\bf Anomalous Perturbations}}\label{anomalous1}
\vspace{.5cm}

Recall that for a subset $A$ of $X$ we let $A^{\circ}$ and $\overline{A}$ denote the interior and closure, respectively, of $A$ in $X$.
If $A \subset C \subset X$ we let  $Int_C A$ and $Cl_C A$ be the interior  and the closure of a subset $A$ of $C$ taken with respect to the
relative topology on $C$.

Given $\ep > 0$  we defined the relations on $X$
\begin{equation}\label{eqanom00}
V_{\ep} \ = \ \{ (x,y) : d(x,y) < \ep \} \quad \text{and} \quad \bar V_{\ep} = \{ (x,y) : d(x,y) \le \ep \}.
\end{equation}

A closed subset $C$ of $X$ is a \emph{regular closed subset}\index{subset!regular closed} when $C = \overline{C^{\circ}}$. Observe that
for any closed subset $C$
\begin{align}\label{eqanom01}\begin{split}
C^{\circ} \ \subset \ \overline{C^{\circ}} \quad &\Longrightarrow \quad C^{\circ} \ \subset \ (\overline{C^{\circ}})^{\circ}, \\
 \overline{C^{\circ}} \ \subset \ \ C \quad &\Longrightarrow \quad (\overline{C^{\circ}})^{\circ} \ \subset \ C^{\circ}. \\
\text{and so} \qquad  &(\overline{C^{\circ}})^{\circ} \ = \ C^{\circ}.
\end{split}\end{align}

\begin{lem}\label{lemanom01} Let $C$ be a  closed subset of $X$ and $A$ be a  subset of $C$.

(a) We have $Cl_C A = \overline{A}$. That is, $A$ is closed in the relative topology of $C$ if and only if it is closed in $X$.
On the other hand, $A^{\circ} = C^{\circ} \cap Int_C A$. If $C$ is a regular closed subset, then
$A^{\circ}$ is dense in $Int_C A$.

(b) If $C$ is a regular closed subset, and $A \subset C$ is closed, then $A$ is regular in $C$ if and only if it is regular in $X$,
i.e. $A = \overline{Int_C A}$ if and only if $A = \overline{A^{\circ}}$.

(c) If a closed set $A$ is nowhere dense in $C$, then  it is nowhere dense in $X$, i.e. if $Int_C A = \emptyset$, then $A^{\circ} = \emptyset$.
If $C$ is a regular closed subset, then the converse holds, i.e. $A$ is nowhere dense in $C$ if it is nowhere dense in $X$. \end{lem}

\begin{proof} (a): That  $Cl_C A = \overline{A}$ is obvious when $C$ is closed, i.e. the notion of closed set is the same for the relative
topology on $C$ and for the original topology on $X$.

That  $C^{\circ} \cap Int_C A = A^{\circ}$ is (\ref{eqConley9aa}).

If $C$ is regular, then $C^{\circ}$ is a dense subset of $C$. Hence, $A^{\circ} = C^{\circ} \cap Int_C A$ is dense in the relatively open set
$Int_C A$.

(b): From (a) $A^{\circ}$ is dense in $Int_C A$.

 (c): If $A^{\circ} \not= \emptyset$, then $Int_C A \supset A^{\circ} \not= \emptyset$. If $C$ is regular, then by (b)
 $A^{\circ}$ is dense in $Int_C A$. So  $Int_C A \not= \emptyset$ implies $A^{\circ} \not= \emptyset$.

\end{proof}

If $C$
is an inward set for a closed relation $F$, then
\begin{equation}\label{eqanom02}
F(\overline{C^{\circ}}) \ \subset \ F(C) \ \subset \ C^{\circ} \ = \ (\overline{C^{\circ}})^{\circ}\ \subset \ \overline{C^{\circ}}.
\end{equation}
Thus, $\overline{C^{\circ}}$ is an inward set with the same attractor as that of $C$.

Similarly, if $C$ is an isolating neighborhood for $C_{\pm}$, then
\begin{equation}\label{eqanom03}
C_{\pm} \ \subset \ C^{\circ} \ = \ (\overline{C^{\circ}})^{\circ}.
\end{equation}
As any bi-infinite solution path in $C$ lies in $C_{\pm}$, it lies in  $\overline{C^{\circ}}$.  It then follows that
\begin{equation}\label{eqanom04}
C_{\pm} \ = \ (\overline{C^{\circ}})_{\pm}.
\end{equation}
Hence, $\overline{C^{\circ}}$ is an isolating neighborhood for the same isolated viable subset.

%For our purposes, the significance of regularity for a closed set comes from the following.

A subset $A$ of $X$ is $\ep$ dense in $X$ if for every $x \in X, V_{\ep}(x) \cap A \not= \emptyset$ or, equivalently, $V_{\ep}(A) = X$.

\begin{lem}\label{lemanom02} If $K$ is a closed, nowhere dense subset of $X$ and $\ep > 0$, then there exists $\d > 0$ such that
$A = X \setminus V_{\d}(K)$ is a  closed set
 which is $\ep$ dense and disjoint from $K$. \end{lem}

\begin{proof} Let $\{ x_1, x_2,\dots \}$ be a  sequence dense in $X \setminus K$. Since $K$ is nowhere dense, the sequence is dense in $X$. Hence,
$\{ V_{\ep}(x_1), V_{\ep}(x_2) \dots \}$ is an open cover of $X$. So there exists a positive integer $N$ such that
$\{ V_{\ep}(x_1), V_{\ep}(x_2) \dots , V_{\ep}(x_N) \}$ is an open cover.
Let  $\d = \min_{i=1}^N \ \{ d(x_i,K) \}$. Since $\{ x_1, \dots, x_N \} \subset A$, it follows that $A$ is $\ep$ dense.

\end{proof}

For $A$ a nonempty closed subset of $X$, define the closed \emph{retraction relation}\index{retraction}\index{relation!retraction} $R_A$ to be
\begin{equation}\label{eqanom05}
R_A \ = \ \{ (x,y) \in X \times A : d(x,y) = d(x,A) \}.
\end{equation}

The retraction relation satisfies the following conditions.
\begin{itemize}
\item[(i)] $(x,y) \in R_A \quad \Longrightarrow \quad y \in A$.
\item[(ii)] $(x,y) \in R_A \ \ \text{and} \ \ x \in A \quad \Longrightarrow \quad y = x$.
\end{itemize}
If, in addition, $A$ is $\ep$ dense, then
\begin{itemize}
\item[(iii)] $(x,y) \in R_A \quad \Longrightarrow \quad d(x,y) < \ep$. \\
Thus, $R_A \subset  V_{\ep}$ and $1_X \subset V_{\ep} \circ R_A$.
\end{itemize}\vspace{.5cm}

Now let $F$ be a closed relation on $X$ with $X = Dom(F)$, i.e. $F(x) \not= \emptyset $ for all $x \in X$. If $C$ is a nonempty inward set,
then $Dom(F) = X$ implies that the associated attractor is nonempty. Let $\ep = d(F(C),X \setminus C) > 0$. If $G$ is a closed relation
with $Dom(G) = X$ and $G \subset V_{\ep} \circ F$, then $C$ is inward for $G$ and so contains a nonempty attractor for $G$. Notice
that in this case $C_+ = C$.

We saw in Theorem \ref{theoConley09b} that if $C$ is an isolating neighborhood for a closed relation $F$, then it remains an isolating neighborhood
for any sufficiently close perturbation of $F$. However, we now show that the associated
 isolated invariant set  can in rather general circumstances be eliminated by arbitrarily small perturbations.

\begin{theo}\label{theoanom03} Let $F$ be a closed relation on $X$ with $Dom(F) = X$ and let $C$ be an isolating neighborhood for $F$.
If $C_{+}$ is a nowhere dense subset of $C$, then for any $\ep > 0$ there exists $G$ a  closed relation on $X$ with $Dom(G) = X$
such that $G \subset V_{\ep} \circ F$ and
$F \subset V_{\ep} \circ G$, but such that there exists a positive integer $N$ with $(G_C)^N = \emptyset$. In particular, with respect to
$G$ we have $C_- = C_+ = C_{\pm} = \emptyset$. \end{theo}

\begin{proof} Notice that if for $F$ either $C_+$ or $C_-$ is nonempty, then by   Proposition \ref{relprop06a} $C_{\pm}$ is nonempty.
Contrapositively, $C_{\pm} = \emptyset$ implies  $C_- = C_+  = \emptyset$. It then follows from Corollary \ref{relcor08aaa}
that $(F_C)^N = \emptyset$ for large enough $N$. Conversely, if $(F_C)^N = \emptyset$ for some $N$, then $C_- = C_+ = C_{\pm} = \emptyset$.

If $C_+$ is nowhere dense, then by Lemma \ref{lemanom02} we can choose an open set $W_+$ containing $C_+$ such that $X \setminus W_+$ is
$\ep$ dense.  As in the proof of  Theorem \ref{theoConley07} we can choose $Q \subset C \cap W_+$ a neighborhood of $C_+$
with respect to $C$ and which is inward for $(F_C)^{-1}$ and which is a regular closed set. That is,
$(F_C)^{-1}(Q) \subset Int_C Q$. Since $C_+$ is the maximum repeller for
$F_C$ it is $(F_C)^{-1}$ invariant and, in particular, $C_+ \subset (F_C)^{-1}(Q)$.

Let $A_1 = C \setminus Int_C Q$ so that $A_1$ is $F_C$ inward.

Because the set $A_1$ is
inward for $F_C$ and disjoint from $C_+$. By Corollary \ref{relcor08aaa} again, it follows that $(F_C)^N(A_1) = \emptyset$ for
large enough $N$.

Because the set $A_1$ is
inward for $F_C$, $F_C(A_1) \subset Int_C (C \setminus Int_C Q) = C \setminus \overline{Int_C Q} = C \setminus Q$
and so $F_C(A_1) \cap Q = \emptyset$.  Since $A_1, Q \subset C$, it follows that $F(A_1) \cap Q = \emptyset$.

Choose an open set $O \subset W_+$ such that $Int_C Q \subset Q \subset O \subset W_+$ and such that $F(A_1) \cap O = \emptyset$. Let $A =
X \setminus O$. Since $O \subset W_+$, the closed set $A$ is $\ep$ dense. Furthermore, $F(A_1) \subset A $ and
$A \cap C  = (X \setminus O) \cap C = C \setminus O \subset C \setminus Int_C Q = A_1$.

Now let $R_A$ be the retraction relation to $A$, and let $G = R_A \circ F$. Because $A$ is $\ep$ dense we have
$G  \subset V_{\ep} \circ F$ and $F \subset V_{\ep} \circ G $.  Because $Dom(R_A) = X$, it follows that $Dom(G ) = X$.

For $x \in A_1$, $F(x) \cap O = \emptyset$.  That is, $F(x) \subset A$ and so $G(x) = F(x)$.

Clearly, $G_C(C) \subset A \cap C \subset A_1$ and on the $F_C$ + invariant set $A_1$, $G = F$ and so $G_C = F_C$. It follows that $(G_C)^{N+1} = \emptyset$.

\end{proof}

While $Dom(G) = X$, it is not true that $Dom(G^{-1}) = X$, because $G(X) \subset A$. Thus, even if $F$ were surjective, $G$ is not.
Notice that if $C_{\pm}$ is a nowhere dense repeller, so that $C_{\pm} = C_+$, then Theorem \ref{theoanom03} applies and
it can be eliminated by this sort of perturbation.
On the other hand it cannot be eliminated by arbitrarily small surjective perturbations because such would provide small perturbations
of $F^{-1}$ with domain equal to $X$. We have seen that the attractors of $F^{-1}$, i.e. the repellers for $F$, cannot be eliminated by
small perturbations with $Dom(F^{-1}) = X$.

\begin{theo}\label{theoanom04} Let $F$ be a closed relation on $X$ with $F$ surjective and let $C$ be an isolating neighborhood for $F$.
If $C_{+}$ and $C_-$ are nowhere dense subsets of $C$, then for any $\ep > 0$ there exists $\hat G$ a surjective closed relation on $X$
such that $\hat G \subset V_{\ep} \circ F$ and
$F \subset V_{\ep} \circ \hat G$, but such that there exists a positive integer $N$ with $(\hat G_C)^N = \emptyset$. In particular, with respect to
$\hat G$ we have $C_- = C_+ = C_{\pm} = \emptyset$. \end{theo}

\begin{proof} We begin with an adjustment of the construction in the proof of Theorem \ref{theoanom03}. Initially, we do not assume that $C_-$ is
nowhere dense.

Given $V$ an open subset of $X$ with $C_{\pm} \subset V \subset C^{\circ}$, we choose $W_+, W_-$ open subsets of
$X$ with  $C_+ \subset W_+, \ C_- \subset W_-$ and $W_+ \cap W_- \subset V$ and such that $X \setminus W_+$ is
$\ep$ dense.

Choose $Q$ an $F_C^{-1}$ inward neighborhood of $C_+$ with $Q \subset C \cap W_+$ as before.
let $A_1 = C \setminus Int_C Q$ so that $A_1$ is $F_C$ inward and so that
 $(F_C)^N(A_1) = \emptyset$ for
large enough $N$ and $F_C(A_1) \subset Int_C (C \setminus Int_C Q) = C \setminus \overline{Int_C Q)} = C \setminus Q$
and so $F_C(A_1) \cap Q = \emptyset$.  Again, since $A_1, Q \subset C$, it follows that $F(A_1) \cap Q = \emptyset$.

Choose an open set $O \subset W_+$ such that $Int_C Q \subset Q \subset O \subset W_+$ and such that $F(A_1) \cap O = \emptyset$. Let $A =
X \setminus O$. Since $O \subset W_+$, the closed set $A$ is $\ep$ dense. Furthermore, $F(A_1) \subset A $ and
$A \cap C  = (X \setminus O) \cap C = C \setminus O \subset C \setminus Int_C Q = A_1$. Let
 $R_A$ be the retraction relation onto $A$.

Now let $P  \subset C \cap W_-$ a neighborhood of $C_-$
with respect to $C$ and with $P$ inward for $F_C$. Because $Q \cap P \subset V \subset C^{\circ}$ Lemma \ref{lemanom01}(a)
implies that $Int_C Q \cap Int_C P = Q^{\circ} \cap P^{\circ}$.

 We assume that the choices have been made
so that $Q$ and $P$ are regular closed sets. Note that
\begin{equation}\label{eqanom05a}
Q \cap P \setminus (Q^{\circ} \cap P^{\circ}) \ = \ (Q^{\circ} \cap \partial_C P) \cup (\partial_C Q \cap P)
\end{equation}
where $\partial_C Q = Q \setminus Int_C Q  \subset A_1$.

Define $B = X \setminus (Q^{\circ} \cap P^{\circ})$, so that $B \supset A_1$ and let
\begin{align}\label{eqanom06} \begin{split}
\hat R \ =_{def} \ 1_{B} \cup &(R_A \cap [(Q \cap P) \times A]), \\
G_1 \ =_{def} \ &\hat R \circ F.
\end{split}\end{align}
Thus, for $x \in Q \cap P \setminus (Q^{\circ} \cap P^{\circ}) \supset Q^{\circ} \cap \partial_C P, \hat R(x) = \{ x \} \cup R_A(x)$
while for $x \in Q^{\circ} \cap P^{\circ}, \hat R(x) = R_A(x)$. Otherwise
$\hat R(x) = \{ x \}$. In particular, $\hat R(x) = \{x \}$ for $x \in A \cup (X \setminus C)$.

As before, we have  $G_1 \subset V_{\ep} \circ F$ and
$F \subset V_{\ep} \circ G_1$. %
%Observe that Proposition \ref{relprop06a} implies any infinite solution path in $S_+(F)$ eventually lies in the open neighborhood
%$Q^{\circ} \cap P^{\circ}$ of $C_{\pm}$.
%
%Let $x_1, x_2, \dots, x_k, x_{k+1}$ be a solution path for $F_C$ with $k$ the minimum such that
% $x_k \in P \cap Q$. If $x_k \in Q^{\circ} \cap \partial P$, then
%$x_1, x_2, \dots x_{k-1}$ is an $(F_2)_C$ solution path and $F_2(x_{k-1}) = \{ x_k \} \cup R_A(x_k)$. Since $F_C(x_k) \subset Int_C P$,
%either $x_{k+1} \in Q^{\circ} \cap P^{\circ}$, in which case $(F_2)_C(x_{k+1}) \subset A$, or $x_{k+1} \in C \setminus Q^{\circ} = A \cap C$.
%If $x_k \in Q^{circ} \cap P^{\circ}$, then $F_2(x_{k-1}) = R_A(x_k) \subset A$. Otherwise, $x_k \not\in Q^{\circ}$ and so $x_k \in A$.
%
%Thus, every $(F_2)_C$ solution path eventually enters $A$.  On the $F_C$ invariant set $A \cap C$, $(F_2)_C = F_C$.  Hence, there are no
%infinite $(F_2)_C$  + solution paths.  That is, for $F_2$  $C_+ = \emptyset$ as before.

If $F$ were surjective with $Dom(F^{-1}) = F(X) = X$, then $Dom(G_1^{-1}) = G_1(X) = B$. So $G_1$ is still not surjective.

Now assume that $C_+$ is nowhere dense.

 Shrinking $\ep$ if necessary, we can assume that the open set
 %$ V_{2\ep}(C_{\pm})$ is  contained in $C$ and that
 $V = V_{\ep}(C_{\pm})$.

We choose $W_-$ so that
$Q \setminus W_-$ is $\ep/2$ dense in $Q$.

Because $P$ and $Q$ are regular closed sets and $C$ was assumed regular, Lemma \ref{lemanom01}(b)
implies that $Q = \overline{Q^{\circ}}, P = \overline{P^{\circ}}$.

Having chosen $Q$ and $P$, we write $Q \cap P$ as the union of nonempty closed sets
$K_1, \dots, K_k$ each of diameter less than $\ep/2$.  For each $K_i$ we choose a point $y_i \in Q \setminus P$ such that
$d(y_i,K_i) < \ep/2$. Because $Q$ is regular we can adjust $y_i$ if necessary to demand that $y_i \in Int_C Q \setminus P$.
In particular,
$y_i \not\in A_1$ for $i = 1, \dots k$.

Because $F$ is surjective, we can choose $x_i \in X$ such that $y_i \in F(x_i)$. The point $x_i$ need not be in $C$. Assume that it is in $C$.
Then $x_i$ is not in the $F_C$ + invariant set $A_1$ because $y_i \not\in A_1$.  Furthermore, $x_i \not\in P$. This is because
$C \setminus Int_C P$ is inward for $(F_C)^{-1}$. Since $y_i \in Q \setminus P \subset C \setminus Int_C P$, it follows that
$x_i \in Int_C(C \setminus Int_C P) = C \setminus \overline{Int_C P} = C \setminus P$. Thus, each
$x_i \in (X \setminus C) \cup (Int_C Q \setminus P)$.
% \subset (X \setminus C) \cup (O \setminus P)$.
Thus, $x_i \not\in A_1$.

For the closed relation
\begin{equation}\label{eqanom07}
M \ =_{def} \ \bigcup_{i=1}^k \ \{ (x_i,z) : z \in \{ y_i \} \cup K_i \},
\end{equation}
\begin{equation}\label{eqanom08}
(x_i,z) \in M \quad \Longrightarrow \quad (x_i,y_i) \in F, \ \ d(y_i,z) < \ep.
\end{equation}

Let
\begin{equation}\label{eqanom09}
\hat G \ =_{def} \ G_1 \cup M.
\end{equation}

We have $Dom(\hat G) \supset Dom(G_1) = X$ and $Dom((\hat G)^{-1}) = \hat G(X) = G_1(X) \cup M(X) = B \cup (P \cap Q) = X$.
Furthermore, $V_{\ep} \circ \hat G \supset V_{\ep} \circ G_1 \supset F$ and $V_{\ep} \circ F \supset G_1 \cup M = \hat G$. \vspace{.5cm}

We now prove that
\begin{equation}\label{eqanom10}
\hat G(P) \ = \ G_1(P) \ \subset \ A_1 \cup (X \setminus C).
\end{equation}

Since $x_i \not\in P$ for $i = 1, \dots, k, \ \ M(P) = \emptyset$ and so $\hat G(P)  =  G_1(P)$.

Now assume that $x \in P$ and $z \in G_1(x) = \hat R (F(x))$.  That is, there exists $y \in F(x)$ such that
$z \in \hat R(y)$.

If $y \in X \setminus C$, then $z = y \in X \setminus C$ because on $X \setminus C$, $\hat R$ is the identity.

If $y \in C \setminus Int_C Q = A_1$, then $z \in \hat R(y) \subset \{ y \} \cup R_A(y) \subset A_1$.

Finally, if $y \in Int_C Q$, then since $P$ is an $F_C$ inward set, $y \in F_C(x) \subset Int_C P$.  That is,
$y \in Int_C Q \cap Int_C P = Q^{\circ} \cap P^{\circ}$. Hence, $z \in R_A(y) \subset A$. \vspace{.25cm}

Now we show that any maximal $(\hat G)_C$ solution path $z_0, z_1, \dots$ has finite length and so $C_{\pm}$ with respect to $\hat G$ is empty.

Notice first that on $A_1, \ F = G_1 = \hat G$
For $x \in A_1$, $F(x) \cap O = \emptyset$.  That is, $F(x) \subset A$ and $M(A_1) = \emptyset$. So $\hat G(x) = G_1(x) = F(x)$.
Hence, on the $F_C$ + invariant set $A_1$, $\hat G_C = F_C$. It follows that on $A_1$, $(\hat G_C)^{N+1} = \emptyset$
Thus, it follows that any solution path for $F_C, (G_1)_C$ or $(\hat G)_C$ which
enters $A_1$ is of finite length.

Any infinite $F_C$ solution path eventually lies in the neighborhood $Q^{\circ} \cap P^{\circ}$ of $C_{\pm}$, by Proposition \ref{relprop06a}.
The relation $\hat R$ sends $Q^{\circ} \cap P^{\circ}$ into $A \subset A_1$. So any $(G_1)_C$ or $\hat G_C$ solution path which is also an $F_C$ solution path
must have finite length.

Now assume that the maximal $(\hat G)_C$ solution path $z_0, z_1, \dots$ is not an $F_C$ solution path and let $k$ be the minimal index such that
$z_{k+1} \not\in F(z_k)$.\vspace{.25cm}

Case 1 ($z_{k+1} = G_1(z_k)$): This implies, there exists $u \in F(z_k)$ with $z_{k+1} \in \hat R(u)$ and $z_{k+1} \not= u$. From the
definition of $\hat R$ it follows that $z_{k+1} \in R_A(u) \subset A_1$. Thus, the solution path enters $A_1$ and so is finite. Furthermore it is
a maximal $(G_1)_C$ solution path.  \vspace{.25cm}

Case 2 ($z_k = x_i$ and $z_{k+1} \in K_i$) Notice that $z_{k+1} \not= y_i$ since $z_{k+1} \not\in F(z_k)$. Since $K_i \subset P$,
(\ref{eqanom10}) implies that $G_1(z_{k+1}) = \hat G(z_{k+1}) \subset A_1 \cup (X \setminus C)$. Thus, the solution path either terminates
at $z_{k+1}$ or else it enters $A_1$.  For either possibility the solution path is finite.\vspace{.5cm}

The same argument shows that any maximal $(G_1)_C$ solution path is finite.

\end{proof}

\begin{ex}\label{exanom05} On $X = \R$ let $f$ be the homeomorphism given by \\ $f(x) = 2x$. \end{ex}

Let $A = \{ x : |x| \ge \ep \}$. For the closed relation. $G = R_A \circ f$, the repeller $\{ 0 \}$ has been eliminated. \vspace{.5cm}

\begin{ex}\label{exanom06} On $X = \R^2$ let $f$ be the homeomorphism given by \\ $f(x,y) = (2x,\frac{1}{2}y)$. \end{ex}

Let $Q = \{ (x,y) : |x| \le \ep \}, P = \{ (x,y) : |y| \le \ep/2 \}, A = \{ (x,y) : |x| \ge \ep \}$ and
$B = A \cup \{ (x,y) : |y| \ge \ep/2 \}$. Define $\hat R = 1_B \cup (R_A \cap [(P \cap Q) \times A])$.
Let $M = \{ (0,\ep) \} \times (P \cap Q)$.

For the surjective closed relation $\hat G = M \cup \hat R \circ f$, the hyperbolic fixed point $\{ (0,0) \}$ has been eliminated.

\vspace{1cm}

\subsection{{\bf Solution Space Dynamics}}\label{solution1}
\vspace{.5cm}

Let $F_1$ $F_2$ be  relations on compact metric spaces $X_1$ and $X_2$. A continuous function $h : X_1 \to X_2$ maps $F_1$ to $F_2$
(written $h : F_1 \to F_2$) when
\begin{equation}\label{eqsol01}
(h \times h)(F_1) \subset F_2, \qquad \text{or, equivalently} \qquad h \circ F_1 \subset F_2 \circ h,
\end{equation}
where $h \times h$ is the product map defined by $(x,y) \mapsto (h(x),h(y))$. Thus, the class of closed relations on compact metric spaces becomes
a category. The morphism $h$ is an \emph{isomorphism} \index{isomorphism} when $h$ is a homeomorphism and the inclusions in (\ref{eqsol01}) are equalities.
Notice that an inclusion between two functions with the same domain is always an equality.

We say that $h : F_1 \to F_2$ satisfies the \emph{pullback condition}\index{pullback condition} when for all $(x_1,x_2) \in X_1 \times X_2$
\begin{equation}\label{eqsol01a}\begin{split}
(x_2,h(x_1)) \in F_2 \qquad \Longrightarrow \hspace{3cm} \\ \text{there exists } \ \ z \in X_1 \ \ \text{such that} \ \  (z,x_1) \in F_1, \ h(z) = x_2.
\end{split}\end{equation}
This condition says that any $F_2$ solution path which terminates at $h(x)$ lifts via $h$ to an $F_1$ solution path which terminates at $x$.

Notice that if $Dom(F_1) = X_1$ and $F_2$ is a map on $X_2$, then $h : F_1^{-1} \to F_2^{-1}$ satisfies the pullback condition.

\begin{prop}\label{propsol01} Let $F_1$, $F_2$ be  closed relations on  $X_1$ and $X_2$, respectively.
Assume that the continuous map $h : X_1 \to X_2$ maps $F_1 $ to $F_2$.

\begin{enumerate}
\item[(a)] $h$ maps $F_1^{-1} $ to $F_2^{-1}$ and for $\A = \O, \NN, \G$ and $\CC$\\ $h$ maps $\A F_1$ to $\A F_2$.

\item[(b)] If $\xx : [n_1,n_2] \to X_1$ is an $F_1$ solution path, then $h \circ \xx : [n_1,n_2] \to X_2$ is an $F_2$ solution path.

\item[(c)] If $S_i$ are the shift maps on $\S_+(F_i)$ and $\S(F_i)$ for $i = 1,2$, then $\xx \mapsto h \circ \xx$ defines continuous maps
$h_* : \S_+(F_1) \to \S_+(F_2)$ and $h_* : \S(F_1) \to \S(F_2)$, each mapping the shift $S_1$ to $S_2$.
If $h$ is an isomorphism then both maps $h_*$ are isomorphisms.

\item[(d)] If $A \subset X_1$ is + viable, - viable or viable for $F_1$, then $h(A) \subset X_2$ satisfies the corresponding property for
$F_2$.

\item[(e)] If $A \subset X_1$ is chain transitive for $F_1$, then $h(A) \subset X_2$ is chain transitive for
$F_2$.

\item[(f)] Assume that $h$ satisfies the pullback condition with $h(X_1) = X_2$. If $A$ is a minimal subset of $X_1$
and either $A = h^{-1}(h(A))$ or $A$ is $F_1^{-1}$
+ invariant, then $h(A)$ is a minimal subset of $X_2$.  In particular, if  $F_1$ is minimal, then $F_2$ is minimal.
 If $B$ is a minimal subset of $X_2$, then there exists $A$ a minimal subset of $X_1$ with $h(A) = B$.  In particular, if  $F_2$ is minimal
  and $h$ is irreducible, then $F_1$ is minimal.

\item[(g)] If $B \subset X_2$ is - viable and $h$ satisfies the pullback condition, then then $h^{-1}(B)$ is - viable for $F_1$.

\item[(h)] If $B \subset X_2$, then for $n \in \Z$
\begin{equation}\label{eqsol02}
F_1^n(h^{-1}(B)) \ \subset \ h^{-1}(F_2^n(B)). \hspace{2cm}
\end{equation}
In particular, if $B$ is + invariant for $F_2$, then $h^{-1}(B)$ is + invariant for $F_1$.

If $h$ satisfies the pullback condition, then equality holds in (\ref{eqsol02}) for $n \ge 1$ and so if $B$ is  invariant for $F_2$,
then $h^{-1}(B)$ is invariant for $F_1$.

 \item[(i)] If $U \subset X_2$ is inward for $F_2$ with associated attractor $U_{\infty} = \bigcap_{n=1}^{\infty} F_2^n(U)$,
 then $h^{-1}(U)$ is inward for $F_1$ with associated attractor contained in $h^{-1}(U_{\infty})$. If $h$ satisfies the pullback
 condition, then the associated attractor for $h^{-1}(U)$ is equal to $h^{-1}(U_{\infty})$.
 \end{enumerate}
 \end{prop}

 \begin{proof}  For (a) see \cite{A93} Proposition 1.17. (b) is obvious and implies (c) and (d).

(e)  If $h$ maps $F_1$ to $F_2$, then clearly $h|A$ maps $(F_1)_A$ to $(F_2)_{h(A)}$ and so by (a) it maps $\CC((F_1)_A)$ to $\CC((F_2)_{h(A)})$.
 So $\CC((F_1)_A) = A \times A$ implies $\CC((F_2)_{h(A)}) = h(A) \times h(A)$.

 (g) If $h(x) \in B$, then there exists a solution path $\xx : [-\infty,0] \to B$ with $\xx(0) = h(x)$.
 This lifts to a solution path which terminates at $x$.
 Hence, $h^{-1}(B)$ is - viable.

 (f) If $B_1$ is a nonempty - viable subset of $B$, then by (g) $h^{-1}(B_1)$ is a nonempty - viable subset of $X$.

 First, assume $B = h(A)$. If $A = h^{-1}(h(A))$, then
 $A \cap h^{-1}(B_1) = h^{-1}(B_1)$.
 If $A$ is $F^{-1}$ + invariant, then $A \cap h^{-1}(B_1)$ is  - viable. So in either case, $A \cap h^{-1}(B_1)$ is a nonempty - viable subset of $A$.
 Because $A$ is minimal, $A \cap h^{-1}(B_1) = A$ and so $B_1 = h(A \cap h^{-1}(B_1)) = B$. In particuar, with $A = X_1$, both conditions are satisfied.

 Now assume that $B$ is a minimal subset of $X_2$. By (g) $h^{-1}(B)$ is - viable and so it contains a nonempty minimal subset $A$. By
 (d) $h(A)$ is a viable subset of $B$ and so, by minimality,
 $h(A) = B$. If $B = X_2$, then $A$ is a closed subset of $X_1$ with $h(A) = X_2$.  So irreducibility implies $A = X_1$.

 (h) The inclusion (\ref{eqsol02})is clear from $h \circ F_1 \subset F_2 \circ h$. If $h$ satisfies the pullback condition and $x \in h^{-1}(F_2^n(B))$ with $n \ge 1$,
 then there is an $F_2$ solution path of length $n$ which begins in $B$ and terminates at $h(x)$. The lift is an $F_1$ solution path
 which begins in $h^{-1}(B)$ and terminates at $x$. So $x \in F_1^n(h^{-1}(B))$.

 (i) If $A \subset \subset B$ in $X_2$, then $h^{-1}(A) \subset \subset h^{-1}(B)$ in $X_1$. Hence, $F_2(U) \subset \subset U$ implies
 $F_1(h^{-1}(U)) \subset \subset h^{-1}(U)$ by (\ref{eqsol02}). The attractor for $h^{-1}(U)$ is $\bigcap_{n=1}^{\infty} \ F_1^n(h^{-1}(U))$
 and $h^{-1}(U_{\infty}) = h^{-1}(\bigcap_{n=1}^{\infty} \ F_2^n(U)) = \bigcap_{n=1}^{\infty} \ h^{-1}(F_2^n(U)))$.
So the attractor results follow from (f).

 \end{proof}

 For a closed relation $F$ on $X$ we define the \emph{derivative} \index{derivative} $F'$ to be the closed relation on $F$ given by
\begin{equation}\label{eqsol02a}
(x_2,y_2) \in F'(x_1,y_1) \quad \Longleftrightarrow \quad x_2 = y_1.
\end{equation}

If we think of $F$ as a directed graph on the set of vertices $X$, then $F'$ is the associated directed graph on the edges.

\begin{prop}\label{propsol02} Let $F$ be a closed relation on $X$.

\begin{enumerate}
\item[(a)]  The map $tw : F \to F^{-1}$ given by $tw(x,y) = (y,x)$ is a homeomorphism which maps
$(F')^{-1}$ on $F$ isomorphically to $(F^{-1})'$ on $F^{-1}$.

\item[(b)]  The projections $\pi_1, \pi_2 : F \to X$ each map $F'$ to $F$. Furthermore, $\pi_1 : F' \to F$ satisfies the pullback condition.

%\item[(c)]  If $B \subset X$, then for $n \in \Z_+$
%\begin{equation}\label{eqsol02a}
%(F')^n(\pi_1^{-1}(B)) \ = \  \pi_1^{-1}(F^n(B)). \hspace{2cm}
%\end{equation}

\item[(c)] If $G \subset F$, then $G$ is a relation on $X$ such that
\begin{align}\label{eqsol03}\begin{split}
G(X) \ = \ X \qquad &\Longrightarrow \qquad F'(G) \ = \ F, \\
G^{-1}(X) \ = \ X \qquad &\Longrightarrow \qquad (F')^{-1}(G) \ = \ F.
\end{split} \end{align}
\begin{align}\label{eqsol04}\begin{split}
\pi_1(F'(G)) \ \subset \  \pi_2(G) \  = & \  G(X), \quad  \text{and} \\
x \in \pi_2(G)\setminus \pi_1(F'(G)) \quad  &\Longrightarrow   \quad F(x) = \emptyset.
\end{split} \end{align}

\item[(d)] If $G \subset F$ is inward for $F'$, then there exists $U \subset X$ inward for $F$ such that
\begin{equation}\label{eqsol05}
F'(G) \subset \subset \pi_1^{-1}(U) \subset \subset G.
\end{equation}
In particular, the associated $F'$ attractors for the $F'$ inward sets $G$ and $\pi_1^{-1}(U)$ agree and equal $\pi_1^{-1}(U_{\infty})$, where
$U_{\infty}$ is the $F$ attractor associated with $U$.

\item[(e)] If $X$ is chain transitive for $F$, then  $F$ is chain transitive for $F'$.

%\item[(f)] $F$ on $X$  is minimal if and only if $F'$ on $F$  is minimal. NO $F$ can be minimal and $F'$ not.

\end{enumerate}
\end{prop}

\begin{proof} (a) follows from
\begin{align}\label{eqsol06}\begin{split}
((x,y),(y,z)) \ &\in \ F' \quad  \quad \ \  \Longleftrightarrow \\
((y,z),(x,y)) \ &\in \ (F')^{-1} \quad  \Longleftrightarrow \\
((z,y),(y,x)) \ &\in \ (F^{-1})'.
\end{split} \end{align}

(b) That $\pi_1$ and $\pi_2 $ each map $ F'$ to $F$ follows from
\begin{align}\label{eqsol08}\begin{split}
(\pi_1 \times \pi_1)((x,y),(y,z))& \ = \ (x,y),  \\(\pi_2 \times \pi_2)((x,y),(y,z))& \ = \ (y,z).
\end{split} \end{align}
 If $(y,z) \in F$ and $(x,y) \in F$, then $((x,y),(y,z)) \in F'$. Since $y = \pi_1(y,z)$ the pullback condition for $\pi_1 : F' \to F$ holds.

%(c) The inclusion one way follows from (\ref{eqsol02}). For the other direction it suffices to consider the case $n = 1$ and use induction.
%
%If $(x,y) \in \pi_1^{-1}(F(B))$, then there exists $z \in B$ such that $(z,x) \in F$. So $(z,x) \in \pi_1^{-1}(B)$ and $(x,y) \in F'(z,x)$.

(c) Assume $(x,y) \in F$. Since $G(X) = X$, there exists $z \in X$ such that $(z,x) \in G$. So $(x,y) \in F'(z,x) \subset F'(G)$.

If $G^{-1} \subset F^{-1}$ and $G^{-1}(X) = X$, then by what we have just shown $(F^{-1})'(G^{-1}) = F^{-1}$. Now apply the homeomorphism
$tw : F^{-1} \to F$ which maps $(F^{-1})'$ to $(F')^{-1}$ by (a). The result follows because $tw(G^{-1}) = G$.

Clearly, $x \in \pi_1(F'(G))$ if and only if there exists $(z,x) \in G$ and $(x,y) \in F$. On the other hand, $x \in \pi_2(G)$ if and
only if there exists $(z,x) \in G$. There does not also exist $y$ such that $(x,y) \in F$ if and only if $F(x) = \emptyset$.

(d) If $G$ is inward for $F'$, there exists $\ep > 0$ such that $V_{2 \ep}(F'(G)) \subset G$. For every $x \in X$, there exists
$\d_x $ with $F(V_{\d_x}(x)) \subset V_{\ep}(F(x))$, i.e. $V_{\d_x}(x)$ is contained in the open set $F^*(V_{\ep}(F(x)))$.

Let $K = G(X) = \pi_2(G)$ which is closed. Let $O = \bigcup \{ V_{\d_x}(x) : x \in K \}$. Assume $(x_1,y_1) \in F$ with $x_1 \in O$.
There exists $x \in K$ with $d(x,x_1) < \d_x < \ep$ and so there exists $y \in F(x)$ such that $d(y,y_1) < \ep$.
By definition, $(x,y) \in F'(G)$ and so the open ball $V_{\ep}(x_1,y_1)$ in $F$ is contained in $V_{2 \ep}(F'(G)) \subset G$.

Now $K$ is closed and $O$ is open. Hence, there exists a closed set $U$ such that $K \subset \subset U \subset \subset O$.
Since $\pi_1(F'(G)) \subset K$ by (\ref{eqsol04}) we have $F'(G) \subset \subset \pi_1^{-1}(U)$. On the other hand,
if $(x_1,y_1) \in \pi_1^{-1}(U)$ then $x_1 \in O$ implies that $V_{\ep}(x_1,y_1) \subset G$.

The equality of the attractors follows from Proposition \ref{propsol01} (i) because $\pi_1$ satisfies the pullback condition.

(e) If $X$ is chain transitive for $F$ then $F$ is surjective and so (\ref{eqsol03}) with $G = F$ implies that $F'$ is surjective.
Hence, $F \not= \emptyset$ and so by Lemma
\ref{rellem08aab} it suffices to show that $F'$ admits no nonempty, proper inward subset.

Suppose instead that $G$ were such an inward subset. Then $F$ surjective implies that $F'(G) \not= \emptyset$. By (d) there would exist
an $F$ inward set $U$ such that $F'(G) \subset \pi_1^{-1}(U) \subset  G$. Since $F'(G)$ is nonempty and $G \subset F$ is proper it follows
that $U$ would be a nonempty, proper $F$ inward subset of $X$. Hence, $F$ is not chain transitive. Contrapositively, no such $G$ can exist.

%(f) Since $\pi_1$ maps $F'$ on $F$ onto $F$ on $X$ and it satisfies the pullback condition, it follows from Proposition \ref{propsol01}(f) that
%$F'$ minimal implies $F$ is minimal.

\end{proof}

For $n_1 \le n_2 \in \Z$ we defined $\S([n_1,n_2])$ (with the relation $F$ understood) to be the set of $F$ solution paths $\xx : [n_1,n_2] \to X$.
The paths have \emph{length}\index{solution path!length} $n = n_2 - n_1$. On $\S([0,n])$ we define the shift relation $S_n$.
\begin{align}\label{eqsol09}\begin{split}
\S([0,n])   \ =_{def} \  \{ (x_0, \dots, x_n) \in X^{n+1} &: (x_{i-1},x_i) \in F \ \ \text{for} \ \ i = 1, \dots n \}, \\
 (y_0, \dots, y_n) \in S_n(x_0, \dots,x_n) \ \ \text{when} &\ \  y_{i-1} = x_i \in F \ \ \text{for} \ \ i = 1, \dots n.
\end{split}\end{align}
 On $\S([n_1,n_2])$  we define $S_n$ using the translation homeomorphism from $\S([n_1,n_2])$ to $\S([0,n])$.

 Thus, $S_n$ consists of pairs $((x_0, \dots,x_n),(x_1, \dots,x_n, y))$. Mapping such a pair to $(x_0, \dots,x_n,y)$ we obtain
 a homeomorphism from $S_n \subset \S([0,n]) \times \S([0,n])$ onto $\S([0,n+1])$. The map induces an isomorphism from the
 derivative $S_n'$ on $S_n$ to $S_{n+1}$ on $\S([0,n+1])$. Observe that $S_0$ on $\S([0,0])$ is $F$ on $X$ and $S_1$ on $\S([0,1])$
 is $F'$ on $F = \S([0,1])$. Thus, by induction we obtain

 \begin{prop}\label{propsol03} For a closed relation $F$, the relation $S_n$ on $\S([n_1,n_2])$ with $n_2 - n_1 = n$ is
 isomorphic to the $n^{th}$ derivative
 $F^{(n)}$ on $F^{(n-1)}$. \end{prop} \vspace{.5cm}

 For $-\infty \le n_1 \le m_1 \le m_2 \le n_2 \le \infty$ the map $\pi_{[m_1,m_2]}: \S([n_1,n_2]) \to \S([m_1,m_2])$ is the
 projection map by restriction which maps the shift relation on the domain to the shift relation on the range.

 \begin{lem}\label{lemsol03a} For $-\infty < n_1  \le m_2 \le n_2 \le \infty$ the map $\pi_{[n_1,m_2]}: \S([n_1,n_2]) \to \S([n_1,m_2])$
 satisfies the pullback condition. \end{lem}

 \begin{proof}: Proceed as in Proposition \ref{propsol02} (b).

 \end{proof}

 \begin{cor}\label{corsol04} Assume $F$ is a closed relation on $X$ and $n \in \Z_+$.
 \begin{enumerate}
 \item[(a)] If $G \subset \S([0,n])$ is inward for $S_n$, then there exists $U \subset X$ inward for $F$ such that
$ S_n^n(G) \subset \subset \pi_0^{-1}(U) \subset  \subset G$ and
 the associated $S_n$ attractors for the $S_n$ inward sets $G$ and $\pi_0^{-1}(U)$ agree and equal $\pi_0^{-1}(U_{\infty})$, where
$U_{\infty}$ is the $F$ attractor associated with $U$.

\item[(b)] If $X$ is chain transitive for $F$, then  $\S([0,n])$ is chain transitive for $S_n$.
\end{enumerate} \end{cor}

\begin{proof} (a) For $n = 1$, this is Proposition \ref{propsol02} (e).

Using the isomorphism of Proposition \ref{propsol03}, Proposition \ref{propsol02} (e) implies there exists $\tilde G \subset \S([0,n-1])$
inward for $S_{n-1}$ such that $S_n(G) \subset \subset \pi_{[0,n-1]}^{-1}(\tilde G) \subset \subset G.$

For clarity, write $\pi_{0}^{(k)} = \pi_0 : \S([0,k]) \to X$ so that $\pi_0^{(n)} = \pi_0^{(n-1)} \circ \pi_{[0,n-1]}$.

By induction hypothesis, there exists $U$ inward for $F$ such that
$S_{n-1}^{n-1}(\tilde G) \subset \subset (\pi_0^{(n-1)})^{-1}(U) \subset \subset \tilde G$.

Apply $(\pi_{[0,n-1]})^{-1}$, noting that from the pullback condition \\
$(\pi_{[0,n-1]})^{-1}(S_{n-1}^{n-1}(\tilde G)) = S_n^{n-1}((\pi_{[0,n-1]})^{-1}(\tilde G)$. Furthermore this set contains
$S_n^n(G)$.

The attractor results follow from the pullback condition, as in Proposition \ref{propsol01} (i).

(b) follows from by induction and (g) of
Proposition \ref{propsol02} using the isomorphism of Proposition \ref{propsol03}.

\end{proof}

\begin{theo}\label{theosol05} Let $F$ be a surjective closed relation on $X$.
\begin{enumerate}
\item[(a)] If $G \subset \S_+(F) = \S([0,\infty])$ is inward for the shift map $S$, then there exists $U \subset X$ inward for $F$ and
a positive integer $N$ such that $S^N(G) \subset \subset \pi_0^{-1}(U) \subset \subset G$.  If $A = \bigcap_{k=1}^{\infty} \ F^k(U)$ is
the associated attractor for $U$, then $\pi_0^{-1}(A) \ = \ \bigcap_{k=1}^{\infty} S^k(G)$ is the associated attractor for $G$.

\item[(b)]  If $G \subset \S(F) = \S([-\infty,\infty])$ is inward for the shift homeomorphism $S$, then there exists $U \subset X$ inward for $F$ and
 positive integers $N, m$ such that $S^N(G) \subset \subset \pi_{-m}^{-1}(U) \subset \subset G$.  If $A = \bigcap_{k=1}^{\infty} \ F^k(U)$ is
the associated attractor for $U$, then the associated attractor for $G$ is $\S(F_A)$, the subspace of all bi-infinite orbit sequences contained in $A$.

\item[(c)] If $X$ is chain transitive for $F$, then each of $\S_+(F)$ and $\S(F)$ is chain transitive for the corresponding shift map.

\item[(d)] If $B \subset X$ is minimal for $F$, then there exist unique $A_+ \subset \S_+(F_B)$ and $A \subset \S(F_B)$ which are the unique minimal
subsets of $\S_+(F)$ and $ \S(F)$ which are mapped onto $B$ by the projections $\pi_0$.
\end{enumerate}\end{theo}

\begin{proof}  We are using the metric (\ref{eqrel17a}) on the infinite solution path spaces.

(a) There exists $\ep > 0$ such that $V_{2 \ep}(S(G)) \subset G$. Fix $m > \frac{1}{ \ep}$. Let
\begin{equation}\label{eqsol10}
K \  = \ \pi_{[1,m+1]}(G) \ = \ \pi_{[0,m]}(S(G)) \ \subset \ \S([0,m]).
\end{equation}

Let $\tilde G = \bar V_{\ep}(K)$.  By the choice of $\ep$, $V_{\ep}(\pi_{[0,m]}^{-1}(\tilde G)) \subset G$.
For suppose $x, x' \in \S_+(F), x'' \in S(G)$ and
such that $d(x,x') < \ep$ and $d(\pi_{[0,m]}(x'),\pi_{[0,m]}(x'')) < \ep$. Thus, $d(x(i),x'(i)) < \ep$ for all $i \le \frac{1}{ \ep}$ and \\
$d(x'(i),x''(i)) < \ep$ for all $i \le m$. Since $m > \frac{1}{ \ep}$ it follows that $d(x(i),x''(i)) < 2\ep$ for all  $i \le \frac{1}{ \ep}$
and so for all $i \le \frac{1}{2 \ep}$. That is, $d(x,x'') < 2 \ep$ and so $x \in G$.

Hence,
\begin{equation}\label{eqsol11}
S(G) \subset \pi_{[0,m]}^{-1}(K) \subset \subset \pi_{[0,m]}^{-1}(\tilde G) \subset \subset G.
\end{equation}

It then follows that $\tilde G$ is inward in $\S([0,m])$ because, since $\pi_{[0,m]}$ onto implies $\tilde G \subset \pi_{[0,m]}(G)$,
\begin{equation}\label{eqsol11aa}
S(\tilde G) \subset S(\pi_{[0,m]}(G)) = \pi_{[0,m]}(S(G)) = K \subset \subset \tilde G.
\end{equation}

By Corollary \ref{corsol04} there exists $U \subset X$ inward for $F$ such that $S^m(\tilde G) \subset \subset (\pi^{(m)}_0)^{-1}(U) \subset \subset \tilde G$.
By Proposition \ref{propsol01}(i)
\begin{align}\label{eqsol12} \begin{split}
S^{m+1}(G) \subset S^m(\pi_{[0,m]}^{-1}(\tilde G)) &\subset \pi_{[0,m]}^{-1}(S^m(\tilde G))\\
 \subset \subset \pi_0^{-1}(U) \subset \subset &\pi_{[0,m]}^{-1}(\tilde G) \subset G,
 \end{split}\end{align}
 as required with $n = m+1$.

 Because $\pi_0$ satisfies the pullback condition, $\pi_0^{-1}(A) = \bigcap_{k=1}^{\infty} S^k(\pi_0^{-1}(U))$ which equals
 $\bigcap_{k=1}^{\infty} S^k(G)$, the associated attractor of $G$.

 (b) The beginning is essentially the same as that of (a).

 There exists $\ep > 0$ such that $V_{2 \ep}(S(G)) \subset G$. Fix $m > \frac{1}{\ep}$. Let
\begin{equation}\label{eqsol10a}
K \  = \ \pi_{[-m+1,m+1]}(G) \ = \ \pi_{[-m,m]}(S(G)) \ \subset \ \S([-m,m]).
\end{equation}

Let $\tilde G = \bar V_{\ep}(K)$.  By the choice of $\ep$, $V_{\ep}(\pi_{[-m,m]}^{-1}(\tilde G)) \subset G$. So
\begin{equation}\label{eqsol11a}
S(G) \subset \pi_{[-m,m]}^{-1}(K) \subset \subset \pi_{[-m,m]}^{-1}(\tilde G) \subset \subset G.
\end{equation}

As above, $\tilde G$ is inward in $\S([-m,m])$.

By Corollary \ref{corsol04} there exists $U \subset X$ inward for $F$ such that
$S^{2m+1}(\tilde G) \subset \subset \pi_{-m}^{-1}(U) \subset \subset \tilde G$.
By Proposition \ref{propsol01}(i)
\begin{align}\label{eqsol12a} \begin{split}
S^{2m+2}(G) \subset S^{2m+1}(\pi_{[-m,m]}^{-1}(\tilde G)) &\subset \pi_{[-m,m]}^{-1}(S^{2m+1}(\tilde G))\\
 \subset \subset \pi_{-m}^{-1}(U) \subset \subset &\pi_{[-m,m]}^{-1}(\tilde G) \subset G,
 \end{split}\end{align}
 as required with $n = 2m+2$.

 Notice that $\pi_{[-m,m]}$ and $\pi_{-m}$ need not satisfy the pullback condition so we used the inclusion form of (\ref{eqsol02}).

 It is still true that the associated attractors for $\pi_{-m}^{-1}(U)$ and for $G$ agree, but it is $ \bigcap_{k=1}^{\infty} S^k(\pi_{-m}^{-1}(U))$
 which is usually a proper subset of $\pi_0^{-1}(A) = \pi_0^{-1}(\bigcap_{k=1}^{\infty} F^k(U)))$.

 If $\xx \in \S(F_A)$, then $\xx(i) \in A \subset U$ for all $i$. Hence, $\pi_{-m}(S^{-k}(\xx)) \in U$ for all $k \in \Z$, i.e.
 $\xx \in S^k((\pi_{-m})^{-1}(U))$ for all $k$.

 On the other hand, suppose that $S^{-k}(\xx) \in \pi_{-m}(U))$ for all $k > 0$. That is, $\xx(-m-k) \in U$ for all $k > 0$.
 This says $\xx(i) \in U$ for all $i < -m$. On the other hand, $U$ is + invariant for $F$ and so it follow that $\xx(i) \in U$ for all $i \in \Z$.
 Now $\xx(i) \in F^k(\xx(i-k)) \subset F^k(U)$.  So for every $i$, $\xx(i) \in  \bigcap_{k=1}^{\infty} S^k(U) = A$.  Thus, $\xx \in \S(F_A)$.

  (c) Since  $F$ is surjective, by Lemma \ref{rellem08aab} it suffices to show that neither solution path
  space contains a proper inward subset and this follows from (b) and (c).

  (d) By Proposition \ref{relprop06ac} $F_B$ is irreducible and there exists $W$ a dense $G_{\d}$
  invariant subset of $B$ on which $F_W$ is a homeomorphism
  and with $F_B(x)$ and $F_B^{-1}(x)$ singletons for $x \in W$. Let $A_+$ and $A$ be the closures
  in $\S_+(F_B)$ and $\S(F_B)$ of  $\S_+(F_W)$ and $\S(F_W)$.
  It is clear that $\S_+(F_W)$ and $\S(F_W)$ are both viable with the projections $\pi_0$
  homeomorphisms onto $W$. So their closures $A_+$ and $A$ are viable.
  % with   $\pi_0 : A_+ \to B$ and $\pi_0 : A \to B$ irreducible mappings.
  Clearly, $\S_+(F_B) \cap (\pi_0)^{-1}(W) = \S_+(F_W)$ and $\S(F_B) \cap (\pi_0)^{-1}(W) = \S(F_W)$.
  So if $C$ is a minimal subset $\S_+(F_B)$ then $C$ contains
  $\S_+(F_W)$ and so contains $A_+$.  Since $A_+$ is viable, it follows that it is the unique minimal
  subset of $\S_+(F_B)$. Similarly, $A$ is
  the unique minimal subset of $\S(F_B)$. If $C$ is a + invariant subset of $\S_+(F)$ (or an invariant
  subset of $\S(F)$) with $\pi_0(C) \subset B$, then
  $C \subset \S_+(F_B)$ (resp. $C \subset \S(F_B)$). Thus, $A_+$ and $A$ are the unique minimal subsets
  of $\S_+(F)$, and $\S(F_B)$ respectively, which map
  onto $B$.

  \end{proof}

\textbf{  Remark:} If $f$ is a minimal homeomorphism on an infinite $X$ and $x_0 \in X$, then for
$F = f \cup \{ (x_0,x_0) \}$ the only minimal subset of $X$ is
$\{ x_0 \}$.  On the other hand, $\S(f)$ is a minimal subset of $\S(F)$. It maps onto $X$ which is
a minimal subset for $f \subset F$.  On the other hand, suppose that
$X = \{ 0, 1 \}$. With $F = X \times X$, the only minimal subsets of $X$ are $\{ 0 \}$ and $\{ 1 \}$
which are the images of the two fixed points for $S$ in $\S(F)$.
If $F_1 = \{ (0,1),(1,0) \}$, then $X$ is minimal for $F_1$ and is the image of the periodic orbit of
period two in $\S(F)$. If $A$ is any minimal subset of $\S(F)$ other than
these three then since neither the word 01 nor the word 10 can be excluded from $\xx \in A$ it follows that $\pi_0(A) = X$.\vspace{.5cm}

  \begin{cor}\label{corsol06} Let $F$ be a surjective closed relation on  $X$
 \begin{enumerate}
\item[(a)] A subset  $K \subset \S_+(F) $ is an attractor for the map $S$ if and only if there exists an attractor $A$ for $F$ such that
$K = \pi_0^{-1}(A)$. The subset $K$ is a chain component for the map $S$, if and only if $C = \pi_0(K)$ is a chain component for $F$
and $K = \S_+(F_C)$.

\item[(b)]  A subset  $K \subset \S(F) $ is an attractor (or repeller) for the homeomorphism $S$ if and only if there exists an attractor
 (resp. a repeller) $A$ for $F$ such that
$K = \S(F_A)$. The subset $K$ is a chain component for the homeomorphism $S$, if and only if $C = \pi_0(K)$ is a chain component for $F$
and $K = \S(F_C)$.
\end{enumerate}
\end{cor}

\begin{proof}  The attractor results in (a) and (b) follow directly from (a) and (b) of Theorem \ref{theosol05}.

The function $\xx \mapsto \bar \xx$ with $\bar \xx(n) = \xx(-n)$ for all $n \in \Z$ provides an isomorphism from
$S^{-1}$ on $\S(F)$ to $S$ on $\S(F^{-1})$.  Hence, the repeller results on $\S(F)$ follow from the attractor results on $\S(F^{-1})$.

By \cite{A93} Corollary 4.11 the chain components for a map $S$ are $S$ invariant subsets.

Assume $K$ is a chain component for $S$ on $\S_+(F)$ with $C = \pi_0(K)$. If $\xx \in K$, then $S^k(\xx) \in K$ for all $k \in \Z_+$ and
so $\xx(k) = \pi_0(S^k(\xx)) \in C$.  That is, $\xx \in \S_+(F_C)$.

By Proposition \ref{propsol01}(e) $C$ is chain transitive for $F$ and so by Corollary \ref{relcor10aaa}, $C$ is contained in
some chain component $\tilde C$ for $F$. By Theorem \ref{theosol05}(c) $\S_+(F_{\tilde C})$ is chain transitive and it
contains $\S_+(F_C) \supset K$. Maximality of the chain component $K$ (see Corollary \ref{relcor10aaa} again)
implies that $\S_+(F_{\tilde C}) = \S_+(F_C) = K$. Since $F_{\tilde C}$ is surjective
it follows that $C = \pi_0(K) = \pi_0(\S_+(F_{\tilde C})) = \tilde C$.

Assume $K$ is a chain component for $S$ on $\S(F)$ with $C = \pi_0(K)$. If $\xx \in K$, then $S^k(\xx) \in K$ for all $k \in \Z$ and
so $\xx(k) = \pi_0(S^k(\xx)) \in C$.  That is, $\xx \in \S(F_C)$. The rest of the proof for $\S(F)$ is identical to the $\S_+(F)$ proof
with $\S_+$ replaced by $\S$ throughout.

\end{proof}

\section{ \textbf{Semiflow Relations}}\label{semiflow}\vspace{.5cm}

For semiflow relations we are essentially following \cite{BK} adjusted to the relation notation.

%Let $\R_+ = [0,\infty), \R_+^* = [0,\infty]$.

For $\Phi$ a closed subset of $ X \times \R_+ \times X$ and
$t \in \R_+$ we let $\phi^t = \{ (x,y) : (x,t,y) \in \Phi \} $
so that each $\phi^t$ is a closed relation on $X$.
With $X$ a compact metric space we call $\Phi$ a \emph{semiflow relation on $X$}\index{semiflow relation}
when it satisfies the following two conditions:
\vspace{.5cm}

(i) \textbf{Initial Value Condition:} \index{Initial Value Condition}$\phi^0 = 1_X$ or, equivalently, for every $x \in X$,
\begin{equation}\label{eqDom}
(x,0,y) \in \Phi \qquad \Longleftrightarrow \qquad y = x.
\end{equation}\vspace{.5cm}

(ii) \textbf{Kolmogorov Condition:} \index{Kolmogorov Condition} For all $t_1, t_2 \in \R_+$ $\phi^{t_1} \circ \phi^{t_2} = \phi^{t_1 + t_2}$
or, equivalently, for $ x, y \in X$
\begin{align}\label{eqKol}\begin{split}
(x,t_1+t_2,y) \in \Phi &\qquad \Longleftrightarrow \\
\text{there exists} \ \
z \in X \ \ \text{such that} &\ \ (x,t_1,z), (z,t_2,y) \in \Phi.
\end{split}\end{align} \vspace{.5cm}

We define
 \begin{equation}\label{eqsemi00}
 \overline{\Phi}  \ =_{def} \  \{ (y,t,x) : (x,t,y) \in \Phi \} \hspace{2cm}
 \end{equation}
 % \index{ $\overline{\Phi}$ }
 so that
$(\bar \phi)^t = (\phi^t)^{-1}$. If $\Phi$ is a semiflow relation,
then $\overline{\Phi}$ is a semiflow relation which we call the \emph{reverse of $\Phi$}\index{semiflow relation!reverse}.

A semiflow relation is called \emph{complete}\index{semiflow relation!complete} when it satisfies the following the additional condition:\vspace{.5cm}

(iii) \textbf{Completeness Condition:} \index{Completeness Condition} For all $t \in \R_+$ $Dom(\phi^{t}) = X$, or, equivalently
 regarded as a relation
$\Phi : X \times \R_+ \to X$, the domain $Dom(\Phi) = X \times \R_+$. That is, for every
$(x,t) \in X \times \R_+$ there exists $y \in X$ such that $(x,t,y) \in \Phi$. \vspace{.5cm}

 We call $\Phi$ a \emph{flow relation}\index{flow relation} when it is a semiflow relation such
 that both $\Phi$ and $\overline{\Phi}$ are complete. In \cite{BK} the authors assume completness, i.e. they restrict attention to flow relations.

 Just as  continuous maps and homeomorphisms are special cases of closed relations, we can regard semiflows and flows as special
 cases of semiflow relations.

 \begin{df}\label{semidf00} A semiflow \index{semiflow} $\Phi$ is a semiflow relation such that $\Phi : X \times \R_+ \to X$ is a map, or, equivalently,
 for all $t \in \R_+$  $\phi^t$ is a map. It is a flow \index{flow} when each $\phi^t$ is a homeomorphism. \end{df} \vspace{.5cm}

 If  $\Phi$ is a semiflow, then $\Phi$ is complete because each $\phi^t$ is a map.
 If both $\Phi$ and $\overline{\Phi}$ are semiflows, then for all $t$ both $\phi^t$ and $(\phi^t)^{-1}$ are maps
 and so  it follows
 that each $\phi^t$ is a homeomorphism. That is, $\Phi$ is a flow if and only if $\Phi$ and $\overline{\Phi}$ are semiflows.

\vspace{1cm}

Now we fix a semiflow relation $\Phi$. A crucial tool is the following.

\begin{theo}\label{semitheo01}\textbf{Equicontinuity Property:}\index{Equicontinuity Property}
For every $\ep > 0$ there exists $\d > 0$ such that $(x,t,y) \in \Phi$ and $t \le \d$ implies $d(x,y) < \ep$.
 \end{theo}

 \begin{proof}  Let $\hat \Phi = \{ (x,t,x,y) : (x,t,y) \in \Phi, t \le 1 \}$.
 We can regard $\hat \Phi$ as a closed relation
 from $X \times I$ to $X \times X$. By (\ref{eqrel06})
 \begin{equation}\label{eqsemi01}
 \hat \Phi^*(V_{\ep}) = \{ (x,t) \in X \times I : (x,y) \in V_{\ep} \ \ \text{for all} \ \ y \ \ \text{with} \ \ (x,t,y) \in \Phi \}
\end{equation}
is an open set and it contains $X \times \{ 0 \}$ by the Initial Value Condition.
By compactness, there exists $\d > 0$ such that
 $X \times [0,\d] \subset  \hat \Phi^*(V_{\ep})$.

 \end{proof} \vspace{.5cm}

For $t_1 \le t_2 \in \{ - \infty \} \cup \R \cup \{ \infty \}$ we let $[t_1,t_2]$ denotes the \emph{$\R$ interval}
\index{interval} $\{ t \in \R : t_1 \le t \le t_2 \}$.
If $t_1, t_2 \in \R$ then $t_2 - t_1$ is the \emph{length} of the interval. \index{interval!length}  Otherwise, it is an
\emph{infinite interval}. \index{interval!infinite} We will let context determine whether we are using a $\Z$ interval or a $\R$ interval.

 \begin{df}\label{semidf01a} Let $\Phi$ be a semiflow relation on $X$. If  $T \subset \R$  then a function $\xx : T  \to X$ is called a
 \emph{partial solution path}\index{solution path!partial} if $t_1 < t_2 \in T $ implies
 $(\xx(t_1), t_2 - t_1, \xx(t_2)) \in \Phi$. A \emph{solution path}\index{solution path} is a partial solution
 path with domain $T $ a closed interval in $\R$. It is an \emph{infinite solution path}\index{solution path!infinite}
  when the interval is infinite.

 We will write $\S([t_1,t_2],\Phi)$ (or just $\S([t_1,t_2])$ when $\Phi$ is understood) for the set of
 solution paths defined on the interval $[t_1,t_2]$. \end{df} \vspace{.5cm}

  As in the discrete case, there are obvious operations on solution paths.
 \begin{itemize}
 \item {\bf Translation} If $\xx : [t_1,t_2] \to X$ is an solution path and $a \in \R$, then the translate $Trl_a(\xx) : [t_1-a,t_2-a] \to X$
 given by $Trl_a(\xx)(t) = \xx(t+a)$ is a solution path.\index{solution path!translation}

 \item {\bf Composition} If $\xx : [t_1,t_2] \to X$ and $\yy: [t_2,t_3] \to X$ are solution paths, with $\xx(t_2) = \yy(t_2)$ then
 the composition $\xx \oplus \yy : [t_1,t_3] \to X$
 is the solution path such that $ \xx \oplus \yy | [t_1,t_2] = \xx$ and $ \xx \oplus \yy| [t_2,t_3] = \yy$.\index{solution path!composition}

\item {\bf Inversion} If $\xx : [t_1,t_2] \to X$ is a solution path for $\Phi$, then $\bar \xx : [-t_2,-t_1] \to X$
defined by $ \bar \xx(t) = \xx(-t)$ is a solution path for $\overline{\Phi}$.
\end{itemize}

 That the composition of solution paths is still a solution path follows from the Kolmogorov Condition.

 \begin{cor}\label{semicor02} Any collection of partial solution paths is uniformly \\equicontinuous. \end{cor}

 \begin{proof} Given $\ep > 0$, choose $\d > 0$ as in Theorem \ref{semitheo01}.
 If $\yy$ is a partial solution path defined
 on $T$ and $t_1 < t_2 \in T$ with $t_2 - t_1 \le \d$, then $(\yy( t_1), t_2 - t_1, \yy( t_2)) \in \Phi$
 implies $d(\yy( t_1),\yy( t_2)) < \ep$. This is uniform equicontinuity
 because the choice of $\d$ does not depend on $\yy$ or the points of the domain of $\yy$.

 \end{proof}

 \begin{cor}\label{semicor02a} If $T$ is a dense subset of $[t_1,t_2]$ and $\xx_0$ is a
 partial solution path defined on $T$, then there is a unique continuous function
 $\xx : [t_1,t_2] \to X$ with $\xx|T = \xx_0$. Furthermore $\xx$ is a solution path on $[t_1,t_2]$.
 \end{cor}

  \begin{proof} By Corollary \ref{semicor02} $\xx_0$ is uniformly continuous on the dense set
  $T$ and so has a unique continuous extension $\xx$ to $[t_1,t_2]$.

   The set
   $\{ (u,v) \in [t_1,t_2] \times [t_1,t_2] : u \le v, (\xx( u),v - u,\xx( v)) \in \Phi \}$
is a closed subset of $\{ (u,v) \in [t_1,t_2] \times [t_1,t_2] : u \le v \}$
and contains the dense set $\{ (u,v) \in T \times T : u \le v \}$.
So it is all of $\{ (u,v) \in [t_1,t_2] \times [t_1,t_2] : u \le v \}$. Thus, $\xx$ is a
solution path on $[t_1,t_2]$.

\end{proof}

\begin{theo}\label{semitheo03} For $t_1 < t_2 \in \R$ let $T$ be a closed
subset of $[t_1,t_2]$ with $t_1, t_2 \in T$. If $\xx_0 : T \to X$ is a partial solution path,
then there exists a solution path $\xx : [t_1,t_2] \to X$ which extends $\xx_0  $.
\end{theo}

\begin{proof}  Let $q_1, q_2, \dots$ be a count of the set $Q$ of rationals in
$[t_1,t_2] \setminus T$ and with $T_0 = T$ let $T_k = T \cup \{ q_1, \dots , q_k \}$.
Beginning with $\xx_0 $ on $T_0$ we define by induction the extension $\xx_{k+1}$ of $\xx_k$
to a partial solution path on $T_{k+1}$.

Assume that $\xx_k$ has been defined as required. Because $T_k $ is closed and does not
contain $q_{k+1}$ there exist $u, v \in T_k$
such that $u < q_{k+1} < v$ and the open interval $(u,v)$ does not meet $T_k$. Since
$(\xx_k(u), v-u, \xx_k(v)) \in \Phi$ by induction hypothesis,
the Kolmogorov Condition implies that there exists $x \in X$ such that
$(\xx_k(u), q_{k+1}-u, x)), (x, v-q_{k+1}, \xx_k(v)) \in \Phi$.
Let $\xx_{k+1}(q_{k+1}) = x$. If $t \in T_k \cap [t_1,a)$ then by induction
hypothesis $(\xx_k(t),u - t, \xx_k(u)) \in \Phi$ and so
the Kolmogorov Condition implies that $(\xx_k(t),q_{k+1} - t, x) \in \Phi$.
Proceed similarly if $t \in T_k \cap (v,t_2]$. Thus,
$\xx_{k+1}$ is a partial solution path defined on $T_{k+1}$.

With $T_{\infty} = Q \cup T$ the union $\xx_{\infty} = \bigcup \xx_k$ is a partial solution
path defined on the dense set $T_{\infty}$.
By Corollary \ref{semicor02a}  $\xx_{\infty}$ extends uniquely to a solution path on $[t_1,t_2]$.

 \end{proof}

 \begin{cor}\label{semicor03a}  If $(x,t,y) \in \Phi$ and $t_1, t_2 \in \R$ with $t_2 - t_1 = t$,
 then there exists a solution path $\xx : [t_1,t_2] \to X$
 with $\xx( t_1) = x$ and $\xx( t_2) = y$. \end{cor}

 \begin{proof}: Apply Theorem \ref{semitheo03} with $T = \{ t_1,t_2 \} \subset [t_1,t_2]$.

 \end{proof}

 \begin{lem}\label{semilem04} Assume $\Phi$ is complete and $t_1 < t_2 < t_3$. If $\xx $ is a solution path on $[t_1,t_2]$,
 then there exists $\yy$ a solution path on $[t_1,t_3]$ such that $\yy|[t_1,t_2] = \xx$. \end{lem}

\begin{proof} By the Completeness Condition, there exists $y$ such that
$(\xx( t_2),t_3 - t_2, y) \in \Phi$ and so by Corollary \ref{semicor03a}
there is a solution path $\yy_1 \in \S([t_2,t_3])$ with $\yy_1(t_2) = \xx( t_2)$ and $\yy_1(t_3) = y$. Let $\yy = \xx \oplus \yy_1$.

\end{proof}

 Let $C([t_1,t_2]; X)$ be the complete metrizable space of continuous functions from $[t_1,t_2]$
 equipped the topology of uniform convergence on compacta.
 Let $\S([t_1,t_2],\Phi)$ (or just $\S([t_1,t_2])$ when $\Phi$ is understood) denote the subset of solution paths for $\Phi$. Clearly, $\S([t_1,t_2])$
 is a point-wise closed subset of $C([t_1,t_2]; X)$.
Corollary \ref{semicor02} and the Arzela-Ascoli Theorem, see \cite{K} Theorem 7.17,  imply that
 $\S([t_1,t_2])$ is a compact subset of  $C([t_1,t_2]; X)$.

 We will write $\S_+(\Phi), \S_-(\Phi),$ and $ \S(\Phi)$ for $\S([t_1,t_2],\Phi)$ with $[t_1,t_2] = [0,\infty], [-\infty,0] $ and $[-\infty,\infty]$,
 respectively.

 \begin{lem}\label{semilem04a} Assume that $\S_0$ is a uniformly equicontinuous collection of  paths in $X$ such that
 \begin{itemize}
 \item[(i)] For  each $x \in X$, the map $0 \mapsto x$ is an element of $\S_0([0,0])$.

 \item[(ii)] For each $t_1 \le t_2 \in \R$, $\S_0([t_1,t_2])$ is a closed, and hence compact, subset of $C([t_1,t_2]; X)$.

 \item[(iii)] If  $\xx : [t_1,t_2] \to X$ is a continuous path such that for all $s_1, s_2$ withe $t_1 < s_1 < s_2 < t_2$ the restriction
 $\xx|[s_1,s_2]$ lies in $\S_0$, then $\xx \in \S_0$.

 \item[(iv)] The collection $\S_0$ is
 closed under restriction to subintervals, under translation and  under composition.
\end{itemize}

 The set
\begin{equation}\label{eqsemi01aa} \begin{split}
\Phi_0   \ =_{def} \   \{ (x,t,y) \in X \times \R_+ \times X : \hspace{2cm} \\
\text{there exists} \  \xx \in \S_0([0,t]) \ \ \text{with} \ \ \xx( 0) = x, \xx( t) = y \}
\end{split}\end{equation}
is a semiflow relation on $X$. Furthermore, for each $t_1 \le t_2 \in \R$, $\S_0([t_1,t_2]) \ = \ \S([t_1,t_2],\Phi_0)$.
\end{lem}

\begin{proof} Since the constant map at $x$ lies in $\S_0([0,0])$ it follows that $(x,0,x) \in \Phi_0$ and the Initial Value Condition holds.

Assume $\{ (x_n,t_n,y_n) \} $ is a sequence in $\Phi_0$ converging to $(x,t,y)$ and that
$\xx_n \in \S_0([0,t_n])$ with $ \xx_n(0) = x_n, \xx_n(t_n) = y_n$. If $t = 0$, then uniform equicontinuity
 implies that $y = x$ and so $(x,t,y) \in \Phi_0$.

Assume now that $t > 0$.  Let $\{ \ep_k\}$ be a decreasing sequence in $(0,t)$
converging to $0$. For each $k$,
eventually, $t_n > t - \ep_k$. By compactness of the $\S_0$ path spaces, and a diagonal
process we can assume, by going to a subsequence, that for each $k$ $\{ \xx_n|[0,t-\ep_k] \}$ converges to
$\yy_k \in \S_0([0,t-\ep_k])$.
These fit together to obtain a limit  path $\yy$ on $[0,t)$ which extends
to a continuous path on $[0,t]$ by uniform equicontinuity. By condition (iii) the extension lies in $\S_0$.

We check that $\yy( t) = y$.

Given $\ep > 0$ choose $\d$ an $\ep/2$ modulus of uniform continuity.
Choose $k$ large enough that $\ep_k < \d/2$.
There exists $N \in \Z_+$ so
that for $n \ge N$, $ 0 < t - t_n + \ep_k < \d$ and $d(\yy_n(t_n),y) = d(y_n,y) < \ep/2$. By choice of
$\d$, $d(\yy_n(t - \ep_k) - \yy_n(t_n)) < \ep/2$ and so $d(\yy_n(t - \ep_k) < \ep$.
Letting $n$ tend to $\infty$ we have $d(\yy( t - \ep_k), y) \le \ep$. Letting $k$ tend to $\infty$
we have $d(\yy( t),y) \le \ep$. As $\ep > 0$ was arbitrary, it follows that $\yy( t) = y$.

It follows that $(x,t,y) \in \Phi_0$ and so $\Phi_0$ is closed.

If $t = t_1 + t_2$ and $\xx \in \S_C([0,t])$ with $\xx( 0) = x, \xx( t) = y$ then the restriction
$\yy_1 = \xx|[0,t_1] \in \S_0([0,t_1])$ and the translate
$\yy_2   \in \S_0[0,t_2])$ with $\yy_2(s) = \xx( t_1 + s)$. Thus, with
$z = \xx( t_1)$, $(x,t_1,z), (z,t_2,y) \in \Phi_C$.

Conversely, if $\yy_1  \in \S_0([0,t_1]), \yy_2  \in \S_0([0,t_2])$ with
$\yy_1(0) = x, \yy_1(t_1) = z = \yy_2(0), \yy_2(t_2) = y$, then with $\yy_3(s) = \yy_2(s - t_1)$
$\yy_3 \in \S_0([t_1,t_1 + t_2])$. By condition (iv) $\xx = \yy_1 \oplus \yy_3 \in \S_C([0,t_1 + t_2])$ with
$\xx( 0) = x, \xx( t_1 + t_2) = y$.  Hence, $(x, t_1 + t_2,y) \in \Phi_0$.

Thus, $\Phi_0$ satisfies the Kolmogorov Condition and so is a semiflow relation on $X$.

Clearly, $\S_0([t_1,t_2]) \subset \S([t_1,t_2],\Phi_0)$.

Now assume $\xx \in \S([t_1,t_2],\Phi_0)$.  Let $\{ T_n \}$ be an increasing sequence of finite subsets of $[t_1,t_2]$ with
$T_0 = \{ t_1, t_2 \}$ and such that $T = \bigcup_n \ T_n$ is dense in $[t_1,t_2]$.

For each $n$  let $T_n = \{ t_1 = s_0 < \dots < s_{k_n} = t_2$. For $i = 1, \dots, k_n$ e can choose an element of $\S_0([s_{i-1},s_i])$
connecting $\xx(s_{i-1})$ to $\xx(s_i)$ and then compose them to get $\yy^n \in \S_0([t_1,t_2])$ so that
$\yy_n|T_n = \xx|T_n$. By going to subsequence we can assume that $\{ \yy^n \}$ converges to some $\yy \in  \S_0([t_1,t_2])$.
Because $\yy|T = \xx|T$ and $T$ is dense, it follows that $\yy = \xx$ and so $\xx \in   \S_0([t_1,t_2])$.

\end{proof}

 For $K$ a compact subset of $\R_+$, define

 \begin{align}\label{eqsemi02}\begin{split}
\phi^K(\Phi)  \ =_{def} \  &\{ (x,y) : (x,t,y)  \in \Phi  \ \text{for some} \ t \in K \}\\ = \ &\pi_{13}[\Phi  \cap (X \times K \times X)] \hspace{1cm}
\end{split}\end{align}
and so, by compactness,  $\phi^K(\Phi)$ is a closed relation on $X$. Clearly, $(x,y) \in \phi^K(\Phi)$ if and only if there exists
a solution path $\xx: [t_1,t_2] \to X$ with $\xx(t_1) = x, \xx(t_2) = y$ and $t_2 - t_1 \in K$.  It follows that

\begin{equation}\label{eqsemi02a}
\phi^K(\overline{\Phi}) \ = \ \phi^K(\Phi)^{-1}. \hspace{2cm}
\end{equation}

When the context is clear we will write $\phi^K$ for $\phi^K(\Phi)$.

\begin{lem}\label{semilem06} If $0 \le t_1 < t_2, t_3 < t_4$, then
\begin{equation}\label{eqsemi03}
\phi^{[t_3,t_4]} \circ \phi^{[t_1,t_2]} \ = \  \phi^{[t_1 + t_3,t_2 + t_4]}
\end{equation}
\end{lem}

\begin{proof} If $(x, s_1, z), (z, s_2, y) \in \Phi$ with $s_1 \in [t_1,t_2], s_2 \in [t_3,t_4]$ then $(x, s_1 + s_2, y) \in \Phi$
with $s_1 + s_ 2 \in [t_1 + t_3,t_2 + t_4]$. Conversely, if $(x,t,y) \in \Phi$ with $t \in [t_1 + t_3,t_2 + t_4]$ we can choose
$s_1 \in [t_1,t_2], s_2 \in [t_3,t_4]$ such that $s_1 + s_2 = t$ and then choose $z$ so that $(x, s_1, z), (z, s_2, y) \in \Phi$.

\end{proof}

In particular, for $I = [0,1], J = [1,2]$

\begin{equation}\label{eqsemi04} \begin{split}
\phi^I = \{ (x, y) : (x,t,y) \in \Phi \ \text{for some} \ t \in [0,1] \}. \\
\phi^J =  \{ (x, y) : (x,t,y)  \in \Phi  \ \text{for some} \ t \in [1,2] \}.
\end{split}\end{equation}

Observe that $1_X \subset \phi^I$, i.e. $\phi^I$ is reflexive, and
\begin{equation}\label{eqsemi05}\begin{split}
\phi^I \circ \phi^I \ = \ \phi^{[0,2]} \ = \ \phi^I \cup \phi^J. \hspace{2cm}\\
\phi^J \circ \phi^I \ = \ \phi^I \circ \phi^J \ = \ \phi^{[1,3]} \ \subset \ \phi^J \cup (\phi^J)^2.
\end{split}\end{equation}\vspace{.5cm}

We define:
\begin{equation}\label{eqsemi05a}
\O \phi  \ =_{def} \  \{ (x,y) : (x,t,y)  \in \Phi  \ \text{for some} \ t \in \R_+ \} \ = \ \pi_{13}(\Phi).
\end{equation}

\begin{prop}\label{semiprop05}  For the semiflow relation $\Phi$ the following hold.
\begin{enumerate}
\item[(a)] $(x,y) \in \O (\phi^I)$ if and only if $(x,t,y) \in \Phi$ for some $t \in \R_+$ and
$(x,y) \in \O (\phi^J)$ if and only if $(x,t,y) \in \Phi$ for some $t \ge 1$.

\item[(b)] $\O \phi \ = \ \O (\phi^I) \ = \ \phi^I \cup \O (\phi^J).$

\item[(c)]  For $\A = \O, \G, \CC$
\begin{equation}\label{eqsemi06}
(\A (\phi^J)) \circ \phi^I \ = \ \A (\phi^J) \ = \  \phi^I \circ (\A (\phi^J)). \hspace{1cm}
\end{equation}
Each $\phi^I \cup \A (\phi^J)$ is a transitive relation.

\item[(d)]  Although $\phi^I \cup \CC (\phi^J) $ is a closed, transitive relation,  it is usually a proper subset of $\CC( \phi^I \cup \phi^J )$.
On the other hand,
\begin{equation}\label{eqsemi07}
\phi^I \cup \G (\phi^J)  \ = \ \G( \phi^I \cup \phi^J ) \ = \ \G (\phi^I). \hspace{1cm}
\end{equation}

\item[(e)] For $\A = \O, \G, \CC$, if $(x,y), (y,x) \in (\phi_C)^I \cup \A (\phi^J)$ and $x \not= y$, then
$(x,y), (y,x), (x,x), (y,y) \in \A (\phi^J)$.

\item[(f)]  For $\A = \O, \G, \CC$, if $A \subset X$ is a closed $\A (\phi^J) $ + invariant set,
then $\phi^I(A)$ is a closed $\phi^I \cup \A (\phi^J) $ invariant set and so is
$\A (\phi^J)$ + invariant. Furthermore, $(\A (\phi^J)) (\phi^I(A)) = (\A (\phi^J)) (A) = \phi^J (A)$.
In particular, $A$ is then $\A (\phi^J)$ invariant if and only if it is $\phi^J $ invariant.
\end{enumerate}
\end{prop}

\begin{proof}(a) A number $t \in \R_+$ can be written as a finite sum of elements of $I$ and a finite sum of elements of $J$ if $t \ge 1$.

(b) Obvious from (a).

(c) By (\ref{eqrel09c}) $\phi^J \cup ((\A (\phi^J)) \circ \phi^J)  =  \A (\phi^J)  = \phi^J \cup (\phi^J \circ (\A (\phi^J)))$
for $\A = \O, \G, \CC$. From (\ref{eqsemi05})
and transitivity of $\A (\phi^J)$ it follows that
$ (\phi^I \circ \A (\phi^J)) $ and $ (\A (\phi^J) \circ \phi^I) \subset \A (\phi^J)$. The reverse inclusions follow because $\phi^I$ is reflexive.
From (\ref{eqsemi05}) it follows that $\phi^I \circ \phi^I \subset \phi^I \cup \phi^J \subset \phi^I \cup \A (\phi^J)$.
Together with (\ref{eqsemi06}) this implies that
$\phi^I \cup \A (\phi^J) $ is transitive.

(d) We clearly have $\phi^I  \subset \phi^I \cup \G \phi^J \subset \G (\phi^I \cup \phi^J)$. Since the closed relation
$\phi^I \cup \G \phi^J$ is transitive by (a),
it contains $\G (\phi^I \cup \phi^J)$. Finally, $\phi^I \cup (\phi^I)^2 = \phi^I \cup \phi^J$ and so $\G (\phi^I \cup \phi^J) = \G \phi^I$.

In a connected space $X$, $\CC 1_X = X \times X$ and so if $X$ is connected, $\CC( \phi^I \cup \phi^J) = X \times X$ which is usually
larger than $\phi^I \cup \CC \phi^J$.

(e) If $(x,y), (y,x) \in \A (\phi^J)$, then the result follows from transitivity of $\A (\phi^J)$.  So we may assume $(x,y) \in \phi^I$.

Case 1: If $(y,x) \in \A (\phi^J)$, then $(x,x) \in \A (\phi^J) \circ \phi^I$ and $(y,y) \in \phi^I \circ \A (\phi^J)$. By (\ref{eqsemi06})
$\A (\phi^J) = \A (\phi^J) \circ \phi^I = \phi^I \circ \A (\phi^J)$. Then $(x,y) \in \phi^I \circ \A (\phi^J) = \A (\phi^J)$.

Case 2: $(x,y), (y,x) \in \phi^I$. This means there exist $0 < t_1, t_2 \le 1$ and solution paths $\xx_1 \in \S([0,t_1)], \xx_2 \in \S([0,t_2)]$
with $\xx_1(0) = \xx_2(t_2) = x, \xx_1(t_1) = \xx_2(0) = y$. Concatenating we can obtain a $t_1 + t_2$ periodic solution path
$\xx : [0,\infty) \to C$, with $x = \xx( n(t_1 + t_2)), y = \xx( t_1 + n(t_1 + t_2))$ for all  $n \in \Z_+$.
 Since $t_1 + t_2 > 0$ and $ (x,n(t_1 + t_2),x) \in \Phi$ for every $n \in \Z_+$
we see that $(x,x) \in \O (\phi^J)\subset \A (\phi^J)$. Similarly, $(y,y) \in \O (\phi^J)$. $(x,t_1 + n(t_1 + t_2),y) $ and so $(x,y) \in \O (\phi^J)$
and similarly for $(y,x)$.  In fact, any pair of points on a periodic solution lies in $\O (\phi^J) \subset \A (\phi^J)$.

(f) Since $\phi^I$ is reflexive and closed, $\phi^I(A)$ is closed and contains $A$. From (\ref{eqsemi06}) and (\ref{eqsemi05}) we have
$(\phi^I \cup \A (\phi^J)) \circ \phi^I = \phi^I \cup \A (\phi^J)$ and so $(\phi^I \cup \A (\phi^J))(\phi^I(A)) = (\phi^I \cup \A (\phi^J))(A) = \phi^I(A)$ since $A$ is $\A (\phi^J)$ + invariant.

 By (\ref{eqrel09c}) $\A (\phi^J)  = \phi^J \cup (\phi^J \circ (\A (\phi^J)))$. Since $\A (\phi^J)(A) \subset A$, it follows that
$ \A (\phi^J)(A) = \phi^J(A)$. By (\ref{eqsemi06}) $(\A (\phi^J)) (\phi^I(A)) = (\A (\phi^J)) (A)$.

\end{proof} \vspace{.5cm}

A subset $A \subset X$ is call $\Phi$ + invariant (or $\Phi$ invariant)\index{subset!invariant} when $\phi^t(A) \subset A$ for all $t \in \R_+$
(resp. $\phi^t(A) = A$ for all $t \in \R_+$). That is, $A$ is  + invariant (or invariant) for the semiflow relation $\Phi$ when it is
+ invariant (resp. invariant) for each of the closed relations $\phi^t$. So  $A$ is $\Phi$ + invariant when $\O \phi(A) = A$ (Note that
$1_X \subset \phi^I \subset \O \phi$ ).

\begin{prop}\label{semiprop05a} Let $\Phi$ be semiflow relation on $X$ and $A$ be a subset of $X$.
\begin{enumerate}
\item[(a)]The following conditions are equivalent,
\begin{itemize}
\item[(i)]  $A$ is $\Phi$ + invariant.

\item[(ii)]  For some $\ep > 0$ $A$ is $\phi^t$ + invariant for all $t $ with $0 < t \le \ep$.

\item[(iii)] $A$  is invariant for the relation $\phi^I$.

\item[(iv)] Whenever $\xx : [t_1,t_2] \to X$ is a solution path with $\xx(t_1) \in A$, $\xx(t) \in A$ for all $t \in [t_1,t_2]$.

\item[(v)]  The collection of sets $\{ \phi^t(A) \}$ is decreasing for $t \in \R_+$.
\end{itemize}

\item[(b)] If $A$ is closed and $\Phi$ + invariant,
 then $A_{\infty} = \bigcap_{t = 0}^{\infty} \ \phi^t(A) = \bigcap_{k = 1}^{\infty} \ (\phi^J)^k(A)$
is a  $\Phi$ invariant subset of $A$ which contains any other $\Phi$ invariant subset of $A$.

\item[(c)] If $A$ is closed and $\phi^J$  + invariant, then $\phi^I(A)$ is $\Phi$  + invariant with $\phi^J(\phi^I(A)) = \phi^J(A) \subset A \subset \phi^I(A)$.
$A_{\infty} = \bigcap_{t = 0}^{\infty} \ \phi^t(\phi^I(A)) = \bigcap_{k = 1}^{\infty} \ (\phi^J)^k(A)$ is a nonempty $\Phi$ invariant
subset of $A$ which contains any other $\Phi$ invariant subset of $\phi^I(A)$. In particular, if $A$ is inward for the relation $\phi^J$,
then the associated $\phi^J$ attractor is $\Phi$ invariant.

\item[(d)] The following conditions are equivalent,
\begin{itemize}
\item[(i)] $A$ is $\Phi$  invariant.

\item[(ii)]  $A$ is $\Phi$ + invariant and $\phi^t(A) = A$ for some $t > 0$.

\item[(iii)] $A$ is $\phi^J$ invariant.
\end{itemize}

\end{enumerate}
\end{prop}

\begin{proof} (a) (i) $\Leftrightarrow$ (iv), (i) $\Rightarrow$ (iii) $\Rightarrow$ (ii) and (v)$\Rightarrow$ (i)  are obvious (Note
that always $A \subset \phi^I(A)$).
 As in Proposition \ref{semiprop05}(b), $\O \phi = \O \phi^{[0,\ep]}$ and so (ii) $\Rightarrow$ (i).

 When $A$ is $\Phi$ + invariant, and $t > s$ then $\phi^{t-s}(A) \subset A$ implies $\phi^t(A) \subset  \phi^s(A)$.  That is,
 the collection of sets $ \{\phi^t(A) \} $ is decreasing in $t$, i.e. (i) $\Rightarrow$ (v).

 (b) $\Phi$ + invariance implies $\phi^J$ + invariance and so $ \{\phi^t(A) \} $ is decreasing in $t$ and $ \{ (\phi^J)^k(A) \} $ is decreasing in $k$.
Since $\phi^{2k}(A) \subset (\phi^J)^k(A) \subset \phi^k(A)$ the two intersections agree. Furthermore, for any fixed $s > 0$
$A_{\infty} = \bigcap_{k = 1}^{\infty} \ \phi^{ks}(A)$. So the result follows from Corollary
\ref{relcor00a} applied to $\phi^s$.

(c) If $A$ is $\phi^J$ + invariant, then $\phi^J(A) = [\phi^J \cup (\phi^J)^2](A)$ and $A \subset \phi^I(A) = [\phi^I \cup \phi^J](A)]$.
From (\ref{eqsemi05}) it follows
\begin{equation}\label{eqsemi08aaa}\begin{split}
\phi^I(\phi^I(A)) \ = \ \phi^I(A), \hspace{3cm}\\
 \phi^I(\phi^J(A)) \ = \ \phi^J(\phi^I(A)) \ = \ \phi^J(A) \ \subset \ A, \\
 \text{For} \ \ k = 1,2,\dots, (\phi^J)^k(\phi^I(A)) \ = \ (\phi^J)^k(A).
\end{split}\end{equation}
Thus, $\phi^I(A)$ is $\phi^I$ invariant and  so (a) implies $\phi^I(A)$ is $\Phi$ + invariant. The rest follows from (b) applied to
$\phi^I(A)$.

(d) (i) $\Rightarrow$ (ii), (iii) are obvious.

When $A$ is $\Phi$ + invariant,
 the collection of sets $ \{\phi^t(A) \} $ is decreasing in $t$. So if
$\phi^t(A) = A$, we have $\phi^s(A) = A$ for all $s \in [0,t]$. If $s > t$ then it can be written as a finite
sum of elements of $[0,t]$ and so again $\phi^s(A) = A$. Thus, (ii) $\Rightarrow$ (i).

If $A$ is $\phi^J$ invariant, then by (c) $A = \bigcap_{k = 1}^{\infty} \ (\phi^J)^k(A)$ is $\Phi$ invariant, i.e.(iii) $\Rightarrow$ (i).

\end{proof}

Following (c) we call $A$ a $\Phi$ attractor (or repeller) when it is a $\phi^J$ attractor (resp. $\phi^J$ repeller)\index{attractor}\index{repeller}.
If $U$ is inward for $\phi^J$ then by (c) $\phi^I(U)$ is $\Phi$ + invariant and is inward for
$\phi^J$.  If $U$ is $\Phi$ + invariant and is inward for $\phi^J$
then for all $t \in \R_+, \  \phi^t(U) \subset U$ and for $t \ge 1 \ \phi^t(U) \subset \phi^1(U) = \phi^J(U) \subset \subset U$. We sharpen this condition
defining $U$ to be \emph{inward for $\Phi$}\index{subset!inward}\index{inward}  when
\begin{equation}\label{eqsemiinward}
\phi^t(U) \subset  \subset U \quad \text{for all} \ \ t > 0.
\end{equation}
That is $U$ is inward for every relation $\phi^t$ with $t > 0$.

We will use Lyapunov functions to construct $\Phi$ inward neighborhoods for $\Phi$ attractors.

\begin{theo}\label{reltheo02semi} Let $\Phi$ be a semiflow relation on $X$. Let $\A = \G$ or $\CC$.

(a) Assume that $A, B$ are disjoint, closed subsets of $X$ with
$A$   invariant for $\phi^I \cup \A (\phi^J)$ and $B$  + invariant for $(\phi^I \cup \A (\phi^J))^{-1}$.

 There exists a continuous function $L: X \to [0,1]$ with $B = L^{-1}(0)$, $A = L^{-1}(1)$ and
such that if $(x,y) \in \phi^I \cup \A (\phi^J)$ with $x \not= y$, then $L(y) \ge L(x)$ with equality only when
 \begin{equation}\label{eqrellyap02semi}
x,y \in A, \quad x,y \in B, \ \ \text{or} \ \ (y,x) \in \A (\phi^J).
\end{equation}
 In particular, $L$ is a Lyapunov function for $\A (\phi^J)$
with $|\A (\phi^J)| \subset |L| \subset |\A (\phi^J)| \cup A \cup B$.

(b) There exists a continuous function $L: X \to [0,1]$
such that if $(x,y) \in \phi^I \cup \A (\phi^J)$ with $x \not= y$, then
$L(y) \ge L(x)$ with equality only when, in addition, $(y,x) \in \A (\phi^J)$. In particular, $L$ is a Lyapunov function
with $|L| = |\A (\phi^J)|$. \end{theo}

\begin{proof}  This is Theorem \ref{reltheo02} applied to $\phi^I \cup \A (\phi^J)$.  Notice that Proposition \ref{semiprop05}(e)
implies that $x \not= y$ and $(x,y) \in (\phi^I \cup \A (\phi^J)) \cap (\phi^I \cup \A (\phi^J))^{-1}$ implies
$(x,y) \in (\A (\phi^J)) \cap (\A (\phi^J))^{-1}$.

\end{proof}

\begin{cor} \label{relcor03semi} Assume that $(A,B)$ is an attractor-repeller pair for the  a semiflow relation $\Phi$ on $X$.
 There exists a continuous function $L: X \to [0,1]$ with $B = L^{-1}(0)$, $A = L^{-1}(1)$ and
such that if $(x,y) \in \phi^I \cup \CC (\phi^J)$ with $x \not= y$, then $L(y) \ge L(x)$ with equality only when
$ x,y \in A,$ or $ x,y \in B$.

 In particular, $L$ is a Lyapunov function for $\CC (\phi^J)$
with $ |L| = A \cup B$.  Furthermore, for all $a$ such that $0 < a < 1$, the set $U_a = \{ x : L(x) \ge a \}$ is an inward subset for $\Phi$ with
associated attractor $A$. If $V$ is any neighborhood of $A$, there exists $0 < a < 1$ such that $U_a \subset V$. \end{cor}

\begin{proof} Apply Theorem \ref{reltheo02semi} with $\A = \CC$. Notice that if $(x,y) \in \CC (\phi^J) \cap \CC (\phi^J)^{-1}$, then
$x$ and $y$ are chain recurrent points lying in the same chain component. It follows that either $x,y \in A$ or $x,y \in B$. Consequently,
$ |L| \subset A \cup B$. If $x \in A$, then $\Phi$ invariance implies there exists $y \in A$ such that $(y,x) \in \phi^J$. Hence,
$L(y) = L(x) = 1$ and so $x \in |L|$.  Similarly, $x \in B$ implies that $x \in |L|$.  Thus, $ |L| = A \cup B$.

Now assume that $0 < L(x) < 1 $ so that $x \not\in A \cup B$ and so $x \not\in |L|$.  It follows that $(x,t,y) \in \Phi$ with $t > 0$  implies
$L(y) > L(x) = a$. So for any $a$ with $0 < a < 1$,
 \begin{equation}\label{eqrellyap02aasemi}
 \inf (L(\phi^t(U_a)) > a, \quad \text{and so} \quad  \phi^t(U_a) \subset \{ x : L(x) > a \} \subset U_a^{\circ}.
 \end{equation}
Thus, for every $t > 0, 0 < a < 1,  \ \phi^t(U_a) \subset  \subset U_a$ and so each $U_a$ is $\Phi$ inward.

 Let $A_1$ be the attractor associated with $U_a$. As $A_1$ is the maximum invariant subset of $U_a$ it follows that
 $A_1 \supset A$.  Choose $x \in A_1$ such that $L(x) = \min \{ L(y) : y \in A_1 \}$. Since $U_a$ is
 disjoint from $B$, $L(x) > 0$. By invariance of $A_1$ there exists $z \in A_1$ with $(z,x) \in \phi^J$. If $L(x)$ were less than $1$
 then $x \not\in |L|$ implies that $L(z) < L(x)$ contradicting the minimality of $L(x)$. Hence, $L(x) = 1$ which implies
 $A_1 \subset L^{-1}(1) = A$. Thus, $A_1 = A$.

 Because $\bigcap_{0 < a < 1} \ U_a \ = \ A$, it follows that if $V$ is any neighborhood of $A$, then $U_a \subset V$ for some $0 < a < 1$.

 \end{proof}

\vspace{.5cm}

In particular, $X$ is inward for $\Phi$ and following (\ref{eqrel17}) we define:
\begin{align}\label{eqsemi08aax}\begin{split}
X_-  \ =_{def} \  \bigcap_{t = 0}^{\infty} \  &\phi^t(X)  \ = \ \bigcap_{k = 1}^{\infty} \ (\phi^J)^k(X)\\
 X_+   \ =_{def} \   \bigcap_{t = 0}^{\infty} \  &\bar \phi^t(X) = \bigcap_{k = 1}^{\infty} \ (\phi^J)^{-k}(X)\\
 X_{\pm}  \ =_{def} \  &X_- \ \cap \ X_+.
 \end{split}\end{align}
So $X_- $
is the maximum attractor for $\Phi$ and $X_+$ is the maximum repeller.

 Define the \emph{solution path spaces}\index{solution path space}
\begin{align}\label{eqsemi08a} \begin{split}
\S_+(\Phi) \ = \ \S([0,\infty],\Phi), &\quad \S_-(\Phi) \ = \ \S([-\infty,0],\Phi), \\ \S(\Phi) \ = \ \S(&[-\infty,\infty],\Phi).
\end{split}\end{align}

\begin{prop}\label{semiprop06} For a point $x \in X$ the following conditions are equivalent.
\begin{itemize}
\item[(i)] $(x,t,y) \in \Phi$ implies $t = 0$ and so $y = x$.
\item[ (ii)] For no $t > 0$ does there exist $\xx \in \S([0,t])$ with $\xx( 0) = x$.
\end{itemize}
When these conditions hold, we call $x$ a \emph{terminal point}\index{terminal point} for $\Phi$.
\end{prop}

\begin{proof} The equivalence is obvious from Corollary \ref{semicor03a}.

\end{proof}

\begin{prop}\label{semiprop07} If $\xx \in \S([0,t])$, then either there exists $\yy \in \S[0,\infty])$ with
$\yy|[0,t] = \xx$ or else the set
\begin{equation}\label{eqsemi08}
 \{ t_1 \in \R_+ : \exists \yy \in \S([0,t_1]) \text{with} \ t_1 \ge t \ \text{and} \ \yy|[0,t] = \xx \}
\end{equation}
has a finite supremum $t^*$ contained in the set and if $\yy \in \S([0,t^*]) $ which extends $\xx$ then
 $\yy( t^*)$ is a terminal point for $\Phi$. \end{prop}

\begin{proof} Assume that $\xx$ does not extend to an infinite solution path. Let $t_n \to t^*$ be an increasing sequence in $\R_+$
and for each $n$ let $\yy_n \in \S([0,t_n]$ which extends $\xx$, by using a diagonal process we can go to a subsequence $\{ \yy_{n_i} \}$
so that for each $k$ $ \{ \yy_{n_i}|[0,t_k] : n_i \ge k \}$ converges.  So we obtain a solution path $\yy_{\infty} \in \S([0,t^*))$ which extends $\xx$.
From the assumption we see that $t^* < \infty$ and by Corollary \ref{semicor02a}, $\yy_{\infty}$ extends to $\yy \in \S([0,t])$

For any $\yy \in  \S([0,t^*])$ which extends $\xx$, let $y = \yy( t^*)$. If $y$ were not terminal, then there would
exist $\zz \in \S([0,\ep])$ for some $\ep > 0$ with $\zz(0) = y$. Composing $\yy$ with a translate of $\zz$ we would
obtain an element of $\S([0,t^*+\ep])$ extending $\xx$ and this contradicts the definition of $t^*$.  Hence, $\yy( t^*)$ is a terminal point.

\end{proof}

  Define $\t, \bar \t : X \to \R_+ \cup \{ \infty \}$ :
  \begin{align}\label{eqsemi09} \begin{split}
  \t(x) & \ =_{def} \  \sup \{ t \in \R_+ : \text{there exists} \ \ y \in X \ \ \text{such that} \ \ (x,t,y) \in \Phi \}, \\
  \bar \t(x) & \ =_{def} \  \sup \{ t \in \R_+ : \text{there exists} \ \ y \in X \ \ \text{such that} \ \ (y,t,x) \in \Phi \}.
  \end{split} \end{align}
  Thus, a point $x$ is terminal if and only if $\t(x) = 0$.
Clearly, the function $\bar \t$ is $\t$ applied to the reverse relation $\overline{\Phi}$.

  \begin{equation}\label{eqsemi10}
  (x,t,y) \in \Phi \quad \Longrightarrow \quad t + \t(y) \le \t(x).
  \end{equation}

    By Proposition \ref{semiprop07} if $\t(x) = \infty$, then there exists $\xx \in \S([0,\infty])$ with $\xx(0) = x$ and if
  $\t(x) < \infty$ then the set $\{ y : (x,\t(x),y) \in \Phi \}$ is nonempty and consists of terminal points.

    \begin{prop}\label{semiprop08} The functions $\t$ and $\bar \t$ are usc, i.e. $\{ \t < t \}$ and $\{ \bar \t < t \}$ are open sets for any
    $t \in \R_+$.

    The following equations hold for the subsets $X_+, X_-$ and $X_{\pm}$..
     \begin{align}\label{eqsemi11}\begin{split}
     X_+ \ = \ \pi_0(\S_+(\Phi)) \ = \ &\{ x : \t(x) = \infty \} \ = \hspace{1cm} \\
     \bigcap_{k=1}^{\infty} (\phi^J)^{-k}(X) \ &= \ \{ x : (\phi^J)^{k}(x) \not= \emptyset \ \ \text{for all} \ \ k \in \Z_+ \}. \\
     X_- \ = \ \pi_0(\S_-(\Phi)) \ = \ &\{ x : \bar \t(x) = \infty \} \ = \hspace{1cm} \\
     \bigcap_{k=1}^{\infty} (\phi^J)^{k}(X) \ &= \ \{ x : (\phi^J)^{-k}(x) \not= \emptyset \ \ \text{for all} \ \ k \in \Z_+ \}.\\
     X_{\pm} \ = \ X_+ \cap X_- \ &= \ \pi_0(\S(\Phi)) \hspace{2cm}
     \end{split}\end{align}
     \end{prop}

    \begin{proof} Observe that $\t(x) < t$ if and only if $\phi^{\{ t \} }(x) = \emptyset$. So  $\{ \t < t \}$  is
    the open set $(\phi^{\{ t \} })^*(\emptyset)$.

    The equations of (\ref{eqsemi11}) are easy to check using Proposition \ref{semiprop07} and Proposition
    \ref{semiprop05a}.

    \end{proof}

\vspace{1cm}

 \subsection{{\bf Restriction to a Closed Subset}}\label{restriction2}\vspace{.5cm}

 If $C$ is a closed subset of $X$, then for $\Phi$ a semiflow relation on $X$, and $-\infty \le t_1 \le t_2 \le \infty$
 \begin{equation}\label{eqsemi02aa}
 \S_C([t_1,t_2],\Phi)   \ =_{def} \   \{ \xx \in \S([t_1,t_2],\Phi) : \xx( [t_1,t_2]) \subset C \}
 \end{equation}
 is a point-wise closed subset of $C([t_1,t_2]; X)$. Again we write $ \S_C([t_1,t_2])$ when $\Phi$ is understood.

\begin{prop}\label{semiprop05aa}  Let $C$ be a closed subset of $X$.
The set
\begin{equation}\label{eqsemi04aa} \begin{split}
\Phi_C   \ =_{def} \   \{ (x,t,y) \in C \times \R_+ \times C : \hspace{2cm} \\
\text{there exists} \  \xx \in \S_C([0,t]) \ \ \text{with} \ \ \xx( 0) = x, \xx( t) = y \}
\end{split}\end{equation}
is a semiflow relation on $C$ called the \emph{restriction}\index{restriction} of $\Phi$ to $C$.

A path $\xx \in C([t_1,t_2],X)$ is a solution path for $\Phi_C$ if and only if $\xx \in \S_C([t_1,t_2])$.\end{prop}

\begin{proof} This follows from Lemma \ref{semilem04} applied with $X = C$ and $\S_0([t_1,t_2]) = \S_C([t_1,t_2])$.

\end{proof}

It is clear that if we restrict $\overline{\Phi}$ to $C$ we obtain the reverse of $\Phi_C$ and so we can write $\overline{\Phi}_C$ without
ambiguity.

The obvious way of defining the restriction of $\Phi$ to $C$ would be to use the intersection $\Phi \cap (C \times \R_+ \times C)$.
However, this need not be a semiflow relation. On the other hand it leads to an alternative way of obtaining $\Phi_C$.

Call  $\Psi$ a \emph{weak semiflow relation}\index{semiflow relation!weak} on $X$ when it is a closed subset of $X \times \R_+ \times X$
which satisfies the Initial Value Condition and also \vspace{.5cm}

(ii$'$) \textbf{Weak Kolmogorov Condition:}\index{Kolmogorov Condition!Weak} \index{Weak Kolmogorov Condition}For all
$t_1, t_2 \in \R_+$ $\phi^{t_1} \circ \phi^{t_2} \subset \phi^{t_1 + t_2}$
or, equivalently, for $ x, y, z \in X$
\begin{equation}\label{eqKolaaa}
(x,t_1,z), (z,t_2,y) \in \Phi \quad \Longrightarrow \quad (x,t_1+t_2,y) \in \Phi.
\end{equation} \vspace{.5cm}

It is clear that the intersection of any family of weak semiflow relations on $X$ is a weak semiflow
relation on $X$.

For $\Psi \subset X \times \R_+ \times X$, let
 $\Psi' \subset X \times \R_+ \times X$ so that
\begin{align}\label{eqsemi05aa}\begin{split}
(x,t,y) \in \Psi' \quad &\Longleftrightarrow \quad \text{for all} \ \  s \in [0,t]\\
\text{there exists} \ \ z \in X& \ \ \text{such that} \ \ (x,s,z), (z,t-s,y) \in \Psi.
\end{split}\end{align} \vspace{.5cm}

Clearly $\Psi_1 \subset \Psi_2$ implies $\Psi_1' \subset \Psi_2'$.

\begin{lem}\label{semilem06aa} If $\Psi$ is a weak semiflow relation on $X$, then $\Psi'$ is a
weak semiflow relation on $X$ with $\Psi' \subset \Psi$  and $\Psi' = \Psi$ if and only if
$\Psi$ is a semiflow relation on $X$.

If $\Phi$ is a semiflow relation contained in $\Psi$, then $\Phi \subset \Psi'$.
 \end{lem}

\begin{proof}    For $s \in \R_+$, let
$Q_s = \{ (x,z,t,y) : (x, \min(s,t), z),$ \\ $(z, t - \min(s,t),y) \in \Psi \}$. Clearly,
$Q_s$ is a closed subset of $X \times X \times \R_+ \times X$. Let
$R_s = \pi_{134}(Q_s)$ so that $(x,t,y) \in R_s$ if and only if
there exists $z$ such that $(x, \min(s,t), z), (z, t - \min(s,t),y) \in \Psi$.
Because $R_s \cap X \times [0,N] \times X$ is the
image of $Q_s \cap (X \times X \times [0,N] \times X)$ it follows that $R_s$
is closed. Hence, $\Psi' = \bigcap_{s \in \R_+} \ R_s$
is closed.

Since $\Psi$ satisfies the Initial Value Condition, $R_0 = \Psi$ and so $\Psi' \subset \Psi$.
It then follows that $(x,0,y) \in \Psi' $
implies $x = y$. On the other hand, for any $x \in X$, $(x,x,0,x) \in Q_s$ for all $s$.
 Hence $(x,0,x) \in R_s$ for all $s$ and so
$(x,0,x) \in \Psi'$.  Thus, $\Psi'$ satisfies the Initial Value Condition.

Now assume that $(x,t_1,z), (z,t_2,y) \in \Psi'$. Because $\Psi' \subset \Psi$ and $\Psi$
is a weak semiflow relation, $(x,t_1 + t_2,y) \in \Psi$. Let $s \in [0,t_1 + t_2] $ then
either $s \in [0,t_1] $ or $s - t_1 \in [0,t_2]$.

If $0 \le s \le t_1$, then there exists $w \in X$ such that $(x,s,w),(w,t_1-s,z) \in \Psi$
and so by the Weak Kolmogorov Condition $(w,t_1 + t_2 -s,y) \in \Psi$.

If $t_1 \le s \le t_1 + t_2$, then there exists $w \in X$ such that
 $(z, s - t_1,w), (w, t_1 + t_2 -s,y)  \in \Psi$ and so $(x,s,w) \in \Psi$.

 Thus, $(x, t_1 + t_2, y) \in \Psi'$.  It follows that $\Psi'$ is a weak semiflow relation.

Finally, if $\Phi$ is a semiflow relation contained in $\Psi$, then $\Phi \subset \Psi'$
by the Kolmogorov Condition for $\Phi$. In particular, if $\Phi = \Psi$ is a semiflow relation
then $\Psi' = \Psi$.

Conversely,  $\Psi' = \Psi$ and the Weak Kolmogorov Condition together imply the Kolmogorov
Condition and so a weak semiflow relation $\Psi$ with $\Psi' = \Psi$ is a semiflow relation.

\end{proof}

\textbf{Remark:} If $\Q$ is a countable dense subset of $\R_+$ it follows that
$\Psi' = \bigcap_{s \in \Q} \ R_s$ because $\Psi$ is closed. Thus, we can regard
obtaining $\Psi'$ from $\Psi$ as a countable construction.\vspace{.5cm}

With $\Psi_0 = \Psi$, inductively, let  $\Psi_{k+1} = (\Psi_k)'$ for $k \in \Z_+$
and $\Psi_{\infty} = \bigcap_k \ \Psi_k$.

\begin{prop}\label{semiprop06ab} If $\Psi$ is weak semiflow relation, then
 $\{ \Psi_k \}$ is a decreasing sequence of weak semiflow relations. The
 intersection $\Psi_{\infty}$ is the maximum semiflow relation contained in $\Psi$. \end{prop}

\begin{proof}
The first claim of the Proposition then follows from Lemma \ref{semilem06aa} by induction.
 The intersection $\Psi_{\infty}$ is a weak semiflow relation which
contains any semiflow relation $\Phi$ which is contained in $\Psi$.

We complete the proof by checking by showing $(\Psi_{\infty})' = \Psi_{\infty} $.

Assume that  $(x,t,y) \in \Psi_{\infty}$ and $0 \le s \le t$. Let
$A_k = \{ z : (x,s,z), (z,t-s,y) \in \Psi_k \}$. Because $(x,t,y) \in \Psi_{k+1}$
the set $A_k$ is nonempty. It is clear that each $A_k$ is closed and so is compact. Because
$\Psi_{k+1} \subset \Psi_k$ we have
 $A_{k+1} \subset A_k$. By compactness, the intersection $A_{\infty} = \bigcap_k A_k$ is
 nonempty.  If $z \in A_{\infty}$ then $(x,s,z), (z,t-s,y) \in \Psi_k$ for all $k$,
 i.e. $(x,s,z), (z,t-s,y) \in \Psi_{\infty}$. Since $s$ was arbitrary,$(x,t,y) \in (\Psi_{\infty})'$.

 From Lemma \ref{semilem06aa} it follows that $\Psi_{\infty}$ is a semiflow relation.

\end{proof}

\begin{prop}\label{semiprop06ac} If $\Phi$ is a semiflow relation on $X$ and $C$ is a
closed subset of $X$, then $\Phi \cap (C \times \R_+ \times C)$ is a weak semiflow
relation on $C$ with $\Phi_C = (\Phi \cap (C \times \R_+ \times C))_{\infty} $ and
so $\Phi_C$ is the maximum semiflow relation on $C$ contained in $\Phi \cap (C \times \R_+ \times C)$.
\end{prop}

\begin{proof} It is clear that $\Phi_C \subset (\Phi \cap (C \times \R_+ \times C))$ and
that the latter is a weak semiflow relation on $C$ because $\Phi$ is a semiflow relation on
$X$.  It follows from Proposition \ref{semiprop06ab} that
$\Phi_C \subset (\Phi \cap (C \times \R_+ \times C))_{\infty}$.

If $(x,t,y) \in (\Phi \cap (C \times \R_+ \times C))_{\infty}$ then applying
Proposition \ref{semiprop05} to the semiflow relation $(\Phi \cap (C \times \R_+ \times C))_{\infty}$
we obtain $\xx$ a \\ $(\Phi \cap (C \times \R_+ \times C))_{\infty}$ solution path on $[0,t]$
with $\xx( 0) = x, \xx( t) = y$.  Since
$(\Phi \cap (C \times \R_+ \times C))_{\infty} \subset \Phi \cap (C \times \R_+ \times C)$ it follows
that $\xx \in \S_C([0,t])$.  Hence, $(x,t,y) \in \Phi_C$.

\end{proof}

Notice that we can characterize $\Phi$ + invariance using $\Phi_C$:
\begin{equation}\label{eqsemi04Q}
C \ \ \text{is} \ \ \Phi \ \ \text{+ invariant} \qquad \Longleftrightarrow \qquad \Phi_C \ = \ \Phi \cap (C \times \R_+ \times X).
\end{equation}

Define
\begin{equation}\label{eqsemi04aax}
\Phi_{C+}   \ =_{def} \  \Phi_C \cup \{ (x,0,x) : x \in X \}.
\end{equation}
Clearly, $\Phi_{C+}$ is a semiflow relation on $X$.

For $t_1  \le t_2$ in $\R_+$, we define, following (\ref{eqsemi02}):

\begin{equation}\label{eqsemi05aax}
(\phi_C)^{[t_1,t_2]} \ = \  \phi^{[t_1,t_2]}(\Phi_{C+}) \ = \ \pi_{13}((\Phi_{C+}) \cap (X \times [t_1,t_2] \times X))
\end{equation}
so that $ (\phi_C)^{[t_1,t_2]}$ is a closed relation on $X$. If $t_1 = 0$, then $(\phi_C)^{[t_1,t_2]}$ is reflexive, i.e. $1_X \subset (\phi_C)^{[t_1,t_2]}$.
If $0 < t_1$, then $\phi^{[t_1,t_2]}$ is a closed relation on $C$ and we can use $\Phi_C$ instead of $\Phi_{C+}$ in the definition (\ref{eqsemi05aax}) and
so if $0 < t_1$
\begin{equation}\label{eqsemi05aaa}\begin{split}
(\phi_C)^{[t_1,t_2]} \ = \
 \{ (x,y) : \text{there exist} \ \ t \in [t_1,t_2], \xx \in \S_C([0,t]) \\ \text{with} \ \ \xx( 0) = x, \ \ \xx( t) = y \}. \hspace{2cm}
\end{split}\end{equation}

\begin{note}\label{seminote} The parentheses in (\ref{eqsemi05aax}) play an important role because $(\phi_C)^{[t_1,t_2]}$ is usually a proper subset of
$(\phi^{[t_1,t_2])})_C$. \end{note}
A pair $(x,y)$ is in the latter relation when $x, y \in C$ and there exists a $\Phi$ solution path of length
$t \in [t_1,t_2]$ from $x$ to $y$, whereas $(x,y) \in (\phi_C)^{[t_1,t_2]}$ requires that such a solution path exist which runs entirely in $C$.

From Lemma \ref{semilem06} applied to $\Phi_{C+}$ we obtain
for $0 \le t_1 < t_2, t_3 < t_4$ in $\R_+$, then
\begin{equation}\label{eqsemi06aa}
(\phi_C)^{[t_3,t_4]} \circ (\phi_C)^{[t_1,t_2]} \ = \  (\phi_C)^{[t_1 + t_3,t_2 + t_4]}
\end{equation}

As before we let $I = [0,1]$ and $J = [1,2]$, and observe
 that $1_X \subset (\phi_C)^I$, i.e. $(\phi_C)^I$ is reflexive, and

\begin{equation}\label{eqsemi07aa}\begin{split}
(\phi_C)^I \circ (\phi_C)^I \ = \ (\phi_C)^{[0,2]} \ = \ (\phi_C)^I \cup (\phi_C)^J. \hspace{2cm}\\
(\phi_C)^J \circ (\phi_C)^I \ = \ (\phi_C)^I \circ (\phi_C)^J \ = \ (\phi_C)^{[1,3]} \ \subset \ (\phi_C)^J \cup ((\phi_C)^J)^2.
\end{split}\end{equation}

\begin{prop}\label{semiprop07aa} A point $x \in C$ is a \emph{terminal point}\index{terminal point} for $\Phi_C$ when the following equivalent conditions
\begin{itemize}
\item[(i)] $(x,t,y) \in \Phi_C$ implies $t = 0$ and so $y = x$.
\item[ (ii)] For no $t > 0$ does there exist $\xx \in \S_C([0,t])$ with $\xx( 0) = x$.
\end{itemize}
We then call $x$ a \emph{terminal point} of $C$.
\end{prop}

\begin{proof} See Proposition \ref{semiprop06}.

\end{proof}

\begin{prop}\label{semiprop08aab} If $\Phi$ is a complete semiflow on $X$ and $C$ is a closed subset of $X$, then
a terminal point of $C$ is contained in $\partial C = C \cap \overline{X \setminus C} = C \setminus C^{\circ}$. \end{prop}

\begin{proof}  Fix $t > 0$. Because $\Phi$ is complete, Lemma \ref{semilem04} implies there exists $\xx \in \S([0,t])$ with
$\xx( 0) = x$.  For any $\ep$ with $0 < \ep \le t$ the restriction $\xx|[0,\ep] \not\in \S_C([0,\ep])$ and so there
exists $\d$ with $0 < \d \le \ep$ such that $\xx( \d) \in X \setminus C$. Since $\xx( \d) \to x$ as $\ep \to 0$, it follows that
$x \in \overline{X \setminus C}$.

\end{proof}

  Define the usc functions $\t_C, \bar \t_C : C \to \R_+ \cup \{ \infty \}$ using (\ref{eqsemi09}) for $\Phi_C$ so that
  \begin{align}\label{eqsemi08aa} \begin{split}
  \t_C(x) & \ =_{def} \  \sup \{ t \in \R_+ : \text{there exists} \ \ y \in X \ \ \text{such that} \ \ (x,t,y) \in \Phi_C \}, \\
  \bar \t_C(x) & \ =_{def} \  \sup \{ t \in \R_+ : \text{there exists} \ \ y \in X \ \ \text{such that} \ \ (y,t,x) \in \Phi_C \}.
  \end{split} \end{align}
  Thus, a point $x \in C$ is terminal if and only if $\t_C(x) = 0$.
Clearly, the function $\bar \t_C$ is $\t_C$ applied to the reverse relation $\overline{\Phi}_C$.

    By Proposition \ref{semiprop07} if $\t_C(x) = \infty$ then there exists $\xx \in \S_C([0,\infty])$ with $\xx(0) = x$ and if
  $\t_C(x) < \infty$ then the set $\{ y : (x,\t_C(x),y) \in \Phi_C \}$ is nonempty and consists of terminal points.

   Define the \emph{solution path spaces}\index{solution path space} following (\ref{eqsemi08a})
\begin{align}\label{eqsemi08aab}\begin{split}
\S_+(\Phi_C) \ = \ \S_C([0,\infty]), &\quad \S_-(\Phi_C)\ = \ \S_C([-\infty,0]), \\ \S(\Phi_C) = \S_C(&[-\infty,\infty]).
\end{split}\end{align}

Following Proposition \ref{semiprop08} we
define the subsets $C_+, C_-$ and $C_{\pm}$ for $\Phi_C$.
     \begin{align}\label{eqsemi11a}\begin{split}
    C_+ \ = \ \pi_0(\S_+(\Phi_C)) \ = \ &\{ x \in C : \t_C(x) = \infty \} \ = \hspace{1cm} \\
     \bigcap_{k=1}^{\infty} ((\phi_C)^J)^{-k}(X) \ &= \ \{ x : ((\phi_C)^J)^{k}(x) \not= \emptyset \ \ \text{for all} \ \ k \in \Z_+ \}. \\
   C_- \ = \ \pi_0(\S_-(\Phi_C)) \ = \ &\{ x \in C : \bar \t_C(x) = \infty \} \ = \hspace{1cm} \\
     \bigcap_{k=1}^{\infty} ((\phi_C)^J)^{k}(X) \ &= \ \{ x : ((\phi_C)^J)^{-k}(x) \not= \emptyset \ \ \text{for all} \ \ k \in \Z_+ \}.\\
     C_{\pm} \ = \ C_+ \cap C_- \ &= \ \pi_0(\S(\Phi_C)) \hspace{2cm}
     \end{split}\end{align}

 \begin{df}\label{semidf09} Let $\Phi$ be a semiflow relation  on $X$ and $C$ be a  (not necessarily closed) subset of $X$.

 We say that $C$ is \emph{+ viable} for $\Phi$. \index{subset!+ viable} when for every $x \in C$ there exists a $\Phi$ solution path
 $\xx : [0, \infty] \to X$ with $\xx(0) = x$ and $\xx(t) \in C$ for all $t \ge 0$.

  We say that $C$ is \emph{- viable} for $\Phi$ when it is + viable for $\overline{\Phi}$. \index{subset!- viable}

  We say that $C$ is \emph{viable} for $\Phi$ (= viable for $\overline{\Phi}$) when it is both + and - viable. \index{subset!viable}
So $C$ is viable when through every point $x \in C$ there exists a bi-infinite $\Phi$ solution path which is contained in $C$.
\end{df} \vspace{.5cm}

   \begin{prop}\label{semiprop09a} Let $\Phi$ be a semiflow relation  on $X$ and $C$ be a closed subset of $X$.\vspace{.25cm}

The following conditions are equivalent.
  \begin{itemize}
   \item[(i)] $C$ is + viable for $\Phi$.
  \item[(ii)] $C = Dom((\phi_C)^J)$, i.e. $C$ is + viable for the relation $(\phi_C)^J$.
  \item[(iii)] There exists $t > 0$, such that for all $x \in C$ there exists a solution path $\xx : [0,t] \to C$ with $\xx(0) = x$.
  \item[(iv)] $C = C_+$.
  \item[(v)] $\pi_0(\S_+(\Phi_C)) = C$
    \end{itemize} \vspace{.25cm}

  So  $C$ is \emph{- viable} for $\Phi$ when
  when $C = C_-$. \vspace{.25cm}

 The following conditions are equivalent.
   \begin{itemize}
   \item[(vi)] $C$ is \emph{viable} for $\Phi$.
   \item[(vii)] $(\phi_C)^J$ is a surjective relation on $C$, i.e. $C$ is  viable for the relation $(\phi_C)^J$.
   \item[(viii)] $C = C_{\pm}$.
   \item[(ix)] $\pi_0(\S(\Phi_C)) = C$.
   \end{itemize}

 \end{prop} \vspace{.5cm}

\begin{proof}  The equivalences among the various conditions are clear from the above descriptions. Notice that + viability for $\Phi$ or, equivalently,
 for $(\phi_C)^J$ is a stronger condition than viability for $\phi^J$ or, equivalently, for $(\phi^J)_C$, see Notice \ref{seminote}.

 \end{proof}

 Since the notions of viability are the same for $\Phi$ and $(\phi_C)^J$, the notions of minimality agree as well.

 We have the following version of Proposition \ref{relprop06aaa} and Lemma \ref{rellem06aa}

  \begin{prop}\label{semiprop10} Let $\Phi$ be a semiflow relation  on $X$.

  (a) If a  subset $A$ is $\Phi$ + invariant, then it $\Phi$ invariant if and only if it is - viable.
  If $A$ is $\overline{\Phi}$ + invariant,then it is $\overline{\Phi}$ invariant if and only if it is + viable.
  In particular, an attractor is - viable and a repeller is + viable.

  (b) If $A$ is  + invariant and $B$ is + viable, then $A \cap B$ is + viable.

  (c) If $A$ is invariant for $\Phi$, e.g. an attractor, and $B$ is invariant for $\overline{\Phi}$, e.g. a repeller,
  then $A \cap B$ is viable for $\Phi$.

  (d)  If $C$ is any  subset, then  \\  $C_+$ is + viable, $C_-$ is - viable
  and $C_{\pm}$ is viable.

  (e) Let $\{ C_i \}$ be a collection of subsets of $X$.  If all are + viable, or all - viable or
    all viable, then $C = \bigcup \{ C_i \}$ satisfies the corresponding property.
  \end{prop}

\begin{proof} (a) A + invariant set $A$ is invariant if and only if for all $x \in A$ there exists a $\Phi$ solution path  $\xx : [-\infty,0] \to A$ with
$\xx(0) = x$.  This is the same as - viability.

(b) If $x \in A \cap B$ then there exists a $\Phi$ solution path $\xx :[0,\infty] \to B$ with $\xx(0) = x$.
Because $A$ is + invariant, $\xx(t) \in A$ for all $t \ge 0$.
Thus, $\xx(t) \in A \cap B$ for all $t$.

(c) $A$ is + invariant and $B$ is + viable by (a) and so by (b) $A \cap B$ is + viable.  Similarly,
$A$ is - viable and $B$ is - invariant and so $A \cap B$ is - viable.

(d) If $x \in C_+$ there exists a $\Phi$ solution path $\xx :[0,\infty] \to C$ with $\xx(0) = x$. The translate $Trsl_{t}(\xx)$ is a solution path
in $C$ with
$Trsl_{t}(\xx)(0) = \xx(t)$. Thus, $\xx(t)  \in C_+$ for all $t \ge 0$. The proofs for $C_-$ and $C_{\pm}$ are similar.

(e) Obvious.

    \end{proof}

    \begin{prop}\label{semiprop11} (a) If $\K$ is a nonempty subset of $\S_+(\Phi_C)$, then
       \begin{equation}\label{eqsemi13}
      \om [\K]  \ =_{def} \   \bigcap_{n=0}^{\infty} \ \overline{ \{ \xx(t) : \xx \in \K, t \geq n \}}
      \end{equation}
      is a nonempty, closed, viable  subset of $C_{\pm}$.

       In particular, if $\xx \in \S_+(\Phi_C)$, then
         \begin{equation}\label{eqsemi13a}
      \om [\xx]  \ =_{def} \   \bigcap_{n=1}^{\infty} \ \overline{ \{ \xx(t) : t \geq n \}}
      \end{equation}
      is a nonempty, closed, viable  subset of $C_{\pm}$.

        (b) Assume that $X_+ = X$, i.e. $\Phi$ is complete, and that $A$ is a closed subset of $X$.
        Let $\K(A)  \ =_{def} \  \{ \xx \in \S_+(\Phi) : \xx(0) \in A \}$.

      \begin{itemize}
        \item[(i)] $\om [\K(A)] = Lim sup \{ \phi^t(A) \} = \bigcap_{n = 1}^{\infty} \overline{\bigcup \{ \phi^t(A) : t \ge n \}}$.

      \item[(ii)] If $\om[\K(A)] \subset A$, then $\om[\K(A)]$ is $\Phi$ invariant and is the maximum - viable subset of $A$.

      \item[(iii)] If $A$ is $\Phi$ + invariant, then $\om[\K(A)] = \bigcap_{k=1}^{\infty} \{ \phi^t(A) \}$.

      \item[(iv)] If $\om[\K(A)] \subset \subset A$, then $\om[\K(A)]$ is $\Phi$ attractor.
      \end{itemize}
      \end{prop}

      \begin{proof} (a) If $\xx \in \S_+(\Phi_C)$  and $\{ t_k \}$ is a sequence in $\R_+$ with $t_{k+1} - t_k > 1$,
      then there is an orbit sequence $\yy$ of $(\phi_C)^J$ and $\{ n_k \}$ an increasing sequence in $\Z_+$ such that
    $\yy(n_k) = \xx(t_k) \}$ for all $k$  We say that such an $(\phi_C)^J$  orbit sequence $\yy$ is contained in $\xx$.
    If we let $\hat \K$ consist of all of the $(\phi_C)^J$ orbit sequences which are contained in some $\xx \in \K$, it is clear
       that $\om[\hat \K]$ for the closed relation $(\phi_C)^J$ is equal to $\om[\K]$. So Proposition \ref{relprop06a}(a) implies that
       $\om[\K]$ is a nonempty, closed, viable  subset of $C_{\pm}$ since these concepts are the same for $\Phi_C$ and for $(\phi_C)^J$.

       (b) (i) is clear from completeness of $\Phi$ as in Proposition \ref{relprop06a}(b)(i) and this clearly implies (iii) since
       $\Phi$ + invariance implies that that collection $\{ \phi^t(A) \}$ is decreasing in $t$.

       (ii) and (iv) follow from the corresponding results in \ref{relprop06a}(b) applied to $\phi^J$ and to $\om[\hat \K]$.
       A $\Phi$ attractor is an $\phi^J$ attractor, see Proposition \ref{semiprop05a}.

      \end{proof}

      If $\xx : [-\infty,0] \to X$ is a solution path for $\Phi_C$ so that $\bar \xx \in S_+(\overline{\Phi}_C)$ then
            \begin{equation}\label{eqsemi13c}
      \a[\xx]  \ =_{def} \   \bigcap_{n=1}^{\infty} \ \overline{ \{ \xx(t) : t \le -n \}} = \om[\bar \xx].
      \end{equation}
       is a nonempty, closed, viable  subset of $C_{\pm}$ by Proposition \ref{semiprop11} applied to $\bar \xx$.

\vspace{1cm}%

\subsection{\textbf{Isolated  Subsets and the Conley Index}}\label{Conleysem}\vspace{.5cm}

 For a closed subset $C$ we follow (\ref{eqConley07a}) and define:

  \begin{equation}\label{eqConleysem01}
\d_{\Phi}(C) \ =_{def} \ \bigcap_{\ep > 0} \ \r_{\phi^{[0,\ep]}}(C).
\end{equation}

 We call $\d_{\Phi}(C)$ the \emph{$\Phi$ boundary }\index{$\Phi$ boundary} of $C$

\begin{prop}\label{propConleysem01}  Let $C$ be a closed subset of $X$.

(a) The subset $C$ is $\Phi$ + invariant if and only if  $\d_{\Phi}(C) \ = \emptyset$.

(b) A point $x$ lies in $\d_{\Phi}(C)$ if and only if
there exists a sequence $\{ (x_n,t_n) \in C \times [0,1] \}$ converging to $(x,0)$
and for each $n$ there exists $y_n \in X \setminus C$ such that $(x_n,t_n,y_n) \in \Phi$.
In particular, $x \in C$ and $\{ y_n \}$ converges to $x$.

(c) The $\Phi$ boundary satisfies $ \d_{\Phi}(C)  \ \subset \ \partial C$ and so
  \begin{equation}\label{eqConleysem02}
\d_{\Phi}(C) \ = \ \bigcap_{\ep > 0} \ \d_{\phi^{[0,\ep]}}(C).
\end{equation}

(d) If $\Phi$ is complete and $x$ is a terminal point of $C$, i.e. $\t_C(x) = 0$, then $x \in \d_{\Phi}(C)$.

(e) If $A$ is a closed $\Phi_C$ + invariant subset of $C$, then $\d_{\Phi}(A) \ \subset \ \d_{\Phi}(C)  \ \subset \ \partial C$.
\end{prop}

\begin{proof}
%If $x \in C^{\circ}$, then there exists $\ep > 0$ such that $V_{2\ep}(x) \subset C^{\circ}$ and there exists $\ep_1 > 0$ such that
%$y \in \phi^{[0,\ep_1]}(x)$ implies $d(x,y) < \ep$. Hence, $\phi^{[0,\ep_1]}(V_{\ep}(x)) \subset C^{\circ}$. Hence,
%$x \not\in\d_{\Phi}(C)$. Moreover, if $A$ is $\Phi_C$ + invariant and $x \in A$, then $\phi^{[0,\ep_1]}(A \cap V_{\ep}(x)) \subset A \cap C^{\circ}$.
%Thus, $x \not\in\d_{\Phi}(A)$. So we have $\d_{\Phi}(C) \subset \partial C$ and if $A$ is $\Phi_C$ + invariant, then
%$\d_{\Phi}(A) \subset \partial C$.

(a): $C$ is $\Phi$ + invariant if and only if for some $\ep > 0$, $\phi^{[0,\ep]}(C) \subset C$ and so $\r_{\phi^{[0,\ep]}}(C) = \emptyset$.
Since $\d_{\Phi}(C)$ is the decreasing intersection of the compacta $\{ \r_{\phi^{[0,\ep]}}(C) \}$, if it is empty then for $\ep > 0$ small enough
$\r_{\phi^{[0,\ep]}}(C) = \emptyset$.

(b), (c): The sequence criterion for a point of $\d_{\Phi}(C)$ is easy to check. Since $C$ is closed and $x$ is the limit of the sequence
$\{ x_n \}$ in $C$, it follows that $x \in C$. If $y$ is any limit point of the sequence $\{ y_n \}$, then
$(x,0,y) \in \Phi$ and so $y = x$. That is,  $\{ y_n \}$ converges to $x$. This implies $x \in \partial C$. Thus, we may intersect with
$\partial C$ to obtain (\ref{eqConleysem02}).

(d): If $\Phi$ is complete and $x \in C$ then there exists $\xx \in \S([0,1])$ with $\xx(0) = x$.  If $x$ is a terminal point of $C$, then for
every $\ep > 0$ there exists $t$ with $0 < t \le \ep$ such that $\xx(t) \not\in C$. Hence, $x \in \r_{\phi^{[0,\ep]}}(C)$ for every $\ep > 0$.

(e): If $A$ is $\Phi_C$ + invariant and $x \in \d_{\Phi}(A)$ then by (b) there is a sequence $\{ (x_n,t_n,y_n) \in \Phi \}$ converging to $(x,0,x)$
with $(x_n,y_n) \in A \times (X \setminus A)$ for all $n$.  Let $\xx_n \in \S([0,t_n])$ with $\xx_n(0) = x_n$ and $\xx_n(t_n) = y_n$.  Since
$A$ is $\Phi_C$ + invariant, there exists $s_n$ such that $0 < s_n \le t_n$ and $z_n = \xx_n(s_n) \not\in C$.  Since $\{ (x_n,s_n,z_n) \}$ converges
to $(x,0,x)$ it follows from (b) that $x \in \d_{\Phi}(C)$.
%$\ep > 0$, then for $x \in A$, $(\phi_C)^{[0,\ep]}(x)  = (\phi_A)^{[0,\ep]}(x) \subset A$. So if
%$\phi^{[0,\ep]}(x) \subset C$, then $\phi^{[0,\ep]}(x) \subset A$.
Thus,   $\d_{\Phi}(A) \ \subset \ \d_{\Phi}(C)$.

\end{proof}

Recall that, from  (\ref{eqrel19}) applied to $F = F_C = (\phi_C)^J$ on $C$, we have
  \begin{equation}\label{eqConley04xxx}
  \G ((\phi_C)^J) \ \subset \ \CC ((\phi_C)^J) \ \subset \ \O ((\phi_C)^J) \ \cup \ (C_+ \times C_-). \hspace{2cm}
  \end{equation}

    \begin{prop}\label{propConley06xxx}(a) If $K$  is a closed subset of $C$ such that\\ $K \cap  C_+ = \emptyset$,
 then $[(\phi_C)^I \cup \CC((\phi_C)^J)](K) = [[(\phi_C)^I \cup \G ((\phi_C)^J)](K) = \O ([(\phi_C)^I)(K)$ is a
 closed $\Phi$ + invariant subset of $C$ which is
  disjoint from $C_+$.

  (b) If $A$  a closed, $\Phi$ + invariant subset of $C$ such that $A \cap  C_{\pm} = \emptyset$, then $A \cap  C_{+} = \emptyset$.

 \end{prop}

 \begin{proof} (a) Because $K$ is disjoint from $C_+$, (\ref{eqConley04}) implies that
 $[(\phi_C)^I \cup \CC((\phi_C)^J)](K) = [[(\phi_C)^I \cup \G (F_C)](K) = \O ([(\phi_C)^I)(K)$ and
 it is closed because $[(\phi_C)^I \cup \CC((\phi_C)^J)]$ and $K$ are closed.
 Since $C_+$ is $\bar \Phi_C$ invariant, it follows that $\O ([(\phi_C)^I)(K)$ is disjoint from
 $C_+$.

 (b): If $x \in A \cap C_+$, then there exists a solution path $\xx \in \S_+(\Phi_C)$ with $\xx(0) = x$. Because $A$ is $\Phi$ + invariant,
 $\xx(t) \in A$ for all $t \in \R_+$. Because $A$ is closed, it follows that $\om[\xx] \subset A$. By Proposition \ref{relprop06a}
 $\om[\xx]$ is a nonempty subset of $C_{\pm}$. Hence, $A \cap C_{\pm} \not= \emptyset$.

 \end{proof}

Recall from Definition \ref{semidf09} and Proposition \ref{semiprop10} that $K \subset C$ with $C$ closed is $\Phi$ (or $\Phi_C$)
invariant if and only if it is $\phi^J$ invariant (resp. $(\phi_C)^J$ invariant).
The closed set  $C$ is + viable, - viable or viable
for $\Phi$  if and only if it satisfies the corresponding property for $(\phi_C)^J$. Finally, for  $C$ the definitions of the sets
$C_+, C_-, C_{\pm}$ for $\Phi$ and for $(\phi_C)^J$ agree. So we can use Definition \ref{dfConley05a} applied to $(\phi_C)^J$ to define
 isolating neighborhoods and isolated sets for $\Phi$.

\begin{df}\label{dfConley05asemi} Let $C$ be a closed subset of $X$ and $\Phi$ be a semiflow relation on $X$.

(a) The set $C$ is  called an \emph{isolating neighborhood}
\index{isolating neighborhood} when $C_{\pm} \subset C^{\circ}$, i.e. its maximum viable subset is contained
in its interior. In that case, the viable set $A = C_{\pm}$ is called an
\emph{isolated viable set}\index{isolated viable set}\index{viable set!isolated}.

 $C$ is called a \emph{simple isolating neighborhood}  \index{isolating neighborhood!simple} when
every $x \in \partial C = C \setminus C^{\circ}$ is either a terminal point for $\Phi_C$ or is a terminal point for $\bar \Phi_C$,
i.e. the function $\t_C \cdot \bar \t_C \ = \ 0$ on $\partial C$.

(b) The set $C$ is  called a \emph{- isolating neighborhood} (or a \emph{+ isolating neighborhood})
 when $C_{-} \subset C^{\circ}$ (resp. $C_{+} \subset C^{\circ}$). In that case, the - viable set $C_{-}$ is called an
\emph{isolated - viable set} (resp. the + viable set $C_{+}$ is called an
\emph{isolated + viable set}).

 $C$ is called a \emph{simple - isolating neighborhood}   (or a \emph{simple + isolating neighborhood}) when
every $x \in \partial C$ a terminal point for $\bar \Phi_C$, i.e. $\bar \t_C \ = \ 0$ on $\partial C$ (resp. $\t_C \ = \ 0$ on $\partial C$ ) . \end{df}
\vspace{.5cm}

 For an isolated viable set for $\Phi$ we define the associated \emph{stable subset} and \emph{unstable subset}\index{subset!stable}\index{subset!unstable}
just as for a closed relation.

\begin{theo}\label{theoConley07newsemi} Let  $C$ be an isolating neighborhood for
$C_{\pm}$, i.e. $C_{\pm} \subset C^{\circ}$.

Define
\begin{align}\label{eqrel19new3semi}\begin{split}
W^s(C_{\pm}) \ &=_{def} \ \bigcup_{n \in \Z_+} \ ((\phi_C)^J)^{-k}(C_+), \\
W^u(C_{\pm}) \ &=_{def} \ \bigcup_{n \in \Z_+} \ ((\phi_C)^J)^k(C_-).
\end{split}\end{align}
$W^s(C_{\pm})$ is a + viable subset for $\Phi$, $W^u(C_{\pm})$ is a - viable subset for $\Phi$ and

\begin{align}\label{eqrel19new4semi}\begin{split}
x \in W^s(C_{\pm}) \  &\Longleftrightarrow \  \text{there exists} \  \xx \in \S_+(\Phi), \ \text{with} \ \xx(0) = x, \ \om[\xx] \subset C_{\pm} \\
x \in W^u(C_{\pm}) \ &\Longleftrightarrow \  \text{there exists} \  \xx \in \S_-(\Phi), \ \text{with} \ \xx(0) = x, \ \a(\xx) \subset C_{\pm}.
\end{split}\end{align}

\end{theo}

\begin{proof}  We apply Theorem \ref{theoConley07new} with $F_C = (\phi_C)^J$.  We leave the details to the reader.

\end{proof}

For $C$ an isolating neighborhood for a viable subset $A $, we call a pair $(P_1, P_2)$ of
 closed subsets of $X$  a $\Phi$ \emph{index pair} \index{index pair}  rel $C$
 for $A$ when
 the following conditions are satisfied:
 \begin{itemize}
 \item[(i)] $P_2 \ \subset \ P_1 \ \subset \ C$.

 \item[(ii)] $P_1$ and $P_2$ are $\Phi_C$ + invariant.

 \item[(iii)] $A = C_{\pm} \ \subset \ P_1^{\circ} \setminus P_2 $..

  \item[(iv)] $P_1 \setminus P_2 \ \subset  \ C^{\circ}$, or, equivalently, $P_1 \cap \partial C \subset P_2$.
 \end{itemize}
 We will sometimes consider the following strengthening of (iv).
  \begin{itemize}
 \item[(iva)] $\overline{P_1 \setminus P_2} \subset  \ C^{\circ}$, or,
 equivalently, $P_1 \cap \partial C$ is contained in the $P_1$ interior of $P_2$.
 \end{itemize}

 We call a pair $(P_1, P_2)$ of closed subsets of $X$  a $\Phi$ \emph{index pair} \index{index pair}
 for a viable set $A$ when there exists an isolating neighborhood $C$ for $A$ such that   $(P_1, P_2)$ is an index pair   rel $C$
 for $A$.

The following is the semiflow relation version of Theorem \ref{theoConley07}.

 \begin{theo}\label{theoConleysem03} Given $\Phi$ a semigroup relation on $X$, assume that $C$ is an isolating neighborhood for a  viable set $A$.
 \begin{enumerate}
 \item[(a)] If $(P_1,P_2)$ is a $\Phi$ index pair rel $C$ for $A$, then
 $C_- \subset P_1$ and $C_+ \cap P_2 = \emptyset$. In addition, $ \d_{\Phi}(P_1) \subset P_1 \cap \partial C \subset P_2$.

In addition, if (iva) holds for $(P_1,P_2)$, then for some \\ $\ep > 0, \ \ \d_{\phi^{[0,\ep]}}(P_1) \ \subset \ P_2$.

  \item[(b)] If $U$ and $V$ are open subsets of $X$ with $C_- \subset U$,  and $C_{\pm} \subset V \subset C$, then there exists
  a simple isolating neighborhood $C_0$ and
 a $\Phi$ index pair $(P_1,P_2)$  rel $C$ for $A$ with $P_1 \subset U$, and  such
 that   $\overline{P_1 \setminus P_2} \subset C_0 \subset V$.
 In particular, (iva) holds for $(P_1,P_2)$.

 \end{enumerate}
 \end{theo}

 \begin{proof} (a): The first part follows directly from Theorem \ref{theoConley07} applied to $(\phi_C)^J$. By Proposition \ref{propConleysem01}
 $\d_{\Phi}(P_1) \subset P_1 \cap \d_{\Phi}(C) \subset P_1 \cap \partial C$.

 By (\ref{eqConleysem02})
$\{  \d_{\phi^{[0,\ep]}}(P_1) \}$ is a decreasing family of compacta in $\partial P_1 \subset P_1$ with intersection $ \d_{\Phi}(P_1) \subset P_2$.
So if (iva) holds then for sufficiently small $\ep > 0, \  \d_{\phi^{[0,\ep]}}(P_1) \subset P_2$.

% By Proposition \ref{propConleysem02} $\g_{\Phi}(P_2) \subset \g_{\Phi}(P_1) \subset \g_{\Phi}(C)$. On the other hand, $\g_{\Phi}(P_1) \subset
% P_1 \cap \partial C \subset P_2$. The proof there showed that for $x \in P_2$, $x$ is terminal for $P_2$ if and only if it is terminal for $P_1$.
% Since $\g_{\Phi}(P_1) \subset P_2$ it follows that the set of terminal points for $P_1$ and for $P_2$ agree. Taking the closure we have
% $\g_{\Phi}(P_2) = \g_{\Phi}(P_1)$.

  (b): The proof follows that of Theorem \ref{theoConley07} (b). We sketch it, leaving the details to the reader.

  Choose $W_-, W_+$ relatively open subsets of $C$ as before.

     Because $C_-$ is an attractor for $\Phi_C$, Corollary \ref{relcor03semi} implies that there exists $P_1$ a $\Phi$ inward closed neighborhood
     (with respect to $C$)
   of $C_-$ with $P_1 \subset W_-$.  Similarly, as $C_+$ is a repeller,
   there exists $Q_1$ a $\bar \Phi$ inward for closed neighborhood (with respect to $C$)
   of $C_+$ with $Q_1 \subset W_+$.

    $C \setminus Int_C(Q_1)$ is $\Phi$ inward.
      Let $P_2 = P_1 \cap (X \setminus Int_C(Q_1)) = P_1 \setminus Int_C(Q_1) $ and let
   $C_0 = P_1 \cap Q_1$ so that   $\overline{P_1 \setminus P_2}  \subset C_0 \subset  V$.

      Since $C_{\pm} \subset P_1^{\circ} \setminus P_2 \subset (C_0)^{\circ}$ and $C_0 \subset C$, we have $(C_0)_{\pm} = C_{\pm}$.

      Because $P_1$ is inward for $\Phi_C$ it easily follows that $\bar \t_{P_1} = 0$ on $\partial_C(P_1)$ and similarly $\t_{Q_1} = 0$
      on $\partial_C(Q_1)$.  It follows that $\t_{C_0} \cdot \bar \t_{C_0} = 0 $ on $\partial_C(C_0) \subset \partial_C(P_1) \cup \partial_C(Q_1)$
      and $C_0 \subset V$ implies $\partial_C(C_0) = \partial C_0$.

    \end{proof}

    The following are the semiflow relation versions of Theorem \ref{theoConley08} and Theorem \ref{theoConley09}

   \begin{theo}\label{theoConleysem04}  A pair $(P_1, P_2)$  of closed subsets of $X$  is a $\Phi$ index pair,
   i.e. there exists a viable set $A$ and a closed neighborhood $C$ of $A$ such that
   $(P_1, P_2)$ is a $\Phi$ index pair rel $C$ for $A$ if
 the following conditions are satisfied:
 \begin{itemize}
 \item[(i$'$)] $P_2 \subset P_1$.

 \item[(ii$'$)] $P_2$ is $\Phi_{P_1}$ + invariant.

 \item[(iii$'$)] $(P_1)_{\pm} \subset  P_1^{\circ} \setminus P_2 $.

  \item[(iv$'$)] For some $\ep > 0, \d_{\phi^{[0,\ep]}}(P_1) \ \subset \ P_2$.
 \end{itemize}

  In addition,
 $C$ can be chosen so that (iva) holds if
  \begin{itemize}
   \item[($iva'$)] $\d_{\Phi}(P_1)$ contained in the $P_1$ interior of $ P_2$.
 \end{itemize}
 \end{theo}

 \begin{proof}  The proof is completely analogous to that of Theorem \ref{theoConley08} although the stronger condition (iv$'$) is needed for
 the proof rather than just the necessary condition $\d_{\Phi}(P_1) \subset P_2$. As shown in the proof of part (a) of Theorem \ref{theoConleysem03},
 condition (iva$'$) implies (iv$'$). Clearly, (iva$'$) is necessary to obtain (iva) for $C$.

First, just as before, we can find $C_1$ so that $P_1 \subset \subset C_1$ and $(P_1)_{\pm} = (C_1)_{\pm}$.

 Fix such a $C_1$ and choose $\ep > 0$ small enough that for all $x \in P_1$, $V_{\ep}(x) \subset C_1$ and, in addition,
 $\d_{\phi^{[0,\ep]}}(P_1) \ \subset \ P_2$. Hence, the closed set $\r_{\phi^{[0,\ep]}}(P_1)$ satisfies
 $P_1 \cap \r_{\phi^{[0,\ep]}}(P_1) = \d_{\phi^{[0,\ep]}}(P_1) \subset P_2$. Let $ C $ be the closure of the set
  \begin{align}\label{eqConleysem13} \begin{split}
\{ \ \  y \in X :\ \  &\text{there exists} \ \ x \in P_1 \ \ \text{such that } \\  d(y,x) \ &\leq \ \frac{1}{2} \min[\ep, d(x,\r_{\phi^{[0,\ep]}}(P_1))]\ \  \}.
 \end{split}\end{align}
   For $y \in P_1$ we can use $x = y$ which shows that $P_1 \subset C$. Notice next that the definition of
   $\ep $ implies that $C \subset C_1$ and so $C_{\pm} = (P_1)_{\pm}$ and then (iii$'$) implies
    that $C$ is an isolating neighborhood for $A = (P_1)_{\pm}$.  By definition, $x \in C^{\circ}$ for all $x \in P_1 \setminus \d_{\phi^{[0,\ep]}}(P_1)$
    because for such $x$, $d(x,\r_{\phi^{[0,\ep]}}(P_1)) > 0$. Contrapositively, $P_1 \cap \partial C \subset \d_{\phi^{[0,\ep]}}(P_1)$ which implies (iv)
    and (iva) follows from (iva$'$) if the latter holds.

  Conditions (i) and(iii) follow from (i$'$) and (iii$'$) since $C_{\pm} = (P_1)_{\pm}$.

    Now suppose that $y \in C \cap \r_{\phi^{[0,\ep]}}(P_1)$. There is a sequence of pairs $\{ (x_n,y_n) \}$ with
    $x_n \in P_1$, $d(y_n,x_n) \leq \frac{1}{2} \min[\ep, d(x_n,\r_{\phi^{[0,\ep]}}(P_1))]$
    for all $ n$ and $ \{ y_n \} \to y$.  By going to a subsequence we may assume $ \{ x_n \} \to x \in P_1$
    and so $d(y,x) \leq \frac{1}{2} \min(\ep, d(x,\r_{\phi^{[0,\ep]}}(P_1)))$.
    Since $y \in \r_{\phi^{[0,\ep]}}(P_1)$ this can only happen if $d(y,x) = 0$, i.e. $y = x$, and so $y \in \d_{\phi^{[0,\ep]}}(P_1)$.

    If $x \in P_1$ and $y \in (\phi_C)^{[0,\ep]}(x) \subset  C \cap (\phi)^{[0,\ep]}(x)$, then if $y$ were
    not in $P_1$, it would be in $\r_{\phi^{[0,\ep]}}(P_1)$ and so,
    by the argument of the preceding paragraph, in $\d_{\phi^{[0,\ep]}}(P_1) \subset P_1$.
    So $P_1$ is $(\phi_C)^{[0,\ep]}$ + invariant. Since $\ep > 0$ this implies $P_1$ is $\Phi_C$ + invariant.

     If $x \in P_2$ and $\xx \in \S([0,a],\Phi_C)$ with $\xx(0) = x$, then for all $t \in [0,a]$ $\xx(t) \in P_1$ because $P_1$ is $\Phi_C$ + invariant.
     Hence, $\xx \in \S([0,a],\Phi_{P_1})$.  It follows that $\xx(t) \in P_2$ for all $t \in [0,a]$ because $P_2$ is $\Phi_{P_1}$ + invariant.
     Thus, $P_2$ is   is $\Phi_C$ + invariant. This completes the proof of (ii).

    From $\Phi_C$ + invariance and Proposition \ref{propConleysem01}(e) together with the inclusion above we obtain:
    \begin{equation}\label{eqConley13a}
    \d_{\Phi}(P_1) \ \subset \ \partial P_1 \cap  \d_{\Phi}(C) \ \subset \ \d_{\phi^{[0,\ep]}}(P_1) \ \subset \ \partial P_1 \cap P_2.
    \end{equation}

    \end{proof}

    \begin{theo}\label{theoConleysem05} For a closed subset $P_1$ of $X$, there exists $P_2$ such that $(P_1,P_2)$ satisfies (i$'$) - (iv$'$)
    of Theorem \ref{theoConleysem04}, and so a $\Phi$ index pair if and only if the following conditions
     hold.
      \begin{itemize}

 \item[(i$''$)] $(P_1)_{\pm} \subset  P_1^{\circ} $, i.e. $P_1$ is an isolating neighborhood for $(P_1)_{\pm}$.

  \item[(ii$''$)] $\d_{\Phi}(P_1) \cap (P_1)_{+} = \emptyset$.
 \end{itemize}
 Furthermore, $P_2$ can be chosen so that (iva$'$) holds for the pair.
 \end{theo}

 \begin{proof}   By (ii$'$) $P_1 \setminus (P_1)_{+}$ is a relatively open subset of $P_1$ which contains $\d_{\Phi}(P_1)$,
%  the decreasing
% intersection of the compact subsets $\{ \d_{\phi^{[0,\ep]}}(P_1) \}$ of $P_1$. Hence, for $\ep > 0$ sufficiently small, $\d_{\phi^{[0,\ep]}}(P_1)$
% is disjoint from $(P_1)_+$.
 so we can choose $P_0$ a closed subset of $P_1$ which contains $\d_{\Phi}(P_1)$  in its $P_1$ interior  and which is disjoint from $ (P_1)_{+}$.
 If $P_0$ is any such set, then with
$P_2 = [(\phi_{P_1})^I \cup \CC((\phi_{P_1})^J)](P_0) = [[(\phi_{P_1})^I \cup \G (\phi_{P_1})](P_0) = \O ((\phi_{P_1})^I)(P_0)$
the pair  $(P_1,P_2)$  satisfies (i$'$) - (iv$'$) and (iva$'$)
    of Theorem \ref{theoConleysem04}.

 \end{proof}

As before, a closed set $P_1$ which satisfies (i$''$) and (ii$''$) is a special sort of isolating neighborhood which
 we will call an \emph{isolating neighborhood of index type}
 \index{isolating neighborhood!of index type}. From Theorem \ref{theoConleysem03} it follows that every
 isolated viable subset admits a neighborhood base of
 isolating neighborhoods of index type.

 For semiflow relations it is not necessarily true that a simple isolating neighborhood is an isolating neighborhood of index type. \vspace{.5cm}

 For a pair $(P_1,P_2)$ of closed sets with $P_2 \subset P_1$
 recall that $P_1/P_2$ is the quotient space obtained by identifying the subset $P_2$ to a point $[P_2]$. In the case
 when $P_2 = \emptyset$, $[P_2]$ is an isolated point separate from $P_1 \setminus P_2$ which equals $P_1$ in this case.  We denote by
 $u, v$ points of $P_1/P_2$ so that $u \in P_1 \setminus P_2$ or $u = [P_2]$.  If $\{ u_n \}$ is a sequence in $P_1/P_2$ converging to $u$, then
 if $u \in P_1 \setminus P_2$, the sequence eventually lies in the $P_1$ open set $P_1 \setminus P_2$. The relative topologies on
 $P_1 \setminus P_2$ induced from $P_1$ (or from $X$) and from $P_1/P_2$ agree.  In particular, $\{ u_n \}$ converges to $u$ in $P_1$.
 If $u = [P_2]$ then either $u_n = [P_2]$ eventually or else $\{ u_n \in P_1 \setminus P_2 \}$ is a subsequence which eventually enters
 every open set which contains $P_2$.  In particular, the set of $P_1$ limit points of the subsequence is contained in $P_2$.

 Now assume that $(P_1,P_2)$ is an index pair for a complete semiflow relation $\Phi$ on $X$. We define the induced relation
 $\Phi_{P_1/P_2} \subset (P_1/P_2) \times \R_+ \times (P_1/P_2)$ by $(u,t,v) \in \Phi_{P_1/P_2}$ when
  \begin{align}\label{eqConleysem13b} \begin{split}
  u \ = v \ &= \ [P_2],\\
  u, v \in P_1 \setminus P_2 \quad &\text{and} \quad (u,t,v) \in \Phi_{P_1}, \\
   u \in P_1 \setminus P_2, \ v \ = \ &[P_2] \quad \text{and there exists} \\
 (s,y) \in [0,t] \times P_2 \quad &\text{such that} \quad (u,s,y)  \in \Phi_{P_1}.
 \end{split} \end{align}
 In particular, if $u = [P_2]$, then $(u,t,v)  \in \Phi_{P_1/P_2}$ if and only if $v = [P_2]$.
 Also, $(u,s,[P_2]) \in \Phi_{P_1/P_2}$ implies $(u,t,[P_2]) \in \Phi_{P_1/P_2}$ for all $t \ge s$.

\begin{theo}\label{theoConleysem06} Assume that $\Phi$ is a complete semiflow relation on $X$  and that $(P_1, P_2)$ is an index
pair for $F$. The  relation $\Phi_{P_1/P_2}$ is a complete semiflow relation on  $P_1/P_2$. In particular, any solution path
for $\Phi_{P_1/P_2}$ on an interval $[0,t]$ extends to an element of $\S_+(\Phi_{P_1/P_2})$.\vspace{.25cm}

If $\xx \in \S_+(\Phi_{P_1/P_2})$, then one of the following holds.
\begin{itemize}
\item The path $\xx$ lies in $P_1 \setminus P_2$ in which case, $\xx \in \S_+(\Phi_{P_1})$ with $\xx(t) \in (P_1)_+$ for all $t \in \R_+$.

\item The path $\xx$ is constant at $[P_2]$, i.e. $\xx(t) = [P_2]$ for all $t \in \R_+$.

\item There exists $t^* > 0$ such that $\xx(t) = [P_2]$ for all $t \ge t^*$ and there exists $\yy \in \S([0,t^*],\Phi_{P_1})$
with $\xx(t) = \yy(t) \in P_1 \setminus P_2$ for all $t \in [0,t^*)$ and $\yy(t^*) \in P_2$.
\end{itemize}

Conversely, if $\yy$ is a maximal solution path for $\Phi_{P_1}$ defined on an interval $[0,s]$ with $s \le \infty$, then one of the following holds.
\begin{itemize}
\item ($s = \infty$): The path $\yy \in \S_+(\Phi_{P_1})$ in which case $\yy(t) \in (P_1)_+ \subset P_1 \setminus P_2$ for all $t \in \R_+$ and
$\xx(t) = \yy(t)$ defines $\xx \in \S_+(\Phi_{P_1/P_2})$.

\item ($s < \infty$): The path ends at $\yy(s) $ a terminal point for $P_1$. There exists $t^*$ with $0 \le t^* \le s$ such that
$\yy(t) \in P_1 \setminus P_2$ for $0 \le t < t^*$ and $\yy(t) \in P_2$ for $t^* \le t \le s$ and $\xx(t) = \yy(t)$ for $0 \le t < t^*$,
$\xx(t) = [P_2]$ for all $t \ge t^*$ defines $\xx \in \S_+(\Phi_{P_1/P_2})$.
\end{itemize}\vspace{.25cm}

The singleton $\{ [P_2] \}$ is
an attractor for $\Phi_{P_1/P_2}$ with dual repeller $(P_1)_+\subset P_1 \setminus P_2$. Furthermore, $\pi((P_1)_-) \cup \{ [P_2] \}$ is
an attractor for $\Phi_{P_1/P_2}$ with  dual repeller $\emptyset$.
If $P_1$ is $\Phi$ + invariant and so $P_2 = \emptyset$, then the isolated point $\{ [P_2] \}$ is also a repeller for $\Phi_{P_1/P_2}$ with dual
attractor $(P_1)_- = (P_1)_{\pm} \subset P_1 \setminus P_2$.
 \end{theo}

 \begin{proof} The Initial Value Condition for $\Phi_{P_1/P_2}$ is clear.

 If $u = [P_2]$ and $s \le t$, then $(u,t,v) \in \Phi_{P_1/P_2}$ if and only if $v = [P_2]$ and so if and only if
 $(u,s,[P_2]), ([P_2],t-s,v) \in \Phi_{P_1/P_2}$.

 If $u,v \in P_1 \setminus P_2$, then $(u,t,v) \in \Phi_{P_1/P_2}$ if and only if $(u,t,v) \in \Phi_{P_1}$ and so if and
 only if there exists $z \in P_1$ such that $(u,s,z),(z,t-s,v) \in \Phi_{P_1}$. Since $v \in P_1 \setminus P_2$, $z \in P_1 \setminus P_2$.
 Thus, $(u,t,v) \in \Phi_{P_1/P_2}$ if and only if there exists $w \in P_1/P_2$ such that $(u,s,w),(w,t-s,v) \in \Phi_{P_1/P_2}$.
 Note $v \in P_1 \setminus P_2$ implies $w \not= [P_2]$.

 If $u,v \in P_1 \setminus P_2$ and $(u,t,[P_2]) \in \Phi_{P_1/P_2}$, then there exist $(t_1,y) \in [0,t] \times P_2$ such that
  $(u,t_1,y) \in \Phi_{P_1}$. Assume that $t_1$ is the smallest such element of $[0,t]$. If $s < t_1$ then there exists $z \in P_1$
  such that $(u,s,z), (z,t_1-s,y) \in \Phi_{P_1}$. By minimality of $t_1$, $z \in P_1 \setminus P_2$. Hence,
  $(u,s,z),(z,t-s,[P_2]) \in \Phi_{P_1/P_2}$.  If $s \in [t_1,t]$, then $(u,s,[P_2]),([P_2],t-s,[P_2]) \in \Phi_{P_1/P_2}$.

  On the other hand, if $u \in P_1 \setminus P_2$ and $(u,s,z),(z,t-s,[P_2]) \in \Phi_{P_1/P_2}$, then $z = [P_2]$ implies
  $(u,t,[P_2])$. If, instead,  $z \in P_1 \setminus P_2$, then there exists $(t_1,y) \in [s,t] \times P_2$ such that
  $(z,t_1-s,y) \in \Phi_{P_1}$. Since $(u,s,z),(z,t_1-s,y) \in \Phi_{P_1}$, $(u,t_1,z) \in \Phi_{P_1}$ and so $(u,t,[P_2])  \in \Phi_{P_1/P_2}$.

  This completes the proof that $\Phi_{P_1/P_2}$ satisfies the Kolmogorov Condition. \vspace{.25cm}

  We now show that $\Phi_{P_1/P_2}$ is closed.

  Assume that $\{(u_n,t_n,v_n) \in \Phi_{P_1/P_2}$ converges to $(u,t,v)$. Note that if $u_n = [P_2]$ then $v_n = [P_2]$. Hence, if
  $u_n = [P_2]$ infinitely often then $u = v = [P_2]$ and so $(u,t,v) \in \Phi_{P_1/P_2}$.

  Now assume, after discarding finitely many terms, that $u_n \in P_1 \setminus P_2$ for all $n$.

  If  $v_n = [P_2]$ infinitely often we may go to a subsequence and assume $v_n = [P_2]$ for all $n$ and so $v = [P_2]$. Then there
  exist $(s_n,y_n) \in [0,t_n] \times P_2$ such that $(u_n,s_n,y_n) \in \Phi_{P_1}$. If $(s,y)$ is a limit point of the sequence $\{(s_n,y_n)$,
  then $(s,y) \in [0,t] \times P_2$ and $(u,s,y) \in \Phi_{P_1}$. Hence, $(u,t,v) = (u,t,[P_2]) \in \Phi_{P_1/P_2}$.

  We may now assume that $u_n, v_n \in P_1 \setminus P_2$ for all $n$ so that $(u_n,t_n,v_n) \in \Phi_{P_1}$ for all $n$. By going to
  a subsequence we may assume it converges to $(x,t,y) \in \Phi_{P_1}$.

  If $y \in P_1 \setminus P_2$, then $x \in P_1 \setminus P_2$ and so $(u,t,v) = (x,t,y) \in  \Phi_{P_1/P_2}$.

If $y \in P_2$ and $x \in P_1 \setminus P_2$, then $u = x, v = [P_2]$ and $(u,t,v) = (u,t,[P_2]) \in \Phi_{P_1/P_2}$.

If $x,y \in P_2$, then $(u,t,v) = ([P_2],t,[P_2]) \in \Phi_{P_1/P_2}$.

The shows that $\Phi_{P_1/P_2}$ is closed.\vspace{.25cm}

 The solution space results easily follow after one recalls two facts.  First, by completeness, if $x$ is a terminal point for
 $P_1$, then $x \in \d_{\Phi}(P_1) \subset P_2$. Second, if $\yy$ is a $\Phi_C$ solution path defined on the half-open interval
 $[0,s)$, then it extends uniquely to a solution path on $[0,s]$.\vspace{.5cm}

 The attractor results follow just as in Theorem \ref{theoConley11a}.

\end{proof}

\vspace{1cm}

\section{ \textbf{Hybrid Systems}}\label{hybrid}\vspace{.5cm}

On the compact metric space $X$ a \emph{hybrid dynamical system}\index{hybrid system}  $\H = (\Phi_C,G)$ is a pair where
$\Phi_C$ is the restriction of a semiflow relation $\Phi$ on $X$ to a nonempty closed subset
$C$ of $X$ and $G$ is a closed relation on $X$ with domain $D$.
The reverse system is $\overline{\H} = (\overline{\Phi}_C,G^{-1})$.

We call $\H = (\Phi_{C},G)$ a \emph{complete hybrid system} \index{hybrid system!complete} when
\begin{itemize}
\item  $\Phi_{C}$ is the
restriction of a complete semiflow relation $\Phi$ on $X$ to $C$.
% (see (\ref{eqsemi04aax})).

\item For the closed relation $G$ on $X$ the domain $D = Dom(G)$ satisfies $D \cup C = X$.
\end{itemize}

For complete hybrid system
the closed set $D$  contains $X \setminus C$  and so it contains
$\partial C  = C \cap \overline{X \setminus C}$. In particular, Proposition \ref{semiprop08aab} implies that for a complete hybrid system
$x \in D$ for any terminal point $x$ of $C$.

A point of $C$ can move continuously using the semiflow relation $\Phi_C$ and a point of $D$ can move with discrete jumps via the
relation $G$.  A point in the overlap $C \cap D$ can move either way.

The solution paths for a hybrid system $\H$ are parameterized by certain special subsets of $\R \times \Z$ called \emph{hybrid time domains}.
\index{hybrid time domains}  On $\R^* \times \Z^*$ we define certain relations and associated intervals.
Assume \\ $-\infty \le t_1 \le t_2 \le \infty, -\infty \le n_1 \le n_2 \le \infty $.

\begin{align}\label{eqhyb01}\begin{split}
(t_1,n_1) \preceq (t_2,n_2) \quad &\text{when} \quad  t_1 \le t_2 \ \ \text{and} \ \ n_1 \le n_2, \\
(t_1,n_1) \le_h (t_2,n_2) \quad &\text{when} \quad  t_1 \le t_2 \ \ \text{and} \ \ n_1 = n_2 \not= \pm \infty, \\
(t_1,n_1) \le_v (t_2,n_2) \quad &\text{when} \quad  t_1 = t_2 \not= \pm \infty \ \ \text{and} \ \ n_1 < n_2, \\
(t_1,n_1) \le (t_2,n_2) \quad &\text{when either} \quad (t_1,n_1) \le_h (t_2,n_2) \ \ \text{or} \ \ (t_1,n_1) \le_v (t_2,n_2).
\end{split}\end{align}
Observe that the relations $\preceq, \le_h, \le_v$ are transitive while $\le$ is not.

When $(t_1,n_1) \preceq (t_2,n_2) $ we write
\begin{align}\label{eqhyb01a} \begin{split}
[(t_1,n_1),(t_2,n_2)] \ = \ &[t_1,t_2] \times [n_1,n_2] \ = \\
 \{ (t,n) \in \R \times \Z : (t_1,n_1)\  &\preceq \ (t,n) \ \preceq \ (t_2,n_2) \}.
\end{split}\end{align}
with length $(t_2 - t_1) + (n_2 - n_1)$.

Thus, when $(t_1,n_1) \le_h (t_2,n_2), \ [(t_1,n_1),(t_2,n_2)]$ is the connected
\emph{horizontal interval}\index{horizontal interval}\index{interval!horizontal}
$[t_1,t_2] \times \{ n_1 \}$. When $(t_1,n_1) \le_v (t_2,n_2)$, \\ $[(t_1,n_1),(t_2,n_2)]$ is the discrete
\emph{vertical interval}\index{vertical interval}\index{interval!vertical}
$\{ t_1 \} \times \{ [n_1,n_2] \}$. So when $(t_1,n_1) \le (t_2,n_2)$ the relation $\preceq$ is a total order on $[(t_1,n_1),(t_2,n_2)]$.
Notice that a horizontal interval may be trivial, i.e. a singleton with length $0$, but a vertical interval always has length at least $1$.

With $[i_1,i_2]$ an interval in $\Z$ a hybrid time interval is the union $E \ = \ \bigcup_{i \in [i_1,i_2-1]} \ [(t_i,n_i),(t_{i+1},n_{i+1})]$
with $\{ (t_i,n_i) : i \in [i_1,i_2] \}$ a finite, infinite or bi-infinite sequence such that $(t_i,n_i) \le (t_{i+1},n_{i+1})$ for
all $i \in [i_1,i_2]$. When $i_1 \in \Z$ and $(t_{i_1},n_{i_1}) \in \R \times \Z$, then $(t_{i_1},n_{i_1})$ is the \emph{left end-point}
\index{end-point!left} of $E$ and when $i_2 \in \Z$ and $(t_{i_2},n_{i_2}) \in \R \times \Z$, then $(t_{i_2},n_{i_2})$ is the \emph{right end-point}
\index{end-point!right} of $E$. The time interval $E$ is compact when it has both a left and right end-point in which case its
length $(t_{i_2} - t_{i_1}) + (n_{i_2} - n_{i_1})$ is the sum of the lengths of the horizontal and vertical pieces $[(t_i,n_i),(t_{i+1},n_{i+1})]$.
The sequence $\{ (t_i,n_i) \} $ is not uniquely defined by the set $E$, but by combining the successive horizontal pieces and successive vertical
pieces, we obtain the unique \emph{simple sequence}\index{simple sequence} for $E$  with no consecutive $\le_h$ relations and no consecutive $\le_v$
relations. In addition, if the length of $E$ is positive, then all of the horizontal pieces are nontrivial.

If $(t,n) \in E$, then $E_1 = E \cap [(-\infty,-\infty),(t,n)]$ and $E_2 = E \cap [(t,n),(\infty,\infty)]$ are non-overlapping hybrid
time intervals with intersection the common end-point $(t,n)$.

For $(t_1,n_1) \preceq (t_2,n_2) \in \R \times \Z$ we will write $E = [[(t_1,n_1),(t_2,n_2)]]$
when $E$ is a hybrid time interval with left endpoint $(t_1,n_1)$ and
right endpoint $(t_2,n_2)$ and so it has length $(t_2 - t_1) + (n_2 - n_1)$.  In contrast with intervals in $\R$ and $\Z$,
 such a time interval is not uniquely determined by its endpoints. When $t_1 < t_2$ and $n_1 < n_2$ there are multiple alternative
 ways of getting from $(t_1,n_1)$ to $(t_2,n_2)$.

 For $(t,n) \in \R \times \Z$ we will write $E = [[(t,n),\infty]]$ for a time interval of infinite length with left endpoint $(t,n)$. This
 includes among other possibilities $[t,\infty] \times \{n \}$ and $\{ t \} \times [n,\infty]$, a single infinite horizontal interval and
 a single infinite vertical interval, respectively.  Similarly, we will write $E = [[-\infty,(t,n)]]$ for a time interval of infinite length
 with right endpoint $(t,n)$.

A \emph{hybrid solution path}\index{solution path!hybrid} for $\H = (\Phi_C,G)$ is a function $\xx : E \to X$ such that $E$ is hybrid time interval and
for $(t_1,n_1), (t_2,n_2) \in E$
\begin{itemize}
\item If $(t_1,n_1) \le_h (t_2,n_2)$, or, equivalently, if $[(t_1,n_1), (t_2,n_2)]$ is a horizontal portion of $E$, then
$t \mapsto \xx(t,n_1)$ is a $\Phi_C$ solution path defined on $[t_1,t_2]$.

\item If $(t_1,n_1) \le_v (t_2,n_2)$, or, equivalently, if $[(t_1,n_1), (t_2,n_2)]$ is a vertical portion of $E$, then
$n \mapsto \xx(t_1,n)$ is a $G$ solution path defined on $[n_1,n_2]$.
\end{itemize}

If $\xx : E \to X$ is a hybrid solution path with a simple sequence $\{ (t_i,n_i) : i \in [i_1,i_2] \}$ and $(s,m) \in \R \times \Z$,
then the translated path $Trl_{(s,m)}\xx$ on $E - (s,m)$ with simple sequence $\{ (t_i - s,n_i - m) : i \in [i_1,i_2] \}$ is defined by
$(Trl_{(s,m)}\xx)(t,n) = \xx(t + s, n + m)$.

If $\xx_1 : E_1 \to X, \xx_2 : E_2 \to X$ are hybrid solution paths such that $(s,m)$ is a right end-point for $E_1$ and a left end-point
for $E_2$, then $E_1 \cap E_2 = \{(s,m) \}$ and $E = E_1 \cup E_2$ is a hybrid time interval.  If $\xx_1(s,m) = \xx_2(s,m)$, then
the \emph{composition}\index{composition}  $\xx = \xx_1 \oplus \xx_2$ is the hybrid solution path such that $\xx|E_i = \xx_1$ and
$\xx|E_2 = \xx_2$.

If $\xx : E \to X$ is a hybrid solution path for $(\Phi_C,G)$, then $-E$ is a hybrid time interval and the reverse of $\xx$,
$\bar \xx : -E \to X$ given by $\bar \xx (t,n) = \xx(-t,-n)$, is a hybrid solution path for $(\overline{\Phi_C},G^{-1})$.

We let $\S([[(t_1,n_1),(t_2,n_2)]],\H)$, or just $\S([[(t_1,n_1),(t_2,n_2)]])$ when $\H$ is understood, to be the set of
all hybrid solution paths defined on some time interval of the form  $[[(t_1,n_1),(t_2,n_2)]]$. We let $\S_+(\H)$, or just
$\S_+$, to be the set of all hybrid solution paths defined on some infinite time interval of the form $[[(0,0),\infty]]$ and
we let $\S_-(\H)$, or just
$\S_-$, to be the set of all hybrid solution paths defined on some infinite time interval of the form $[[-\infty,(0,0)]]$.
Clearly, $\xx \in \S_-(\H)$ if and only if $\bar{\xx} \in \S_+(\overline{\H})$. \vspace{.5cm}

With $(\phi_C)^I, (\phi_C)^J$ defined using equation (\ref{eqsemi05aax}) for $I = [0,1], J = [1,2]$ we define the
\emph{Associated Relation for the Hybrid System}$(\Phi_C,G)$:
\begin{equation}\label{eqhyb01aaa}
H(\H) \quad  \ =_{def} \  \quad ((\phi_C)^I \circ G \circ (\phi_C)^I) \cup (\phi_C)^J.
\end{equation}
Recall that $(\phi_C)^I$ is reflexive on $X$, while $Dom((\phi_C)^J) \subset C$. In particular, it follows that $G \subset H$.
We will just write $H$ for $H(\H)$ when
the hybrid system is understood.

\begin{prop}\label{hybprop00aaa} If $\H$ is a complete hybrid system, then domain of $H$ equals $X$. \end{prop}

\begin{proof}  If  $x \in D$  there exists $y \in X$ with $(x,y) \in G$. Since $(x,x), (y,y) \in (\phi_C)^I$,
we have $(x,y) \in (\phi_C)^I \circ G \circ (\phi_C)^I$.

Now let $x \in C \supset X \setminus D$.

If there exists $\xx \in \S_C([0,t])$ with $t \ge 1$ and $\xx( 0) = x$, then with $y = \xx( 1)$ we have $(x,y) \in (\phi_C)^J$.
If $x \in C$ but no such $\xx$ exists, then by Proposition \ref{semiprop07} there exists $\xx \in \S_C([0,t])$ such that
$0 \le t < 1$, $\xx( 0) = x$ and $\xx( t) = y_1$ is a terminal point. Since $t < 1$, $(x,y_1) \in (\phi_C)^I$. By
Proposition \ref{semiprop08aab}, $y_1 \in \partial C \subset Dom (G)$ and so there exists $y \in X$ such that $(y_1,y) \in G$.
Since $(x,y_1), (y,y) \in (\phi_C)^I$ we have $(x,y) \in  (\phi_C)^I \circ G \circ (\phi_C)^I$.

\end{proof}

If $\yy \in \S([0,k], H)$ with $ k \in \Z_+$ and $\xx \in \S([[(t_1,n_1),(t_2,n_2)]],\H)$, we say that $\xx$ \emph{spans} \index{spans}
$\yy$ when there is a sequence $(s_0,m_0) \prec (s_1,m_1) \prec \dots \prec (s_k,m_k)$ in $[[(t_1,n_1),(t_2,n_2)]]$ with $(s_0,m_0) = (t_1,n_1),
(s_k,m_k) = (t_2,n_2)$ and $\xx(s_i,m_i) = \yy(i)$ for $i = 0, 1, \dots, k$. Similarly, if $\yy \in \S_+(H)$  and $\xx \in \S_+(\H)$,
we say that $\xx$ \emph{spans}
$\yy$ when there is a sequence $\{ (s_i,m_i) \in \R_+ \times \Z_+ \} $ with $(s_0,m_0) = (0,0)$ such that $ (s_i,m_i) \prec  (s_{i+1},m_{i+1})$
 and $\xx(s_i,m_i) = \yy(i)$ for all $i \in \Z_+$.

\begin{theo}\label{hybtheo01solpath} Let $\H = (\Phi_C,G)$ be a hybrid system on $X$ with associated relation $H$ and let $x, y \in X$.
\begin{enumerate}
\item[(a)] If $(x,y) \in (\phi_C)^I$, then there exists a hybrid solution path $\xx : [0, \ell]\times \{ 0 \} \to X$ from $x$ to $y$
with length $\ell$ satisfying $0 \le \ell \le 1$.  Conversely, if there exists a solution path $\xx : E \to X$ from $x$ to $y$ with
length $\ell$ satisfying $0 \le \ell < 1$, then $(x,y) \in (\phi_C)^I$.

\item[(b)] If for some positive integer $k$, $(x,y) \in H^k$, then there exists a hybrid solution path $\xx : E \to X$ from $x$ to $y$
with length $\ell$ satisfying $k \le \ell \le 3k$. In detail, if $\yy \in \S([0,k], H)$ with $\yy(0) = x, \yy(k) = y$ then there
exists a hybrid solution path $\xx : E \to X$ which spans $\yy$ and with length $\ell$ satisfying $k \le \ell \le 3k$.

\item[(c)] If there exists a solution path $\xx : E \to X$ from $x$ to $y$ with
length  $\ell \ge 1$, then there exists a positive integer $k$ satisfying $\frac{\ell}{3} \le k \le \ell$ such that  $(x,y) \in H^k$. In detail,
there exists $\yy \in \S([0,k], H)$ such that $\xx$ spans $\yy$ and with $k$ a positive integer satisfying $\frac{\ell}{3} \le k \le \ell$.

\item[(d)] If $\yy \in \S_+(H)$, then there exists $\xx \in \S_+(\H)$ such that $\xx$ spans $\yy$. Conversely, if $\xx \in \S_+(\H)$,
then there exists   $\yy \in \S_+(H)$  such that $\xx$ spans $\yy$.
\end{enumerate}\end{theo}

\begin{proof} (a) $(x,y) \in (\phi_C)^I$ if and only if there exists a solution path $\yy \in \S_C([0,\ell],\Phi)$ with $\yy(0) = x, \yy(\ell) = y$ and
$0 \le \ell \le 1$. Let $\xx(t,0) = \yy(t)$.  Conversely, since the length of any vertical interval is a positive integer, it follows that if
$\xx$ has length $\ell < 1$, its simple sequence consists of a single horizonal interval of length $\ell$. Clearly, $(x,y) \in (\phi_C)^I$.

(b) If $(x,y) \in (\phi_C)^J$ then there is a single horizontal solution path connecting $x$ to $y$ with length $\ell \in J$.
If $(x,y) \in (\phi_C)^I \circ G \circ (\phi_C)^I$, then with
$E = ([0,\ell_1] \times \{ 0 \}) \cup  \{\ell_1\} \times [0,1] \cup ([\ell_1,\ell_1+\ell_2] \times \{ 1\})$,
with suitable choice of $\ell_1, \ell_2 \in I$, we can build a hybrid solution path from $x$ to $y$ with length
$\ell = \ell_1 + 1 + \ell_2 \in [1,3]$. Composing we see that if $\yy \in \S([0,k], H)$ then there is a solution path $\xx$ with length
between $k$ and $3k$ which spans $\yy$.

(c) Assume that $\xx : E \to X$ is a hybrid solution path with length $\ell$ such that $1 \le \ell$ beginning at $x$ and terminating at $y$.

The simple sequence for $E$ consists of
alternating horizontal intervals and vertical intervals with length $\ell_0,j_1,\ell_1,j_2, \dots, j_n,\ell_n$. The initial and
final horizontal intervals can be trivial so that $\ell_0 $ and $\ell_n$ may equal $0$. The other horizontal lengths $\ell_i$
 are all positive reals and the vertical lengths $j_i$ are all all positive integers.

 If there are no vertical intervals, then there is a single horizontal interval of length $\ell = \ell_0 \ge 1$. Let $k$ be the largest integer
 such that $\ell/k \ge 1$. By assumption $\ell/1 \ge 1$.  Now if $\ell/k \ge 2$ then $\ell \ge 2k \ge k + 1$ and so $\ell/(k+1) \ge 1$.
 As this contradicts the choice of $k$, we have $2 > \ell/k \ge 1$. So we can subdivide the horizontal interval into $k$ intervals of
 length $\ell/k$. The endpoints of a horizontal solution path of length between $1$ and $2$ are related by $(\phi_C)^J$.  Hence, we obtain
 $\yy \in \S([0,k],(\phi_C)^J)$ with $\ell \ge k > \ell/2 > \ell/3$ such that $\xx$ spans $\yy$.

 Now assume there is at least one vertical interval.  Let $k_i = [\ell_i/2]$ for $i = 0, \dots, n$. There exist $a_i, b_i \in I$ such that
 $\ell_0 = 2k_0 + a_0, \ell_1 = b_1 + 2k_1 + a_1, \dots, \ell_{n-1} = b_{n-1} + 2k_i + a_{n-1}, \ell_n = b_n + 2k_n$. For $i = 1, \dots n$
the piece consisting of the $a_{i-1}$ horizontal step followed by $j_i$ vertical jumps and then the $b_i$ horizontal step
 spans an element of $\S([0,j_i],(\phi_C)^I \circ G \circ (\phi_C)^I)$. The portion consisting of the horizontal piece of length
 $2k_i$ spans an element of $\S([0,k_i],(\phi_C)^J)$. Note that this includes the possibility that $k_i = 0$ when $\ell_i < 2$.

 Translating and composing we obtain $\yy \in \S([0,k],H)$
 such that $\xx$ spans $\yy$  with $k = k_0 + j_1 + k_1 + \dots + j_n + k_n$,
 while $\ell = \ell_0 + j_1 + \dots + j_n + \ell_n = 2k_0 + ( a_0 + j_1 + b_1) + 2k_1 + \dots + (a_{n-1} + j_n + b_n)+ 2k_n$. So $\ell \ge k $.
 With $a, b \in I$, and $j$ a positive integer $j \ge (j+2)/3 \ge (a + j + b)/3$. Also $k_i \ge 2k_i/3 $. Hence, $k \ge \ell/3$.

 (d) Cut the infinite paths into an infinite sequence of finite paths, apply (b) and (c) to each piece, then translate and compose to obtain
 an increasing sequence of paths. The required $\xx$ or $\yy$ is then the union.

\end{proof}

As suggested by Andrew Teel, there is an alternative way of obtaining a version of the associated relation.
We define an analogue of the semiflow relation for the entire hybrid
system.

\begin{df}\label{hybdefTeel01}  For a hybrid system $\H = (\Phi_C,G)$ we define the subset $\Psi(\H)$ (or just $\Psi$ when $\H$ is understood) by
$$\Psi \subset X \times \R_+ \times \Z_+ \times X$$ such that $(x,(t,n),y) \in \Psi$ when
there exists a hybrid solution path $\xx : [[(0,0),(t,n)]] \to X$ with
$\xx(0,0) = x$ and $\xx(t,n) = y$. Let $\psi^{(t,n)} = \{ (x,y) : (x,(t,n),y) \in \Psi \}$. \end{df} \vspace{.25cm}

\begin{prop}\label{hybpropTeel02} The relation $\Psi(\H)$ is a closed subset of \\ $X \times \R_+ \times \Z_+ \times X$
with $\psi^{(0,0)} = 1_{C \cup D}$ and
for all $(t_1,n_1), (t_2,n_2) \in \R_+ \times \Z_+, \quad \psi^{(t_1,n_1)} \circ \psi^{(t_2,n_2)} \subset \psi^{(t_1+t_2,n_1+n_2)}$. \end{prop}

\begin{proof}  The analogue of the Initial Value Condition is obvious.  The analogue of the Weak Kolmogorov Condition
follows by using translation and composition of the
hybrid solution paths.  Proving closure requires a bit more work.

Let $\{ (x_i,(t_i,n_i),y_i) \} $ be a sequence in $\Psi$ converging to the
point $\{(x,(t,n),y)\}$.  Since eventually $n_i = n$ we may assume $n_i = n$ for all $i$ and
prove that $(x,(t,n),y) \in \Psi$ by induction on $n$. Let
$\xx_i : E_i = [[(0,0),(t_i,n)]] \to X$ be a hybrid solution path connecting $x_i$ to $y_i$.

If $n = 0$, then $\{ (x_i,t_i,y_i) \} $ is a sequence in $\Phi_C$.  As the latter is closed, $(x,t,y) \in \Phi_C$ and so $(x,(t,0),y) \in \Psi$.

Now assume that $n \ge 1$ and let $s_i$ be the maximum value such that the horizontal interval
$[0,s_i] \times \{ 0 \} \subset E_i$. So in $\xx_i$ a jump occurs at
$(s_i,0)$. Hence, $(s_i,1) \in E_i$ and if $u_i = \xx_i(s_i,0), v_i = \xx_i(s_i,1)$,
then $(u_i,v_i) \in G$.  By going to a subsequence we may assume that
$\{ s_i \}, \{  u_i \}$ and $\{  v_i \}$ converge to $s, u$ and $v$.
From in the $n = 0$ case, we have $(x, (s,0), u) \in \Psi$. Since $G$ is closed $(u,v)  \in G$
and so $(u,(0,1),v) \in \Psi$. By truncating $\xx_i$ at $(s_i,1)$ we obtain a sequence
$\yy_i : [[(s_i,1),(t_i,n)]] \to X$ which translates to a sequence
$\xx_i' :[[(0,0),(t_i-s_i,n-1)]] \to X$ which shows that $(v_i,(t_i-s_i,n-1),y_i) \in \Psi$.
By inductive hypothesis, the limit $(v,(t-s,n-1),y) \in \Psi$.

From the analogue of the Weak Kolmogorov Property, we see that $(x,(t,n),y) \in \Psi$.

Thus, $\Psi$ is closed.

\end{proof}

Now define

\begin{equation}\label{hybeqTeel03}\begin{split}
\tilde H \ =_{def} \ \psi^{[1,3]} \ = \  \{ (x,y): (x,(t,n),y) \in \Psi \\ \ \text{for some} \ (t,n) \ \text{with} \ \ 1 \le t+n \le 3 \}. \hspace{.5cm}
\end{split}\end{equation}
That is, $\psi^{[1,3]}$ is the union of the $\psi^{(t,n)}$'s with $(t,n)$ varying in the
compact subset $\{ (t,n) : 1 \le t+n \le 3 \}$ of $\R_+ \times \Z_+$.
It follows that $\tilde H$ is a closed relation on $X$.

We have
\begin{equation}\label{hybeqTeel04}
H \ \subset \ \tilde H \ \subset \ H \cup H^2 \cup H^3. \hspace{2cm}
\end{equation}
The first inclusion is obvious and the second follows from Theorem \ref{hybtheo01solpath}(c).

It follows that for $\A = \O, \G, \CC, \quad \A H \ = \ \A \tilde H$.

While $\tilde H$ is perhaps more intuitive, we will see below that $H$ is easier to work with.

\begin{prop}\label{hybprop01aaa}   Let $H$ be the Associated Relation.

(a) The following hold:
\begin{equation}\label{eqhyb02aaaa}
G \ \subset \ G \circ (\phi_C)^I, \ (\phi_C)^I \circ G \  \subset \ H
\end{equation}
\begin{equation}\label{eqhyb02baaa}
H \circ (\phi_C)^I \cup (\phi_C)^I \circ H \ \subset \ H \cup H^2. \hspace{1cm}
\end{equation}
\begin{equation}\label{eqhyb02caaa}
(\phi_C)^I \cup \O H \ = \ \O ((\phi_C)^I \cup H) = \O((\phi_C)^I \cup G).
\end{equation}

(b) For $\A = \O, \G, \CC$
\begin{equation}\label{eqhyb02daaa}
\A H \circ (\phi_C)^I \ = \ \A H \ = \  (\phi_C)^I \circ \A H. \hspace{1cm}
\end{equation}
Each $(\phi_C)^I \cup \A H$ is a transitive relation.

(c) Although $(\phi_C)^I \cup \CC H$ is a closed, transitive relation,  it is usually a proper subset of $\CC( (\phi_C)^I \cup H)$.
On the other hand,
\begin{equation}\label{eqhyb02eaaa}
(\phi_C)^I \cup \G H \ = \ \G ((\phi_C)^I \cup H) \ = \ \G ((\phi_C)^I \cup G). \hspace{2cm}
\end{equation}

(d) For $\A = \O, \G, \CC$, if $A \subset X$ is a closed $\A H $ + invariant set,
then $(\phi_C)^I(A)$ is a closed $(\phi_C)^I \cup \A H$ invariant set and so is
$\A H$ + invariant. Furthermore, $\A H(A) = H(A)$. In particular, $A$ is $\A H$ invariant if and only if it is $H$ invariant.

\end{prop}

\begin{proof} (a) The inclusions \ref{eqhyb02aaaa} follow because $(\phi_C)^I$ is reflexive.

The inclusions (\ref{eqhyb02baaa}) follow from (\ref{eqsemi07aa}).

By \ref{eqhyb02aaaa}, we have $(\phi_C)^I \cup G \subset (\phi_C)^I \cup H \subset (\phi_C)^I \cup \O (H) \subset \O(\phi^I_C \cup G)$.
Apply the operator $\O$ and observe that $\O \O = \O$. From \ref{eqhyb02baaa} it follows that $(\phi_C)^I \cup \O H$ is transitive and
so equals $\O((\phi_C)^I \cup H)$.

(b) By (\ref{eqrel09c}) $H \cup ((\A H) \circ H)  =  \A H  = H \cup (H \circ (\A H))$ for $\A = \O, \G, \CC$. From (\ref{eqhyb02baaa})
and transitivity of $\A H$ it follows that
$ ((\phi_C)^I \circ \A H) $ and $ (\A H \circ (\phi_C)^I) $ are subsets of $ \A H$. The reverse inclusions follow because $(\phi_C)^I$ is reflexive.
From (\ref{eqsemi07aa}) it follows that $(\phi_C)^I \circ (\phi_C)^I \subset (\phi_C)^I \cup H \subset (\phi_C)^I \cup \A H$.
Together with (\ref{eqhyb02daaa}) this implies that
$\phi^I \cup \A H$ is transitive.

(c) We clearly have $(\phi_C)^I \cup  H \subset (\phi_C)^I \cup \G H \subset \G ((\phi_C)^I \cup H)$.
Since the closed relation $(\phi_C)^I \cup \G H$ is transitive by (b),
it contains $\G ((\phi_C)^I \cup H)$. Since $(\phi_C)^I \cup G \subset (\phi_C)^I \cup H \subset ((\phi_C)^I \cup G)^3$,
it follows that $\G((\phi_C)^I \cup G) = \G((\phi_C)^I \cup H)$.

In a connected space $X$, $\CC 1_X = X \times X$ and so if $X$ is connected, $\CC( (\phi_C)^I \cup H) = X \times X$ which is usually
larger than $(\phi_C)^I \cup \CC H$.

(d) Since $(\phi_C)^I$ is reflexive and closed, $(\phi_C)^I(A)$ is closed and contains $A$. From (\ref{eqhyb02daaa}) and (\ref{eqhyb02baaa}) we have
$((\phi_C)^I \cup \A H) \circ (\phi_C)^I = (\phi_C)^I \cup \A H$ and so
$((\phi_C)^I \cup \A H)((\phi_C)^I(A)) = ((\phi_C)^I \cup \A H)(A) = (\phi_C)^I(A)$ since $A$ is $\A H$ + invariant.

 By (\ref{eqrel09c}) $\A H  = H \cup (H \circ (\A H))$. Since $\A H(A) \subset A$, it follows that
$ \A H(A) = H(A)$.

\end{proof}

\vspace{.5cm}

\begin{cor}\label{hybcor02aaa} For $\A = \O, \G, \CC$, if $(x,y), (y,x) \in (\phi_C)^I \cup \A H$ and $x \not= y$, then
$(x,y), (y,x), (x,x), (y,y) \in \A H$. \end{cor}

\begin{proof}  If $(x,y), (y,x) \in \A H$, then the result follows from transitivity of $\A H$.  So we may assume $(x,y) \in (\phi_C)^I$.

Case 1: If $(y,x) \in \A H$, then $(x,x) \in \A H \circ (\phi_C)^I$ and $(y,y) \in (\phi_C)^I \circ \A H$. By (\ref{eqhyb02daaa})
$\A H = \A H \circ (\phi_C)^I = (\phi_C)^I \circ \A H$. Then $(x,y) \in (\phi_C)^I \circ \A H = \A H$.

Case 2: $(x,y), (y,x) \in (\phi_C)^I$. This means there exist $0 < t_1, t_2 \le 1$ and solution paths $\xx_1 \in \S_C([0,t_1)], \xx_2 \in \S_C([0,t_2)]$
with $\xx_1(0) = \xx_2(t_2) = x, \xx_1(t_1) = \xx_2(0) = y$. Concatenating we can obtain a $t_1 + t_2$ periodic solution path
$\xx : [0,\infty) \to C$, with $x = \xx( n(t_1 + t_2)), y = \xx( t_1 + n(t_1 + t_2))$ for all  $n \in \Z_+$. Since the path is periodic,
it follows that $\t(x) = \t(y) = \infty$. Since $t_1 + t_2 > 0$ and $ (x,n(t_1 + t_2),x) \in \Phi_C$ for every $n \in \Z_+$
we see that $(x,x) \in \O (\phi_C)^J \subset \A H$. Similarly, $(y,y) \in \O (\phi_C)^J$. $(x,t_1 + n(t_1 + t_2),y) $ and so $(x,y) \in \O (\phi_C)^J$
and similarly for $(y,x)$.  In fact, any pair of points on a periodic solution lies in $\O (\phi_C)^J \subset \A H$.

\end{proof}

This is a good moment to explain why we are doing this work to distinguish $H$ and $(\phi_C)^I \cup H$.  Why not use the latter closed relation
or just $(\phi_C)^I \cup G$ and
be done with it? The answer is that because $(\phi_C)^I$ is reflexive, $|(\phi_C)^I \cup G| = X$ and so every point is a fixed point of $(\phi_C)^I \cup G$, whereas we
want to observe the recurrence due to $H$.

\begin{cor}\label{hybcor03aaa} For $\A =  \G, \CC$, the closed equivalence relation on $X$ associated with $(\phi_C)^I \cup \A H$ satisfies
\begin{equation}\label{eqhyb03aaa}
((\phi_C)^I \cup \A H) \cap ((\phi_C)^I \cup \A H)^{-1} \ = \ 1_X \cup (\A H \cap \A H^{-1}).
\end{equation}
In particular, an equivalence class of $((\phi_C)^I \cup \A H) \cap ((\phi_C)^I \cup \A H)^{-1}$ containing more than one point
is an equivalence class of $\A H \cap \A H^{-1}$
contained in $|\A H|$. \end{cor}\vspace{.5cm}

 \begin{proof}  This is immediate from Corollary \ref{hybcor02aaa}.

 \end{proof}

\begin{theo}\label{hybtheo05aaa} For $x, y \in X$ there exists   a hybrid solution path, $\xx : E \to X$, which begins at $x$ and
terminates at $y$ if and only if $(x,y) \in (\phi_C)^I \cup \O H$. There exists  a hybrid solution path,
$\xx : E \to X$ with length of $E$ greater than or equal to $1$, which begins at $x$ and
terminates at $y$ if and only if $(x,y) \in  \O H$.
%
%In particular,  $A \subset X$ is an $(\phi_C)^I \cup \O H$ + invariant subset if and only if any hybrid solution
%path which begins in $A$ terminates in $A$.
\end{theo}

\begin{proof}  This follows directly from Theorem \ref{hybtheo01solpath}.

\end{proof}

A subset  $A$ of $X$ is called \emph{+ invariant}\index{+ invariant} for $\H$ when any hybrid solution path which begins at a point
of $A$ remains in $A$.  That is, if $\xx : E = [[(t_1,n_1),[t_2,n_2)]] \to X$ is a hybrid solution path with $\xx(t_1,n_1) \in A$, then
$\xx(E) \subset A$. The subset is \emph{invariant}\index{invariant} for $(\Phi_C,G)$ when, in addition, for every point $x \in A$, there
exists  $\xx : E = [[(t_1,n_1),[t_2,n_2)]] \to X$  a hybrid solution path of length at least $1$ with $\xx(t_2,n_2) = x$ and $\xx(E) \subset A$.

\begin{prop}\label{hybpropinvar} Let $\H = (\Phi_C,G)$ be a hybrid dynamical system on $X$.
\begin{enumerate}
\item[(a)] For $A \subset X$ the following are equivalent:

\begin{itemize}
\item[(i)] $A$ is + invariant for $\H$.

\item[(ii)] $A$ is + invariant for $\Phi_C$ and is + invariant for $G$.

\item[(iii)] $A$ is  invariant for $(\phi_C)^I \cup G$, i.e. $A$ is  invariant for $(\phi_C)^I$ and + invariant for  $G$.

\item[(iv)] $A$ is  invariant for $(\phi_C)^I \cup H$, i.e. $A$ is  invariant for $(\phi_C)^I$ and + invariant for  $H$.

\item[(v)] $A$ is  invariant for $(\phi_C)^I \cup \O H$.
\end{itemize}

\item[(b)] For $A \subset X$ the following are equivalent:

\begin{itemize}

\item[(i)] $A$ is invariant for $\H$.

\item[(ii)]  $A$ is  invariant for $H$.

\item[(iii)] $A$ is + invariant for $\H$ and for every point $x \in A$, there
exists  $\xx : E = [[-\infty,[t_2,n_2)]] \to X$  an infinite hybrid solution path with $\xx(t_2,n_2) = x$ and $\xx(E) \subset A$.
\end{itemize}

When these conditions hold, $A = (\phi_C)^I(A)$.

\item[(c)] If $A \subset X$ is + invariant for $H$, then $(\phi_C)^I(A)$  is + invariant for $\H$ with
$H((\phi_C)^I(A)) = H(A) \subset A$. Furthermore, $A_{\infty} \ = \ \bigcap_{k=1}^{\infty} \ H^k(A) \ = \bigcap_{k=1}^{\infty} \ H^k((\phi_C)^I(A))$
is $\H$ invariant and contains any $H$ invariant subset of $(\phi_C)^I(A)$.

In particular, if $U \subset X$ is inward for $H$, then $(\phi_C)^I(U)$ is inward for $H$ and the $\H$ invariant set
$A \ = \ \bigcap_{k=1}^{\infty} \ H^k(U) \ = \bigcap_{k=1}^{\infty} \ H^k((\phi_C)^I(U))$ is the associated attractor for $H$.

\end{enumerate}
\end{prop}

\begin{proof} (a) A set is + invariant for a relation $F$ if and only if it is + invariant for $\O F$. So (iii), (iv) and (v) are equivalent by
(\ref{eqhyb02caaa}). The equivalence of (i) and (iv) follows from Theorem \ref{hybtheo05aaa}. The equivalence of (ii) and (iii) follows from
Proposition \ref{semiprop05a}.

(b) (iii) $\Rightarrow$ (i) is obvious.

(i) $\Rightarrow$ (ii) follows from the second part of Theorem \ref{hybtheo05aaa}.

(ii) $\Rightarrow$ (iii): In any case, $A \subset ((\phi_C)^I)(A)$. If $H(A) \subset A$ and so $H^2(A) \subset H(A)$, then
\begin{equation}\label{eqhyb03bbb}
H(A) \subset (((\phi_C)^I)H(A))\cup H(((\phi_C)^I)(A)) \subset (H \cup H^2)(A) = H(A)
\end{equation}
by (\ref{eqhyb02baaa}).  Thus, if $H(A) = A$, then $A = ((\phi_C)^I)(A)$.

That $A$ is + invariant for $\H$ now follows from (a).  If $x \in A$, then $H$ invariance
implies there exists $\yy \in \S([-\infty,0],H)$  with $\yy(-i) \in A$ for all $i$ and $\yy(0) = x$.
Now we can use Theorem \ref{hybtheo01solpath}(c) (applied to $\overline{\H}$) to construct $\xx \in \S([[-\infty,(0,0)]],\H)$ which spans $\yy$.
Since $A$ is + invariant for $\H$ and $\xx(s_i,m_i) = \yy(i) \in A$, it follows that $\xx(s,m) \in A$ for all $(s,m) \in [[-\infty,(0,0)]]$
with $(s_i,m_i) \prec (s,m)$. Letting $i$ tend to infinity we see that $\xx(s,m) \in A$ for all $(s,m) \in [[-\infty,(0,0)]]$.

(c) When $A$ is $H$ + invariant, (\ref{eqhyb03bbb}) implies $H(((\phi_C)^I)(A)) = H(A)$. So for $k \in \Z_+$, $H^k(((\phi_C)^I)(A)) = H^k(A)$
and the results follow from Corollary \ref{relcor00a} with $F = H$. Note that $H$ invariance is the same as $\H$ invariance by (b).

\end{proof}

Motivated by the above results we will call $A$ an \emph{ attractor}\index{attractor} (or \emph{repeller})\index{repeller}
 for $\H$ when it is an attractor (resp. a repeller) for $H$.

 Thus, $A$ is $\H$ + invariant when $G(A) \subset A$ and the family of subsets $\{ (\phi_C)^t(A) : t \in \R_+ \}$ is decreasing in $t$.
Such a set $A$ is then inward for $H$ if and only if $G(A) \subset \subset A$ and $ (\phi_C)^1(A) \subset \subset A$ which then
implies $ (\phi_C)^t(A) \subset \subset A$ for all $t \ge 1$.

We call $U$ \emph{inward for $\H$}\index{subset!inward}\index{inward} when $G(A) \subset \subset A$
and  $ (\phi_C)^t(A) \subset \subset A$ for all $t > 0$.
That is, $U$ is inward for $G$ and inward for $\Phi_C$.

 As with semiflow relations, we will use Lyapunov functions to construct $\H$ inward neighborhoods for $\H$ attractors.

\begin{theo}\label{reltheo02semihyb} Let $\H = (\Phi_C,G)$ be a hybrid dynamical system on $X$. Let $\A = \G$ or $\CC$.

(a) Assume that $A, B$ are disjoint, closed subsets of $X$ with
$A$  + invariant for $(\phi_C)^I \cup \A H$ and $B$  + invariant for $((\phi_C)^I \cup \A H)^{-1}$.

 There exists a continuous function $L: X \to [0,1]$ with $B = L^{-1}(0)$, $A = L^{-1}(1)$ and
such that if $(x,y) \in (\phi_C)^I \cup \A H$ with $x \not= y$, then $L(y) \ge L(x)$ with equality only when
 \begin{equation}\label{eqrellyap02semihyb}
x,y \in A, \quad x,y \in B, \ \ \text{or} \ \ (y,x) \in \A H.
\end{equation}
 In particular, $L$ is a Lyapunov function for $\A H$
with $|\A H| \subset |L| \subset |\A H| \cup A \cup B$.

(b) There exists a continuous function $L: X \to [0,1]$
such that if $(x,y) \in (\phi_C)^I \cup \A H$ with $x \not= y$, then
$L(y) \ge L(x)$ with equality only when, in addition, $(y,x) \in \A H$. In particular, $L$ is a Lyapunov function
with $|L| = |\A H|$. \end{theo}

\begin{proof}  This is Theorem \ref{reltheo02} applied to $(\phi_C)^I \cup \A H$ just as in Theorem \ref{reltheo02semi}.
  Notice that Corollary \ref{hybcor02aaa}
implies that $x \not= y$ and $(x,y) \in ((\phi_C)^I \cup \A H )\cap ((\phi_C)^I \cup \A H )^{-1}$ implies
$(x,y) \in \A H \cap \A H^{-1}$.

\end{proof}

\begin{cor} \label{relcor03semihyb} Assume that $(A,B)$ is an attractor-repeller pair for the  a hybrid system $\H$ on $X$.
 There exists a continuous function $L: X \to [0,1]$ with $B = L^{-1}(0)$, $A = L^{-1}(1)$ and
such that if $(x,y) \in (\phi_C)^I \cup \CC H$ with $x \not= y$, then $L(y) \ge L(x)$ with equality only when
$ x,y \in A,$ or $ x,y \in B$.

 In particular, $L$ is a Lyapunov function for $\CC H$
with $ |L| = A \cup B$.  Furthermore, for all $a$ such that $0 < a < 1$, the set $U_a = \{ x : L(x) \ge a \}$ is an inward subset for $\H$ with
associated attractor $A$. If $V$ is any neighborhood of $A$, there exists $0 < a < 1$ such that $U_a \subset V$. \end{cor}

\begin{proof} Apply Theorem \ref{reltheo02semihyb} with $\A = \CC$. Notice that if $(x,y) \in \CC H \cap \CC H^{-1}$, then
$x$ and $y$ are chain recurrent points lying in the same chain component. It follows that either $x,y \in A$ or $x,y \in B$. Consequently,
$ |L| \subset A \cup B$. If $x \in A$, then $\H$ invariance implies there exists $y \in A$ such that $(y,x) \in H$. Hence,
$L(y) = L(x) = 1$ and so $x \in |L|$.  Similarly, $x \in B$ implies that $x \in |L|$.  Thus, $ |L| = A \cup B$.

Now assume that $0 < L(x) < 1 $ so that $x \not\in A \cup B$ and so $x \not\in |L|$.  It follows that
$(x,y) \in G$ or $(x,t,y) \in \Phi_C$  with $t > 0$  implies
$L(y) > L(x) = a$. So for any $a$ with $0 < a < 1$,
 \begin{align}\label{eqrellyap02aasemihyb}\begin{split}
 \inf (L((\phi_C)^t(U_a)) > a, \quad &\text{and so} \quad  (\phi_C)^t(U_a) \subset \{ x : L(x) > a \} \subset U_a^{\circ}, \\
 \text{and} \qquad \inf (L|G(A)) > a \quad &\text{and so} \quad  G(U_a) \subset \{ x : L(x) > a \} \subset U_a^{\circ}.
 \end{split}\end{align}
Thus, for every $t > 0, 0 < a < 1,  \ (\phi_C)^t(U_a) \cup G(U_a) \subset  \subset U_a$ and so each $U_a$ is $\H$ inward.

The remaining results are proved just as for Corollary \ref{relcor03semi}.

 \end{proof}

 As usual $X$ and $\emptyset$ are inward for $\H$ and for $\overline{\H}$. We define
\begin{equation}\label{eqrel17hyb}\begin{split}
X_-  \ =_{def} \  \bigcap_{n=1}^{\infty} \ H^n(X) \ = \ \{ x : H^{-n}(x) \not= \emptyset \ \ \text{for all} \ n \in \Z_+ \}\\
 X_+   \ =_{def} \  \bigcap_{n=1}^{\infty} \ H^{-n}(X)\ = \ \{ x : H^{n}(x) \not= \emptyset \ \ \text{for all} \ n \in \Z_+ \}\\
 X_{\pm}  \ =_{def} \  X_- \ \cap \ X_+. \hspace{5cm}
 \end{split}\end{equation}
$X_-$ is the maximum $\H$ invariant subset of $X$.
It is an attractor with $\emptyset$ as dual repeller. On the other hand $X_+ $, the
maximum $\overline{\H}$ invariant subset, is  a repeller dual to the attractor $\emptyset$.

\vspace{.5cm}

%{proper $k$ step hybrid sequence} is a sequence $\{(t_i,j_i,x_i) : i = 0,\dots, k \}$

We will write
 \begin{align}\label{eqhyb04aaaax} \begin{split}
(t_1,n_1) \le_{z,t} (t_2,n_2) \quad &\text{when} \quad  0 \le t - t_1, t_2 - t < 1  \ \ \text{and} \ \ n_1 < n_2, \\
(t_1,n_1) \le' (t_2,n_2) \quad &\text{when either} \quad (t_1,n_1) \le_h (t_2,n_2), \ \ (t_1,n_1) \le_v (t_2,n_2) \\
\text{or} \ \ (t_1,n_1) \le_{z,t} &(t_2,n_2) \ \ \text{for some} \ \ t \in \R.
\end{split}\end{align}
When $(t_1,n_1) \le_{z,t} (t_2,n_2)$ the associated hybrid time interval is \\ $[(t_1,n_1),(t,n_1)] \cup [(t,n_1),(t,n_2)] \cup [(t,n_2),(t_2,n_2)]$.

Let $(t_0,n_0) \le' (t_1,n_1) \le' \dots (t_k,n_k)$ be a finite sequence of length $k$ in $\R \times \Z$ with associated compact time interval $E$.
If  $x_0, x_1, \dots x_k \in X$, then there exists a hybrid solution path $\xx : E \to X$ with $\xx(t_i,n_i) = x_i$ for $i = 0,1, \dots, k$ if and
only if for $i = 1,\dots, k$:
 \begin{align}\label{eqhyb05caaa}\begin{split}
(t_{i-1},n_{i-1}) \le_h (t_i,n_i) \quad &\Longrightarrow \quad (x_{i-1},t_i - t_{i-1},x_i) \in \Phi_C, \\
(t_{i-1},n_{i-1}) \le_v (t_i,n_i) \quad &\Longrightarrow \quad (x_{i-1},x_i) \in \O G, \\
(t_{i-1},n_{i-1}) \le_{z,t} (t_i,n_i) \quad &\Longrightarrow \quad \text{there exist} \ \ y_1, y_2 \in X \ \ \text{such that}\\
(x_{i-1},t - t_{i-1},y_1),\  (&y_2,t_i - t,x_i) \in \Phi_C \ \ \text{and} \ \ (y_1,y_2) \in \O G.
\end{split} \end{align}

A $k$ step $\ep$ $\H$ chain from $x$ to $y$, is based on a sequence in $\R_+ \times \Z_+ \times X \times X$ :
 $\{ (t_i,n_i,z_i,w_i) : i = 0,1, \dots, k \}$ such that

  \begin{align}\label{eqhyb05daaa}\begin{split}
  (t_0,n_0) \le' (t_1,n_1) &\le' \dots (t_k,n_k). \\
 d(z_i, w_i) \ < \ \ep  \quad &\text{for} \quad i = 0, \dots, k. \\
(t_{i-1},n_{i-1}) \le_h (t_i,n_i) \quad \Longrightarrow &\quad (w_{i-1},t_i - t_{i-1},z_i) \in \Phi_C \ \  \text{and} \ t_i - t_{i-1} \ge 1 \\
(t_{i-1},n_{i-1}) \le_v (t_i,n_i) \quad \Longrightarrow &\quad (w_{i-1},z_i) \in \O G  \\
(t_{i-1},n_{i-1}) \le_{z,t} (t_i,n_i) \quad \Longrightarrow &\quad \text{there exist} \ \ y_1, y_2 \in X \ \ \text{such that}\\
(w_{i-1},t - t_{i-1},y_1), \ (y_2,t_i - t,z_i) &\in \Phi_C \ \ \text{and} \ \ (y_1,y_2) \in \O G \qquad \text{for} \ \ i = 1,\dots, k,\\
 x = z_0 \quad \text{and} \quad &y = w_k.
\end{split}\end{align}

From \cite{A93} Proposition 1.8, it follows that $\CC H$ can be written
\begin{equation}\label{eqhyb07baaa}
\CC H \quad = \quad \bigcap_{\ep > 0} \ (\O (V_{\ep} \circ H) )\circ V_{\ep}.
\end{equation}
(Compare (\ref{eqrel07}).

Following Theorem \ref{hybtheo01solpath} it easily follows that

\begin{prop}\label{hybtheo08aaa} For $x, y \in X$ $(x,y) \in \CC H$ if and only if for every $\ep > 0$ there exists
a $k$ step $\ep$ $\H$ chain from $x$ to $y$ for some $k = 1, 2, \dots$. \end{prop} \vspace{1cm}

\subsection{\textbf{Restriction to a Closed Subset}}\label{restriction3}\vspace{.5cm}

For $K$ a closed subset of $X$.  The restriction of the hybrid system $\H = (\Phi_C,G)$ to $K$ is the hybrid system
$\H_K = (\Phi_{K \cap C},G_K)$ with $G_K = G \cap (K \times K)$ so that $Dom(G_K) = \{ x \in K : G(x) \cap K \not=  \emptyset \}  \subset K \cap D$.

It is easy to check that for a hybrid time domain $E$ the map $\xx : E \to X$ is a hybrid solution path for $\H_K$ if and only if
it is a hybrid solution path for $\H$ which is contained entirely in $K$, i.e. $\xx(E) \subset K$.

We will label the associated relation as $H|K$ so that with $I = [0,1], J = [1,2]$

\begin{align}\label{eqhybsub01}\begin{split}
H|K \ =_{def} \ &H(\H_K) \   \ = \  \\  ((\phi_{C \cap K})^I \circ G_K \circ &(\phi_{C \cap K})^I) \cup (\phi_{C \cap K})^J.
\end{split}\end{align}

We use this notation because $H|K$ is usually a proper subset of $H_K = H \cap (K \times K)$.

For an arbitrary subset $K$ of $X$ we define $K_-, K_+, K_{\pm}$ by
\begin{itemize}
  \item  $x \in K_-$ if and only if there exists a hybrid solution path $\xx : [[-\infty,(0,0)]] \to K$ for $\H$
with $x = \xx(0,0)$.\vspace{.5cm}

 \item  $x \in K_+$ if and only if there exists a hybrid solution path $\xx : [[(0,0),\infty]] \to K$ for $\H$
with $x = \xx(0,0)$.\vspace{.5cm}

\item $x \in K_{\pm}$ if and only if there exists a bi-infinite hybrid solution path for $\H$
which passes through $x$ and is contained in $K$.\vspace{.5cm}
\end{itemize}

When $K$ is a closed subset Proposition \ref{hybpropinvar} allows us to use
 $H|K$  to obtain
\begin{align}\label{eqrel17hybaa}\begin{split}
K_-  \ &=_{def} \  \bigcap_{k=1}^{\infty} \ (H|K)^k(K) \ = \\
\{ x \in K : \  &(H|K)^{-k}(x) \not= \emptyset  \ \ \text{for all} \ n \in \Z_+ \} \\
K_+   \ &=_{def} \  \bigcap_{k=1}^{\infty} \ (H|K)^{-k}(K)\ = \\
\{ x \in K : \  &(H|K)^{k}(x) \not= \emptyset  \ \ \text{for all} \ n \in \Z_+ \}\\
K_{\pm}  \ &=_{def} \   K_- \ \cap  \ K_+,
 \end{split}\end{align}
 with $K_+$ the maximum repeller and $K_-$ the maximum attractor for $\H_K$ or, equivalently, for $H|K$.

  \begin{df}\label{reldf07hyb} Let  $\H = (\Phi_{ C},G)$ be a hybrid system on $X$ and $K$ a (not necessarily closed) subset of $X$.

We say that $K$ is \emph{+ viable} for $\H$ \index{subset!+ viable}
when $K = K_+$.  When $K$ is closed, the following are equivalent.
  \begin{itemize}
  \item[(i)] $K = Dom(H|K)$, i.e. $K$ is + viable for $H|K$.
  \item[(ii)] $K \subset (H|K)^{-1}(K)$.
  \item[(iii)] $(H|K)(x)  \not= \emptyset$ for all $x \in K$.
  \item[(iv)] $K = K_+$, i.e. $K$ is + viable for $\H$
  \end{itemize}

  We say that $K$ is \emph{- viable} for $\H$ when it is + viable for $\overline{\H}$, \index{subset!- viable}
  or, equivalently, when $K = K_-$.

  We say that $K$ is \emph{viable} for $\H$ (= viable for $\overline{\H}$)
  when it is both + and - viable or, equivalently, when $K = K_{\pm}$  \index{subset!viable}
When $K$ is closed  the following conditions are equivalent.
   \begin{itemize}
   \item[(v)] $H|K$ is a surjective relation on $K$,  i.e. $K$ is viable for $H|K$.
  \item[(vi)]  $K = K_{\pm}$, i.e. $K$ is  viable for $\H$.
  \end{itemize}
 \end{df} \vspace{.5cm}

 As before, the equivalences are clear. Thus, for a closed subset $K$, +, -  viability and viability are the same for the hybrid systems $\H$, $\H_K$
 and for the  relation $H|K$.

  \begin{prop}\label{semiprop10hyb} Let $\H = (\Phi_{ C},G)$  be a hybrid system on $X$.

  (a) If a closed subset $K$ is $\H$  + invariant, then it $\H$  invariant if and only if it is - viable.
  If $C$ is $\overline{\H}$ + invariant,then it is $\overline{\H}$ invariant if and only if it is + viable.
  In particular, an attractor is - viable and a repeller is + viable.

  (b) If $A$ is  + invariant and $B$ is + viable, then $A \cap B$ is + viable.

  (c) If $A$ is invariant for $\H$ , e.g. an attractor, and $B$ is invariant for $\overline{\H}$, e.g. a repeller,
  then $A \cap B$ is viable for $\H$.

  (d)  If $K$ is any subset, then for  $\H$  \\  $K_+$ is + viable, $K_-$ is - viable
  and $K_{\pm}$ is viable. If $K$ is closed, each is equivalent to the corresponding condition for $\H_K$.

  (e) Let $\{ K_i \}$ be a  collection of  subsets of $X$.  If all are + viable, or all - viable or
    all viable, then $K = \bigcup \{ K_i \}$ satisfies the corresponding property.
  \end{prop}

\begin{proof} (a) A + invariant set $K$ is invariant if and only if for all $x \in K$ there exists $\xx \in \S([[-\infty,(0,0)]])$ with
$\xx(0,0) = x$ and $\xx([[-\infty,(0,0)]]) \subset K$.  This is the same as - viability.

(b) If $x \in A \cap B$ then, because $B$ is + viable, there exists $\xx \in \S([[(0,0),\infty]])$ with $\xx(0) = x$
and $\xx([[(0,0),\infty]]) \subset B$. Because $A$ is + invariant, $\xx([[(0,0),\infty]]) \subset A$ because $\xx(0,0) = x \in A$.
Thus, $\xx([[(0,0),\infty]]) \subset A \cap B$.

(c) $A$ is + invariant and $B$ is + viable and so by (b) $A \cap B$ is + viable.  $A$ is - viable and $B$ is - invariant and so $A \cap B$ is - viable.

(d) If $x \in K_+$ there exists $\xx \in \S([[(0,0),\infty]])$ with $\xx(0,0) = x$ and  $\xx([[(0,0),\infty]]) \subset K$.
For any $(t,n) \in E$ the translate $Trsl_{(t,n)}(\xx) \in \S([[(0,0),\infty]])$ with
$Trsl_{t}(\xx)(0,0) = \xx(t,n)$. Thus, $\xx(t,n) \in K_+$ for all $(t,n)$ and so $\xx([[(0,0),\infty]]) \subset K_+$.
The proofs for $K_-$ and $K_{\pm}$ are similar.

(e) Obvious.

    \end{proof}

   \begin{prop}\label{semiprop11hyb} Let $\H$  be a hybrid system on $X$ and $K$ be a closed subset of $X$.
If $\xx : E \to X$ is a hybrid solution path in $ \S_+(\H_K)$, then
         \begin{equation}\label{eqsemi13hyba}
      \om [\xx]  \ =_{def} \   \bigcap_{(t,n) \in E} \ \overline{ \{ \xx(s,m)) : (t,n) \prec (s,m) \ \  \text{in} \ \ E \}}
      \end{equation}
      is a nonempty, closed, viable  subset of $K_{\pm}$.

      \end{prop}

      \begin{proof} It is easy to see that $x \in \om [\xx]$ if and only if $x \in \om[\yy]$ for some $\yy \in \S_+(H|K)$ with $\xx$ spanning $\yy$.
      Since viability for
      $H|K$ and for  $\H$ agree, it follows from Proposition \ref{semiprop11} that each such $\om[\yy]$ is a viable subset
      of $K_{\pm}$. The union $\om[\xx]$ is viable by Proposition \ref{semiprop10hyb} (e).

      \end{proof}

      As usual applying the result to the reverse system we see that if $\xx : E = [[-\infty,(0,0)]] \to X$ is a hybrid solution path for
      $\H$, then

            \begin{equation}\label{eqsemi13hybb}
      \a [\xx]  \ =_{def} \   \bigcap_{(t,n) \in E} \ \overline{ \{ \xx(s,m)) : (s,m) \prec (t,n) \ \  \text{in} \ \ E \}}
      \end{equation}
      is a nonempty, closed, viable  subset of $K_{\pm}$. \vspace{1cm}

      \subsection{\textbf{Isolated  Subsets and the Conley Index}}\label{Conleyhyb}\vspace{.5cm}

      Let $\H = (\Phi_C,G)$ be a hybrid dynamical system on $X$ and $K$ be a closed subset of $X$.  Following \ref{eqConley07a} and
      \ref{eqConleysem02} we define

      \begin{align}\label{eqConleyhyb01}\begin{split}
      \d_{\H}(K) \ =_{def} \ &\d_{G}(K) \cup \d_{\Phi_C}(K) \ = \\
     [(\overline{G(K) \setminus K})\cap K]\ &\cup \ [\bigcap_{\ep > 0} \ \d_{(\phi_C)^{[0,\ep]}}(K)].
       \end{split}\end{align}

       \begin{prop}\label{propConleyhyb01} Let $\H = (\Phi_C,G)$ be a hybrid dynamical system on $X$. For $K$ a closed subset of $X$
       let $\H_K = (\Phi_{C \cap K},G_K)$ be the restriction of $\H$ to $K$. If $P$ is a closed $\H_K$ + invariant subset of $K$, then
       \begin{equation}\label{eqConleyhyb02}
       \d_{\H}(P) \ \subset \ \d_{\H}(K) \ \subset  \partial K. \hspace{1cm}
       \end{equation}
       \end{prop}

       \begin{proof} Since $P$ is $\H_K$ + invariant, it is $G_K$ + invariant and so \ref{eqConley07b} and \ref{eqConley07c}
       imply that $\d_{G}(P) \subset \d_{G}(K) \subset \partial K$.  It is also $\Phi_{C \cap K}$ + invariant and so
       by Proposition \ref{propConleysem01}(e) $\d_{\Phi_C}(P) \subset \d_{\Phi_C}(K) \subset \partial K$.

       \end{proof}

Recall from  Proposition \ref{hybpropinvar} that $K$ is $\H$ invariant if and only if it is $H$ invariant.  By Definition \ref{reldf07hyb}
$K$ is + viable, - viable or viable
for $\H$  if and only if it satisfies the corresponding property for $H|K$. Finally, for a closed set $K$ the definitions of the sets
$K_+, K_-, K_{\pm}$ for $\H$ and for $H|K$ agree. So we can use Definition \ref{dfConley05a} applied to $H|K$ to define
 isolating neighborhoods and isolated sets for $\Phi$.

\begin{df}\label{dfConley05ahyb} Let $\H = (\Phi_C,G)$ be a hybrid dynamical system on $X$ and $K$ be a closed subset of $X$

(a) The set $K$ is  called an \emph{isolating neighborhood}
\index{isolating neighborhood} when $K_{\pm} \subset K^{\circ}$, i.e. its maximum viable subset is contained
in its interior. In that case, the viable set $A = K_{\pm}$ is called an
\emph{isolated viable set}\index{isolated viable set}\index{viable set!isolated}.

(b) The set $K$ is  called a \emph{- isolating neighborhood} (or a \emph{+ isolating neighborhood})
 when $K_{-} \subset K^{\circ}$ (resp. $K_{+} \subset K^{\circ}$). In that case, the - viable set $K_{-}$ is called an
\emph{isolated - viable set} (resp. the + viable set $K_{+}$ is called an
\emph{isolated + viable set}).
 \end{df}
\vspace{.5cm}

 For an isolated viable set for $\H$ we define the associated \emph{stable subset} and \emph{unstable subset}\index{subset!stable}\index{subset!unstable}
just as for a closed relation and for a semiflow relation.

\begin{theo}\label{theoConley07newhyb} Let  $K$ be an isolating neighborhood for
$K_{\pm}$, i.e. $K_{\pm} \subset K^{\circ}$.

Define
\begin{align}\label{eqrel19new3hyb}\begin{split}
W^s(K_{\pm}) \ &=_{def} \ \bigcup_{n \in \Z_+} \ (H|K)^{-k}(K_+), \\
W^u(K_{\pm}) \ &=_{def} \ \bigcup_{n \in \Z_+} \ (H|K)^k(K_-).
\end{split}\end{align}
$W^s(K_{\pm})$ is a + viable subset for $\Phi$, $W^u(K_{\pm})$ is a - viable subset for $\Phi$ and

\begin{align}\label{eqrel19new4hyb}\begin{split}
x \in W^s(K_{\pm}) \  &\Longleftrightarrow \  \text{there exists} \  \xx \in \S_+(\H), \ \text{with} \ \xx(0,0) = x, \ \om[\xx] \subset C_{\pm} \\
x \in W^u(K_{\pm}) \ &\Longleftrightarrow \  \text{there exists} \  \xx \in \S_-(\H), \ \text{with} \ \xx(0,0) = x, \ \a(\xx) \subset C_{\pm}.
\end{split}\end{align}

\end{theo}

\begin{proof}  As before apply Theorem \ref{theoConley07new} with $F_C = H|K$, using Theorem \ref{hybtheo01solpath}(d). We leave the details to the reader.

\end{proof}

For $K$ an isolating neighborhood for a viable subset $A $, we call a pair $(P_1, P_2)$ of
 closed subsets of $X$  an $\H$ \emph{index pair} \index{index pair}  rel $K$
 for $A$ when
 the following conditions are satisfied:
 \begin{itemize}
 \item[(i)] $P_2 \ \subset \ P_1 \ \subset \ K$.

 \item[(ii)] $P_1$ and $P_2$ are $\H_K$ + invariant.

 \item[(iii)] $A = K_{\pm} \ \subset \ P_1^{\circ} \setminus P_2 $..

  \item[(iv)] $P_1 \setminus P_2 \ \subset  \ K^{\circ}$, or, equivalently, $P_1 \cap \partial K \subset P_2$.
 \end{itemize}
 We will sometimes consider the following strengthening of (iv).
  \begin{itemize}
 \item[(iva)] $\overline{P_1 \setminus P_2} \subset  \ K^{\circ}$, or,
 equivalently, $P_1 \cap \partial K$ is contained in the $P_1$ interior of $P_2$.
 \end{itemize}

 We call a pair $(P_1, P_2)$ of closed subsets of $X$  an $\H$ \emph{index pair} \index{index pair}
 for a viable set $A$ when there exists an isolating neighborhood $K$ for $A$ such that   $(P_1, P_2)$ is an index pair   rel $K$
 for $A$.

The following is the hybrid system version of Theorems \ref{theoConley07} and \ref{theoConleysem03}.

 \begin{theo}\label{theoConleyhyb03} Given $\H = (\Phi_C,G)$ a hybrid dynamical relation on $X$,
 assume that $K$ is an isolating neighborhood for a  viable set $A$.
 \begin{enumerate}
 \item[(a)] If $(P_1,P_2)$ is an $\H$ index pair rel $K$ for $A$, then
 $K_- \subset P_1$ and $K_+ \cap P_2 = \emptyset$. In addition, $ \d_{\H}(P_1) \subset P_1 \cap \partial K \subset P_2$.

In addition, if (iva) holds for $(P_1,P_2)$, then for some \\ $\ep > 0, \ \ \d_{(\phi_C)^{[0,\ep]}}(P_1) \ \subset \ P_2$.

  \item[(b)] If $U$ and $V$ are open subsets of $X$ with $C_- \subset U$,  and $C_{\pm} \subset V \subset C$, then there exists
 an $\H$ index pair $(P_1,P_2)$  rel $K$ for $A$ with $P_1 \subset U$, and  such
 that   $\overline{P_1 \setminus P_2} \subset V$.
 In particular, (iva) holds for $(P_1,P_2)$.

 \end{enumerate}
 \end{theo}

 \begin{proof} This follows the proof of Theorem \ref{theoConleysem03}. The existence of the required inward sets follows from
 Corollary \ref{relcor03semihyb}.

    \end{proof}

    The following are the hybrid system versions of Theorem \ref{theoConleysem04} and Theorem \ref{theoConleysem05}

   \begin{theo}\label{theoConleyhyb04}  A pair $(P_1, P_2)$  of closed subsets of $X$  is an $\H$ index pair,
   i.e. there exists a viable set $A$ and a closed neighborhood $K$ of $A$ such that
   $(P_1, P_2)$ is an $\H$ index pair rel $K$ for $A$ if
 the following conditions are satisfied:
 \begin{itemize}
 \item[(i$'$)] $P_2 \subset P_1$.

 \item[(ii$'$)] $P_2$ is $\H_{P_1}$ + invariant.

 \item[(iii$'$)] $(P_1)_{\pm} \subset  P_1^{\circ} \setminus P_2 $.

  \item[(iv$'$)] For some $\ep > 0, \d_{G}(K) \cup \d_{(\phi_C)^{[0,\ep]}}(P_1) \ \subset \ P_2$.
 \end{itemize}

  In addition,
 $K$ can be chosen so that (iva) holds if
  \begin{itemize}
   \item[($iva'$)] $\d_{\H}(P_1)$ contained in the $P_1$ interior of $ P_2$.
 \end{itemize}
 \end{theo}

 \begin{proof}  The proof is completely analogous to that of Theorem \ref{theoConleysem04}. As above,
 condition (iva$'$) implies (iv$'$). Clearly, (iva$'$) is necessary to obtain (iva) for $K$.

First, just as before, we can find $K_1$ so that $P_1 \subset \subset K_1$ and $(P_1)_{\pm} = (K_1)_{\pm}$.

 Fix such a $K_1$ and choose $\ep > 0$ small enough that for all $x \in P_1$, $V_{\ep}(x) \subset K_1$ and, in addition,
 $\d_{(\phi_C)^{[0,\ep]}}(P_1) \ \subset \ P_2$. Hence, the closed set $\r_G(P_1) \cup \r_{(\phi_C)^{[0,\ep]}}(P_1)$ satisfies
 $P_1 \cap (\r_G(P_1) \cup \r_{(\phi_C)^{[0,\ep]}}(P_1)) = \d_G(P_1) \cup \d_{\phi^{[0,\ep]}}(P_1) \subset P_2$. Let $ K$ be the closure of the set
  \begin{align}\label{eqConleyhyb13} \begin{split}
\{ \ \  y \in X :\ \  &\text{there exists} \ \ x \in P_1 \ \ \text{such that } \\
d(y,x) \ \leq \ \frac{1}{2} &\min[\ep, d(x,\r_G(P_1) \cup \r_{(\phi_C)^{[0,\ep]}}(P_1))]\ \  \}.
 \end{split}\end{align}

 Complete the proof as before.

    \end{proof}

    \begin{theo}\label{theoConleyhyb05} For a closed subset $P_1$ of $X$, there exists $P_2$ such that $(P_1,P_2)$ satisfies (i$'$) - (iv$'$)
    of Theorem \ref{theoConleyhyb04}, and so is an $\H$ index pair if and only if the following conditions
     hold.
      \begin{itemize}

 \item[(i$''$)] $(P_1)_{\pm} \subset  P_1^{\circ} $, i.e. $P_1$ is an isolating neighborhood for $(P_1)_{\pm}$.

  \item[(ii$''$)] $\d_{\H}(P_1) \cap (P_1)_{+} = \emptyset$.
 \end{itemize}
 Furthermore, $P_2$ can be chosen so that (iva$'$) holds for the pair.
 \end{theo}

 \begin{proof}   The proof follows that of Theorem \ref{theoConleysem05}.

 \end{proof}

As before, a closed set $P_1$ which satisfies (i$''$) and (ii$''$) is a special sort of isolating neighborhood which
 we will call an \emph{isolating neighborhood of index type}
 \index{isolating neighborhood!of index type}. From Theorem \ref{theoConleyhyb03} it follows that every
 isolated viable subset admits a neighborhood base of
 isolating neighborhoods of index type.
\vspace{1cm}

   \section{\textbf{Appendix: Continuity Conditions}}\label{appendix}    \vspace{1cm}

   For $X$ a compact metric space with diameter $D$ we let $2^X$ denote the set of compact subsets of $X$ and define for $A, B \in 2^X$
   \begin{align}\label{eqapp01}\begin{split}
   d(A/B) \ =_{def} \ &\min(D + 1, \inf \{ \ep \ge 0 : V_{\ep}(A) \supset B \}), \\
   d(A,B) \ =_{def} \ &\max( d(A/B), d(B/A) ).
   \end{split} \end{align}
   Thus, $d(A/B) = D+1$ if and only if $A = \emptyset$ and $B \not= \emptyset$.

   The metric $d$ is the \emph{Hausdorff metric}\index{Hausdorff metric} on $2^X$ with $\emptyset$ an isolated point. Equipped with
   $d$, $2^X$ is a compact metric space, see e.g. \cite{A93} Chapter 7.

   \begin{prop}\label{propapp01} Let $X,Y$ be  metric spaces with $X$ compact and $f$ be a function from $Y$ to $2^X$.

   (a) The following are equivalent and when they hold we call $f$ \emph{uppersemicontinuous} or \emph{usc}
   \index{function!uppersemicontinuous} \index{usc}.
    \begin{itemize}
    \item[(i)]  If $\{y_n\}$ is a sequence in $Y$ converging to $y$, then $\{ d(f(y)/f(y_n)) \}$ converges to $0$.
    \item[(ii)] If $U$ is an open subset of $X$, then $\{ y : f(y) \subset U \}$ is an open subset of $Y$.
    \item[(iii)] The relation $F = \{ (y,x) : x \in f(y) \} \subset Y \times X$ is closed.
    \end{itemize}

    (b)   The following are equivalent.
    \begin{itemize}
    \item[(i)] $f : Y \to 2^X$ is a continuous function
    \item[(ii)]  If $\{y_n\}$ is a sequence in $Y$ converging to $y$, then $\{ d(f(y),f(y_n)) \}$ converges to $0$.
    \item[(iii)] If $U$ is an open subset of $X$, then $\{ y : f(y) \subset U \}$ and $\{ y : f(y) \cap U \not= \emptyset \}$
    are open subsets of $Y$.
        \end{itemize}
    \end{prop}

    \begin{proof} See  \cite{A93} Proposition 7.11.

    \end{proof}

    The following is easy to check with details left to the reader.

    \begin{lem}\label{lemapp02} Assume $\{ A_n \}$ is a sequence in $2^X \setminus \{ \emptyset \}$ converging to $A$, i.e. $\{ d(A_n,A)) \}$ converges to $0$.
    If $x \in A$, then there exists $x_n \in A_n$ such that the sequence $\{ x_n \}$ in $X$ converges to $x$. Conversely, if $x \in X$ is a limit point
    of a sequence  $\{ x_n \in A_n\}$, then $x \in A$.\end{lem}

    For $t \ge 0$ in $\R$ let $\I$ denote the set of intervals in $2^{[0,t]}$. By an interval we mean a nonempty, closed interval, but it might be
    trivial, consisting of a single point.

    \begin{prop}\label{propapp02a} The set $\I$ is a closed subset of $2^{[0,t]} \setminus \{ \emptyset \}$.
    The functions $I \mapsto \sup I$ and $I \mapsto \inf I$ are
    continuous functions from $2^{[0,t]} \setminus \{ \emptyset \}$ to $[0,t]$. \end{prop}

    \begin{proof} If $U_1, U_2$ are disjoint open subsets of $[0,t]$, then the conditions $I \cap U_1 \not= \emptyset, I \cap U_2 \not= \emptyset$
    and $I \subset U_1 \cup U_2$ are open conditions on $I \in 2^{[0,t]}$. So if $\{ I_i \}$ converges to $I$ and $I$ is disconnected, then
    eventually the $I_i$'s are disconnected. Contrapositively, $\I$ is closed.

    If $\{ I_i \}$ converges to $I$, $s = \sup I$ and $s_i = \sup I_i$, then Lemma \ref{lemapp02} implies there is a sequence
    $\{ u_i \in I_i \}$ which converges to $s$. If $\tilde s$ is any limit point of the sequence $\{ s_i \}$, then
    $u_i \le s_i$ for all $i$ implies that $s \le \tilde s$. By Lemma \ref{lemapp02} again $\tilde s \in I$ and so
    $s = \tilde s$. Thus, $\{s_i \}$ converges to $s$, proving continuity of $I \mapsto \sup I$.

    The proof for $\inf$ is similar.
    \end{proof}

     \vspace{1cm}

     \subsection{\textbf{Semiflow Relations}}\label{appsemi} \vspace{1cm}

     Let $X$ be a compact metric space. For $\Phi \subset X \times \R_+ \times X$ we let $\phi^t = \{ (x,y) : (x,t,y) \in \Phi \}$.
     From Proposition \ref{propapp01} it follows that $\Phi$ is a closed subset if and only if the map $\phi^{\#}$ from $\R_+$ to $2^{X \times X}$
     given by $t \mapsto \phi^t$ is usc. In particular, if $t \mapsto \phi^t$ is continuous, then $\Phi$ is closed.

     We consider the extent to which the converse is true when $\Phi$ is a semiflow relation.

     The space of continuous maps from $X$ to itself is a topological monoid (i.e. a semigroup with identity) since composition is
     continuous. $\Phi$ is a semiflow on $X$ exactly when the map $\phi^{\#}$ from $\R_+$ to the space of continuous maps is
     a continuous monoid homomorphism.

     The space $2^{X \times X}$ of closed relations on $X$ is also a monoid under composition although now composition is only usc and not
     usually continuous (see \cite{A93} Proposition 7.16). If $\Phi$ is a semiflow relation on $X$, then the map $\phi^{\#} : \R_+ \to 2^{X \times X}$
     given by $t \mapsto \phi^t$ is a monoid homomorphism by the Kolmogorov Condition.

     \begin{theo}\label{theoapp03} If $\Phi$ is a semiflow relation on $X$, then the map $\phi^{\#}$ is continuous from the left at every $t \in \R_+$.
     If $\Phi$ is a complete semiflow relation, then $\phi^{\#} : \R_+ \to 2^{X \times X}$ is continuous. \end{theo}

     \begin{proof} Fix $\ep > 0$.  As in Theorem \ref{semitheo01} for every $t \in \R_+$
     $\phi^t = \bigcap_{\d > 0} \{ \pi_{13}(\Phi \cap X \times \R_+ \cap [t-\d,t+\d] \times X) \}$ implies there exists $\d_t$ such that for each
     $(x,s,y) \in \Phi$ with $|s - t| \le \d_t$ there exists $(x_1,t,y_1) \in \Phi$ with $d((x,y),(x_1,y_1)) < \ep$.  In particular, and this is
     exactly Theorem \ref{semitheo01}, for each $(x,s,y) \in \Phi$ with $s \le \d_0$, $d(x,y) < \ep$.

     If $t \le s \le t + \d_0$, then the Kolmogorov Condition implies that for $(x,s,y) \in \Phi$ there exists $z$ such that
     $(x,t,z), (z,s-t,y) \in \Phi$. Thus, for $(x,y) \in \phi^s$ we have $(x,z) \in \phi^t$ with $d((x,y),(x,z)) < \ep$.
     Thus, $\phi^s \subset V_{\ep}(\phi^t)$.

     First, assume that $\Phi$ is complete. Now if $(x,t,y) \in \Phi$, and $t \le s \le t + \d_0$, completeness implies there
     exists $z$ such that $(y,s-t,z) \in \Phi$ and so by the Kolmogorov Condition $(x,s,z) \in \Phi$. Thus, for $(x,y) \in \phi^t$
     we have $(x,z) \in \phi^s$ with $d((x,y),(x,z)) < \ep$. Thus, $\phi^t \subset V_{\ep}(\phi^s)$.

     It follows that in the complete case, $t \le s \le t + \d_0$ implies $d(\phi^t,\phi^s) < \ep$.  Since $\ep,t$ and $s$ are arbitrary,
     the map $\phi^{\#}$ is continuous.

     If we do not have completeness, assume that $\max(0,t - \min(\d_0,\d_t)) \le s \le t$. Because   $s \le t \le s + \d_0$ we have
     $\phi^t \subset V_{\ep}(\phi^s)$.

     On the other hand, if $(x,s,y) \in \Phi$ then $|t - s| \le \d_t$ implies there exists $(x_1,t,y_1) \in \Phi$ with
     $d((x,y),(x_1,y_1)) < \ep$. So $\phi^s \subset V_{\ep}(\phi^t)$.

     It follows that  $t - \min(\d_0,\d_t) \le s \le t$ implies $d(\phi^t,\phi^s) < \ep$. Because $\d_t$ depends on $t$ we only obtain
     continuity of the map $\phi^{\#}$ from the left.

     \end{proof}

     Without completeness, continuity need not hold. Trivially, consider $\Phi = \{ (x,0,x) : x \in X \}$ so that $\phi^0 = 1_X$ and
     $\phi^t = \emptyset$ for all $t > 0$. A more interesting example on $X \ = \ [-1,1] \ \subset \R$ is given by
        \begin{equation}\label{eqapp02}
        \Phi \ = \ \{ (x,t,y) : y = x - t \ge 0, 0 \le t \le 1 \} \cup \{ (-1,t,-1) : t \in \R_+ \}
        \end{equation}
        In this case $\phi^1 = \{ (1,0), (-1,-1) \}$ and $\phi^t = \{ (-1,-1) \}$ for all $t > 1$. Here $\phi^{\#}$ is discontinuous
        at $0$ and $1$.

            \vspace{1cm}

     \subsection{\textbf{Hybrid Solution Paths}}\label{apphyb} \vspace{1cm}

     On $\R_+ \times \Z_+$  we defined (see \ref{eqhyb01}) the closed partial order $\preceq$ and associated orders by
      \begin{align}\label{eqapp03}\begin{split}
(t_1,n_1) \preceq (t_2,n_2) \quad &\text{when} \quad  t_1 \le t_2 \ \ \text{and} \ \ n_1 \le n_2, \\
(t_1,n_1) \le_h (t_2,n_2) \quad &\text{when} \quad  t_1 \le t_2 \ \ \text{and} \ \ n_1 = n_2 \not= \pm \infty, \\
(t_1,n_1) \le_v (t_2,n_2) \quad &\text{when} \quad  t_1 = t_2 \not= \pm \infty \ \ \text{and} \ \ n_1 < n_2, \\
(t_1,n_1) \le (t_2,n_2) \quad &\text{when either} \quad (t_1,n_1) \le_h (t_2,n_2) \ \ \text{or} \ \ (t_1,n_1) \le_v (t_2,n_2).
\end{split}\end{align}
%Observe that the relations $\preceq, \le_h, \le_v$ are transitive while $\le$ is not.

When $(t_1,n_1) \preceq (t_2,n_2) $ we write
\begin{align}\label{eqapp04} \begin{split}
[(t_1,n_1),(t_2,n_2)] \ = \ &[t_1,t_2] \times [n_1,n_2] \ = \\
 \{ (t,n) \in \R \times \Z : (t_1,n_1)\  &\preceq \ (t,n) \ \preceq \ (t_2,n_2) \}.
\end{split}\end{align}
%with length $(t_2 - t_1) + (n_2 - n_1)$.

%Thus, when $(t_1,n_1) \le_h (t_2,n_2), \ [(t_1,n_1),(t_2,n_2)]$ is the connected
%\emph{horizontal interval}\index{horizontal interval}\index{interval!horizontal}
%$[t_1,t_2] \times \{ n_1 \}$. When $(t_1,n_1) \le_v (t_2,n_2)$, \\ $[(t_1,n_1),(t_2,n_2)]$ is the discrete
%\emph{vertical interval}\index{vertical interval}\index{interval!vertical}
%$\{ t_1 \} \times \{ [n_1,n_2] \}$. So when $(t_1,n_1) \le (t_2,n_2)$ the relation $\preceq$ is a total order on $[(t_1,n_1),(t_2,n_2)]$.
%Notice that a horizontal interval may be trivial, i.e. a singleton with length $0$, but a vertical interval always has length at least $1$.

With $[i_1,i_2]$ a finite interval in $\Z$ a compact hybrid time interval is the union $E \ = \ \bigcup_{i \in [i_1,i_2-1]} \ [(t_i,n_i),(t_{i+1},n_{i+1})]$
with $\{ (t_i,n_i) : i \in [i_1,i_2] \}$ a finite sequence such that $(t_i,n_i) \le (t_{i+1},n_{i+1})$ for
all $i \in [i_1,i_2-1]$ with $(t_{i_1},n_{i_1})$ is the left end-point
$(t_{i_2},n_{i_2})$ is the right end-point
 of $E$. We write $E = [[(t_1,n_1),(t_2,n_2)]]$ when $E$ is a hybrid time interval from $(t_1,n_1)$ to $(t_2,n_2)$.
 Notice that for $E$ a hybrid time interval $\preceq$ restricts to a total order on $E$.

 In general, for $E \subset \R \times \Z$, $\preceq$ restricts to a total order on $E$ if and only if
 \begin{equation}\label{eqapp05}
 E \times E \ \subset \ \preceq \cup \preceq^{-1}.
 \end{equation}

 \begin{theo}\label{theoapp04} Let $(t,n) \in \R_+ \times \Z_+$ and let $E \subset [(0,0),(t,n)]$.

 The following conditions are equivalent.
 \begin{itemize}
 \item[(i)] $E$ is a hybrid time domain from $(0,0)$ to $(t,n)$.

 \item[(ii)] $E$ is a maximal subset of $[(0,0),(t,n)]$ on which $\preceq$ restricts to a total order.

 \item[(iii)] $E$ is a closed subset of $[(0,0),(t,n)]$ such that
\begin{align}\label{eqapp06} \begin{split}
 E \times E \ &\subset \ \preceq \cup \preceq^{-1}, \\
 \pi_1(E) \ &= \ [0,t] \ \subset \R_+, \\
  \pi_2(E) \ &= \ [0,n] \ \subset \Z_+.
  \end{split}\end{align}
  \end{itemize}\end{theo}

  \begin{proof} (i) $\Rightarrow$ (iii) is obvious.

  (ii) $\Rightarrow$ (iii): If $E \times E \subset  \preceq \cup \preceq^{-1}$ then the closure $\overline{E}$ satisfies
  $ \overline{E} \times \overline{E} \subset  \preceq \cup \preceq^{-1}$ because $\preceq$ is closed.  Hence, by maximality
  $E = \overline{E} $, Thus, $E$ is closed.

  If $t_1 \not\in \pi_1(E)$, let $E_1 = \{ (s,m) \in E : s \le t_1 \}$ and $E_2 = \{ (s,m) \in E : s \ge t_1 \}$. $E_1$ and $E_2$ are
  disjoint closed sets with union $E$. Because $\preceq$ is a
  total order on $E$, $(s_1,n_1) \in E_1$ and $(s_2,n_2) \in E_2$ implies $(s_1,n_1) \preceq (s_2,n_2)$. If $m$ equals either
  the minimum of $\pi_2(E_2)$ or the maximum of $\pi_2(E_1)$, then $\preceq$ restricts to a total order on
  $E_1 \cup \{ (t_1,m) \} \cup E_2$ which contains $E$ as a proper subset.  This violates maximality of $E$.

  Similarly, if $n_1 \not\in \pi_2(E)$, let $E_1 = \{ (s,m) \in E : m \le n_1 \}$ and $E_2 = \{ (s,m) \in E : m \ge n_1 \}$.
  If $s $ equals either $\inf \pi_1(E_2)$ or $\sup \pi_1(E_1)$, then $\preceq$ restricts to a total order on
  $E_1 \cup \{ (s,n_1) \} \cup E_2$ which contains $E$ as a proper subset.  This violates maximality of $E$.

  This completes the proof that (ii) implies (iii).

  (iii) $\Rightarrow$ (i): For each $m \in [0,n] \subset \Z_+$, the set
  \begin{equation}\label{eqapp07}
  I_m(E) \ =_{def} \  \{ s : (s,m) \in E \}
  \end{equation}
  is a nonempty closed interval in $[0,t]$
  (although it might be a trivial interval consisting of a single point).
  It is nonempty because $m \in \pi_2(E)$.  It is a closed set because $E$ is closed. Now assume $s_1 \le s \le s_2$ with
   $(s_1,m),(s_2,m) \in E$. Since $ s \in \pi_1(E)$ there exists $k$ with $(s,k) \in E$.  Because $\preceq$ is total on
   $E$, $s_1 \le s$ implies $m \le k$ and $s \le s_2$ implies $k \le m$.  That is, $k = m$ and so $s \in I_m$.

   If $m < n$, then $\sup I_m = \inf I_{m+1}$. Because $\preceq$ is total on
   $E$, $\sup I_m \le \inf I_{m+1}$.  However, if there exists $t_1$ with $\sup I_m < t_1 < \inf I_{m+1}$, then with $(t_1,k) \in E$
   we obtain $ m \le k$ and $k \le m+1$. Hence, either $k = m$ or $k = m+1$ violating the definition of the sup or the inf.

   Similarly, $\inf I_0 = 0$ and $\sup I_n = t_1$.

   Let $s_i = \inf I_i$ for $i = 0,1,\dots,n$ and $s_{n+1} = t$. Clearly,
  \begin{align}\label{eqapp08}\begin {split}
  (s_0,0) \le_{h} &(s_1,0) \le_v (s_1,1) \le_h (s_2,1) \le_v (s_2,2) \\
   \dots &(s_n,n-1) \le_v (s_n,n) \le_h (s_{n+1},n)
  \end{split}\end{align}
   is a sequence which defines the hybrid interval $E$.

    (iii) $\Rightarrow$ (ii): Assume that $(s,m) \in  [(0,0),(t,n)] \setminus E$.

    Case 1: If $s_1 = \sup I_m < s$, then $m < n$, since $\sup I_n = t_1$, and so $(s_1,m+1) \in E$. Since
    $s_1 < s$ and $m+1 > m$, it follows that $\preceq$ is not total on $E \cup \{ (s,m) \}$.

    Case 2: If $s_1 = \inf I_m > s$, then $m > 0$, since $\inf I_0 = 0$, and so $(s_1,m-1) \in E$. Again
    $\preceq$  is not total on $E \cup \{ (s,m) \}$.

    It follows that $E$ is a maximal subset on which $\preceq$  is total.

    \end{proof}

    \begin{theo}\label{theoapp05} Given $(t,n) \in \R_+ \times \Z_+$ let $\E(t,n)$ be the set of hybrid time intervals from $(0,0)$ to $(t,n)$.
    For $m = 1, \dots, n$ and $E \in \E(t,n)$ let $j_m(E) = \inf I_m(E)$. The set $\E(t,n)$ is a closed subset of $2^{[(0,0),(t,n)]}$ and for
    each $m$, the map $j_m : \E(t,n) \to [0,t]$ is continuous. \end{theo}

    \begin{proof} The conditions of (\ref{eqapp06}) are closed conditions and so $\E$ is a closed subset of $2^{[(0,0),(t,n)]}$.

    Now assume that $\{ E_i \}$ is a sequence in $\E$ converging to $E$. Let $s = j_m(E) = \sup I_{m-1}(E)$.
    Hence, $(s,m), (s,m-1) \in E$, and so     Lemma \ref{lemapp02}
    implies there exist a sequence $\{ (u_i,m_i) \in E_i\} $ converging to $(s,m)$  and so eventually $m_i = m$.  By discarding initial terms
    we may assume $m_i = m$ for all $i$.  Similarly, there exists a sequence $\{ (v_i,m-1) \in E_i \}$ converging to $(s,m-1)$.
    Since $\preceq$ is total on $E_i$ we have $v_i \le j_m(E_i) \le u_i$.  By the Squeeze Theorem $\{ j_m(E_i) \}$ converges to $s = j_m(E)$, proving
    continuity.

    \end{proof}

    Because the finite set $[0,n] \subset \Z_+$ is discrete, it is clear that for $\{ E_i \}$ and $E$ in $\E$
 \begin{align}\label{eqapp09}\begin {split}
 \{ E_i \} \ \to \ E \quad \text{in} \ \ &2^{[(0,0),(t,n)]} \qquad \Longleftrightarrow \\
\{ I_m(E_i) \} \ \to \ I_m(E) \quad \text{in} \  \ &2^{[0,t]} \quad \text{for} \ \ m = 0, 1, \dots n.
\end{split}\end{align}

In particular, continuity of the maps $j_m$ also follows from Proposition \ref{propapp02a}.
\vspace{1cm}

Now let $\Phi$ be a semiflow relation on $X$. For $I \in \I \subset 2^{[0,t]}$ a $\Phi$ solution path is a function $\xx: I \to X$ such that
\begin{equation}\label{eqapp10}
\inf I \le s_1 \le s_2 \le \sup I \quad \Longrightarrow \quad (\xx(s_1), s_2 - s_1, \xx(s_2)) \in \Phi.
\end{equation}
We will repeatedly use the uniform equicontinuity of such solution paths, see Corollary \ref{semicor02}.

We can regard such a map $\xx$ as an element of $2^{[0,t] \times X}$, the space of closed relations from $[0,t]$ to $X$.

\begin{theo}\label{theoapp06} For $\Phi$ a semiflow relation on $X$, let $\{ \xx_i : I_i \to X \}$ be a sequence of
$\Phi$ solution paths defined on intervals $I_i \in \I \subset 2^{[0,t]}$. Assume that in $2^{[0,t] \times X}$
the sequence $\{ \xx_i \}$ converges to $\yy \in 2^{[0,t] \times X}$.

\begin{itemize}
\item[(i)] The sequence $\{ I_i \}$ converges to some $I \in \I$.

\item[(ii)] $\yy$ is a $\Phi$ solution path defined on the interval $I$.

\item[(iii)] The sequences $\{ \xx_i(\inf I_i) \}$ and $\{ \xx_i(\sup I_i) \}$ converge in $X$ to
$ \yy(\inf I)$ and $ \yy(\sup I)$, respectively.
\end{itemize}
\end{theo}

\begin{proof} (i): The continuous map $\pi_1 : [0,t] \times X \to [0,t]$ induces a continuous map
$(\pi_1)_*: 2^{[0,t] \times X} \to 2^{[0,t]}$ by $K \mapsto \pi(K)$ (see \cite{A93} Proposition 7.16).
Hence, $\{ I_i = (\pi_1)_*(\xx_i) \}$ converges to $I =_{def} (\pi_1)_*(\yy)$. In particular, $I$ is the
domain of the closed relation $\yy$. \vspace{.5cm}

(ii) and (iii): Let $r_i = \inf I_i, s_i = \sup I_i$ so that $I_i = [r_i,s_i]$.  Similarly, let $I = [r,s]$. \vspace{.5cm}

If $I$ is a singleton, then (ii) holds trivially. Assume now that $I$ is nontrivial so that $r < s$.

Claim: If $u $ is in the open interval $(r,s)$, then there exists $N_u$ such that $u \in (r_i,s_i)$ for $i \ge N_u$. \vspace{.5cm}

Proof of the Claim:  There exists $\ep > 0$ such that $r + \ep < u < s - \ep$. Choose $N_u$ so that for $i \ge N_u, \ \ I \subset V_{\ep/2}(I_i)$.
Then for $i \ge N_u$ there exist $a_i, b_i \in I_i$ such that $|a_i - r|, |b_i - s| < \ep/2$. It follows that
$r_i \le a_i < u < b_i \le s_i$.  \vspace{.5cm}

Now let $z \in X$ so that $(u,z) \in \yy$. By Lemma  \ref{lemapp02} there exists a sequence $\{ u_i \in I_i \}$ such that
 $\{ (u_i,\xx(u_i)) \}$ converges to $(u,z)$. Since $|u_i - u| \to 0$, uniform equicontinuity implies that for $i \ge N_u$
 $d(\xx_i(u_i),\xx_i(u)) \to 0$.  It follows that $\{ \xx_i(u) : i \ge N_u \}$ converges to $z$. In particular, the restriction of $\yy$
 to the open interval   $ (r,s)$ is a function to $X$.

 If $r < u_1 < u_2 < s$ then for $i \ge \max (N_{u_1}, N_{u_2})$ we have $u_1, u_2 \in (r_i,s_i)$ and so the closed interval
 $[u_1,u_2]$ is contained in $I_i$ and so $\xx_i$ restricts to a $\Phi$ solution path on $[u_1,u_2]$ for such $i$. The sequence
 of restrictions converges pointwise to the restriction of $\yy$ to $[u_1,u_2]$ and so from uniform equicontinuity the convergence
 is uniform. It follows that the restriction  $\yy$ to $[u_1,u_2]$ is a $\Phi$ solution path.

 Consequently, the restriction of $\yy$ to the open interval $(r,s)$ is a $\Phi$ solution path and so by Corollary \ref{semicor02a},
 it extends continuously to a solution path $\tilde \yy$ on $[r,s]$.

 Now whether $I$ is nontrivial or a singleton, we let $z \in X$ so that $(s,z) \in \yy$. As before,
 there exists a sequence $\{ u_i \in I_i \}$ such that
 $\{ (u_i,\xx(u_i)) \}$ converges to $(s,z)$. By Proposition \ref{propapp02a} $\{ s_i \}$ converges to $s$.
 Since $|s_i - u_i| \le |s_i - s| + |u_i - s| \to 0$ it follows by uniform equicontinuity that $\{ \xx(s_i) \}$ converges to $z$.
 In particular, $z$ is uniquely defined as the limit.

 In particular, if $I$ is a singleton, $\yy = \{ (s,z) \}$ with $z$ the limit of the sequence $\{ \xx(s_i) \}$.

When $I$ is nontrivial, we must show that  $z = \tilde \yy(s)$.

Choose for $\ep > 0$ an $\ep$ modulus of uniform equicontinuity $\d > 0$. There exists $u \in (r,s)$ with $s - \d/2 < u < s$.
Since $\tilde \yy$ is a $\Phi$ solution path, we have $d(\tilde \yy(u),\tilde \yy(s)) < \ep$. Now choose $N > N_u$ so that
$i \ge N$ implies $|s_i - s| < \d/2$ and so $|s_i - u| < \d$. Hence, $d(\xx_i(s_i),\xx_i(u)) < \ep$. Letting
$i$ tend to infinity, we obtain $d(z,\yy(u)) \le \ep$. Since $\yy(u) = \tilde \yy(u)$, we see that $d(z,\tilde \yy(s)) < 2 \ep$.
As  $\ep > 0$ was arbitrary, it follows that $z = \tilde \yy(s)$.

 With a similar argument for the infimum, we see that $\yy$ is a function from $I$ to $X$ with $\yy = \tilde \yy$.

 That is, the solution path $\tilde \yy$ on $I$ is the same as the limit $\yy$.

 \end{proof}

 We immediately obtain

 \begin{cor}\label{corapp07} For $\Phi$ a semiflow relation on $X$, the collection of $\Phi$ solution paths on
 intervals $I \in \I \subset 2^{[0,t]}$ is a closed subset of the set $2^{[0,t] \times X}$ of closed relations from
 $[0,t]$ to $X$. \end{cor}
 \vspace{1cm}

 Now let $\H =(\Phi_C,G)$ be a hybrid dynamical system on $X$. For $E \in \E(t,n) \subset 2^{[(0,0),(t,n)]}$ an $\H$ solution path
 is a function $\xx : E \to X$ such that
 \begin{align}\label{eqapp11}\begin {split}
 s \mapsto \xx(s,m) \ \  \text{is a} \ \ \Phi_C \ \ &\text{solution path on} \ \ I_m \quad \text{for} \ \ m = 0,1 \dots, n, \\
 (\xx(j_m(E),m-1), &\xx(j_m(E),m)) \in G \quad \text{for} \ \ m = 1,2 \dots, n.
 \end{split}\end{align}
 Recall that $j_m = \inf I_m$ is equal to $\sup I_{m-1}$ for $m = 1,2 \dots, n.$

We can regard such a map $\xx$ as an element of $2^{[(0,0),(t,n)] \times X}$, the space of closed relations from $[(0,0),(t,n)]$ to $X$.
 For $\xx \in 2^{[(0,0),(t,n)] \times X}$, i.e. a closed subset of $[0,t] \times [0,n] \times X \subset \R_+ \times \Z_+ \times X$ we
 define $\xx^m \in 2^{[0,t] \times X}$ for $m = 0, 1,\dots,n$ by
 \begin{equation}\label{eqapp12}
 (s,y) \in \xx^m \quad \Longleftrightarrow \quad ((s,m),y) \in \xx.
 \end{equation}
 Thus, when $\xx$ is an $\H$ solution path, $\xx^m$ is the $\Phi_C$ solution path on $I_m(E)$ given by $s \mapsto \xx(s,m)$.

 \begin{theo}\label{theoapp08} For $\H = (\Phi_C,G)$ a hybrid dynamical system on $X$, let $\{ \xx_i : E_i \to X \}$ be a sequence of
$\H$ solution paths defined on hybrid time intervals $E_i \in \E(t,n) \subset 2^{[(0,0),(t,n)]}$. Assume that in $2^{[(0,0),(t,n)] \times X}$
the sequence $\{ \xx_i \}$ converges to $\yy \in 2^{[(0,0),(t,n)] \times X}$.

\begin{itemize}
\item[(i)] The sequence $\{ E_i \}$ converges to some $E \in \E(t,n)$.

\item[(ii)] $\yy$ is a $\H$ solution path defined on the hybrid time interval $E$.
\end{itemize}
\end{theo}

\begin{proof} As in the proof of Theorem \ref{theoapp06} (i), by projecting, we see that the sequence $\{ E_i \}$ in $\E(t,n)$
converges to some $E \in 2^{[(0,0),(t,n)]}$ with $E$ lying in $\E(t,n)$ because, by Theorem \ref{theoapp05}, the latter subset is
closed. As in Theorem \ref{theoapp06}, $E$ is the domain of the closed relation $\yy$.

As in (\ref{eqapp09}), the discreteness of the interval $[0,n] \subset \Z_+$ implies that from the convergence of $\{ \xx_i \}$ to $\yy$
we obtain for each $m = 0,1, \dots n$ the convergence of $\{ \xx^m_i \}$ to $\yy^m$.

It now follows from Theorem \ref{theoapp06}(ii) that $\yy^m$ is a $\Phi_C$ solution path for each $m = 0,1, \dots, n$.

In addition, Theorem \ref{theoapp06}(iii) implies that $\{ \xx^m_i(j_m(E_i)) \}$ converges to  $\yy^m(j_m(E))$ and
$\{ \xx^{m-1}_i(j_m(E_i)) \}$ converges to $\yy^{m-1}(j_m(E))$ for each $m = 1,2,\dots n$.

Because each $(\xx^{m-1}_i(j_m(E_i)),\xx^{m}_i(j_m(E_i))) \in G$ and $G$ is closed, it follows that
  $(\yy^{m-1}(j_m(E)),\yy^{m}(j_m(E))) \in G$ for each $m = 1,2,\dots n$.

  Thus, $\yy : E \to X$ is an $\H$ solution path.

 \end{proof}

 As before we obtain

  \begin{cor}\label{corapp09} For $\H = (\Phi_C,G)$ a hybrid dynamical system on $X$, the collection of $\H$ solution paths on
  hybrid time  intervals $E \in \E(t,n) \subset 2^{[(0,0),(t,n)]}$ is a closed subset of the set $ 2^{[(0,0),(t,n)] \times X}$ of closed relations from
 $[(0,0),(t,n)]$ to $X$. \end{cor}

\vspace{3cm}

%\newpage

\bibliographystyle{amsplain}

\begin{thebibliography}{10}



\bibitem{A93}
E. Akin,  {\bfseries The general topology of dynamical systems}, Grad. Studies in
Mathematics v. 01,(1993) Amer. Math. Soc.; Second printing 1996. \vspace{.25cm}
\vspace{.5cm}

%\bibitem{A97}
%E. Akin,  {\bfseries Recurrence in topological Dynamics: Furstenberg families and Ellis actions}, (1997)Plenum Press, New York.
%\vspace{.5cm}
%
% \bibitem{A04}
%E. Akin,   \emph{ Lectures on Cantor and Mycielski Sets for dynamical systems} in {\bfseries Topological Dynamics and Applications},
% (ed. I. Assani), Amer. Math. Soc. Contemp. Math.(2004) {\bf 356}: 21-79.
% \vspace{.5cm}

\bibitem{Au}
J. Auslander,  {\bfseries Minimal flows and their extensions},(1988) North Holland \vspace{.25cm}
\vspace{.5cm}

 \bibitem{BEGK}
I. Banic, G. Erceg, R. Gril Rodina and J. Kennedy, \emph{Minimal dynamical systems with closed relations}, (2022) ArXiv:2205.02907v1.
 \vspace{.5cm}

\bibitem{BM}
B. Batko and M. Mrozek, \emph{Weak index pairs and the Conley index for discrete multivalued
dynamical systems.} SIAM J. Appl. Dyn. Syst., (2016) {\bf 15(2)}: 1143-1162.
 \vspace{.5cm}

 \bibitem{B}
B. Batko, \emph{Weak index pairs and the Conley index for discrete multivalued dynamical
systems. Part II: Properties of the index.} SIAM J. Appl. Dyn. Syst.,(2017) {\bf 16(3)}: 1587-1617.
 \vspace{.5cm}



\bibitem{BK}I. U. Bronstein and A. Ya. Kopanskii, \emph{Chain recurrence in dynamical systems, without uniqueness},
Nonlinear Anal. (1988) {\bf 12}: 147-154.
\vspace{.5cm}

\bibitem{C}
C. C. Conley,  \textbf{Isolated invariant sets and the Morse index}, CBMS Regional
Conference Series in Mathematics, vol. 38,  (1978) Amer. Math. Soc., Providence, R.I.
\vspace{.5cm}


\bibitem{E}
R. Engelking  {\bfseries General topology}, (1989) Heldermann Verlag, Berlin.
\vspace{.5cm}

\bibitem{F}
J. Franks, \emph{A variation on the Poincar\'e-Birkhoff theorem}. Contemp. Math, (1988) {\bf 81}:111-117,
\vspace{.5cm}

\bibitem{GST}
R. Goebel, R. G. Sanfelice,  and A. R. Teel,  \textbf{Hybrid dynamical systems - modeling, stability and robustness}, (2012) Princeton University Press.
\vspace{.5cm}

\bibitem{KM}
T. Kacynski and M. Mrozek, \emph{Conley index for discrete multi-valued dynamical systems}, Topol. and its Appl., (1995) {\bf 65}: 83-96.
\vspace{.5cm}

\bibitem{K}
J. Kelley,  {\bfseries General topology}, (1955) Van Nostrand, New York.
\vspace{.5cm}

\bibitem{KST}
S. Kolyada,L' Snoha and S. Trofimchuk, \emph{Noninvertible minimal maps}, Fundamenta Math.,(2001) {\bf 168(2)}: 141-163.
\vspace{.5cm}

\bibitem{M}
R. P. McGehee, \emph{Attractors for closed relations on compact Hausdorff spaces}, Indiana U. Math. J., (1992){\bf 41}:  1165-1209.
\vspace{.5cm}

\bibitem{MW}
R. P. McGehee and T. Wiandt, \emph{Conley decomposition for closed relations}, J. of Difference Equations and Appls.,(2006), {\bf 12(1)}: 1-47.
\vspace{.5cm}

\bibitem{T}
C. Thieme, \emph{Conley index theory and the attractor-repller decomposition for differential inclusions}, (2020) ArXiv:2009.00696v2.



\end{thebibliography}

\printindex

\end{document}